\documentclass[12pt]{amsart}
\usepackage{latexsym, amscd, amsfonts, eucal, mathrsfs, amsmath, amssymb, amsthm, xypic,xr, stmaryrd, amsxtra, MnSymbol}
\newcommand{\coch}{\mathrm{Chain}(\mathbb{Z})}
\newcommand{\cochtfr}{\mathrm{Chain}(\mathbb{Z})^{\mathrm{tf}}}
\newcommand{\dhp}{\widehat{D(\Z_p)}}
\newcommand{\lotimes}{\stackrel{L}{\otimes}}

\newcommand{\cochfr}{\mathrm{Chain}(\Z)^{\mathrm{free}}}
\newcommand{\spec}{\mathrm{Spec}}
\renewcommand{\phi}{\varphi}

\newcommand{\Ab}{\mathrm{Ab}}

\usepackage{url}
\usepackage{verbatim}

\usepackage{hyperref}

\newtheorem{theorem}{Theorem}[subsection]
\newtheorem{lemma}[theorem]{Lemma}
\newtheorem{proposition}[theorem]{Proposition}

\newtheorem{corollary}[theorem]{Corollary}

\renewcommand{\mathbb}{\mathbf}
\theoremstyle{definition}

\newtheorem{definition}[theorem]{Definition}
\newtheorem{construction}[theorem]{Construction}
\newtheorem{example}[theorem]{Example}

\newtheorem{notation}[theorem]{Notation}

\newtheorem{warning}[theorem]{Warning}
\newtheorem{remark}[theorem]{Remark}
\newtheorem{variant}[theorem]{Variant}

\newcommand{\etale}{\'{e}tale}
\newcommand{\sm}{\mathrm{small}}

\def\Z{\mathbf{Z}}
\def\C{\mathbf{C}}
\def\F{\mathbf{F}}

\DeclareMathOperator{\VPro}{ \mathbf{VPC} }
\DeclareMathOperator{\Res}{Res}
\DeclareMathOperator{\sheafF}{{\mathscr F}}

\DeclareMathOperator{\sheafO}{{\mathscr O}}

\DeclareMathOperator{\RGamma}{R\Gamma}
\DeclareMathOperator{\calU}{\mathcal U}
\DeclareMathOperator{\calC}{{\mathcal C} }
\DeclareMathOperator{\calN}{{\mathcal{N}}}
\DeclareMathOperator{\aff}{aff}
\DeclareMathOperator{\calD}{{\mathcal D}}

\DeclareMathOperator{\op}{op}
\DeclareMathOperator{\HH}{H}
\DeclareMathOperator{\dR}{dR}

\DeclareMathOperator{\WittScript}{ \mathcal{W} }
\DeclareMathOperator{\WOmega}{ \WittScript \Omega }
\newcommand{\WrOmega}[1]{ \WittScript_{#1} \Omega } 
\DeclareMathOperator{\WSaturate}{\WittScript \Saturate}

\DeclareMathOperator{\crys}{crys}
\DeclareMathOperator{\FrobComp}{ \mathbf{DC} }
\DeclareMathOperator{\FrobAlg}{ \mathbf{DA} }
\DeclareMathOperator{\sat}{sat}
\DeclareMathOperator{\str}{str}
\DeclareMathOperator{\Cart}{Cart}
\DeclareMathOperator{\FrobCompSat}{ \mathbf{DC}_{\sat} }
\DeclareMathOperator{\FrobAlgSat}{ \mathbf{DA}_{\sat} }
\DeclareMathOperator{\FrobCompComplete}{ \mathbf{DC}_{\str} }
\DeclareMathOperator{\FrobAlgComplete}{ \mathbf{DA}_{\str} }

\DeclareMathOperator{\red}{red}

\DeclareMathOperator{\Ext}{Ext}
\DeclareMathOperator{\Saturate}{Sat}
\DeclareMathOperator{\id}{id}
\DeclareMathOperator{\im}{im}

\DeclareMathOperator{\Spf}{Spf}
\DeclareMathOperator{\Spec}{Spec}

\DeclareMathOperator{\Fun}{Fun}
\DeclareMathOperator{\Mod}{Mod}

\DeclareMathOperator{\Hom}{Hom}

\DeclareMathOperator{\perf}{perf}

\DeclareMathOperator{\Tot}{Tot}
\DeclareMathOperator{\supp}{supp}
\DeclareMathOperator{\CAlg}{CAlg}
\DeclareMathOperator{\Vect}{Vect}

\DeclareMathOperator{\coker}{coker}

\newcommand{\et}{\'{e}t}

\newcommand{\mathet}{\text{\et}}
\newcommand{\SCRf}{\mathrm{SCR}_{\mathbb{F}_p}}

\newcommand{\gr}{\mathrm{gr}}

\newcommand{\Adjoint}[4]{\xymatrix@1{#2 \ar@<.4ex>[r]^-{#1} & #3 \ar@<.4ex>[l]^-{#4}}}

\title{Revisiting the de~Rham--Witt complex}

\author{Bhargav Bhatt}
\author{Jacob Lurie}
\author{Akhil Mathew}

\begin{document}
\begin{abstract}
The goal of this paper is to offer a new construction of the de~Rham--Witt
complex of a smooth variety over a perfect field of characteristic $p > 0$.

We introduce a category of cochain complexes
which are equipped with an endomorphism $F$ of underlying graded abelian groups satisfying $dF = pFd$, whose homological algebra we study in detail. 
To any such object satisfying an abstract analog of the Cartier isomorphism, an
elementary homological process  associates a generalization of the de~Rham--Witt
construction. Abstractly, the homological algebra can be viewed as a
calculation of the fixed points of the Berthelot--Ogus operator $L \eta_p$ on
the $p$-complete derived category. 
We give various applications of this approach, including a simplification of the
crystalline comparison for the $A \Omega$-cohomology theory introduced in \cite{BMS}.
\end{abstract}

\maketitle

\setcounter{tocdepth}{2}

\tableofcontents

\newcommand{\FrobTow}{\mathbf{TD}}
\newcommand{\W}{\mathcal{W}}
\newpage
\section{Introduction}

\subsection{The de~Rham--Witt complex}
Let $X$ be a smooth algebraic variety defined over a field $k$. The {\it algebraic de~Rham cohomology} 
$\mathrm{H}^{\ast}_{\dR}(X/k)$ is defined as the hypercohomology of the de~Rham complex
$$ \Omega^{0}_{X/k} \xrightarrow{d} \Omega^1_{X/k} \xrightarrow{ d}
\Omega^2_{X/k} \xrightarrow{d} \cdots.$$
When $k = \C$ is the field of complex numbers, Grothendieck
\cite{GrothendieckdR} showed that the algebraic de~Rham cohomology
$\mathrm{H}^{\ast}_{\dR}(X/k)$ is isomorphic to the usual de~Rham cohomology of
the underlying complex manifold $X(\C)$ (and therefore also to the singular
cohomology with complex coefficients of the topological space $X(\C)$). However,
over fields of characteristic $p > 0$, algebraic de~Rham cohomology is a less
satisfactory invariant.  This is due in part to the fact that it takes values in
the category of vector spaces over $k$, and therefore consists entirely of
$p$-torsion. To address this point, Berthelot \cite{Berthelot74} and
Grothendieck \cite{Grothendieck68} introduced the theory of {\it crystalline
cohomology}. Suppose (for the rest of the introduction) that $k$ is perfect and
characteristic $p > 0$. To every smooth algebraic variety over $k$, the theory
of crystalline cohomology associates
cohomology groups $\mathrm{H}^{\ast}_{\crys}( X )$ which are modules over the
ring of $p$-typical Witt vectors $W(k)$, and can be regarded as integral versions of the de~Rham cohomology groups
$\mathrm{H}^{\ast}_{\dR}(X/k)$. 

The crystalline cohomology groups of a variety $X$ were originally defined as the cohomology groups of the structure sheaf of a certain ringed topos, called the 
{\it crystalline topos} (of $X$). However, Bloch \cite{Bloch} (in the case of
small dimension) and Deligne-Illusie
\cite{illusie} later gave an alternative description of crystalline cohomology, which is closer in spirit to the definition of algebraic de~Rham cohomology.
More precisely, they showed that the crystalline cohomology of a smooth variety
$X$ over $k$  can be realized as the hypercohomology of a complex of sheaves
$$ W\Omega^{0}_{X} \xrightarrow{d} W\Omega^1_{X} \xrightarrow{ d}
W\Omega^2_{X} \xrightarrow{d} \cdots$$
called the {\it de~Rham--Witt complex of $X$}. The complex $W \Omega^{\ast}_{X}$
is related to the usual de~Rham complex by the following pair of results:

\begin{theorem}\label{maintheoA}
Let $X$ be a smooth algebraic variety defined over $k$. Then
there is a canonical map of cochain complexes of sheaves $W \Omega^{\ast}_{X} \rightarrow \Omega^{\ast}_{X}$, which
induces a quasi-isomorphism $W\Omega^{\ast}_{X} / p W\Omega^{\ast}_{X}
\rightarrow \Omega^{\ast}_{X/k}$.
\end{theorem}

\begin{theorem}\label{maintheoB}
Let $\mathfrak{X}$ be a smooth formal scheme over $\Spf(W(k))$ with central fiber $X = \Spec(k) \times_{ \Spf(W(k))} \mathfrak{X}$. 
Moreover, suppose that the Frobenius map $\varphi_X\colon X \rightarrow X$ extends to a map of formal schemes $\varphi_{ \mathfrak{X} }\colon \mathfrak{X} \rightarrow \mathfrak{X}$. 
Then there is a natural quasi-isomorphism $\Omega^{\ast}_{ \mathfrak{X}/W(k) }
\rightarrow W\Omega^{\ast}_{ X}$ of cochain complexes of abelian sheaves on the
underlying topological space of $X$ (which is the same as the underlying
topological space of $\mathfrak{X}$). Here $\Omega^{\ast}_{ \mathfrak{X}/W(k)}$ denotes the de~Rham complex of the formal scheme $\mathfrak{X}$ relative to $W(k)$. 
\end{theorem}

\begin{warning}
In the situation of Theorem \ref{maintheoB}, the map of cochain complexes
$\Omega^{\ast}_{ \mathfrak{X}/W(k)} \rightarrow W \Omega^{\ast}_{ X}$ depends on the
choice of the map
$\varphi_{ \mathfrak{X} }$, though the induced map on cohomology groups (or even in the derived category) is independent of that choice.
\end{warning}

The proofs of Theorems \ref{maintheoA} and \ref{maintheoB} which appear in \cite{illusie} depend on some rather laborious calculations. Our aim in this paper is to present
an alternate construction of the complex $W \Omega^{\ast}_{X}$ (which agrees with the construction of Deligne and Illusie for smooth varieties; see Theorem \ref{maintheoC})
which can be used to prove Theorems \ref{maintheoA} and \ref{maintheoB} in an essentially calculation-free way: the only input we will need is the Cartier isomorphism
(which we recall as Theorem~\ref{theo71}) and some elementary homological algebra.

The de~Rham--Witt complexes constructed in this paper agree with those of \cite{illusie} in the smooth case, but differ in general. For this reason, we distinguish the two constructions notationally: we write $W\Omega^*_X$ for the construction from \cite{illusie} and refer to it as the {\em classical de~Rham--Witt complex}, while we write $\W\Omega^*_X$ for our construction and refer to it as the {\em saturated de~Rham--Witt complex}.

\subsection{Overview of the Construction}
Assume for simplicity that $k$ is a perfect field  of characteristic $p>0$ and that $X = \Spec(R)$ is the spectrum of a smooth $k$-algebra $R$. In this case,
we can identify $\mathrm{H}^{\ast}_{\dR}(X/k)$ with the cohomology of a cochain complex of $k$-vector spaces
$( \Omega^{\ast}_{R}, d)  = (\Omega^{\ast}_{R/k}, d)$. This complex admits both concrete and abstract descriptions:
\begin{itemize}
\item[$(a)$] For each $n \geq 0$, the $R$-module $\Omega^{n}_{R}$ is given by the $n$th exterior power of the module
of K\"{a}hler differentials $\Omega^1_{R/k}$, which is a projective $R$-module of finite rank. The de~Rham differential 
$d\colon \Omega^{n}_{R} \rightarrow \Omega^{n+1}_{R}$ is then given concretely by the formula
$$ d( a_0 (da_1 \wedge da_2 \wedge \cdots \wedge da_n) ) = da_0 \wedge da_1 \wedge \cdots \wedge da_n$$
for $a_0, a_1, \ldots, a_n \in R$.

\item[$(b)$] The de~Rham complex $(\Omega^{\ast}_{R}, d)$ has the structure of a commutative differential graded algebra over $\F_p$ (see \S \ref{frobalgsec}). Moreover, it is characterized by the following universal property: if $(A^{\ast}, d)$ is any commutative differential graded algebra over $\F_p$, then every ring homomorphism $R \rightarrow A^{0}$ extends uniquely to a map of commutative differential graded algebras
$\Omega^{\ast}_{R} \rightarrow A^{\ast}$.
\end{itemize}

Both of these descriptions have analogues for our saturated de~Rham--Witt complex $\W \Omega_{R}^{\ast}$.
Let us begin with $(a)$. Fix an isomorphism $R \simeq \widetilde{R} / p \widetilde{R}$, where
$\widetilde{R}$ is a $W(k)$-algebra which is $p$-adically complete and $p$-torsion-free (such an isomorphism always exists, by virtue
our assumption that $R$ is smooth over $k$). Let $\widehat{ \Omega }^{\ast}_{\widetilde{R}}$ denote the $p$-adic completion 
of the de~Rham complex of $\widetilde{R}$ (or, equivalently, of the de~Rham complex of $\widetilde{R}$ relative to $W(k)$). 
Then $\widehat{ \Omega}^{\ast}_{ \widetilde{R} }$ is a commutative differential graded algebra over $W(k)$, and we have a canonical isomorphism
$\Omega^{\ast}_{R} \simeq \widehat{ \Omega}^{\ast}_{ \widetilde{R} } / p \widehat{ \Omega}^{\ast}_{ \widetilde{R} }$. Beware that
the cochain complex $\widehat{\Omega}_{\widetilde{R}}^{\ast}$ depends on the choice of $\widetilde{R}$, and does not depend functorially
on $R$.

We now make a further choice: a ring homomorphism $\varphi\colon \widetilde{R} \rightarrow \widetilde{R}$ which is compatible with the Frobenius endomorphism of $R$
(so that $\varphi( x) \equiv x^{p} \pmod{p}$); such a choice exists by
the smoothness of $R$ over $k$. Then $\varphi^*$ is divisible by $p^n$ on $\Omega^n_{\widetilde{R}}$, so we obtain a homomorphism of graded algebras
\begin{equation}
\label{FIntro}
F\colon\widehat{\Omega}_{\widetilde{R}}^{\ast} \to \widehat{\Omega}_{\widetilde{R}}^{\ast}  
\end{equation}
by the formula
\begin{align*}
F( a_0 da_1 \wedge da_2 \wedge \cdots \wedge da_n ) &:= ``\frac{1}{p^n} \cdot \varphi^* (a_0 da_1 \wedge da_2 \wedge \cdots \wedge da_n)" \\
										     &:= \varphi(a_0) \frac{ d \varphi(a_1)}{p} \wedge \cdots \wedge \frac{ d \varphi(a_n)}{p}.
\end{align*}
The homomorphism $F$ is not a map of {\em differential} graded algebras because it does not commute with the differential $d$: instead, 
we have the identity $d F(\omega) = p F( d \omega )$ for each $\omega \in \widehat{ \Omega}^{\ast}_{R}$. 

Motivated by the preceding observation, we make the following definition, depending on a prime $p$ which is implicitly fixed throughout the sequel. 

\begin{definition}[See  Definition~\ref{def1}]
A {\it Dieudonn\'{e} complex} is a cochain complex of abelian groups $(M^{\ast},d)$
equipped with a map of graded abelian groups $F\colon M^{\ast} \rightarrow
M^{\ast}$ satisfying $dF = p Fd$. 
\end{definition} 
We devote \S \ref{section2} of this paper to a general study of the homological algebra of Dieudonn\'{e} complexes.

For any Dieudonn\'{e} complex $M^{\ast} = (M^{\ast}, d, F)$, the Frobenius map $F$ determines a map of cochain complexes
\begin{equation}\label{Cartier.reference} ( M^{\ast} / p M^{\ast}, 0) \rightarrow (M^{\ast} / p M^{\ast}, d ). \end{equation}
In the case where $M^{\ast} = \widehat{\Omega}^{\ast}_{ \widetilde{R} }$ is the completed de~Rham complex of
$\widetilde{R}$, this map turns out to be a quasi-isomorphism: that is, on
cohomology groups, it induces a (Frobenius-semilinear)
isomorphism from the ring of differential forms $\Omega^{\ast}_{R}$ to the de~Rham
cohomology $\mathrm{H}^{\ast}_{\dR}( \Spec(R)/k )$. This is (the inverse of) the classical
\emph{Cartier isomorphism}. Motivated by this observation, we introduce the following:

\begin{definition}[See Definition~\ref{def:Cartiertype}] 
A Dieudonn\'{e} complex $M^{\ast}$ is {\it of Cartier type} if it is $p$-torsion-free and the map $(\ref{Cartier.reference})$ is a quasi-isomorphism.
\end{definition} 

Given an arbitrary Dieudonn\'{e} complex $M^{\ast}$, we will construct a new (generally much larger)
Dieudonn\'{e} complex $\WSaturate( M^{\ast} )$, which we call the {\it completed saturation} of $M^{\ast}$. 
This is a combination of two simpler constructions: a saturation process which
produces a Dieudonn\'e complex with an additional \emph{Verschiebung} operator
$V$, and a second procedure which enforces (loosely speaking) completeness with respect to $V$.
We show that the completed saturation construction (while difficult to control
in general) is especially well-behaved
for Dieudonn\'e complexes of Cartier type. 

\begin{theorem}[See Theorem~\ref{theo60} and Proposition~\ref{prop18}] 
Let $M^{\ast}$ be a Dieudonn\'{e} complex. If $M^{\ast}$ is of Cartier type, the map $M^{\ast} \to \W
\mathrm{Sat}(M^{\ast})$ induces an isomorphism 
on cohomology with modulo $p$ coefficients (that is, it becomes a
quasi-isomorphism  after reducing modulo $p$). 
\end{theorem} 

Our saturated de~Rham--Witt complex $\W \Omega_{R}^{\ast}$ can be defined
as the completed saturation of $\widehat{ \Omega}^{\ast}_{ \widetilde{R} }$: that is,
we have a canonical isomorphism $\W \Omega_R^{\ast} \simeq  \W \mathrm{Sat}(  \widehat{ \Omega}^{\ast}_{
\widetilde{R} } )$. We will see that $\W \Omega_R^{\ast}$ is equipped
with an algebra structure which is (in a suitable sense) compatible with its structure as a Dieudonn\'{e} complex,
exhibiting it as an example of the following:

\begin{definition}[See Definition~\ref{SDA} ]
A \emph{strict Dieudonn\'e algebra} is a Dieudonn\'e complex $(A^{\ast}, d, F)$
with the structure of a differential graded algebra such that: 
\begin{enumerate}
\item $A^i = 0$ for $i < 0$.  
\item $F \colon  A^{\ast} \to A^{\ast}$ is a homomorphism of graded rings. 
\item $A^{\ast}$ is $p$-torsion-free. 
\item The map $F \colon A^i \to A^i$ is injective and has image given by those $x \in
A^i$ such that $dx$ is divisible by $p$. 
\item For $x \in A^0$, we have $Fx \equiv x^p \ ( \mathrm{mod}  \ p)$. 
\item $A^{\ast}$ satisfies a certain completeness condition with respect
to the Verschiebung (see Definition~\ref{def22}).  
\end{enumerate}
\end{definition}

The Frobenius map on $R$ induces a ring homomorphism 
\[ R \to \HH^0(\W\Omega^*_R/p),\]
which is analogous to observing that the de~Rham differential on $\Omega^*_R$ is linear over the subring of $p$-th powers in $R$. In the setting of strict Dieudonn\'{e} algebras, we can use this map to characterize $\W\Omega^*_R$ by a universal property in the spirit of $(b)$ (see Definition~\ref{def80} and 
Corollary~\ref{cor70}): 

\begin{theorem}[The universal property of $\W \Omega_R^{\ast}$] 
If $A^{\ast}$ is any strict Dieudonn\'{e} algebra, then every ring homomorphism
$R \rightarrow \HH^0(A^*/p A^*)$ can be extended uniquely to a map of strict Dieudonn\'{e} algebras
$\W \Omega_{R}^{\ast} \rightarrow A^{\ast}$.
\end{theorem} 

Consequently, the saturated de~Rham--Witt complex $\W\Omega^*_R$
depends functorially on $R$, and is independent of the choice of the $W(k)$-algebra $\widetilde{R}$ lifting $R$ or the map
$\varphi\colon \widetilde{R} \rightarrow \widetilde{R}$ lifting the Frobenius.

\begin{remark}[Extension to all $\F_p$-algebras]
\label{extensionFpalgebras}
Using the characterization of $\W \Omega^{\ast}_{R}$ by the universal property mentioned above, we can extend the definition of
$\W \Omega^{\ast}_{R}$ to the case where $R$ is an arbitrary $\F_p$-algebra, not necessarily smooth over
a perfect field $k$. In the general case, we will see that the saturated de~Rham--Witt complex $\W\Omega^*_R$ differs from the classical one. For instance, $\W\Omega^*_R$ only depends on the reduction $R^{\red}$ of $R$ (see Lemma~\ref{lem27}), while the analogous statement fails for the classical de~Rham--Witt complex $W\Omega^*_R$. In \S \ref{sec:Seminormal}, we will prove a stronger result: $\W \Omega^*_{R}$ depends only on the seminormalization $R^{\mathrm{sn}}$ of $R$. Moreover, we can identify $R^{\mathrm{sn}}$ with the $0$th cohomology group
$\mathrm{H}^0(\W \Omega^*_R/p)$. We are not aware of a similar description for the higher cohomology groups of $\W\Omega^*_R/p$. 
\end{remark}

\begin{remark}[Frobenius untwisted characterization]
It is possible to give a formulation of the universal property of
$\W\Omega^*_R$ which is closer in spirit to $(b)$. Any saturated Dieudonn\'e
complex $(M^*,d,F)$ is equipped with a {\it Verschiebung map} $V\colon M^* \to M^*$,
which is characterized by the formula $FV = p$ (Proposition~\ref{Vexists}). If $(M^*,d,F)$
is saturated and $M^{i} = 0$ for $i < 0$, then $F$ induces an isomorphism
$M^0/VM^0 \xrightarrow{\sim} \HH^0(M^*/pM^*)$ (Proposition~\ref{prop11}). It follows that the
Frobenius map $R \to \HH^{0}( \W \Omega_R^{\ast} / p )$ factors uniquely as a composition $R \rightarrow \W\Omega^{0}_{R} / V \W \Omega^{0}_{R} \xrightarrow{F} \HH^{0}( \W \Omega^{0}_{R} / p )$. The universal property $(\ast)$ then translates to the following:

\begin{itemize}
\item  If $A^{\ast}$ is a strict Dieudonn\'{e} algebra, then every ring homomorphism $R \rightarrow A^0/VA^0$ can be extended uniquely to a map of strict Dieudonn\'{e} algebras $\W \Omega_{R}^{\ast} \rightarrow A^{\ast}$.
\end{itemize}

Here we can think of the map $\gamma\colon  R \to \W \Omega_{R}^{0} / V \W \Omega^{0}_{R}$ as an analogue of the identification $R \simeq W(R) / V W(R)$ for the ring of Witt vectors $W(R)$.
Beware, however, that the map $\gamma$ need not be an isomorphism when $R$ is
not regular: in fact (as mentioned in Remark~\ref{extensionFpalgebras} above), it exhibits $W \Omega_{R}^{0} / V W \Omega_{R}^{0}$ as the seminormalization $R^{\mathrm{sn}}$ (Theorem~\ref{dRWsn2}).
\end{remark}

\begin{remark}[Differences with the classical theory]
Let us briefly contrast our definition of the saturated de~Rham--Witt complex $\W \Omega^{\ast}_{R}$ with the classical
de~Rham--Witt complex $W\Omega^{\ast}_{R}$ of \cite{illusie}. Recall that $W \Omega^{\ast}_{R}$ is defined as the inverse limit
of a tower of differential graded algebras $\{ W_r \Omega_{R}^{\ast} \}_{r \geq 0}$, which is defined to be an initial
object of the category of \emph{$V$-pro-complexes} (see \S \ref{compareclass} for a review of the relevant definitions). Our approach differs in several respects:

\begin{itemize}
\item Our saturated de~Rham--Witt complex can also be described as the inverse
limit of a tower $\{ \W_r \Omega_{R}^{\ast} \}_{r \geq 0}$, but for most purposes we find it more convenient to work directly
with the inverse limit $\W \Omega_{R}^{\ast} = \varprojlim_{r} \W_r \Omega_{R}^{\ast}$. This allows us to sidestep certain complications:
for example, the limit $\W \Omega_{R}^{\ast}$ is $p$-torsion free, while the
individual terms $\W_r \Omega_{R}^{\ast}$ are essentially never $p$-torsion-free
(or even flat over $\Z / p^{r} \Z$). 

\item In the construction of the classical de~Rham--Witt pro-complex $\{ W_r \Omega_{R}^{\ast} \}_{r \geq 0}$, the Verschiebung operator $V$
is fundamental and the Frobenius operator $F$ plays an ancillary role. In our presentation, the opposite is true:
we regard the Frobenius operator as an essential part of the structure $\W \Omega_{R}^{\ast}$, while the Verschiebung operator is determined by requiring $FV = p = VF$.

\item The notion of a $V$-pro-complex introduced in \cite{illusie} is essentially non-linear: the axioms make sense only for differential graded algebras, rather than for general cochain complexes. By contrast, our Dieudonn\'{e} complexes form an additive category $\FrobComp$,
and our Dieudonn\'{e} algebras can be viewed as commutative algebras in the
category $\FrobComp$ (satisfying a mild additional condition).
\end{itemize}
\end{remark}

\begin{remark}
The theory of the de~Rham--Witt complex has taken many forms since its original
introduction in \cite{illusie}, and our approach is restricted to the classical
case of algebras over $\mathbb{F}_p$. In particular, we do not discuss the
relative de~Rham--Witt complexes of Langer--Zink \cite{LZ}, or the absolute de
Rham--Witt complexes considered by Hesselholt--Madsen \cite{HMabsolute} in
relation to topological Hochschild homology. 
\end{remark}

\subsection{Motivation via the $L\eta$ Functor}

In this section, we briefly discuss the original (and somewhat more highbrow) motivation that led us to the notion of strict Dieudonn\'e complexes, and some related results. To avoid a proliferation of Frobenius twists,  we work over the field $\F_p$ (instead of an arbitrary perfect field $k$ of characteristic $p$) for the rest of this section.

Let $R$ be a smooth $\F_p$-algebra and let $\Omega^{\ast}_{R} = \Omega^{\ast}_{R/\F_p}$ denote the algebraic de~Rham complex of $R$. Then $\Omega^{\ast}_{R}$ is a cochain complex of vector spaces over $\F_p$, and
can therefore be viewed as an object of the derived category $D( \F_p )$. The
Berthelot--Grothendieck theory of crystalline cohomology provides an object
$\RGamma_{\crys}( \Spec(R) )$ of the derived category $D(\Z)$, which plays the role of a ``characteristic zero lift'' of $\Omega^{\ast}_{R}$ in the sense that there is
a canonical isomorphism
$$\alpha\colon  \F_p \otimes^{L}_{\Z} \RGamma_{\crys}( \Spec(R) ) \rightarrow \Omega^{\ast}_{R}$$
in the category $D( \F_p )$. Explicitly, the object $\RGamma_{\crys}( \Spec(R) ) \in D( \Z )$ can be realized as the global sections of $\mathscr{I}^{\ast}$, where $\mathscr{I}^{\ast}$ is any injective resolution of the structure sheaf on the crystalline site of $\Spec(R)$. Beware that, when regarded as a cochain complex, $\RGamma_{\crys}( \Spec(R) )$ depends on the choice
of the injective resolution $\mathscr{I}^{\ast}$: different choices of injective resolutions  generally yield nonisomorphic cochain complexes, which are nevertheless (canonically) isomorphic
as objects of the derived category $D( \Z )$. 

From the above perspective, it is somewhat surprising that the object $\RGamma_{\crys}( \Spec(R) ) \in D(\Z)$ admits a canonical representative at the level of cochain complexes (given by the de~Rham--Witt complex $W \Omega^{\ast}_{R}$). One of our goals in this paper is to give an explanation for this phenomenon: the cochain complex $W \Omega^{\ast}_{R}$ can actually be functorially recovered from
its image $\RGamma_{\crys}( \Spec(R) ) \in D( \Z )$, together with an additional structure which can also be formulated intrinsically at the level of the derived category. 

To describe this additional structure, recall that the derived category of abelian groups $D(\Z)$ is equipped with an additive (but non-exact) endofunctor
$L \eta_{p}\colon  D(\mathbb{Z}) \to D(\mathbb{Z})$ (see Construction~\ref{etaconstruction} for the definition). This operation was discovered 
in early work on crystalline cohomology stemming from Mazur's resolution \cite{Mazur73} of a conjecture of Katz, giving a relationship between the Newton and Hodge
polygons associated to a smooth projective variety $X$. In \cite{BO78}, Ogus showed (following a suggestion of Deligne) that Mazur's result can be obtained from the following stronger local assertion.

\begin{theorem}[Ogus] 
The Frobenius endomorphism of $R$ endows the crystalline cochain complex $\RGamma_{\crys}(\Spec(R))$ of a smooth affine variety $\Spec(R)$ with a canonical isomorphism
\begin{equation}
\label{OgusDivFrob} 
\widetilde{\varphi}_R\colon \RGamma_{\crys}(\Spec(R)) \simeq L\eta_p \RGamma_{\crys}(\Spec(R))
\end{equation}
in the derived category $D(\Z)$.
\end{theorem} 

It follows that we can regard the pair $(\RGamma_{\crys}(\Spec(R)), \widetilde{\varphi}_R)$ as a \emph{fixed point} of the $L\eta_p$ operator acting on $D(\mathbf{Z})$. 
It was observed in \cite{BMS} that for each $K \in D(\Z)$, the object $(\Z / p^{n} \Z) \otimes_{\Z}^{L} (L \eta_{p})^n K$ of the derived category $D( \Z / p^{n} \Z)$ 
admits a {\em canonical} representative in the category of cochain complexes, given by the Bockstein complex $(\mathrm{H}^*(\Z/p^n \Z \otimes_{\Z}^L K), \mathrm{Bock}_{p^n})$
Consequently, if the object $K \in D(\Z)$ is equipped with an isomorphism $\alpha\colon  K \simeq L \eta_p K$ (in the derived category), then
each tensor product $$(\Z / p^{n} \Z) \otimes_{\Z}^{L} K \simeq ( \Z / p^{n} \Z) \otimes_{\Z}^{L} (L \eta_p)^{n} K$$ 
also has a canonical representative by cochain complex of $\Z / p^{n} \Z$-modules. One might therefore hope that these canonical representatives can be amalgamated
to obtain a representative for $K$ itself (at least up to $p$-adic completion).
We will verify this heuristic expectation in the
following result.

\begin{theorem}[See Theorem~\ref{fixedpointDZ} below]
The category of {\it strict Dieudonne complexes} (Definition~\ref{def22}) is equivalent to the category $\widehat{D(\mathbb{Z})}_p^{L
\eta_p}$ of fixed points for the endofunctor $L \eta_p$ on the $p$-completed derived category of abelian groups (see Definition~\ref{definition.dhp}).
\end{theorem} 
From this optic, Ogus's isomorphism $\widetilde{\varphi}$ in \eqref{OgusDivFrob} above guarantees that
$\RGamma_{\crys}( \Spec(R) )$ admits a canonical presentation by a strict Dieudonne complex, which can then be identified with the
de~Rham--Witt complex $W \Omega^{\ast}_{R}$ of Deligne-Illusie (this observation traces back to the work of Katz, cf. \cite[\S III.1.5]{IllusieRaynaud}).

We give two applications of the equivalence of categories $\FrobCompComplete \simeq \widehat{D(\mathbb{Z})}_p^{L
\eta_p}$. The first is internal to the considerations of this
paper: in \S \ref{sec:sddRW}, we give an alternative construction of the
saturated de~Rham--Witt complex $\W\Omega^*_R$ as a {\it saturation} of the
derived de~Rham--Witt complex $LW\Omega_R$ (for an arbitrary $\F_p$-algebra $R$). 
That is, we observe the derived de~Rham--Witt complex $L W \Omega_R$ is
equipped with a natural map $\alpha_R\colon  L W \Omega_R \to L \eta_p( L W \Omega_R)$
and that $\W \Omega^{\ast}_{R}$ can be identified with the (homotopy) colimit of the diagram
\[ \left( L W \Omega_R \xrightarrow{\alpha_R} L \eta_p( L W \Omega_R)
\xrightarrow{ L\eta_p( \alpha_R)} L\eta_{p^2} ( L W \Omega_R ) \to \dots
\right)^{\widehat{}}_p 
,\]
formed in the $p$-complete derived category $\widehat{D( \Z)}_p$. As a result, we deduce an analog of the Berthelot--Ogus isogeny theorem for $\W\Omega^*_R$ (Corollary~\ref{FrobSatdRWIsogeny}). Our second application concerns integral $p$-adic Hodge theory: in \S \ref{cryscompsec}, we apply the universal property of saturated de~Rham--Witt complexes to give a relatively simple construction of the crystalline comparison map for the $A\Omega$-complexes introduced in \cite{BMS}.

\begin{remark}
One particularly striking feature of the de~Rham--Witt complex $W \Omega^{\ast}_{R}$ is that it provides a cochain-level model for the cup product on crystalline cohomology, where
the associativity and (graded)-commutativity are visible at the level of cochains: that is, $W \Omega^{\ast}_{R}$ is a commutative differential graded algebra. From a homotopy-theoretic
point of view, this is a very strong condition. For example, there can be no analogous realization for the $\F_{\ell}$-{\etale} cohomology of algebraic varieties (or for the singular cohomology of topological spaces), due to the existence of nontrivial Steenrod operations.
\end{remark}

\subsection{Notation and Terminology}

Throughout this paper, we let $p$ denote a fixed prime number. 
A cochain complex $(M^*,d)$ is {\em $p$-torsion-free} if each $M^i$ is $p$-torsion-free. 

Given a cochain complex $(M^{\ast}, d)$ and an integer $k \in \mathbb{Z}$, we
write $\tau^{\leq k} M^{\ast}$ for the canonical truncation in degrees $\leq k$, i.e., $\tau^{\leq k} M^\ast$ is the subcomplex 
\[ \left\{ \dots \to M^{k-2} \to M^{k-1} \to \ker( d\colon  M^k \to M^{k+1}) \to 0 \to \dots \right\}\] 
of $M^{\ast}$; similarly for $\tau^{\geq k} M^\ast$. 

For a commutative ring $R$, we write $D(R)$ for the (unbounded) derived category of $R$-modules, viewed as a symmetric monoidal category by the usual (derived) tensor product of chain complexes. As it is convenient to use the language of $\infty$-categories in the later sections of this paper (especially when discussing the comparison with derived crystalline cohomology), we write $\mathcal{D}(R)$ for the derived $\infty$-category of $R$-modules, so that $D(R)$ is the homotopy category of $\mathcal{D}(R)$. Note that every cochain complex of $R$-modules can be regarded as an object of the derived category $D(R)$. To avoid confusion, we will generally follow the convention of using the notation
$M^{\ast}$ to denote a cochain complex, and $M$ (with no superscript) for its image in the category $D(R)$ (or the $\infty$-category $\mathcal{D}(R)$).

\subsection{Prerequisites}
A primary goal of this paper is to give an account of 
the theory of the de~Rham--Witt complex for $\F_p$-algebras which is as
elementary and accessible as possible. 
In particular, part 1 (comprising sections 1 through 6 of the paper), which
contains the construction of
$\W \Omega_R^{\ast}$, presupposes only a knowledge of elementary
commutative and homological algebra. 

By contrast, part 2 will require additional prerequisites. In section 7, we assume the reader is familiar with derived categories and their
$\infty$-categorical enhancements. In section 9, we will use the theory of nonabelian derived functors (such as the cotangent complex
and derived de~Rham and crystalline cohomology for $\F_p$-algebras). In sections 10 and 11, we will use techniques from the paper \cite{BMS}. 

\subsection{Acknowledgments}

We would like to thank Luc Illusie for his correspondence.
The idea of using $L \eta_p$ to reconstruct the de~Rham--Witt complex has been
independently considered by him. We would also like to thank Vladimir
Drinfeld, Arthur Ogus, Matthew Morrow, and Peter Scholze for helpful discussions. Early discussions of this project took place under the auspices of the Arizona Winter School on {\em Perfectoid Spaces} in March 2017, and we thank the organizers of this school for creating a stimulating environment.
Finally, we most heartily thank Illusie, Ogus and the referee for many helpful comments. 

During the period in which this work was carried
out, the first author was supported by NSF grants \#1501461 and \#1801689, a Packard fellowship and the Simons Foundation grant \#622511, the second author was supported by the National Science Foundation under grant number 1510417, and the third author was supported by a Clay
Research Fellowship. 

\newpage

\part{Construction of $\W \Omega_R^{\ast}$}
\section{Dieudonn\'{e} Complexes}\label{section2}

Our goal in this section is to study the category $\FrobComp$ of
Dieudonn\'e complexes. In particular, we define
the classes of {\it saturated} and {\it strict} Dieudonn\'e complexes, and the operations of
saturation and completed saturation. We also introduce the class of Dieudonn\'{e} complexes of \emph{Cartier type}, for which
the operations of saturation and completed saturation do not change the mod $p$
homology (Theorem~\ref{theo60}), or the homology at all under completeness
hypotheses (Corollary~\ref{cor67}).

\subsection{The Category $\FrobComp$}

\begin{definition}\label{def1}
A {\it Dieudonn\'{e} complex} is a triple $( M^{\ast}, d, F )$, where $( M^{\ast}, d)$ is a cochain complex of abelian groups and $F\colon  M^{\ast} \rightarrow M^{\ast}$ is a map of
graded abelian groups which satisfies the identity $dF(x) = p F(dx)$.

We let $\FrobComp$ denote the category whose objects are Dieudonn\'{e} complexes, where a morphism from $(M^{\ast}, d, F)$ to $(M'^{\ast}, d', F')$ in
$\FrobComp$ is a map of graded abelian groups $f\colon  M^{\ast} \rightarrow
M'^{\ast}$ satisfying $d'f(x) = f(dx)$ and $F'f(x) = f( Fx)$ for each $x \in M^{\ast}$.
\end{definition}

\begin{remark}
Let $(M^{\ast}, d, F)$ be a Dieudonn\'{e} complex. We will generally abuse
notation by simply referring to the underlying graded abelian group
$M^{\ast}$ as a {\it Dieudonn\'{e} complex}; in this case we implicitly assume that the data of $d$ and $F$ have also been specified.
We will refer to $F$ as the {\it Frobenius map} on the complex $M^{\ast}$.
\end{remark}

In the situation of Definition \ref{def1}, it will be convenient to interpret $F$ as a map of cochain complexes.
Note that 
$F$ defines a morphism of complexes $(M^{\ast}, pd) \to (M^{\ast}, d)$; we will
need a slight variant of this in the $p$-torsion-free case. 

\begin{construction}
\label{etaconstruction}
Let $(M^{\ast}, d)$ be a $p$-torsion-free cochain complex of abelian groups,
which we identify with a subcomplex\footnote{In this paper, we will almost
exclusively consider cochain complexes in nonnegative degrees; for such
$M^{\ast}$, we have $(\eta_p M)^{\ast} \subseteq M^{\ast}$.} of the localization
$M^{\ast}[p^{-1}]$. We define another subcomplex $(\eta_p M)^{\ast} \subseteq M^{\ast} [ p^{-1} ]$ as follows: for each integer $n$, we set
$$( \eta_p M)^{n} := \{ x \in p^{n} M^n\colon  dx \in p^{n+1} M^{n+1}
\}.$$
\end{construction}

\begin{remark}\label{rem2}
Let $(M^{\ast}, d, F)$ be a Dieudonn\'{e} complex, and suppose that $M^{\ast}$ is $p$-torsion-free. Then $F$ determines a map of cochain complexes
$$\alpha_{F}\colon  M^{\ast} \rightarrow (\eta_p M)^{\ast},$$ given by the formula $\alpha_{F}(x) = p^{n} F(x)$ for $x \in M^{n}$. 
Conversely, if $(M^{\ast}, d)$ is a cochain complex which is $p$-torsion-free
and $\alpha\colon  M^{\ast} \rightarrow (\eta_p M)^{\ast}$ is a map of cochain complexes,
then we can form a Dieudonn\'{e} complex $(M^{\ast}, d, F)$ by setting $F(x) = p^{-n} \alpha(x)$ for $x \in M^{n}$. These constructions are inverse to one another.
\end{remark}

\begin{remark}[Tensor Products]\label{tprop}
Let $M^{\ast}$ and $N^{\ast}$ be Dieudonn\'{e} complexes. Then the tensor product $M^{\ast} \otimes N^{\ast}$ admits the structure of a Dieudonn\'{e} complex,
with Frobenius given by $F(x \otimes y) = Fx \otimes Fy$. This tensor product operation (together with the usual commutativity and associativity constraints)
determines a symmetric monoidal structure on the category $\FrobComp$ of Dieudonn\'{e} complexes.
\end{remark}

\subsection{Saturated Dieudonn\'{e} Complexes}

\begin{definition}\label{def21}
Let $(M^{\ast}, d, F)$ be a Dieudonn\'{e} complex. We will say that $(M^{\ast}, d, F)$ is {\it saturated} if the following pair of conditions is satisfied:
\begin{itemize}
\item[$(i)$] The graded abelian group $M^{\ast}$ is $p$-torsion-free.
\item[$(ii)$] For each integer $n$, the map $F$ induces an isomorphism of abelian groups $$M^{n} \rightarrow \{ x \in M^{n}\colon  dx \in p M^{n+1} \}.$$
\end{itemize}
We let $\FrobCompSat$ denote the full subcategory of $\FrobComp$ spanned by the saturated Dieudonn\'{e} complexes.
\end{definition}

\begin{remark}\label{rem3}
Let $( M^{\ast}, d, F)$ be a Dieudonn\'{e} complex, and suppose that $M^{\ast}$ is $p$-torsion-free. Then
$(M^{\ast}, d, F)$ is a saturated Dieudonn\'{e} complex if and only if the map $\alpha_{F}\colon  M^{\ast} \rightarrow (\eta_p M)^{\ast}$ is
an isomorphism (see Remark \ref{rem2}).
\end{remark}

\begin{remark}\label{rem5}
Let $(M^{\ast}, d, F)$ be a saturated Dieudonn\'{e} complex. Then the map $F\colon  M^{\ast} \rightarrow M^{\ast}$ is injective.
Moreover, the image of $F$ contains the subcomplex $p M^{\ast}$. It follows that for each element
$x \in M^{n}$, there exists a unique element $Vx \in M^{n}$ satisfying $F(Vx) = px$. We will refer to
the map $V\colon  M^{\ast} \rightarrow M^{\ast}$ as the {\it Verschiebung}. 
\end{remark}

\begin{proposition}\label{Vexists}
Let $(M^{\ast}, d, F)$ be a saturated Dieudonn\'{e} complex. Then the Verschiebung $V\colon  M^{\ast} \rightarrow M^{\ast}$ is an injective map which satisfies the identities
\begin{gather} F \circ V = V \circ F = p ( \id ) \\ F \circ d \circ V = d 
\\ p (d \circ V) = V \circ d .\end{gather}
\end{proposition}

\begin{proof}
The relation $F \circ V = p ( \id)$ follows from the definition, and immediately implies that $V$ is injective (since $M^{\ast}$ is $p$-torsion-free).
Precomposing with $F$, we obtain $F \circ V \circ F = p F$. Since $F$ is injective, it follows that $V \circ F = p (\id)$. 
We have $p d = (d \circ F \circ V ) = p (F \circ d \circ V)$, so that $d = F \circ d \circ V$. Postcomposition with $V$ then yields the identity
$V \circ d = V \circ F \circ d \circ V = p( d \circ V)$.
\end{proof}

\begin{proposition}\label{prop10}
Let $(M^{\ast}, d, F)$ be a saturated Dieudonn\'{e} complex and let $r \geq 0$ be an integer. Then the map $F^{r}$ induces an isomorphism of graded abelian groups
$$M^{\ast} \rightarrow \{ x \in M^{\ast}: dx \in p^{r} M^{\ast+1} \}.$$
\end{proposition}

\begin{proof}
It is clear that $F^{r}$ is injective (since $F$ is injective), and the inclusion $F^{r} M^{\ast} \subseteq \{ x \in M^{\ast}: dx \in p^{r} M^{\ast+1} \}$
follows from the calculation $d(F^r y) = p^{r} F^{r}(dy)$. We will prove the reverse inclusion by induction on $r$. Assume that $r > 0$ and
that $x \in M^{\ast}$ satisfies $dx = p^{r} y$ for some $y \in M^{\ast+1}$. Then $d(p^{r} y) = 0 $, so our assumption that $M^{\ast}$ is $p$-torsion-free guarantees that $dy = 0$. Applying our assumption that $M^{\ast}$ is a saturated Dieudonn\'{e} complex to $x$ and $y$, we can write
$x = F(x')$ and $y = F(y')$ for some $x' \in M^{\ast}$ and $y' \in M^{\ast+1}$. We then compute
$$ p F(dx') = d( F(x') ) = dx = p^{r} y = p^{r} F(y').$$
Canceling $F$ and $p$, we obtain $dx' = p^{r-1} y'$. It follows from our inductive hypothesis that $x' \in \im( F^{r-1} )$, so that $x = F(x') \in \im( F^{r} )$.
\end{proof}

\begin{remark}
Let $(M^{\ast}, d, F)$ be a saturated Dieudonn\'{e} complex. It follows from Proposition \ref{prop10} that every cycle $z \in M^{\ast}$ is infinitely divisible
by $F$: that is, it belongs to $F^{r} M^{\ast}$ for each $r \geq 0$.
\end{remark}

\subsection{Saturation of Dieudonn\'{e} Complexes}\label{substrict}

Let $f\colon  M^{\ast} \rightarrow N^{\ast}$ be a morphism of Dieudonn\'{e} complexes. We will say that $f$ {\it exhibits $N^{\ast}$ as a saturation of $M^{\ast}$}
if $N^{\ast}$ is saturated and, for every saturated Dieudonn\'{e} complex $K^{\ast}$, composition with $f$ induces a bijection
\begin{equation} \label{Wvectuniv} \Hom_{ \FrobComp}( N^{\ast}, K^{\ast} )
\xrightarrow{\sim} \Hom_{ \FrobComp}( M^{\ast}, K^{\ast} ).\end{equation}
In this case, the Dieudonn\'{e} complex $N^{\ast}$ (and the morphism $f$) are determined by $M^{\ast}$ up to unique isomorphism; we will refer to
$N^{\ast}$ as the {\it saturation} of $M^{\ast}$ and denote it by $\Saturate(M^{\ast} )$.

\begin{proposition}\label{prop4}
Let $M^{\ast}$ be a Dieudonn\'{e} complex. Then $M^{\ast}$ admits a saturation $\Saturate( M^{\ast} )$.
\end{proposition}

\begin{proof}
Let $T^{\ast} \subseteq M^{\ast}$ be the graded subgroup consisting of elements $x \in M^{\ast}$ satisfying
$p^{n} x = 0$ for $n \gg 0$. Replacing $M^{\ast}$ by the quotient $M^{\ast} / T^{\ast}$, we can
reduce to the case where $M^{\ast}$ is $p$-torsion-free. In this case, Remark
\ref{rem3} and the fact that $\eta_p$ commutes with filtered colimits implies that the direct limit of the sequence
 \begin{equation} \label{saturationexp} M^{\ast} \xrightarrow{ \alpha_F }
 (\eta_p M)^{\ast} \xrightarrow{ \eta_p( \alpha_F ) } (\eta_p \eta_p M)^{\ast}
 \xrightarrow{ \eta_p( \eta_p( \alpha_F ) )} ( \eta_p \eta_p \eta_p M)^{\ast}
 \rightarrow \cdots \end{equation}
is a saturation of $M^{\ast}$.
\end{proof}

\begin{corollary}\label{cor65}
The inclusion functor $\FrobCompSat \hookrightarrow \FrobComp$ admits a left adjoint, given by the saturation construction $M^{\ast} \mapsto \Saturate( M^{\ast} )$.
\end{corollary}

\begin{remark}
\label{rem:satissubgroup}
Let $(M^{\ast}, d, F)$ be a Dieudonn\'{e} complex. Assume 
that $M^{\ast}$ is $p$-torsion-free and that $F$ is injective
on $M^{\ast}$ (if not, one can replace $M$ by its quotient by the $F$-power
torsion subgroup). Then we can identify $M^{\ast}$ with a subgroup of the graded abelian group
$M^{\ast}[ F^{-1} ]$, defined as the colimit
$M^{\ast} \stackrel{F}{\to} M^{\ast} \stackrel{F}{\to} \dots $, whose elements
can be written as formal expressions of the form $F^{n} x$ for $x \in M^{\ast}$ and
$n \in \Z$. Unwinding \eqref{saturationexp}, we see that the saturation $\Saturate( M^{\ast} )$ can be described more concretely as follows:
\begin{itemize}
\item As a graded abelian group, $\Saturate( M^{\ast} )$ can be identified with the subgroup of $M^{\ast}[F^{-1}]$ consisting of those elements
$x    \in M^{\ast}[F^{-1}]$ satisfying $d (F^{n}(x) ) \in p^{n} M^{\ast}$ for $n \gg 0$ (note that if this condition is satisfied for some nonnegative integer $n$,
then it is also satisfied for every larger integer).
Concretely, if $x = F^{-m}y$ for $y \in M^{\ast}$, then this condition requires that
$d( F^{n-m}y ) \in p^n M^{\ast}$ for $n \gg m$. 

\item The differential on $\Saturate(M^{\ast} )$ is given by the construction $x \mapsto F^{-n} p^{-n} d (F^{n} x)$ for $n \gg 0$.

\item The Frobenius map on $\Saturate(M^{\ast})$ is given by the restriction of the map $F\colon  M^{\ast}[F^{-1}] \rightarrow M^{\ast}[ F^{-1} ]$.
\end{itemize}
\end{remark}
\begin{remark}
\label{integralformsabstract}

We can rephrase Remark~\ref{rem:satissubgroup} as follows.  The operator $d\colon  M^{\ast} \to M^{\ast + 1}$ extends to a map
\[ d\colon  M^{\ast}[F^{-1}] \to M^{\ast+1}[F^{-1}] \otimes_{\mathbb{Z}}
\mathbb{Z}[1/p] \]
given by the formula
\[ d( F^{-n} x) = p^{-n } F^{-n}dx. \]
We claim that $\mathrm{Sat}(M^{\ast}) \subseteq M^{\ast}[F^{-1}]$ can be identified with
the graded subgroup consisting of those elements $y = F^{-n} x \in M^{\ast}[F^{-1}]$
for which $dy \in M^{\ast}[F^{-1}] \otimes_{\mathbb{Z}} \mathbb{Z}[1/p]$ belongs to 
$M^{\ast}[F^{-1}]$. 

To see this, suppose that we can write $d( F^{-n}x) = F^{-m} z$ in
$M^{\ast}[F^{-1}]$ for $x, z \in M^{\ast}$. 
It follows that $d(F^{r-n} x) = d(F^r F^{-n} x)= p^r F^r
d(F^{-n} x) = p^r F^r  (F^{-m} z)$ belongs to $p^{r} M^{\ast}$ for $r \gg 0$, so that
$F^{-n} x $ belongs to $\mathrm{Sat}(M^{\ast})$ (Remark~\ref{rem:satissubgroup}). 
Conversely, if $F^{-n} x$ belongs to  $\mathrm{Sat}(M^{\ast})$, then $d(F^{r-n} x) \in p^r M^{\ast}$ for $r \gg 0$ (Remark~\ref{rem:satissubgroup}),
so that $d(F^{-n} x) = p^{-r}F^{-r}d( F^{r-n} dx)  \in M^{\ast}[F^{-1}]$ as desired. 
\end{remark}

\subsection{The Cartier Criterion}\label{cartcrit}

Let $M^{\ast}$ be a Dieudonn\'{e} complex. For each $x \in M^{\ast}$, the equation $dFx = p F(dx)$ shows that the image of
$Fx$ in the quotient complex $M^{\ast} / p M^{\ast}$ is a cycle, and therefore represents an element in the cohomology $\mathrm{H}^{\ast}( M / p M )$.
This construction determines a map of graded abelian groups $M^{\ast} \rightarrow \mathrm{H}^{\ast}( M / p M)$, which factors uniquely through the quotient
$M^{\ast} / p M^{\ast}$.

\begin{definition} 
\label{def:Cartiertype}
We say that the Dieudonn\'e complex $M^{\ast}$ is of \emph{Cartier type} if 
\begin{itemize}
\item[$(i)$] The complex $M^{\ast}$ is $p$-torsion-free.
\item[$(ii)$] The Frobenius map $F$ induces an isomorphism of graded abelian
groups $$M^{\ast} / p M^{\ast} \xrightarrow{\sim} \mathrm{H}^{\ast}(M / p M).$$ 
\end{itemize}
\end{definition}

\begin{theorem}[Cartier Criterion]\label{theo60}
Let $M^{\ast}$ be a Dieudonn\'{e} complex which is of Cartier type. 
Then the canonical map $M^{\ast} \rightarrow \Saturate( M^{\ast} )$ induces a
quasi-isomorphism of cochain complexes $M^{\ast} / p M^{\ast} \rightarrow \Saturate( M^{\ast} ) / p \Saturate( M^{\ast} )$.
\end{theorem}

\begin{remark} 
The Dieudonn\'{e} complex $\Saturate(M^{\ast})$ itself is essentially never of Cartier
type. 
\end{remark} 

To prove Theorem \ref{theo60}, we will need to review some properties of the construction $M^{\ast} \mapsto (\eta_p M)^{\ast}$.

\begin{construction}\label{con61}
Let $M^{\ast}$ be a cochain complex which is $p$-torsion-free. For each element
$x \in ( \eta_p M)^{k}$, we let $\overline{\gamma}(x) \in \mathrm{H}^{k}( M / pM )$ denote the cohomology class represented by $p^{-k} x$ (which is a cycle in
$M^{\ast} / p M^{\ast}$, by virtue of our assumption that $x \in (\eta_p
M)^{k}$). This construction determines a map of graded abelian groups
$\overline{\gamma}\colon  (\eta_p M)^{\ast} \rightarrow \mathrm{H}^{\ast}( M / pM )$. Since the codomain of $\overline{\gamma}$ is a $p$-torsion group,
the map $\overline{\gamma}$ factors (uniquely) as a composition
$$(\eta_p M)^{\ast} \rightarrow (\eta_p M)^{\ast} / p (\eta_p M)^{\ast} \xrightarrow{\gamma} \mathrm{H}^{\ast}( M / pM ).$$
\end{construction}

We next need a general lemma about the behavior of $\eta_p$ reduced mod $p$. 
This is a special case of a result about d\'ecalage due to  Deligne \cite[Prop. 1.3.4]{HodgeII}, where the filtration is
taken to be the $p$-adic one. 
This result in the case of $L \eta_p$ appears in \cite[Prop. 6.12]{BMS}. 

\begin{proposition}[{}]\label{prop62}
Let $M^{\ast}$ be a cochain complex of abelian groups which is $p$-torsion-free.
Then Construction \ref{con61} determines a quasi-isomorphism of cochain complexes
$$ \gamma\colon  (\eta_p M)^{\ast} / p (\eta_p M)^{\ast} \rightarrow \mathrm{H}^{\ast}( M / p M).$$
Here we regard $\mathrm{H}^{\ast}( M / pM)$ as a cochain complex with respect to the Bockstein map
$\beta\colon  \mathrm{H}^{\ast}( M / pM) \rightarrow \mathrm{H}^{\ast+1}( M/ pM)$ determined the
short exact sequence of cochain complexes
$0 \rightarrow M^{\ast} / p M^{\ast} \xrightarrow{p} M^{\ast} / p^{2} M^{\ast} \rightarrow M^{\ast} / p M^{\ast} \rightarrow 0$.
\end{proposition}

\begin{proof}
For each element $x \in M^{k}$ satisfying $dx \in p M^{k+1}$, let $[x] \in \mathrm{H}^{k}( M / pM)$ denote the cohomology class represented by $x$.
Unwinding the definitions, we see that the map $\beta$ satisfies the identity $\beta( [x] ) = [ p^{-1} dx ]$. It follows that for each
$y \in (\eta_p M)^{k}$, we have 
$$ \beta( \overline{\gamma}(y) ) = \beta( [ p^{-k} y] ) = [ p^{-k-1} dy] = \overline{\gamma}( dy ),$$
so that $\overline{\gamma}\colon  (\eta_p M)^{\ast} \rightarrow \mathrm{H}^{\ast}(M/pM)$ is a
map of cochain complexes.
It follows that $\gamma$ is also a map of cochain complexes. It follows immediately from the definition of $(\eta_p M)^{\ast}$ that the map
$\gamma$ is surjective. Let $K^{\ast}$ denote the kernel of $\gamma$. We will
complete the proof by showing that the cochain complex $K^{\ast}$ is acyclic.
Unwinding the definitions, we can identify $K^{\ast}$ with the subquotient of $M^{\ast} [p^{-1}]$ described by the formula
$$ K^{k} = ( p^{k+1} M^{k} + d( p^{k} M^{k-1} )) / ( p^{k+1} M^{k} \cap d^{-1}( p^{k+2} M^{k+1} ) ).$$
Explicitly, if $x \in (\eta_p M)^k$ and $\gamma(x) = 0$,
then we can write
$p^{-k} x  = dy + pz$ in $M^{k}$, forcing $x = p^{k+1}z + d( p^k y)$, giving
the formula for $K^k$ as desired.  

Suppose that $e \in K^{k}$ is a cocycle, represented by $p^{k+1} x + d( p^{k} y)$ for some $x \in M^{k}$ and $y \in M^{k-1}$.
The equation $de = 0 \in K^{k+1}$ implies that $d(p^{k+1} x) \in p^{k+2} M^{k+1}$. It follows that $e$ is also represented by $d( p^{k} y)$, and is therefore
a coboundary in $K^{k}$.
\end{proof}

\begin{corollary}\label{cor63}
Let $f\colon  M^{\ast} \rightarrow N^{\ast}$ be a map of cochain complexes of abelian groups. Assume that $M^{\ast}$ and $N^{\ast}$ are $p$-torsion-free.
If $f$ induces a quasi-isomorphism $M^{\ast} / p M^{\ast} \rightarrow N^{\ast} / p N^{\ast}$, then the induced map
$$(\eta_p M)^{\ast} / p (\eta_p M)^{\ast} \rightarrow (\eta_p N)^{\ast} /p (\eta_p N)^{\ast}$$ is also a quasi-isomorphism.
\end{corollary}

\begin{proof}
By virtue of Proposition \ref{prop62}, it suffices to show that $f$ induces a quasi-isomorphism
$\mathrm{H}^{\ast}( M / pM ) \rightarrow \mathrm{H}^{\ast}( N / pN)$ (where both are regarded as chain complexes with differential given by the Bockstein operator).
In fact, this map is an isomorphism (by virtue of our assumption that $f$ induces a quasi-isomorphism $M^{\ast} / p M^{\ast} \rightarrow N^{\ast} / p N^{\ast}$).
\end{proof}

\begin{proof}[Proof of Theorem \ref{theo60}]
Let $(M^{\ast},d,F)$ be a Dieudonn\'{e} complex which is $p$-torsion-free. Then the saturation $\Saturate( M^{\ast} )$ can be identified with the colimit of the sequence
$$ M^{\ast} \xrightarrow{ \alpha_F } (\eta_p M)^{\ast} \xrightarrow{ \eta_p( \alpha_F ) } (\eta_p \eta_p M)^{\ast} \xrightarrow{ \eta_p( \eta_p( \alpha_F ) )} ( \eta_p \eta_p \eta_p M)^{\ast} \rightarrow \cdots $$
Consequently, to show that the canonical map $M^{\ast} / p M^{\ast} \rightarrow \Saturate( M^{\ast} ) / p \Saturate( M^{\ast} )$ is a quasi-isomorphism, it will suffice to show that
each of the maps $$(\eta_p^{k} M)^{\ast} / p (\eta_p^{k} M)^{\ast} \rightarrow
(\eta_p^{k+1} M)^{\ast} / p (\eta_p^{k+1} M)^{\ast} \quad k \geq 0$$ is a quasi-isomorphism.
By virtue of Corollary \ref{cor63}, it suffices to verify this when $k = 0$. In this case, we have a commutative diagram
of cochain complexes
$$ \xymatrix{ M^{\ast} / p M^{\ast} \ar[rr]^{ \alpha_F } \ar[dr]_F & & (\eta_p M)^{\ast} / p ( \eta_p M)^{\ast} \ar[dl]^{\gamma} \\
& \mathrm{H}^{\ast}( M / p M ), & }$$
where the map $M^{\ast}/pM^{\ast} \to \mathrm{H}^*(M/pM )$ is induced by the Frobenius.
Since $M^{\ast}$ is assumed to be of Cartier type, this map is an isomorphism. 
 Since $\gamma$ is a quasi-isomorphism
by virtue of Proposition \ref{prop62}, it follows that the map
$\alpha_F\colon M^{\ast} / p M^{\ast} \rightarrow (\eta_p M)^{\ast} / p (\eta_p M)^{\ast}$ is also a quasi-isomorphism.
\end{proof}

\subsection{Strict Dieudonn\'{e} Complexes}
\label{strictdieudonnecomplex}

We now introduce a special class of saturated Dieudonn\'{e} complexes, characterized by the requirement that they are complete with
respect to a filtration determined by the Verschiebung operator $V$.

\begin{construction}[Completion]
\label{completionDieudonne}
Let $M^{\ast}$ be a saturated Dieudonn\'{e} complex. 
For each integer $r \geq 0$, we let $\WittScript_{r}(M)^{\ast}$ denote the quotient of
$M^{\ast}$ by the subcomplex $\im( V^{r} ) + \im ( d V^r )$, where $V$ is the Verschiebung map of Remark \ref{rem5}.
We denote the natural quotient maps $\mathcal{W}_{r+1}(M)^{\ast} \to
\mathcal{W}_r(M)^{\ast}$ by $\Res\colon  \mathcal{W}_{r+1}(M)^{\ast} \rightarrow \mathcal{W}_{r}(M)^{\ast}$.
We let $\WittScript(M)^{\ast}$ denote the inverse limit of the tower of cochain complexes
$$  \rightarrow \cdots \rightarrow \WittScript_{3}(M)^{\ast} \xrightarrow{\Res} \WittScript_{2}(M)^{\ast} \xrightarrow{\Res} \WittScript_{1}(M)^{\ast} \xrightarrow{\Res} \WittScript_0(M)^{\ast} = 0.$$
We will refer to $\WittScript(M)^{\ast}$ as the {\it completion} of the
saturated Dieudonn\'{e} complex $M^{\ast}$.
\end{construction}

\begin{remark}[Frobenius and Verschiebung at Finite Levels]\label{rem13}
Let $M^{\ast}$ be a saturated Dieudonn\'{e} complex. Using the identities $F \circ d \circ V = d$ and $F \circ V = p \id$, 
we deduce that the Frobenius map $F\colon  M^{\ast} \rightarrow M^{\ast}$ carries $\im( V^{r} ) + \im( d V^{r} )$ into $\im( V^{r-1} ) + \im( d V^{r-1} )$.
It follows that there is a unique map of graded abelian groups $F\colon  \WittScript_{r}(M)^{\ast} \rightarrow \WittScript_{r-1}(M)^{\ast}$ for which the diagram
$$ \xymatrix{ M^{\ast} \ar[r]^-{F} \ar@{->>}[d] & M^{\ast} \ar@{->>}[d] \\
\WittScript_{r}(M)^{\ast} \ar[r]^-{F} & \WittScript_{r-1}(M)^{\ast} }$$
commutes. Passing to the inverse limit over $r$, we obtain a map $F\colon  \WittScript(M)^{\ast} \rightarrow \WittScript(M)^{\ast}$ which
endows the completion $\WittScript(M)^{\ast}$ with the structure of a Dieudonn\'e 
complex.

Similarly, the identity $V \circ d = p( d \circ V)$ implies that $V$ carries $\im( V^{r} ) + \im( dV^{r} )$ into
$\im( V^{r+1} ) + \im( d V^{r+1} )$, so that there is a unique map $V\colon  \WittScript_{r}(M)^{\ast} \rightarrow \WittScript_{r+1}(M)^{\ast}$
for which the diagram
$$ \xymatrix{ M^{\ast} \ar[r]^-{V} \ar[d] & M^{\ast} \ar[d] \\
\WittScript_{r}(M)^{\ast} \ar[r]^-{V} & \WittScript_{r+1}(M)^{\ast} }$$
commutes.
\end{remark}

\begin{remark}
The formation of completions is functorial: if $f\colon  M^{\ast} \rightarrow N^{\ast}$ is a morphism of saturated Dieudonn\'{e} complexes, then
$f$ induces a morphism $\WittScript(f)\colon  \WittScript(M)^{\ast} \rightarrow \WittScript(N)^{\ast}$, given by the inverse limit of the tautological maps
$$M^{\ast} / (\im( V^{r} ) + \im( d V^{r} ) ) \xrightarrow{f} N^{\ast} / ( \im( V^{r} ) + \im( d V^{r} ) ).$$
\end{remark}

For every saturated Dieudonn\'{e} complex $M^{\ast}$, the tower of cochain complexes
$$ M^{\ast} \rightarrow \cdots \rightarrow \WittScript_{3}(M)^{\ast} \xrightarrow{\Res} \WittScript_{2}(M)^{\ast} \xrightarrow{\Res} \WittScript_{1}(M)^{\ast} \xrightarrow{\Res} \WittScript_0(M)^{\ast} = 0.$$
determines a tautological map $\rho_{M}\colon  M^{\ast} \rightarrow \WittScript(M)^{\ast}$. It follows immediately from the definitions that $\rho_M$ is a map of Dieudonn\'{e} complexes, which depends functorially on $M^{\ast}$.

\begin{definition}\label{def22}
Let $M^{\ast}$ be a Dieudonn\'{e} complex. We will say that $M^{\ast}$ is {\it strict} if it is saturated and the map $\rho_M\colon  M^{\ast} \rightarrow \WittScript(M)^{\ast}$
is an isomorphism. We let $\FrobCompComplete$ denote the full subcategory of $\FrobComp$ spanned by the strict Dieudonn\'{e} complexes.
\end{definition}

\begin{remark}\label{olos}
Let $M^{\ast}$ be a strict Dieudonn\'{e} complex. Then each $M^{n}$ is a torsion-free abelian group: it is $p$-torsion free by virtue of
the assumption that $M^{\ast}$ is saturated, and $\ell$-torsion free for $\ell \neq p$ because $M^{n}$ is $p$-adically complete.
\end{remark}

\begin{example}\label{pulox}
Let $A$ be an abelian group and let $$M^{\ast} = \begin{cases} A & \text{ if } \ast = 0 \\
0 & \text{otherwise} \end{cases}$$ denote the cochain complex consisting of the abelian group $A$ in degree zero.
Then any  endomorphism $F\colon  A \rightarrow A$ endows $M^{\ast}$ with the structure of a Dieudonn\'{e} complex.
The Dieudonn\'{e} complex $M^{\ast}$ is saturated if and only if $A$ is $p$-torsion-free and the map $F$ is an automorphism. 
If these conditions are satisfied, then $M^{\ast}$ is strict if and only if $A$ is $p$-adically complete: that is,
if and only if the canonical map $A \rightarrow \varprojlim A/ p^{r} A$ is an isomorphism.
\end{example} 

\begin{example}[Free Strict Dieudonn\'{e} Complexes]
Let $M^{0}$ denote the abelian group consisting of formal expressions of the form
$$\sum_{m \geq 0} a_m F^m x + \sum_{n > 0} b_n V^n x$$
where the coefficients $a_m$ and $b_n$ are $p$-adic integers with the property
that the sequence $\{ a_m \}_{m \geq 0}$ converges to zero (with respect to the $p$-adic topology on $\Z_p$). 
Let $M^{1}$ denote the abelian group of formal expressions of the form
$$\sum_{m \geq 0} c_m F^m dx + \sum_{n > 0} d  (d_n V^n x)$$
where the coefficients $c_m$ and $d_n$ are $p$-adic integers with the property that
the sequence $\{ c_m \}_{m \geq 0}$ converges to zero (also with respect to the $p$-adic topology). 
Then we can form a strict Dieudonn\'{e} complex
$$ \cdots \rightarrow 0 \rightarrow 0 \rightarrow M^{0} \xrightarrow{d} M^{1} \rightarrow 0 \rightarrow 0 \rightarrow \cdots,$$
where the differential $d$ is given by the formula
$$ d( \sum_{m \geq 0} a_m F^m x + \sum_{n > 0} b_n V^n x ) = \sum_{m \geq 0} p^{m} a_m F^m dx + 
\sum_{n > 0} b_n dV^{n} x,$$
and where 
$F$ is defined in the natural fashion with the identities $FdV = d$, $FV = p$. 
Moreover, the strict Dieudonn\'{e} complex $M^{\ast}$ is freely generated by $x \in M^0$ in the following
sense: for any strict Dieudonn\'{e} complex $N^{\ast}$, evaluation on $x$ induces a bijection
$$\Hom_{\FrobComp}(M^{\ast},
N^{\ast}) \simeq N^{0}.$$
\end{example}

\subsection{Strict Dieudonn\'{e} Towers}

For later use, it will be convenient to axiomatize some of the features of the complexes
$\W_{r}( M)^{\ast}$ arising from Construction \ref{completionDieudonne}.

\begin{definition}\label{stricTD}
\label{strictTD}
A \textbf{strict Dieudonn\'e tower} is an inverse system of cochain complexes
$$ \cdots \rightarrow X_{3}^{\ast} \xrightarrow{R} X_{2}^{\ast} \xrightarrow{R} X_1^{\ast} \xrightarrow{R} X_0^{\ast}$$
which is equipped with maps of graded abelian groups $F\colon  X_{r+1}^{\ast} \to X_{r}^{\ast}$ and $V\colon 
X_{r}^{\ast} \to X_{r+1}^{\ast}$ which satisfy the following axioms:
\begin{enumerate}
\item The cochain complex $X_0^{\ast}$ vanishes.

\item The map $R\colon  X_{r+1}^{\ast} \to X_{r}^{\ast}$ is surjective for each $r
\geq 0$. 
\item The operator $F\colon  X_{r+1}^{\ast} \to X_{r}^{\ast}$ satisfies $dF = pFd$
for each $r \geq 0$. 
\item The operators $F$, $R$, and $V$ commute with each other. 
\item We have $F(V(x)) = px = V(F(x))$ for each $x \in X_{r}^{\ast}$.

\item Given any $x \in X_r^{\ast}$ such that $dx$ is divisible by
$p$, $x$ belongs to the image of $F \colon  X_{r+1}^{\ast} \to
X_r^{\ast}$. 
\item The kernel of the map $R\colon  X_{r+1}^{\ast} \to X_{r}^{\ast}$ coincides with the submodule
$X_{r+1}^{\ast}[p] \subseteq X_{r+1}^{\ast}$ consisting of those elements $x \in
X_{r+1}^{\ast}$ satisfying $px = 0$. 
\item The kernel of $R\colon  X_{r+1}^{\ast} \to X_{r}^{\ast}$ is the span of the images
of $V^{r}, dV^{r}\colon  X_1^{\ast} \to X_{r+1}^{\ast}$. 
\end{enumerate}
A morphism of strict Dieudonn\'e towers is a morphism of towers of cochain
complexes which are compatible with the $F, V$ operators. 
We let $\FrobTow$ denote the category of strict Dieudonn\'e towers. 
\end{definition}

We now show that Construction \ref{completionDieudonne} gives rise to strict Dieudonn\'{e} towers:

\begin{proposition} 
\label{buildtowerfromsat}
Let $M^{\ast}$ be a saturated Dieudonn\'e complex. Then the diagram
$$ \cdots \rightarrow \mathcal{W}_{3}(M)^{\ast}
\xrightarrow{\Res} \mathcal{W}_2(M)^{\ast} \xrightarrow{\Res} \mathcal{W}_1(M)^{\ast} \xrightarrow{\Res} \mathcal{W}_0(M)^{\ast}$$
is a strict Dieudonn\'{e} tower, when endowed with the Frobenius and Verschiebung operators of Remark \ref{rem13}.
\end{proposition}

The proof of Proposition \ref{buildtowerfromsat} will require some preliminaries.

\begin{lemma}\label{lem6}
Let $M^{\ast}$ be a saturated Dieudonn\'{e} complex and let $x \in M^{\ast}$ satisfy
$d(V^{r} x) \in p M^{\ast+1}$. Then $x \in \im(F)$.
\end{lemma}

\begin{proof}
We have $dx = F^{r} d (V^{r} x ) \in F^{r} p M^{\ast+1} \subseteq p M^{\ast+1}$, so the desired result follows from our assumption that $M^{\ast}$ is saturated.
\end{proof}

\begin{lemma} 
\label{lem:noptorsion}
Let $M^{\ast}$ be a saturated Dieudonn\'e complex and let $x \in M^{\ast}$. If
the image of $px$ vanishes in $\W_{r+1}(M)^{\ast}$, then the image of $x$ vanishes
in $\W_{r}(M)^{\ast}$. 
\end{lemma} 
\begin{proof} 
Write $px = V^{r+1} a + dV^{r+1} b$ for some $a, b \in M^{\ast}$. 
Then $dV^{r+1} a = p(dx)$ is divisible by $p$, so that Lemma~\ref{lem6} allows us to
write $a = F \widetilde{a}$ for some $\widetilde{a} \in M^{\ast}$. This guarantees that
$d V^{r+1} b = p (x - V^{r} \widetilde{a} )$ is divisible by $p$, so that
$b = F \widetilde{b}$ for some $\widetilde{b} \in
M^{\ast}$ (Lemma~\ref{lem6}). Then we can write
\[ px = p V^r \widetilde{a} + d V^r (p \widetilde{b}), \]
so that $x = V^{r} \widetilde{a} + d V^{r} \widetilde{b} \in \im(V^{r}) + \im( d V^{r} )$ as desired. 
\end{proof}

\begin{proof}[Proof of Proposition \ref{buildtowerfromsat}]
Let $M^{\ast}$ be a saturated Dieudonn\'{e} complex; we wish to show that the
inverse system $\{ \W_r ( M)^{\ast} \}_{r \geq 0}$
of Construction \ref{completionDieudonne} is a strict Dieudonn\'{e} tower (when equipped with the Frobenius and Verschiebung operators of
Remark \ref{rem13}). It is clear that this system satisfies axioms $(1)$ through $(5)$ and $(8)$ of Definition \ref{stricTD}. To verify
$(6)$, suppose that we are given an element $\overline{x} \in \W_r( M)^{\ast}$ such that $d \overline{x}$ is divisible by $p$.
Choose an element $x \in M^{\ast}$ representing $\overline{x}$, so that have an identity of the form
$dx = py + V^{r} a + d V^{r} b$ for some $y, a, b \in M^{\ast}$. Then $d(V^r a)$ is divisible by $p$,
so Lemma~\ref{lem6} implies that we can write $a = F \widetilde{a}$ for some $\widetilde{a} \in M^{\ast}$.
We then have
\[ d(x - V^r b) = p (y + V^{r-1} \widetilde{a}), \]
so our assumption that $M^{\ast}$ is saturated guarantees that we can write $x - V^r b = Fz$ 
for some $z \in M^{\ast}$. It follows that $\overline{x} = F \overline{z}$, where $\overline{z} \in \W_{r+1} M^{\ast}$
denotes the image of $z$.

We now verify $(7)$. We first note that if $x$ belongs to $\im( V^r ) + \im( d V^r )$, then
$px$ belongs to 
$$\im( p V^r ) + \im( p dV^r ) = \im( V^{r+1} F) + \im( d V^{r+1} F ) \subseteq \im( V^{r+1}) + \im( d V^{r+1} ).$$
It follows that the kernel of the restriction map $\Res\colon  \W_{r+1} (M)^{\ast}
\rightarrow \W_{r} (M)^{\ast}$ is contained
in $(\W_{r+1} (M)^{\ast})[p]$. The reverse inclusion follows from Lemma \ref{lem:noptorsion}.
\end{proof}

We next establish a sort of converse to Proposition \ref{buildtowerfromsat}: from any strict
Dieudonn\'{e} tower, we can construct a saturated Dieudonn\'{e} complex.

\begin{proposition}\label{buildstrictDieudonne}
Let $\left\{X_r^{\ast}\right\}$ be a strict Dieudonn\'{e} tower and let
$X^{\ast}$ denote the cochain complex given by the inverse limit 
$\varprojlim_r X_r^{\ast}$. Let $F\colon  X^{\ast} \rightarrow X^{\ast}$
denote the map of graded abelian groups given by the inverse limit
of the Frobenius operators $F\colon  X_{r+1}^{\ast} \rightarrow X_{r}^{\ast}$.
Then $( X^{\ast}, F)$ is a saturated Dieudonn\'{e} complex.
\end{proposition} 

\begin{proof} 
It follows from axiom $(3)$ of Definition \ref{stricTD} that the Frobenius map
$F\colon  X^{\ast} \rightarrow X^{\ast}$ satisfies $dF = p Fd$, so that $(X^{\ast}, F)$ is a Dieudonn\'{e} complex.
We claim that the graded abelian group $X^{\ast}$ is $p$-torsion-free. 
Choose any element $x \in X^{\ast}$, given by a compatible system of elements
$\{ x_r \in X^{\ast}_{r} \}_{r \geq 0}$. If $px = 0$, then each $x_r$ satisfies
$px_r = 0$ in $X^{\ast}_{r}$, so that $x_{r-1} = 0$ by virtue of axiom $(7)$ of Definition \ref{stricTD}.
Since $r$ is arbitrary, it follows that $x = 0$.

To complete the proof that $X^{\ast}$ is saturated, it suffices now to show that if $x = \{ x_r \}_{r \geq 0}$ is an element of 
$X^{\ast}$ and $dx$ is divisible by $p$, then $x$ belongs to the image of $F$. 
Using axiom $(6)$ of Definition \ref{stricTD}, we can choose $y_{r+1} \in X_{r+1}^{\ast}$ such that
$Fy_{r+1} = x_r$. Beware that we do not necessarily have $R( y_{r+1} ) = y_r$ for $r > 0$.
However, the difference $R(y_{r+1}) - y_r \in X_{r}^{\ast}$ is annihilated
by $F$, hence also by $p = VF$, hence also by the restriction map $R$ (by virtue of axiom
$(7)$ of Definition \ref{strictTD}). It follows that the sequence $\left\{R(y_{r+1}) \in X_r^{\ast}\right\}$ is
determines an element $y \in X^{\ast}$ satisfying $Fy=x$.
\end{proof} 

\subsection{The Completion of a Saturated Dieudonn\'{e} Complex}\label{lebex}

Let $M^{\ast}$ be a saturated Dieudonn\'{e} complex. It follows from Propositions~\ref{buildtowerfromsat} and \ref{buildstrictDieudonne}
that the completion $\WittScript(M)^{\ast}$ is also a saturated Dieudonn\'{e} complex. Our next goal is to show that
the completion $\WittScript(M)^{\ast}$ is strict (Corollary \ref{realcor14}), and is {\em universal} among strict Dieudonn\'{e} complexes
receiving a map from $M^{\ast}$ (Proposition \ref{prop37}). We begin by studying the homological properties of
Construction \ref{completionDieudonne}.

\begin{proposition}\label{prop11}
Let $M^{\ast}$ be a saturated Dieudonn\'{e} complex and let $r$ be a nonnegative integer. Then the map
$F^{r}\colon  M^{\ast} \rightarrow M^{\ast}$ induces an isomorphism of graded abelian groups
$\WittScript_{r}(M)^{\ast} \xrightarrow{\sim} \mathrm{H}^{\ast}( M^{\ast} / p^{r} M^{\ast} )$.
\end{proposition}

\begin{proof}
The equality $d \circ F^{r} = p^{r} (F^{r} \circ d)$ guarantees that $F^{r}$ carries each element of $M^{\ast}$ to a 
cycle in the quotient complex $M^{\ast} / p^{r} M^{\ast}$. Moreover, the identities 
$$ F^{r} V^{r} (x) = p^{r} x \quad \quad F^{r} d V^{r}(y) = dy$$
imply that $F^{r}$ carries the subgroup $\im( V^{r} ) + \im( d V^{r} )
\subseteq M^{\ast}$ to the set of boundaries in the cochain complex
$M^{\ast} / p^{r} M^{\ast}$. We therefore obtain a well-defined homomorphism $\theta\colon  \WittScript_{r}(M)^{\ast} \rightarrow \mathrm{H}^{\ast}( M^{\ast} / p^{r} M^{\ast} )$.
The surjectivity of $\theta$ follows from Proposition \ref{prop10}. To prove injectivity, suppose that $x \in M^{\ast}$ has the property that
$F^{r} x$ represents a boundary in the quotient complex $M^{\ast} / p^{r} M^{\ast}$, so that we can write $F^{r} x = p^{r} y + dz$ for some $y \in M^{\ast}$ and $z \in M^{\ast-1}$.
Then $$p^{r} x = V^{r} F^{r} x = V^{r} ( p^{r} y + dz ) = p^{r} V^{r}(y) + p^{r} d V^{r}(z),$$ so that $x = V^{r}(y) + dV^{r}(z) \in \im( V^{r} ) + \im( d V^{r} )$.
\end{proof}

\begin{corollary}\label{cor79}
Let $(M^{\ast}, d, F)$ be a saturated Dieudonn\'{e} complex. Then the quotient map
$u\colon  M^{\ast} / pM^{\ast} \rightarrow \WittScript_1(M)^{\ast}$ is a
quasi-isomorphism of cochain complexes.
\end{corollary}

\begin{proof}
Regard $\mathrm{H}^{\ast}(M/pM)$ as a cochain complex with respect to the
Bockstein operator, as in Proposition~\ref{prop62}. 
We have a commutative diagram of cochain complexes
$$ \xymatrix{ M^{\ast} / p M^{\ast} \ar[r]^-{ \alpha_F } \ar[d]^{u} & (\eta_p M)^{\ast} / p (\eta_p M)^{\ast} \ar[d] \\
\WittScript_1(M)^{\ast} \ar[r]^-{\sim} & \mathrm{H}^{\ast}( M / p M), }$$
where the bottom horizontal map is the isomorphism of Proposition
\ref{prop11} (which one easily checks to be a map of cochain complexes as in
Proposition~\ref{prop62}, either directly or using the surjectivity of $u$), the right vertical map is the quasi-isomorphism
of Proposition \ref{prop62}, and $\alpha_F$ is an isomorphism by virtue of our assumption that $M^{\ast}$ is saturated.
It follows that $u$ is also a quasi-isomorphism.
\end{proof}

\begin{remark}
More generally, if $M^{\ast}$ is a saturated Dieudonn\'{e} complex, then the quotient map $M^{\ast} / p^{r} M^{\ast} \rightarrow \WittScript_r(M)^{\ast}$ is a quasi-isomorphism
for each $r \geq 0$. This can be proven by essentially the same argument, using
a slight generalization of Proposition \ref{prop62}. One can also argue directly as follows. 
First, we claim that the map 
$\mathrm{H}^i (M/p^r M) \to \mathrm{H}^i( \WittScript_r(M))$ is a surjection. 
A class in $\mathrm{H}^i( \WittScript_r(M))$ is represented by an element $x \in M^i$
such that we can write $dx = V^r y +  dV^r z$ for $y \in M^{i+1}, z \in M^i$. 
This implies that $dV^r y = 0$, so $dy = F^r dV^r y = 0$ and $y$ is a cycle. 
Since $M^{\ast}$ is saturated, this implies that we can write $y = F^r y'$, so
$dx  = p^r y' + dV^r z$. It follows that $x - V^r z
\in M^i$ is a class which is a cycle modulo $p^r$, and lifts $[x] \in
\mathrm{H}^i( \WittScript_r(M))$. Conversely, 
if $w \in M^i$ is a class with $dw \in p^r M^{i+1}$ 
such that $[w] \in \mathrm{H}^i( M/p^rM)$ maps to zero in $\mathrm{H}^i(
\WittScript_r(M))$, then 
we can  write $w =  dw' + V^r  u + dV^r v$ in $M^i$ for $u \in M^i, w', v
\in M^{i-1}$.  
It follows that $dV^r u \in p^r M^{i+1}$ and hence $du \in p^r M^{i+1}$,  so that we can write $ u = F^r u'$ by
saturation. 
Therefore, $w = d(w' + V^r v) + p^r u'$, implying that $[w] = 0$ in
$\mathrm{H}^i(M/p^r M)$, as desired. 
\end{remark}

\begin{corollary}\label{cor15}
Let $f\colon  M^{\ast} \rightarrow N^{\ast}$ be a morphism of saturated Dieudonn\'{e} complexes. The following conditions are equivalent:
\begin{itemize}
\item[$(1)$] The induced map $M^{\ast} / p M^{\ast} \rightarrow N^{\ast} / p N^{\ast}$ is a quasi-isomorphism of cochain complexes.
\item[$(2)$] The induced map $\WittScript_{1}(M)^{\ast} \rightarrow
\WittScript_1(N)^{\ast}$ is an isomorphism of cochain complexes.
\item[$(3)$] For each $r \geq 0$, the map $M^{\ast} / p^{r} M^{\ast}
\rightarrow N^{\ast} / p^{r} N^{\ast}$ is a quasi-isomorphism of cochain complexes.
\item[$(4)$] For each $r \geq 0$, the map $\WittScript_{r}(M)^{\ast}
\rightarrow \WittScript_{r}(N)^{\ast}$ is an isomorphism of cochain complexes.
\end{itemize}
\end{corollary}

\begin{proof}
The equivalences $(1) \Leftrightarrow (2)$ and $(3) \Leftrightarrow (4)$ follow from Proposition \ref{prop11}. The implication
$(3) \Rightarrow (1)$ is obvious, and the reverse implication follows by induction on $r$.
\end{proof}

\begin{proposition}\label{prop17}
Let $M^{\ast}$ be a saturated Dieudonn\'{e} complex. Then the map $\rho_{M}\colon  M^{\ast} \rightarrow \WittScript(M)^{\ast}$ induces isomorphisms
$\WittScript_{r}(M)^{\ast} \rightarrow \WittScript_{r}( \WittScript(M) )^{\ast}$ for $r \geq 0$.
\end{proposition}

\begin{proof}
Recall that $\WittScript(M)^{\ast}$ is saturated
(Propositions~\ref{buildtowerfromsat} and \ref{buildstrictDieudonne}). 
By virtue of Corollary \ref{cor15}, it will suffice to treat the case $r = 1$. Let $\gamma\colon  \WittScript_1(M)^{\ast} \rightarrow \WittScript_{1}( \WittScript(M) )^{\ast}$ be the map determined by
$\rho_{M}$. We first show that $\gamma$ is injective. Choose an element $\overline{x} \in \WittScript_1(M)^{\ast}$ satisfying $\gamma( \overline{x} ) = 0$; we wish to show
that $\overline{x} = 0$. Represent $\overline{x}$ by an element $x \in M^{\ast}$. The vanishing of $\gamma( \overline{x} )$ then implies that we can write
$\rho_{M}(x) = Vy + dVz$ for some $y \in \WittScript(M)^{\ast}$ and $z \in \WittScript(M)^{\ast-1}$. Then we can identify
$y$ and $z$ with compatible sequences $\overline{y}_m \in \WittScript_{m}(M)^{\ast}$ and $\overline{z}_{m} \in \WittScript_{m}(M)^{\ast-1}$. The equality
$\rho_M(x) = Vy + dVz$ then yields $\overline{x} = V(\overline{y}_0) + dV(\overline{z}_0) = 0$.

We now prove that $\gamma$ is surjective. Choose an element $e \in \WittScript_1( \WittScript(M) )^{\ast}$; we wish to show that 
$e$ belongs to the image of $\gamma$. Let $x \in \WittScript(M)^{\ast}$ be an element representing $e$, which we can identify
with a compatible sequence of elements $\overline{x}_{m} \in \WittScript_m(M)^{\ast}$ represented by elements $x_m \in M^{\ast}$.
The compatibility of the sequence $\{ \overline{x}_m \}_{m \geq 0}$ guarantees that we can write
$x_{m+1} = x_{m} + V^{m} y_m + d V^{m} z_m$ for some elements $y_{m} \in M^{\ast}$ and $z_m \in M^{\ast-1}$.
We then have the identity
$$ x = \rho_M(x_1) + V (\sum_{m \geq 0} V^{m} y_{m+1}) + dV ( \sum_{m \geq 0} V^{m} z_{m+1} )$$
in $\WittScript(M)^{\ast}$, which yields the identity $e = \gamma( \overline{x}_1 )$ in $\WittScript_1( \WittScript(M) )^{\ast}$.
\end{proof}

\begin{corollary}\label{realcor14}
Let $M^{\ast}$ be a saturated Dieudonn\'{e} complex. Then the completion $\WittScript(M)^{\ast}$ is a strict Dieudonn\'{e} complex.
\end{corollary}

\begin{proof}
It follows from Propositions~\ref{buildtowerfromsat} and \ref{buildstrictDieudonne} that
$\WittScript(M)^{\ast}$ is a saturated Dieudonn\'{e} complex. Let $\rho_M\colon  M^{\ast} \rightarrow \WittScript(M)^{\ast}$ be the canonical map, so that we have a pair of maps of Dieudonn\'{e} complexes
$$ \rho_{ \WittScript(M) }, \WittScript( \rho_M)\colon  \WittScript(M)^{\ast} \rightarrow \WittScript(\WittScript(M))^{\ast}.$$
We wish to show that $\rho_{ \WittScript(M) }$ is an isomorphism. It follows from Proposition \ref{prop17} that the map $\WittScript( \rho_M )$ is an isomorphism.
It will therefore suffice to show that $\rho_{ \WittScript(M)}$ and $\WittScript( \rho_M)$ coincide. Fix an element $x \in \WittScript(M)^{\ast}$; we wish to show that
$\rho_{ \WittScript(M) }(x)$ and $\WittScript( \rho_M )(x)$ have the same image in $\WittScript_{r}( \WittScript(M) )^{\ast}$ for each $r \geq 0$.
Replacing $x$ by an element of the form $x + V^{r} y + d V^{r} z$, we can assume
without loss of generality (by Proposition~\ref{prop17}) that $x = \rho_M( x' )$ for some $x' \in M^{\ast}$.
In this case, the desired result follows from the commutativity of the diagram
$$ \xymatrix{ M^{\ast} \ar[r]^{\rho_M} \ar[d]^{\rho_M} & \WittScript(M)^{\ast} \ar[d] \\
\WittScript_r(M)^{\ast} \ar[r] & \WittScript_{r}(\WittScript(M))^{\ast}.  } 
$$
\end{proof}

\begin{proposition}\label{prop37}
Let $M^{\ast}$ and $N^{\ast}$ be saturated Dieudonn\'{e} complexes, where $N^{\ast}$ is strict. Then composition with the map $\rho_{M}$ induces a bijection
$$ \theta\colon  \Hom_{ \FrobComp}( \WittScript(M)^{\ast}, N^{\ast} ) \rightarrow \Hom_{ \FrobComp}( M^{\ast}, N^{\ast} ).$$
\end{proposition}

\begin{proof}
We first show that $\theta$ is injective. Let $f\colon  \WittScript(M)^{\ast} \rightarrow N^{\ast}$ be a morphism of Dieudonn\'{e} complexes such that
$\theta(f) = f \circ \rho_M$ vanishes; we wish to show that $f$ vanishes. For each $r \geq 0$, we have a commutative diagram of cochain complexes
$$ \xymatrix@C=50pt{ M^{\ast} \ar[r]^-{\rho_M} \ar[d] & \WittScript(M)^{\ast} \ar[r]^-{f} \ar[d] & N^{\ast} \ar[d] \\
\WittScript_r(M)^{\ast} \ar[r]^-{ \WittScript_r(\rho_M)} & \WittScript_r( \WittScript(M) )^{\ast} \ar[r]^-{\WittScript_r(f)} & \WittScript_r(N)^{\ast} }$$
Since $\WittScript_r(f) \circ \WittScript_r( \rho_M ) = \WittScript_r( f \circ \rho_M)$ vanishes and $\WittScript_r( \rho_M )$ is an isomorphism (Proposition \ref{prop17}),
we conclude that $\WittScript_r(f) = 0$. It follows that the composite map $\WittScript(M)^{\ast} \xrightarrow{f} N^{\ast} \rightarrow \WittScript_r(N)^{\ast}$ vanishes
for all $r$. Invoking the strictness of $N^{\ast}$, we conclude that $f =0$.

We now argue that $\theta$ is surjective. Suppose we are given a map of Dieudonn\'{e} complexes $f_0\colon  M^{\ast} \rightarrow N^{\ast}$; we wish to show that $f_0$ belongs to the image of $\theta$. We have a commutative diagram of Dieudonn\'{e} complexes
$$ \xymatrix@C=50pt{ M^{\ast} \ar[r]^{f_0} \ar[d]^{ \rho_M} & N^{\ast} \ar[d]^{ \rho_N} \\
\WittScript(M)^{\ast} \ar[r]^{ \WittScript(f_0) } & \WittScript(N)^{\ast}, }$$
where the right vertical map is an isomorphism (since $N^{\ast}$ is strict). It follows that $f_0 = \theta( \rho_{N}^{-1} \circ \WittScript{f}_0 )$ belongs to the image of $\theta$, as desired.
\end{proof}

\begin{corollary}\label{cor66a}
The inclusion functor $\FrobCompComplete \hookrightarrow \FrobCompSat$ admits a
left adjoint, given by the completion functor $M^{\ast} \mapsto
\WittScript(M)^{\ast}$. \qed 
\end{corollary}

\subsection{Comparison of $M^{\ast}$ with $\WittScript(M)^{\ast}$}

We now record a few homological consequences of the results of \S \ref{lebex}.
As before, we write $\WittScript(\cdot)^{\ast}$ for the completion of a
saturated Dieudonn\'e complex. 

\begin{proposition}\label{prop18}
Let $M^{\ast}$ be a saturated Dieudonn\'{e} complex. Then the tautological map
$\rho_{M}\colon  M^{\ast} \rightarrow \WittScript(M)^{\ast}$ induces a quasi-isomorphism $M^{\ast} / p^{r} M^{\ast} \rightarrow
\WittScript(M)^{\ast} / p^{r} \WittScript(M)^{\ast}$ for every nonnegative integer $r$.
\end{proposition}

\begin{proof}
Combine Proposition \ref{prop17} with Corollary \ref{cor15}.
\end{proof}

\begin{corollary}
\label{derivedcompleteDC}
Let $M^{\ast}$ be a saturated Dieudonn\'{e} complex. Then the tautological map 
$\rho_{M}\colon  M^{\ast} \rightarrow \WittScript(M)^{\ast}$ exhibits $\WittScript(M)^{\ast}$ as a $p$-completion of $M^{\ast}$ in the derived category of abelian groups (see \S~\ref{S:derivedreview}).
\end{corollary}

\begin{proof}
By virtue of Proposition \ref{prop18}, it will suffice to show that $\WittScript(M)^{\ast}$ is a $p$-complete object of the derived category. In fact, something stronger is true:
each term in the cochain complex $\WittScript(M)^{\ast}$ is a $p$-complete
abelian group, since $\WittScript(M)^{n}$ is given by an inverse limit
$\varprojlim_r \WittScript_{r}(M)^{n}$,
where $\WittScript_r( M)^{n}$ is annihilated by $p^{r}$.
\end{proof}

\begin{corollary}
\label{cor66}
Let $M^{\ast}$ be a saturated Dieudonn\'{e} complex. The following conditions are equivalent:
\begin{itemize}
\item[$(1)$] The tautological map $\rho_M\colon  M^{\ast} \rightarrow \WittScript(M)^{\ast}$ is a quasi-isomorphism.
\item[$(2)$] The underlying cochain complex $M^{\ast}$ is a $p$-complete object
of the derived category of abelian groups. \qed
\end{itemize}
\end{corollary}

We now combine Corollary \ref{cor66} with our study of the saturation construction $M^{\ast} \mapsto \Saturate( M^{\ast} )$.

\begin{notation}\label{completesat}
Let $M^{\ast}$ be a Dieudonn\'{e} complex. We let $\WSaturate(M^{\ast})$ denote the completion of the saturation $\Saturate( M^{\ast} )$.
Note that the construction $M^{\ast} \mapsto \WSaturate(M^{\ast})$ is left adjoint to the inclusion functor $\FrobCompComplete \hookrightarrow \FrobComp$ 
(see Corollaries \ref{cor65} and \ref{cor66}).
\end{notation}

\begin{corollary}\label{cor67}
Let $M^{\ast}$ be a Dieudonn\'{e} complex of Cartier type, and suppose that each of the abelian groups $M^{n}$ is $p$-adically complete.
Then the canonical map $u\colon  M^{\ast} \rightarrow \WSaturate(M^{\ast} )$ is a quasi-isomorphism.
\end{corollary}

\begin{proof}
Note that the domain and codomain of $u$ are cochain complexes of abelian groups which are $p$-complete and $p$-torsion-free.
Consequently, to show that $u$ is a quasi-isomorphism, it will suffice to show that the induced map
$M^{\ast} / p M^{\ast} \rightarrow \WSaturate(M^{\ast} ) / p \WSaturate(M^{\ast} )$ is a quasi-isomorphism.
This map factors as a composition
$$M^{\ast} / p M^{\ast} \xrightarrow{v} \Saturate(M^{\ast}) / p \Saturate(M^{\ast}) \xrightarrow{v'} \WSaturate(M^{\ast} ) / p \WSaturate(M^{\ast} ).$$
The map $v$ is a quasi-isomorphism by virtue of Theorem \ref{theo60}, and the map $v'$ is a quasi-isomorphism by Proposition \ref{prop18}.
\end{proof}

\subsection{More on Strict Dieudonn\'{e} Towers}

We close this section by showing that the category $\FrobCompComplete$ of strict Dieudonn\'{e} complexes is equivalent
to the category $\FrobTow$ of strict Dieudonn\'{e} towers (Corollary \ref{Dieudtowerstrict}).

\begin{proposition} 
\label{strictTDconverse}
Let $\{ X^{\ast}_{r} \}_{r \geq 0}$ be a strict Dieudonn\'{e} tower and let
$X^{\ast} = \varprojlim_{r} X_{r}^{\ast}$ be its inverse limit, which
we regard as a Dieudonn\'{e} complex as in Proposition \ref{buildstrictDieudonne}.
Then, for each $r \geq 0$, the canonical map $X^{\ast} \rightarrow X_{r}^{\ast}$
induces an isomorphism of cochain complexes
$$ \WittScript_{r}( X )^{\ast} = X^{\ast} / (\im(V^r) + \im( dV^r) )
\xrightarrow{\sim} X_{r}^{\ast}.$$
\end{proposition} 

\begin{proof} 
Axiom $(2)$ of Definition \ref{stricTD} guarantees that the restriction maps in
the tower $\{ X_{k}^{\ast} \}_{k \geq r }$ are surjective, so that
the induced map $\theta\colon  X^{\ast} \rightarrow X_{r}^{\ast}$ is surjective.
Moreover, the vanishing of $X_0^{\ast}$ guarantees that
$\ker(\theta)$ contains $\im( V^{r} ) + \im( d V^r )$. 
We will complete the proof by verifying the reverse inclusion.
Fix an element $x \in \ker(\theta)$, which we can identify
with a compatible sequence of elements $\{ x_k \in X^{\ast}_{k+r} \}_{k \geq 0}$
satisfying $x_0 = 0$. We wish to show that $x$ can be written as a sum $V^{r} y + d V^{r} z$ for
some $y, z \in X^{\ast}$. Equivalently, we wish to show that we can choose
compatible sequences of elements $\{ y_k, z_k \in X^{\ast}_{k} \}$ satisfying
$x_k = V^{r} y_k + d V^{r} z_k$. We proceed by induction on $k$, the case $k=0$ being trivial.
To carry out the inductive step, let us assume that $k > 0$ and that we have
chosen $y_{k-1}, z_{k-1} \in X^{\ast}_{k-1}$ satisfying
$x_{k-1} = V^{r} y_{k-1} + d V^{r} z_{k-1}$. Using Axiom $(2)$ of Definition \ref{stricTD},
we can lift $y_{k-1}$ and $z_{k-1}$ to elements $\overline{y}_{k}, \overline{z}_{k}$
in $X^{\ast}_{k}$. We then have 
$$x_{k} = V^{r} \overline{y}_k + d V^{r} \overline{z}_k + e$$
where $e$ belongs to the kernel of the restriction map $X^{\ast}_{r+k} \rightarrow X^{\ast}_{r+k-1}$.
Applying axiom $(8)$ of Definition \ref{stricTD}, we can write $e = V^{r+k-1} e' + d V^{r+k-1} e''$
for some $e', e'' \in X^{\ast}_{1}$. We conclude the proof by setting $y_{k} = \overline{y}_k + V^{k-1} e'$
and $z_k = \overline{z}_k + V^{k-1} e''$.
\end{proof}

\begin{corollary}\label{cor101}
Let $\{ X^{\ast}_{r} \}_{r \geq 0}$ be a strict Dieudonn\'{e} tower. Then
there exists a saturated Dieudonn\'{e} complex $M^{\ast}$
and an isomorphism of strict Dieudonn\'{e} towers
$\{ X^{\ast}_{r} \}_{r \geq 0} \simeq \{ \WittScript_{r}( M)^{\ast} \}_{r \geq 0}$.
\end{corollary}

\begin{proof}
Combine Proposition \ref{strictTDconverse} with Proposition \ref{buildstrictDieudonne}.
\end{proof}

\begin{corollary}\label{cor102}
Let $\{ X^{\ast}_{r} \}_{r \geq 0}$ be a strict Dieudonn\'{e} tower.
Then the inverse limit $X^{\ast} = \varprojlim_{r} X_{r}^{\ast}$ is a strict
Dieudonn\'{e} complex.
\end{corollary}

\begin{proof}
By virtue of Corollary \ref{cor101}, we can assume that $\{ X^{\ast}_{r} \}_{r \geq 0} =
\{ \WittScript_{r}(M)^{\ast} \}_{r \geq 0}$, where $M^{\ast}$ is a saturated Dieudonn\'{e} complex.
In this case, the desired result follows from Corollary \ref{realcor14}.
\end{proof}

\begin{corollary} 
\label{Dieudtowerstrict}
The composite functor $\FrobCompComplete \hookrightarrow \FrobComp \to \FrobTow$
sending a strict Dieudonn\'e complex $M^{\ast}$ to the strict Dieudonn\'e tower
$\left\{\WittScript_r(M)^{\ast}\right\}_{r \geq 0}$ is an equivalence of categories, with inverse
equivalence given by $\{ X^{\ast}_{r} \}_{r \geq 0} \mapsto \varprojlim_{r} X^{\ast}_{r}$.
\end{corollary}

\begin{proof}
For any saturated Dieudonn\'{e} complex $M^{\ast}$ and every strict Dieudonn\'{e} tower
$\{ X_r^{\ast} \}_{r \geq 0}$, we have canonical maps
$$ M^{\ast} \rightarrow \varprojlim \WittScript_{r}(M)^{\ast} \quad \quad
\WittScript_{r'}( \varprojlim_{r} X_{r}^{\ast} ) \rightarrow X_{r'}^{\ast}.$$
The first of these maps is an isomorphism when $M^{\ast}$ is strict
(by definition), and the second for any strict Dieudonn\'{e} tower
(by virtue of Proposition \ref{strictTDconverse}).
\end{proof} 

\newpage
\section{Dieudonn\'{e} Algebras}\label{section3}

In this section, we introduce the notion of a {\it Dieudonn\'{e} algebra}. Roughly speaking, a Dieudonn\'{e} algebra
is a Dieudonn\'{e} complex $( A^{\ast}, d, F)$ which is equipped with a ring structure which is compatible
with the differential $d$ and the Frobenius operator $F$ (see Definition~\ref{def20}). We show that if
$A$ is a $p$-torsion-free ring equipped with a lift of the Frobenius $\varphi_{
A/pA}\colon  A/pA \to A/pA$, then the absolute de~Rham complex
$\Omega^{\ast}_{A}$ inherits the structure of a Dieudonn\'{e} algebra (Proposition~\ref{prop42}). We also show that
the saturation and completed saturation constructions of \S \ref{section2} have counterparts in the setting of
Dieudonn\'{e} algebras.

\subsection{The Category $\FrobAlg$}\label{frobalgsec}

Recall that a {\it commutative differential graded algebra} is a cochain complex $(A^{\ast}, d)$ which is equipped
with the structure of a graded ring satisfying the following conditions:
\begin{itemize}
\item[$(a)$] The multiplication on $A^{\ast}$ is graded-commutative: that is, for every pair of elements $x \in A^{m}$ and $y \in A^{n}$,
we have $xy = (-1)^{mn} yx$ in $A^{m+n}$.

\item[$(b)$] If $x \in A^{n}$ is a homogeneous element of odd degree, then $x^2$ vanishes in $A^{2n}$.

\item[$(c)$] The differential $d$ satisfies the Leibniz rule $d(xy) = (dx) y + (-1)^{m} x(dy)$ for $x \in A^{m}$.
\end{itemize}

\begin{remark}
For each homogeneous element $x \in A^{n}$ of odd degree, assumption $(a)$ implies that $x^2 = - x^2$. Consequently,
assumption $(b)$ is automatic if $A^{\ast}$ is $2$-torsion free.
\end{remark}

\begin{definition}\label{def20}
A {\it Dieudonn\'{e} algebra} is a triple $( A^{\ast}, d, F )$, where $( A^{\ast}, d)$ is a commutative differential graded algebra and
$F\colon  A^{\ast} \rightarrow A^{\ast}$ is a homomorphism of graded rings
satisfying the following additional conditions: 
\begin{itemize}
\item[$(i)$] For each $x \in A^{\ast}$, we have $dF(x) = p F(dx)$ for $x \in A^{\ast}$.
\item[$(ii)$] The groups $A^{n}$ vanish for $n < 0$.
\item[$(iii)$] For $x \in A^{0}$, we have $Fx \equiv x^{p} \pmod{p}$.
\end{itemize}

Let $( A^{\ast}, d, F)$ and $(A'^{\ast}, d', F' )$ be Dieudonn\'{e} algebras. A {\it morphism of Dieudonn\'{e} algebras}
from $(A^{\ast}, d, F)$ to $(A'^{\ast}, d', F')$ is a graded ring homomorphism $f\colon  A^{\ast} \rightarrow A'^{\ast}$
satisfying $d' \circ f = f \circ d$ and $F' \circ f = f \circ F$. We let $\FrobAlg$ denote the category whose objects are Dieudonn\'{e} algebras
and whose morphisms are Dieudonn\'{e} algebra morphisms.
\end{definition}

\begin{remark}
In the situation of Definition \ref{def20}, we will generally abuse terminology and simply refer to $A^{\ast}$ as a Dieudonn\'{e} algebra.
\end{remark}

\begin{remark}
There is an evident forgetful functor $\FrobAlg \rightarrow \FrobComp$, which carries a Dieudonn\'{e} algebra
$A^{\ast}$ to its underlying Dieudonn\'{e} complex (obtained by neglecting the multiplication on $A^{\ast}$).
\end{remark}

\begin{remark}\label{algdesc}
Let us regard the category of $\FrobComp$ of Dieudonn\'{e} complexes as a symmetric monoidal category
with respect to the tensor product described in Remark \ref{tprop}. Then we can identify the category
$\FrobAlg$ of Dieudonn\'{e} algebras with a full subcategory of the category $\CAlg( \FrobComp )$ of commutative
algebra objects of $\FrobComp$. A commutative algebra object $A^{\ast}$ of $\FrobComp$ is a Dieudonn\'{e} algebra
if and only if it satisfies the following conditions:
\begin{itemize}
\item[$(1)$] The groups $A^{n}$ vanish for $n < 0$.
\item[$(2)$] For $x \in A^{0}$, we have $Fx \equiv x^{p} \pmod{p}$.
\item[$(3)$] Every homogeneous element $x \in A^{\ast}$ of odd degree satisfies $x^2 = 0$.
\end{itemize}
Note that condition $(3)$ is automatic if the Dieudonn\'{e} complex $A^{\ast}$ is strict,
since it is $2$-torsion free by virtue of Remark \ref{olos}.
\end{remark}

\begin{remark}\label{rem30}
Let $A^{\ast}$ be a commutative differential graded algebra, and suppose that $A^{\ast}$ is $p$-torsion-free.
Let $\eta_p A^{\ast} \subseteq A^{\ast}[ p^{-1} ]$ denote the subcomplex introduced in Remark \ref{rem3}. Then $\eta_{p} A^{\ast}$ is closed under multiplication, and
therefore inherits the structure of a commutative differential graded algebra. Moreover, the dictionary of Remark \ref{rem3} establishes a bijection between the following data:
\begin{itemize}
\item Maps of differential graded algebras $\alpha\colon  A^{\ast} \rightarrow \eta_{p} A^{\ast}$
\item Graded ring homomorphisms $F\colon  A^{\ast} \rightarrow A^{\ast}$ satisfying $d \circ F = p (F \circ d)$.
\end{itemize}
\end{remark}

\begin{remark}\label{rem34}
Let $(A^{\ast}, d, F)$ be a Dieudonn\'{e} algebra, and suppose that $A^{\ast}$ is $p$-torsion-free. Then we can regard $\eta_p A^{\ast}$
as a subalgebra of $A^{\ast}$, which is closed under the action of $F$. We claim that $( \eta_p A^{\ast}, d|_{ \eta_p A^{\ast} }, F|_{ \eta_p A^{\ast} } )$ is
also a Dieudonn\'{e} algebra. The only nontrivial point is to check condition $(iii)$ of Definition \ref{def20}. Suppose we are given an element
$x \in (\eta_p A)^{0}$: that is, an element $x \in A^{0}$ such that $dx = py$ for some $y \in A^{1}$. Since $A^{\ast}$ satisfies condition $(iii)$ of Definition \ref{def20}, we can write
$Fx = x^{p} + p z$ for some element $z \in A^{0}$. We wish to show that $z$ also
belongs to $(\eta_p A)^{0}$:   that is, $dz \in p A^{1}$. This follows from the calculation
$$ p( dz ) = d(pz) = d( Fx - x^{p} ) = p F(dx) - p x^{p-1} dx = p^2 F(y) - p^2 x^{p-1} y,$$
since $A^{1}$ is $p$-torsion-free.
\end{remark}

\begin{example}\label{ex24}
Let $R$ be a commutative $\F_p$-algebra and let $W(R)$ denote the ring of Witt vectors of $R$, regarded as a commutative differential graded algebra which is concentrated in degree zero.
Then the Witt vector Frobenius $F \colon W(R) \rightarrow W(R)$ exhibits $W(R)$ as a Dieudonn\'{e} algebra, in the sense of Definition \ref{def20}. The Dieudonn\'{e} algebra $W(R)$ is saturated (as a Dieudonn\'{e} complex) if and only if the $\F_p$-algebra $R$ is perfect. If this condition is satisfied, then $W(R)$ is also strict (as a Dieudonn\'{e} complex).
\end{example}

\subsection{Example: The de~Rham Complex}

We now consider an important special class of Dieudonn\'{e} algebras.

\begin{proposition}\label{prop42}
Let $R$ be a commutative ring which is $p$-torsion-free, and let $\varphi\colon  R \rightarrow R$ be a ring homomorphism satisfying 
$\varphi(x) \equiv x^{p} \pmod{p}$. Then there is a unique ring homomorphism $F\colon  \Omega^{\ast}_{R} \rightarrow \Omega^{\ast}_{R}$
with the following properties:
\begin{itemize}
\item[$(1)$] For each element $x \in R = \Omega^{0}_{R}$, we have $F(x) = \varphi(x)$.
\item[$(2)$] For each element $x \in R$, we have $F(dx) = x^{p-1} dx + d( \frac{ \varphi(x) - x^p}{p} )$
\end{itemize}
Moreover, the triple $(\Omega^{\ast}_{R}, d, F)$ is a Dieudonn\'{e} algebra.
\end{proposition}

\begin{remark}
In the situation of Proposition \ref{prop42}, suppose that the de~Rham complex $\Omega^{\ast}_{R}$ is $p$-torsion-free,
and regard $\eta_p \Omega^{\ast}_{R}$ as a (differential graded) subalgebra of $\Omega^{\ast}_{R}$. 
The ring homomorphism $\varphi$ extends uniquely to a map of commutative differential graded algebras
$\alpha\colon  \Omega^{\ast}_{R} \rightarrow \Omega^{\ast}_{R}$. The calculation
$$ \alpha( dx ) = d \alpha(x) = d \varphi(x) \equiv dx^{p} = p x^{p-1} dx \equiv 0 \pmod{p}$$
shows that $\alpha$ carries $\Omega^{n}_{R}$ into $p^{n} \Omega^{n}_{R}$, and can therefore be regarded as a map of differential graded algebras
$\Omega^{\ast}_{R} \rightarrow \eta_p \Omega^{\ast}_{R}$. Applying Remark \ref{rem30}, we see that there is a unique ring homomorphism
$F\colon  \Omega^{\ast}_{R} \rightarrow \Omega^{\ast}_{R}$ which satisfies condition $(1)$ of Proposition \ref{prop42} and endows $\Omega^{\ast}_{R}$
with the structure of a Dieudonn\'{e} algebra. Moreover, the map $F$ automatically satisfies condition $(2)$ of Proposition \ref{prop42}: this follows
from the calculation
$$p F(dx)  =  d( Fx )  =  d( \varphi(x) ) 
 =  d( x^{p} + p \tfrac{ \varphi(x) - x^{p} }{p}) 
 =  p (x^{p-1} dx + d \textstyle \frac{ \varphi(x) - x^p}{p} ). $$
\end{remark}

\begin{proof}[Proof of Proposition \ref{prop42}]
The uniqueness of the homomorphism $F$ is clear, since the de~Rham complex $\Omega^{\ast}_{R}$ is generated (as a graded ring) by elements
of the form $x$ and $dx$, where $x \in R$. We now prove existence. Define $\delta\colon  R \rightarrow R$ by the formula
$\delta(x) = \frac{ \varphi(x)- x^{p} }{p}$. A simple calculation then gives
\begin{gather}\label{eqn1}
\delta(x+y) = \delta(x) + \delta(y) - \sum_{0 < i < p } \frac{ (p-1)!}{ i!
(p-i)!} x^{i} y^{p-i}, \\
\label{eqn2}
\delta(xy) = \varphi(x) \delta(y) + \delta(x) \varphi(y) - p \delta(x)
\delta(y). \end{gather}
Consider the map $\rho\colon  R \rightarrow \Omega^{1}_R$ given by the formula
$\rho(x) = x^{p-1} dx + d \delta(x)$. We first claim that $\rho$ is a group homomorphism. This follows from the calculation
\begin{eqnarray*}
\rho(x+y) & = & (x+y)^{p-1} d(x+y) + d \delta(x+y) \\
& = & (x+y)^{p-1} d(x+y) + d \delta(x) + d \delta(y) - d ( \sum_{ 0 < i < p }
\tfrac{ (p-1)!}{ i! (p-i)!} x^{i} y^{p-i} ) \\
& = & \rho(x) + \rho(y) + ( (x+y)^{p-1} - x^{p-1} - \sum_{0 < i < p} \tfrac{ (p-1)!}{ (i-1)! (p-i)!} x^{i-1} y^{p-i} ) dx \\
& & +  ( (x+y)^{p-1} - y^{p-1} - \sum_{0 < i < p} \tfrac{ (p-1)!}{ i! (p-i-1)!} x^{i} y^{p-i-1} ) dy \\
& = & \rho(x) + \rho(y). \end{eqnarray*}

We next claim that $\rho$ is a $\varphi$-linear derivation from $R$ into $\Omega^{1}_{R}$: that is,
it satisfies the identity $\rho(xy) = \varphi(y) \rho(x) + \varphi(x) \rho(y)$. This follows from the calculation
\begin{eqnarray*}
\rho(xy) & = & (xy)^{p-1} d(xy) + d \delta(xy) \\
& = & (xy)^{p-1} d(xy) + d( \varphi(x) \delta(y) + \delta(x) \varphi(y) - p \delta(x) \delta(y) ) \\
& = & (x^{p-1} y^{p} dx + \delta(y) d \varphi(x) + \varphi(y) d \delta(x) - p \delta(y) d \delta(x) ) \\
& & + (x^{p} y^{p-1} dy + \varphi(x) d \delta(y) + \delta(x) d \varphi(y) - p \delta(x) d \delta(y) ) \\
& = & ( (y^{p} + p \delta(y)) x^{p-1} dx + \varphi(y) d \delta(x) ) + ( (x^{p} + p \delta(x)) y^{p-1} dy + \varphi(x) d \delta(y) ) \\ 
& = & \varphi(y) (x^{p-1} dx + d \delta(x) ) + \varphi(x) ( y^{p-1} dy + d \delta(y) ) \\
& = & \varphi(y) \rho(x) + \varphi(x) \rho(y).\end{eqnarray*}

Invoking the universal property of the module of K\"{a}hler differentials $\Omega^{1}_{R}$, we deduce that
there is a unique $\varphi$-semilinear map $F\colon  \Omega^{1}_{R} \rightarrow \Omega^{1}_{R}$
satisfying $F(dx) = \rho(x)$ for $x \in R$. For every element $\omega \in \Omega^{1}_{R}$, we have
$F(\omega)^2 = 0$ in $\Omega^{\ast}_{R}$ (since $F(\omega)$ is a $1$-form), so $F$ extends
to a ring homomorphism $\Omega^{\ast}_{R} \rightarrow \Omega^{\ast}_{R}$ satisfying $(1)$ and $(2)$.

To complete the proof, it will suffice to show that $F$ endows $\Omega^{\ast}_{R}$ with the structure of a Dieudonn\'{e} algebra.
It follows from $(1)$ that $F(x) \equiv x^{p} \pmod{p}$ for $x \in \Omega^{0}_{R}$. It will therefore suffice to show
that $d \circ F = p (F \circ d)$. Let $A \subseteq \Omega^{\ast}_{R}$ be the subset consisting of those elements
$\omega$ which satisfy $d F(\omega) = p F( d\omega)$. It is clear that $A$ is a graded subgroup of
$\Omega^{\ast}_{R}$. Moreover, $A$ is closed under multiplication: if $\omega \in \Omega^{m}_{R}$ and
$\omega' \in \Omega^{n}_{R}$ belong to $A$, then we compute
\begin{eqnarray*}
d F( \omega \wedge \omega' ) & = & d( F(\omega) \wedge F(\omega') ) \\
& = & ( (dF\omega ) \wedge F\omega') + (-1)^{m} (F\omega \wedge (dF \omega' )) \\
& = & p( (Fd \omega) \wedge F\omega') + (-1)^{m} (F\omega \wedge (Fd\omega')) \\
& = & pF( d \omega \wedge \omega' + (-1)^{m} \omega \wedge d \omega' ) \\
& = & pFd( \omega \wedge \omega' ).\end{eqnarray*}
Consequently, to show that $A = \Omega^{\ast}_{R}$, it will suffice to show that $A$ contains $x$ and $dx$ for
each $x \in R$. The inclusion $dx \in A$ is clear (note that $Fd dx = 0 = d
F(dx)$, since $F(dx) = x^{p-1} dx + d \delta(x) $ is a closed $1$-form). 
The inclusion $x \in A$ follows from the calculation
\begin{align*}
dFx  & =  d \varphi(x) 
 =  d( x^{p} + p \delta(x) ) \\
& =  p x^{p-1} dx + p d \delta(x) 
 =  p ( x^{p-1} dx + d \delta(x) ) \\
& =  p F(dx).  \qedhere \end{align*}
\end{proof}

The Dieudonn\'{e} algebra of Proposition \ref{prop42} enjoys the following universal property:

\begin{proposition}\label{prop39}
Let $R$ be a commutative ring which is $p$-torsion-free, let $\varphi\colon  R \rightarrow R$ be a ring homomorphism satisfying 
$\varphi(x) \equiv x^{p} \pmod{p}$, and regard the de~Rham complex $\Omega^{\ast}_{R}$ as equipped with the Dieudonn\'{e} algebra
structure of Proposition \ref{prop42}. Let $A^{\ast}$ be a Dieudonn\'{e} algebra which is $p$-torsion-free.
Then the restriction map
$$ \Hom_{ \FrobAlg}( \Omega^{\ast}_{R}, A^{\ast} ) \rightarrow \Hom( R, A^{0} )$$
is injective, and its image consists of those ring homomorphisms $f\colon  R \rightarrow A^{0}$ satisfying $f( \varphi(x) ) = F( f(x) )$ for $x \in R$.
\end{proposition}

\begin{proof}
Let $f\colon  R \rightarrow A^{0}$ be a ring homomorphism. Invoking the universal property of the de~Rham complex $\Omega^{\ast}_{R}$, we see that
$f$ extends uniquely to a homomorphism of differential graded algebras $\overline{f}\colon  \Omega^{\ast}_{R} \rightarrow A^{\ast}$. To complete the proof, it will suffice
to show that $\overline{f}$ is a map of Dieudonn\'{e} algebras if and only if $f$ satisfies the identity $f( \varphi(x) ) = F( f(x) )$ for each $x \in R$. The
``only if'' direction is trivial. Conversely, suppose that $f( \varphi(x) ) = F( f(x) )$ for $x \in R$; we wish to show that
$\overline{f}( F \omega ) = F \overline{f}(\omega)$ for each $\omega \in \Omega^{\ast}_{R}$. The collection of those
elements $\omega \in \Omega^{\ast}_{R}$ which satisfy this identity form a subring of $\Omega^{\ast}_{R}$. Consequently,
we may assume without loss of generality that $\omega = x$ or $\omega = dx$, for some $x \in R$. In the first case,
the desired result follows from our assumption. To handle the second, we compute
\begin{align*}
p \overline{f}( F dx ) & =  p \overline{f} ( x^{p-1} dx + d (\tfrac{
\varphi(x) - x^{p} }{p}) ) \\
& =  \overline{f}( p x^{p-1} dx + d( \varphi(x) - x^{p} ) ) \\
& =  \overline{f}( d \varphi(x) ) 
 =  d \overline{f}( \varphi(x) ) \\
& =  d F f(x) 
 =  p F d(f(x) ).
\end{align*}
Since $A^{1}$ is $p$-torsion-free, it follows that $\overline{f}( F dx ) = F d \overline{f}(x)$, as desired.
\end{proof}

\subsection{The Cartier Isomorphism}

We will need a variant of Proposition~\ref{prop42} for {\it completed} de~Rham complexes. 

\begin{variant}[Completed de~Rham Complexes]\label{var46}
Let $R$ be a commutative ring. We let $\widehat{ \Omega}^{\ast}_{R}$ denote the inverse limit 
$$\varprojlim_{n} \left( \Omega^{\ast}_{R} / p^{n} \Omega^{\ast}_{R} \right) \simeq \varprojlim_{n} \Omega^{\ast}_{ R / p^n R}.$$
We will refer to $\widehat{\Omega}^{\ast}_{R}$ as the {\it completed de~Rham complex} of $R$. Note that there is a unique
multiplication on $\widehat{\Omega}^{\ast}_{R}$ for which the tautological map $\Omega^{\ast}_{R} \rightarrow \widehat{ \Omega}^{\ast}_{R}$ is
a morphism of differential graded algebras.

Now suppose that $R$ is $p$-torsion-free, and that we are given a ring homomorphism $\varphi\colon  R \rightarrow R$
satisfying $\varphi(x) \equiv x^{p} \pmod{p}$. Then the homomorphism $F\colon  \Omega^{\ast}_{R} \rightarrow \Omega^{\ast}_{R}$
of Proposition \ref{prop42} induces a map $\widehat{\Omega}^{\ast}_{R} \rightarrow \widehat{\Omega}^{\ast}_{R}$, which we will also denote by $F$,
and which endows $\widehat{ \Omega}^{\ast}_{R}$ with the structure of a Dieudonn\'{e} algebra. Moreover, this Dieudonn\'{e} algebra enjoys the following
universal property: if $A^{\ast}$ is any Dieudonn\'{e} algebra which is $p$-adically complete and $p$-torsion-free,
then the restriction map
$$ \Hom_{ \FrobAlg}( \widehat{\Omega}^{\ast}_{R}, A^{\ast} ) \rightarrow \Hom( R, A^{0} )$$
is an injection, whose image is the collection of ring homomorphisms $f\colon  R \rightarrow A^{0}$ satisfying
$f \circ \varphi = F \circ f$. This follows immediately from Proposition \ref{prop39}.
\end{variant}

\begin{remark}
Let $R$ be a commutative ring. Then the quotient $\widehat{\Omega}^{\ast}_{R} / p^{n} \widehat{\Omega}^{\ast}_{R}$ can be identified with the de~Rham complex
$\Omega^{\ast}_{ R/p^n R }$.
\end{remark}

\begin{remark}
\label{rem:freenessinsmoothcase}
Let $R$ be a commutative ring which is $p$-torsion-free. Suppose that there exists a perfect $\F_p$-algebra $k$ such that $R/pR$ is a smooth $k$-algebra.
Then each $R / p^n R$ is also a smooth $W_{n}(k)$-algebra, by the universal
lifting property of Witt vectors (cf. \cite[Ch. II, Sec. 5--6]{SerreLF}). 
Our assumption that $k$ is perfect guarantees that the quotient $\widehat{\Omega}^{\ast}_{R} / p^{n} \widehat{\Omega}^{\ast}_{R}$ can
be identified with the de~Rham complex of $R / p^{n} R$ {\em relative} to $W_n(k)$. It follows that each $\widehat{\Omega}^{i}_R / p^n \widehat{\Omega}^{i}_{R}$ is a projective
$R/p^n R $-module of finite rank, and that the transition maps
$$(\widehat{\Omega}^{i}_R / p^n \widehat{\Omega}^{i}_{R}) \otimes_{ \Z / p^{n} \Z} (\Z / p^{n-1} \Z) \rightarrow (\widehat{\Omega}^{i}_R / p^{n-1} \widehat{\Omega}^{i}_{R})$$
are isomorphisms. From this, we conclude that each $\widehat{\Omega}^{i}_{R}$ is a projective module of finite rank over the completion $\widehat{R} = \varprojlim R/p^{n} R$.
In particular, $\widehat{\Omega}^{\ast}_{R}$ is $p$-torsion-free.
\end{remark}

We now recall the classical {\it Cartier isomorphism}, which will be useful for studying the completed de~Rham complexes of
Variant \ref{var46}.

\begin{proposition}\label{cartmapexist}
Let $A$ be a commutative $\F_p$-algebra. Then there is a unique homomorphism of graded algebras
$\Cart\colon  \Omega^{\ast}_{A} \rightarrow \mathrm{H}^{\ast}( \Omega^{\ast}_{A} )$ satisfying
$$ \Cart(x) = x^{p} ,\quad \quad \Cart(dy) = [y^{p-1} dy],$$
where $[y^{p-1} dy]$ denotes the cohomology class of the cocycle $y^{p-1} dy \in \Omega^{1}_{A}$.
\end{proposition}

\begin{proof}
It will suffice to show that the construction $(y \in A) \mapsto [ y^{p-1} dy ] \in \mathrm{H}^{1}( \Omega^{\ast}_A )$
is a derivation with respect to the Frobenius map $A \rightarrow \mathrm{H}^{0}( \Omega^{\ast}_{A} )$. 
This follows from the calculations
\begin{align*}
(xy)^{p-1} d(xy) & =  x^{p} (y^{p-1} dy) + y^{p} (x^{p-1} dx) \\
(x+y)^{p-1} d(x+y) & =  x^{p-1} dx + y^{p-1} dy + d ( \sum_{0 < i < p}
\tfrac{(p-1)!}{i! (p-i)!} x^{i} y^{p-i} ). \qedhere
\end{align*}
\end{proof}

We will refer to the homomorphism $\Cart$ of Proposition \ref{cartmapexist} as the {\it Cartier map}.

\begin{example}\label{zisp}
Let $R$ be a commutative ring which is $p$-torsion-free, let $\varphi\colon  R \rightarrow R$ be a ring homomorphism satisfying
$\varphi(x) \equiv x^{p} \pmod{p}$, and regard $\widehat{\Omega}^{\ast}_{R}$ as a Dieudonn\'{e} algebra as in Variant \ref{var46}.
Then the Frobenius map $F\colon  \widehat{\Omega}^{\ast}_{R} \rightarrow \widehat{\Omega}^{\ast}_{R}$ induces a map of graded
rings 
$$ \Omega^{\ast}_{R/pR} \simeq \widehat{\Omega}^{\ast}_{R} / p \widehat{\Omega}^{\ast}_{R} \rightarrow
\mathrm{H}^{\ast}( \widehat{\Omega}^{\ast}_{R} / p \widehat{\Omega}^{\ast}_{R}  ) \simeq \mathrm{H}^{\ast}( \Omega^{\ast}_{ R/ p } ).$$
Using the formulae of Proposition \ref{prop42}, we see that this map is given concretely by the formula
$$ x_0 dx_1 \wedge \cdots \wedge dx_n \mapsto [ x_0^{p} (x_1^{p-1} dx_1) \wedge \cdots \wedge ( x_n^{p-1} dx_n) ],$$
and therefore coincides with the Cartier map of Proposition \ref{cartmapexist}. In particular, it does not depend on the choice
of $\varphi$.
\end{example}

We refer the reader to \cite[Th. 7.2]{Katz} for a proof of the following:

\begin{theorem}[Cartier Isomorphism]\label{theo71}
Let $k$ be a perfect $\F_p$-algebra and let $A$ be a smooth $k$-algebra. Then the Cartier map
$\Cart\colon  \Omega^{\ast}_{A} \rightarrow \mathrm{H}^{\ast}( \Omega^{\ast}_{A} )$ is an isomorphism.
\end{theorem}

\begin{remark}
By Popescu's smoothing theorem \cite[Tag 07GB]{stacks-project}, 
it follows from 
Theorem \ref{theo71}
that 
the Cartier map $\Cart\colon  \Omega^{\ast}_{A} \rightarrow \mathrm{H}^{\ast}( \Omega^{\ast}_{A} )$ is an
isomorphism whenever $A$ is a regular Noetherian $\F_p$-algebra. 
In \S \ref{sec9sub6}, we will give a proof of this fact (Theorem
\ref{avoidpop}) which is independent of Popescu's theorem.
\end{remark}

\begin{corollary}\label{cor68}
Let $R$ be a commutative ring which is $p$-torsion-free and let $\varphi\colon  R \rightarrow R$ be a ring homomorphism satisfying
$\varphi(x) \equiv x^{p} \pmod{p}$. Suppose that there exists a perfect $\F_p$-algebra $k$ such that $R/pR$ is a smooth algebra over $k$.
Then:
\begin{itemize}
\item[$(1)$] When regarded as a Dieudonn\'{e} complex, the completed de~Rham complex $\widehat{\Omega}^{\ast}_{R}$ is of Cartier type
(in the sense of Definition \ref{def:Cartiertype}).
\item[$(2)$] The canonical map $\widehat{ \Omega}^{\ast}_{R} \rightarrow \WSaturate( \widehat{\Omega}^{\ast}_{R} )$ is a quasi-isomorphism of chain complexes.
\end{itemize}
\end{corollary}

\begin{proof}
Assertion $(1)$ follows from Theorem \ref{theo71} and Example
\ref{zisp}; note also that 
the completed de Rham complex $\widehat{\Omega}^{\ast}_R$ is $p$-torsion-free
by Remark~\ref{rem:freenessinsmoothcase}. 
Assertion $(2)$ follows from $(1)$ and Corollary \ref{cor67}.
\end{proof}

\subsection{Saturated Dieudonn\'{e} Algebras}

\begin{definition}
Let $A^{\ast}$ be a Dieudonn\'{e} algebra. We will say that $A^{\ast}$ is {\it
saturated} if it is saturated when regarded as a Dieudonn\'{e} complex:
that is, if it is $p$-torsion-free and the map
$F\colon  A^{\ast} \rightarrow \{ x \in A^{\ast}: d x \in p A^{\ast} \}$ is a bijection. We let $\FrobAlgSat$ denote the full subcategory of $\FrobAlg$ spanned by the saturated Dieudonn\'{e} algebras.
\end{definition}

In the setting of saturated Dieudonn\'{e} algebras, axiom $(iii)$ of Definition \ref{def20} can be slightly weakened:

\begin{proposition}\label{prop35}
Let $A^{\ast}$ be a saturated Dieudonn\'{e} complex satisfying $A^{\ast} = 0$ for $\ast < 0$. Suppose
that $A^{\ast}$ is equipped with the structure of a commutative differential graded algebra for which the Frobenius map
$F\colon  A^{\ast} \rightarrow A^{\ast}$ is a ring homomorphism. The following conditions are equivalent:
\begin{itemize}
\item[$(a)$] The complex $A^{\ast}$ is a Dieudonn\'{e} algebra: that is, each element $x \in A^{0}$ satisfies
$Fx \equiv x^{p} \pmod{p}$.
\item[$(b)$] Each element $x \in A^{0}$ satisfies $Fx \equiv x^{p} \pmod{VA^{0}}$. 
\end{itemize}
\end{proposition}

\begin{proof}
The implication $(a) \Rightarrow (b)$ is obvious. For the converse, suppose that $(b)$ is satisfied
and let $x \in A^{0}$. Then we can write $Fx = x^{p} + Vy$ for some $y \in A^{0}$. Applying the differential $d$,
we obtain $d(Vy) = d( Fx - x^{p} ) = p (F(dx) - x^{p-1} dx) \in p A^{1}$. Invoking Lemma \ref{lem6}, we can write
$y = Fz$ for some $z \in A^{0}$, so that $Fx = x^{p} + Vy = x^{p} + VFz = x^{p} + pz \equiv x^{p} \pmod{p}$.
\end{proof}

Let $f\colon  A^{\ast} \rightarrow B^{\ast}$ be a morphism of Dieudonn\'{e} algebras. We will say that $f$ {\it exhibits $B^{\ast}$ as a saturation of $A^{\ast}$}
if $B^{\ast}$ is saturated and, for every saturated Dieudonn\'{e} algebra $C^{\ast}$, composition with $f$ induces a bijection
$$ \Hom_{ \FrobAlg}( B^{\ast}, C^{\ast} ) \rightarrow \Hom_{ \FrobAlg}( A^{\ast}, C^{\ast} ).$$

\begin{proposition}\label{prop33}
Let $A^{\ast}$ be a Dieudonn\'{e} algebra. Then there exists a map of Dieudonn\'{e} algebras $f\colon  A^{\ast} \rightarrow B^{\ast}$ which exhibits
$B^{\ast}$ as a saturation of $A^{\ast}$. Moreover, $f$ induces an isomorphism of Dieudonn\'{e} complexes $\Saturate( A^{\ast} ) \rightarrow B^{\ast}$,
where $\Saturate( A^{\ast} )$ is the saturation of \S \ref{substrict}.
\end{proposition}

\begin{remark}
We can summarize Proposition \ref{prop33} more informally as follows: if $A^{\ast}$ is a Dieudonn\'{e} algebra and
$\Saturate( A^{\ast} )$ is the saturation of $A^{\ast}$ in the category
of Dieudonn\'{e} complexes, then $\Saturate( A^{\ast})$ inherits the structure of a Dieudonn\'{e} algebra,
and is also a saturation of $A^{\ast}$ in the category of Dieudonn\'{e} algebras.
\end{remark}

\begin{proof}[Proof of Proposition \ref{prop33}]
Replacing $A^{\ast}$ by the quotient $A^{\ast} / A^{\ast} [ p^{\infty} ]$, we can reduce to the case where $A^{\ast}$ is $p$-torsion-free.
In this case, the Frobenius on $A$ determines a map of Dieudonn\'{e} algebras
$\alpha_F\colon  A^{\ast} \rightarrow (\eta_p A)^{\ast}$ (Remarks \ref{rem30}
and  \ref{rem34}), and
$A^{\ast}$ is saturated if and only if $\alpha_{F}$ is an isomorphism. The saturation of $A$ can then be described as the direct limit of the sequence
$$ A^{\ast} \xrightarrow{ \alpha_F } (\eta_p A)^{\ast} \xrightarrow{ \eta_p( \alpha_F ) } (\eta_p \eta_p A)^{\ast} \xrightarrow{ \eta_p( \eta_p( \alpha_F ) )} ( \eta_p \eta_p \eta_p A)^{\ast} \rightarrow \cdots, $$
which is also the saturation of $A^{\ast}$ as a Dieudonn\'{e} complex.
\end{proof}

\subsection{Completions of Saturated Dieudonn\'{e} Algebras}

\begin{proposition}\label{proposition.dgraded}
Let $A^{\ast}$ be a saturated Dieudonn\'{e} algebra, and let $V\colon  A^{\ast} \rightarrow A^{\ast}$ be the Verschiebung map of
Remark \ref{rem5}. Then:
\begin{itemize}
\item[$(i)$] The map $V$ satisfies the projection formula $x V(y) = V( F(x) y )$.
\item[$(ii)$] For each $r \geq 0$, the sum $\im( V^{r} ) + \im( dV^r ) \subseteq A^{\ast}$ is a (differential graded) ideal.
\end{itemize}
\end{proposition}

\begin{proof}
We first prove $(i)$. Given $x,y \in A^{\ast}$, we compute
$$F( x V(y) ) = F(x) F(V(y)) = p F(x) y = F( V( F(x) y).$$
Since the map $F$ is injective, we conclude that $x V(y) = V( F(x) y)$.

We now prove $(ii)$. Let $x$ be an element of $A^{\ast}$; we wish to show that multiplication by $x$ carries
$\im( V^{r} ) + \im( dV^{r} )$ into itself. This follows from the identities
$$ x V^{r}(y) = V^{r}( F^r(x) y )  $$
$$(-1)^{\mathrm{deg} x}x d(V^{r} y) = d( x V^{r}(y) ) - (dx) V^{r}(y) = d
\left( V^{r}(
F^{r}(x) y )\right) -
V^{r}( F^{r}(dx) y ). \qedhere$$
\end{proof}

\begin{corollary}
Let $A^{\ast}$ be a saturated Dieudonn\'{e} algebra. For each $r \geq 0$, there is a unique ring structure on the quotient
$\WittScript_{r}( A)^{\ast}$ for which the projection map $A^{\ast} \rightarrow
\WittScript_r(A)^{\ast}$ is a ring homomorphism. Moreover, this ring structure
exhibits $\WittScript_r(A)^{\ast}$ as a commutative differential graded
algebra. \qed 
\end{corollary}

\begin{remark}\label{rem25}
Let $A^{\ast}$ be a saturated Dieudonn\'{e} algebra. Since $A^{-1} = 0$, we have $\WittScript_{r}(A)^{0} = A^{0} / V^{r} A^{0}$. In particular,
we have $\WittScript_1(A)^{0} = A^0 / V A^{0}$.
\end{remark}

\begin{construction}
Let $A^{\ast}$ be a saturated Dieudonn\'{e} algebra. We let $\WittScript(A)^{\ast}$ denote the completion of $A^{\ast}$ as a (saturated) Dieudonn\'{e} complex, given by the inverse limit of the tower
$$ \cdots \rightarrow \WittScript_3( A)^{\ast} \rightarrow \WittScript_2(A)^{\ast} \rightarrow \WittScript_1(A)^{\ast} \rightarrow \WittScript_0(A)^{\ast} \simeq 0.$$
Since each $\WittScript_{n}(A)^{\ast}$ has the structure of a commutative differential graded algebra, the inverse limit $\WittScript(A)^{\ast}$ inherits the structure of a commutative differential graded algebra.
\end{construction}

\begin{proposition}
Let $A^{\ast}$ be a saturated Dieudonn\'{e} algebra. Then the completion $\WittScript(A)^{\ast}$ is also a saturated Dieudonn\'{e} algebra.
\end{proposition}

\begin{proof}
It follows from Corollary \ref{realcor14} that $\WittScript(A)^{\ast}$ is a saturated Dieudonn\'{e} complex. Moreover, the Frobenius map
$F\colon  \WittScript(A)^{\ast} \rightarrow \WittScript(A)^{\ast}$ is an inverse limit of the maps $F\colon  \WittScript_{r}(A)^{\ast} \rightarrow \WittScript_{r-1}(A)^{\ast}$ appearing
in Remark \ref{rem13}, each of which is a ring homomorphism (since the Frobenius on $A^{\ast}$ is a ring homomorphism). It follows
that $F\colon  \WittScript(A)^{\ast} \rightarrow \WittScript(A)^{\ast}$ is a ring homomorphism. The vanishing of $\WittScript(A)^{\ast}$ for $\ast < 0$ follows immediately
from the analogous property of $A^{\ast}$. To complete the proof, it will suffice to show that each element $x \in \WittScript(A)^{0}$ satisfies
$Fx \equiv x^{p} \pmod{V \WittScript(A)^{0} }$ (Proposition \ref{prop35}). In other words, we must show that $F$ induces the usual Frobenius
map on the $\F_p$-algebra $\WittScript_1( \WittScript(A)^{0} ) = \WittScript(A)^{0} / V \WittScript(A)^{0}$. This follows from the analogous property for
$A^{\ast}$, since the tautological map $\WittScript_1( A )^{\ast} \rightarrow \WittScript_1( \WittScript(A) )^{\ast}$ is an isomorphism by Proposition \ref{prop17}.
\end{proof}

Let $A^{\ast}$ be a saturated Dieudonn\'{e} algebra. Then the tautological map $\rho_{A}\colon  A^{\ast} \rightarrow \WittScript(A)^{\ast}$ of
\S \ref{strictdieudonnecomplex} is a morphism of (saturated) Dieudonn\'{e} algebras. 

\begin{definition}
\label{SDA}
Let $A^{\ast}$ be a saturated Dieudonn\'{e} algebra. We will say that $A^{\ast}$ is {\it strict} if the map $\rho_{A}\colon  A^{\ast} \rightarrow \WittScript(A)^{\ast}$
is an isomorphism of Dieudonn\'{e} algebras. We let $\FrobAlgComplete$ denote the full subcategory of $\FrobAlgSat$ spanned by the strict Dieudonn\'{e} algebras.
\end{definition}

\begin{remark}
A saturated Dieudonn\'{e} algebra $A^{\ast}$ is strict if and only if it is strict when regarded as a saturated Dieudonn\'{e} complex.
\end{remark}

We have the following nonlinear version of Proposition \ref{prop37}:

\begin{proposition}\label{prop38}
Let $A^{\ast}$ and $B^{\ast}$ be saturated Dieudonn\'{e} algebras, where $B^{\ast}$ is strict. Then composition with the map $\rho_{A}$ induces a bijection
$$ \theta\colon  \Hom_{ \FrobAlg}( \WittScript(A)^{\ast}, B^{\ast} ) \rightarrow \Hom_{ \FrobAlg}( A^{\ast}, B^{\ast} ).$$
\end{proposition}

\begin{proof}
By virtue of Proposition \ref{prop37}, it will suffice to show that if $f\colon  \WittScript(A)^{\ast} \rightarrow B^{\ast}$
is a map of Dieudonn\'{e} complexes for which $f \circ \rho_{A}$ is a ring homomorphism, then $f$ is a ring homomorphism.
Since $B^{\ast}$ is complete, it will suffice to show that each of the composite maps $\WittScript(A)^{\ast} \rightarrow B^{\ast} \rightarrow \WittScript_{r}(B)^{\ast}$ is a ring homomorphism.
This follows by inspecting the commutative diagram
$$ \xymatrix{ A^{\ast} \ar[r]^-{\rho_A} \ar[d] & \WittScript(A)^{\ast} \ar[r]^-{f} \ar[d] & B^{\ast} \ar[d] \\
\WittScript_r(A)^{\ast} \ar[r]^-{\beta}_-{\sim} & \WittScript_r(\WittScript(A))^{\ast} \ar[r] & \WittScript_r(B)^{\ast}, }$$
since the map $\beta$ bijective (Proposition \ref{prop17}).
\end{proof}

\begin{corollary}\label{cor69}
The inclusion functor $\FrobAlgComplete \hookrightarrow \FrobAlgSat$ admits a left adjoint, given by the construction
$A^{\ast} \mapsto \WittScript(A)^{\ast}$. \qed
\end{corollary}

\begin{corollary}\label{oloter}
The inclusion functor $\FrobAlgComplete \hookrightarrow \FrobAlg$ admits a left adjoint, given by the construction $A^{\ast} \mapsto \WSaturate(A)^{\ast}$.
\end{corollary}

\begin{proof}
Combine Proposition \ref{prop33} with Corollary \ref{cor69}.
\end{proof}

\subsection{Comparison with Witt Vectors}

\begin{lemma}\label{lem27}
Let $A^{\ast}$ be a saturated Dieudonn\'{e} algebra. Then the quotient $R = A^{0} / V A^{0}$ is a reduced $\F_p$-algebra.
\end{lemma}

\begin{proof}
It follows from Remark \ref{rem25} that $V A^{0}$ is an ideal in $A^{0}$, so that we can regard $R$ as a commutative ring.
Note that the ideal $V A^{0}$ contains $V(1) = V( F(1) ) = p$, so that $R$ is an $\F_p$-algebra. To show that $R$ is reduced, it will suffice to show that if
$\overline{x}$ is an element of $R$ satisfying $\overline{x}^{p} = 0$, then $\overline{x} = 0$. Choose an element $x \in A^{0}$ representing $\overline{x}$.
The condition $\overline{x}^{p} = 0$ implies that $x^{p} \in V A^{0}$, so that $F(x) \equiv x^{p} \pmod{ p}$ also belongs to $V A^{0}$. We can therefore
write $Fx = Vy$ for some $y \in A^{0}$. Applying the differential $d$, we obtain $d( Vy) = d(Fx) = p F(dx) \in p A^{1}$. Invoking Lemma \ref{lem6},
we can write $y = Fz$ for some element $z \in A^{0}$. Then $Fx = VFz = FVz$, so that $x = Vz \in V A^{0}$ and therefore $\overline{x} = 0$, as desired.
\end{proof}

\begin{proposition}\label{prop28}
Let $A^{\ast}$ be a strict Dieudonn\'{e} algebra and let $R = A^{0} / V A^{0}$. Then there is a unique ring isomorphism
$u\colon  A^{0} \xrightarrow{\sim}  W(R)$ which is compatible with the identification $R = A^0/VA^0$ and for which the Witt vector
Frobenius corresponds to $F$. 
More precisely: 
\begin{itemize}
\item[$(i)$] The diagram
$$ \xymatrix{ A^{0} \ar[r]^-{u} \ar[d] & W(R) \ar[d] \\
A^{0} / V A^{0} \ar[r]^-{\id} & R }$$
commutes (where the vertical maps are the natural projections).

\item[$(ii)$] The diagram 
$$ \xymatrix{ A^{0} \ar[d]^-{F} \ar[r]^-{u} & W(R) \ar[d]^-{F} \\
A^{0} \ar[r]^-{ u } & W(R) }$$
commutes (where the right vertical map is the Witt vector Frobenius).
\end{itemize}
\end{proposition}

\begin{proof}
Since the ring $R$ is reduced (Lemma \ref{lem27}), the ring of Witt vectors
$W(R)$ is $p$-torsion free as it embeds inside $W(R^{1/p^{\infty}})$ where
$R^{1/p^{\infty}}$ denotes the \emph{perfection} 
$$ \varinjlim( R \xrightarrow{ x \mapsto x^p} R \xrightarrow{ x \mapsto x^p} R \xrightarrow{} \cdots )$$
Applying the universal property of $W(R)$ (as the cofree $\delta$-ring on
$R$; see Definition~\ref{lambdapring} below and \cite{Joyal}), we deduce
that there exists a unique ring homomorphism $u\colon  A^{0} \rightarrow W(R)$ which satisfies conditions $(i)$ and $(ii)$. To complete the proof, it will suffice to show that $u$ is an isomorphism.
We first note that $u$ satisfies the identity $u( Vx ) = V u(x)$ for each $x \in A^{0}$ (where the second $V$ denotes the usual Witt vector Verschiebung). To prove this, we begin with the identity
$$ F( u(Vx) ) = u( FVx) = u(px) = p u(x) = F V u(x)$$
and then invoke the fact that the Frobenius map $F\colon  W(R) \rightarrow W(R)$ is injective (by virtue of the fact that $R$ is reduced). 
It follows that $u$ carries $V^r A^0$ into $V^{r} W(R)$, and therefore induces a ring homomorphism $u_{r}\colon   \WittScript_{r}(A)^{0} \rightarrow W_r(R)$ for every nonnegative integer $r$.
Since $A^{\ast}$ is strict, we can identify $u$ with the inverse limit of the tower of maps $\{ u_r \}_{r \geq 0}$. Consequently, to show that $u$ is an isomorphism, it will suffice to show
that each $u_r$ is an isomorphism. We now proceed by induction on $r$, the case $r = 0$ being trivial. We have a commutative diagram of short exact sequences
$$\xymatrix{ 0 \ar[r] & A^{0} / VA^{0} \ar[r]^{ V^{r-1} } \ar[d]^{u_1} & A^{0} / V^{r} A^{0} \ar[r] \ar[d]^{ u_r } & A^{0} / V^{r-1} A^{0} \ar[d]^{ u_{r-1} } \ar[r] & 0 \\
0 \ar[r] & R \ar[r]^-{ V^{r-1} } & W_{r}(R) \ar[r] & W_{r-1} (R) \ar[r] & 0 }$$
where $u_1$ is the identity map from $A^{0} / VA^{0} = R$ to itself, and $u_{r-1}$ is an isomorphism by the inductive hypothesis. It follows that $u_r$ is also an isomorphism, as desired.
\end{proof}

We will also need the following:

\begin{proposition}\label{prop30}
Let $A^{\ast}$ be a strict Dieudonn\'{e} algebra and let $R$ be a commutative $\F_p$-algebra. For every ring homomorphism
$f_0\colon  R \rightarrow A^{0} / V A^{0}$, there is a unique ring homomorphism $f\colon  W(R) \rightarrow A^{0}$ satisfying the following pair of conditions:
\begin{itemize}
\item[$(i)$] The diagram
$$ \xymatrix{ W(R) \ar[r]^-{f} \ar[d] & A^{0} \ar[d] \\
R \ar[r]^-{f_0} & A^{0} / V A^{0} }$$
commutes (where the vertical maps are the natural projections).

\item[$(ii)$] The diagram 
$$ \xymatrix{ W(R) \ar[d]^{F} \ar[r]^{f} & A^{0} \ar[d]^{F} \\
W(R) \ar[r]^{ f } & A^{0} }$$
commutes (where the left vertical map is the Witt vector Frobenius).
\end{itemize}
\end{proposition}

In the statement of Proposition \ref{prop30}, the assumption that $A^{\ast}$ is a strict Dieudonn\'{e} algebra
is needed only to guarantee that $A^{0}$ is isomorphic to the ring of Witt vectors of $A^{0}/VA^{0}$ (Proposition \ref{prop28}). The
{\em existence} of the map $f\colon  W(R) \rightarrow A^{0}$ is clear: we simply apply
the Witt vector functor to the ring homomorphism $f_0\colon  R \rightarrow A^{0} / V A^{0}$.
The uniqueness of $f$ is a consequence of the following elementary assertion:

\begin{lemma}\label{lem29}
Let $R, R' $ be commutative $\F_p$-algebras, and let $g \colon W(R) \to W(R')$ be a
ring homomorphism. 
Suppose that: 
\begin{enumerate}
\item The map $g$ carries the ideal $VW(R) $ into the ideal $VW(R')$, and therefore
induces a map $f \colon R \to R'$ of $\F_p$-algebras; 
\item The map $g$ is compatible with the Witt vector Frobenius on each side; 
\item The ring $R'$ is reduced. 
\end{enumerate}
Then $g = W(f)$. 
\end{lemma}

\begin{proof}
Our hypotheses imply that $g \colon W(R) \to W(R')$ is a map of $\delta$-rings
(Definition~\ref{lambdapring}),
since it commutes with the Witt vector Frobenius on both sides and since
$W(R')$ is $p$-torsionfree. Therefore, $g$
is determined by the composite map of rings $W(R) \xrightarrow{g} W(R') \to 
R'$, thanks to the universal property of the Witt vectors of \cite{Joyal}. 
Similarly, $W(f)$ is determined by the composite map of rings $W(R)
\xrightarrow{W(f)} W(R') \to R'$. But these two composites agree since $g$
carries the ideal $VW(R) $ into  $VW(R')$. 
\end{proof}

\subsection{Aside: Rings with $p$-Torsion}
\label{lambda}

We now describe a generalization of Proposition \ref{prop42} which allows us to
relax the assumption that $R$ is $p$-torsion free. This material will
not be used in the sequel, and can be safely skipped without loss of
continuity. We begin by recalling the notion of a \emph{$\delta$-ring}, introduced by Joyal in \cite{Joyal} and further developed (using the language of \emph{plethories}) in
\cite{BoWi05}. 

\begin{definition}
\label{lambdapring}
A {\it $\delta$-ring} is a commutative ring $R$ equipped with a function $\delta\colon  R \to R$ 
satisfying the identities
$$
\delta(x+y) = \delta(x) + \delta(y) - \sum_{0 < i < p } \frac{ (p-1)!}{ i! (p-i)!} x^{i} y^{p-i}$$
$$ \delta(xy) = x^p \delta(y) + \delta(x) y^p + p \delta(x) \delta(y).$$
$$\delta(1) = \delta(0) = 0.$$
\end{definition}

\begin{remark}
Let $(R, \delta)$ be a $\delta$-ring, and define $\varphi\colon  R \rightarrow R$ by the formula
$\varphi(x) = x^{p} + p \delta(x)$. Then $\varphi$ is a ring homomorphism satisfying the condition
$\varphi(x) \equiv x^{p} \pmod{p}$. Conversely, if $R$ is a $p$-torsion-free ring
equipped with a ring homomorphism $\varphi\colon  R \rightarrow R$ satisfying $\varphi(x) \equiv x^{p} \pmod{p}$,
then the construction 
\[x \mapsto \frac{\varphi(x) - x^p}{p}\]
 determines a map $\delta\colon  R \rightarrow R$
which endows $R$ with the structure of a $\delta$-ring (as noted in the proof of Proposition \ref{prop42}).
This observation does not extend to the case where $R$ has $p$-torsion. For example, if $R$ is an $\F_p$-algebra,
then there is a unique ring homomorphism $\varphi\colon  R \rightarrow R$ satisfying $\varphi(x) \equiv x^p \pmod{p}$.
However, $R$ cannot admit the structure of a $\delta$-ring unless $R \simeq 0$ (this follows from the calculation
$0 = \delta(0) = \delta(p) = 1 - p^{p-1}$).
\end{remark}

\begin{remark}\label{losbit}
The free $\delta$-ring on a single generator $x$ is given as a commutative ring
by $S = \mathbb{Z}[x, \delta(x), \delta^2(x), \dots ]$. More generally, the free $\delta$-ring on
a set of generators $\{ x_i \}_{i \in I}$ can be identified with an $I$-fold tensor product of copies of $S$, and is therefore
isomorphic (as a commutative ring) to a polynomial algebra over $\Z$ on infinitely many generators.
\end{remark}

\begin{definition} 
A \emph{$\delta$-cdga} is a commutative differential graded algebra
$(A^{\ast}, d)$ together with a map of graded rings $F\colon  A^{\ast} \rightarrow A^{\ast}$
and a map of sets $\delta\colon  A^{0} \rightarrow A^{0}$ satisfying the following conditions:
\begin{enumerate}
\item The group $A^{n}$ vanishes for $n < 0$.
\item The map $\delta\colon  A^{0} \rightarrow A^{0}$ endows $A^{0}$ with the structure 
of a $\delta$-ring.
\item For each $x \in A^{0}$, we have $F(x) = x^p + p \delta(x)$.
\item For each $x \in A^{\ast}$, we have $d F(x) = p F(dx)$.
\item For each $x \in A^{0}$, we have $F(dx) = x^{p-1} dx + d \delta(x)$ (note that
this condition is automatic if $A^{1}$ is $p$-torsion-free, since it holds after multiplying by $p$).
\end{enumerate}
A map  of $\delta$-cdgas is a map of commutative differential graded
algebras compatible with $\delta$ and $F$. 
\end{definition} 

\begin{remark}
Let $( A^{\ast}, d, F, \delta)$ be a $\delta$-cdga. Then the triple $(A^{\ast}, d, F)$ is a Dieudonn\'{e} algebra.
Conversely, if $( A^{\ast}, d, F)$ is a $p$-torsion-free Dieudonn\'{e} algebra, then there is a unique
map $\delta\colon  A^{0} \rightarrow A^{0}$ for which $(A^{\ast}, d, F, \delta)$ is a $\delta$-cdga.
\end{remark}

\begin{proposition} 
Let $(R, \delta)$ be a $\delta$-ring. Then there is a unique map $F\colon  \Omega_{R}^{\ast} \rightarrow \Omega_{R}^{\ast}$
which endows the de~Rham complex $\Omega_{R}^{\ast}$ with the structure of a
$\delta$-cdga with respect to the map $\delta$. Moreover, for any 
$\delta$-cdga $A^{\ast}$, we have a canonical bijection
\[ \left\{\delta\text{-}\mathrm{cdga \ maps \ } \Omega_R^{\ast} \to
A^{\ast}\right\}  \simeq \left\{\delta\text{-}\mathrm{ring \ maps \ } R \to
A^0\right\}. \]
In other words, the construction $R \mapsto \Omega_R^{\ast}$ is the left adjoint
of the forgetful functor $A^{\ast} \mapsto A^0$ from $\delta$-cdgas to
$\delta$-rings. 
\end{proposition} 

\begin{proof}
Assume first that $R$ is $p$-torsion free. The existence and uniqueness of $F$ follows from Proposition~\ref{prop42}.
Let $A^{\ast}$ be any $\delta$-cdga. Using the universal property of the 
de~Rham complex $\Omega^{\ast}_{R}$, we see that every morphism of $\delta$-rings $f\colon  R \to A^0$ extends
uniquely to a map of differential graded algebras $\overline{f}\colon  \Omega^{\ast}_R \to A^{\ast}$.
It remains to check that $\overline{f}$ is a map of $\delta$-cdgas: that is, that it commutes with the
Frobenius operator $F$. Since $\Omega^{\ast}_{R}$ is generated by $R$ together with elements of the form
$dx$, we are reduced to showing that $\overline{f}( F dx ) = F \overline{f}(dx)$ for each $x \in R$, 
which follows from the identity $F(dy) = y^{p-1} dy + d \delta(y)$ for $y = f(x) \in A^0$.

To treat the general case, we note that any $\delta$-ring $R$ can be resolved
using free $\delta$-rings. Explicitly, $R$ can be presented as a reflexive coequalizer
$\varinjlim( F' \rightrightarrows F )$, where $F$ and $F'$ are free
$\delta$-rings; this follows from the general theory of (possibly
multi-sorted) finitary algebraic theories \cite[Ch. 6]{Johnstone}. In this case,
$F$ and $F'$ are $p$-torsion-free (Remark \ref{losbit}). It follows from the preceding argument that the de~Rham complexes
$\Omega_{F}^{\ast}$ and $\Omega_{F'}^{\ast}$ can be regarded as $\delta$-cdgas.
Since reflexive coequalizers in the category of $\delta$-cdgas can be computed at the level of the underlying abelian groups,
it follows that $\Omega_{R}^{\ast} \simeq \varinjlim( \Omega_{F'}^{\ast} \rightrightarrows \Omega_{F}^{\ast} )$ inherits
the structure of a $\delta$-algebra with the desired universal property.
\end{proof}

\newpage
\section{The Saturated de~Rham--Witt Complex}

In this section, we introduce the {\it saturated de~Rham--Witt complex} $\WOmega^{\ast}_{R}$ of
a commutative $\F_p$-algebra $R$ (Definition~\ref{def80}). We define $\WOmega^{\ast}_{R}$
to be a strict Dieudonn\'{e} algebra satisfying a particular universal property, analogous
to the characterization of the usual de~Rham complex $\Omega^{\ast}_{R}$ as the commutative differential graded algebra
generated by $R$. In the case where $R$ is a smooth algebra over a perfect ring $k$, we show that $\WOmega^{\ast}_{R}$ is
quasi-isomorphic to the completed de~Rham complex of any $W(k)$-lift of $R$ (Theorem~\ref{theo53}). In the case where
$R$ is a regular Noetherian ring, we show that the saturated de~Rham--Witt complex $\WOmega^{\ast}_{R}$ agrees with
the classical de~Rham--Witt complex $W \Omega^{\ast}_{R}$ constructed in \cite{illusie} (Theorem~\ref{maintheoC}).

\subsection{Construction of $\WOmega^{\ast}_{R}$}

\begin{definition}\label{def80}
Let $A^{\ast}$ be a strict Dieudonn\'{e} algebra and suppose we are given
an $\F_p$-algebra homomorphism $f \colon R \rightarrow A^{0} / VA^{0}$. We will say that
$f$ {\it exhibits $A^{\ast}$ as a saturated de~Rham--Witt complex of $R$} if it satisfies the following universal property:
for every strict Dieudonn\'{e} algebra $B^{\ast}$, composition with $f$ induces a bijection
$$ \Hom_{ \FrobAlg}( A^{\ast}, B^{\ast} ) \rightarrow \Hom( R, B^{0} / V B^{0} ).$$
\end{definition}

\begin{notation}\label{not40}
Let $R$ be an $\F_p$-algebra. It follows immediately from the definitions that if there exists
a strict Dieudonn\'{e} algebra $A^{\ast}$ and a map $f\colon  R \rightarrow A^{0} / V A^{0}$ which
exhibits $A^{\ast}$ as a de~Rham--Witt complex of $R$, then $A^{\ast}$ (and the map $f$) are determined up to unique isomorphism.
We will indicate this dependence by denoting $A^{\ast}$ as $\WOmega^{\ast}_{R}$ and referring to $A^{\ast}$ as
{\it the saturated de~Rham--Witt complex of $R$}. 
\end{notation}

\begin{warning}
The saturated de~Rham Witt complex $\WOmega^{\ast}_{R}$ of Notation \ref{not40} is generally not the same as the de~Rham--Witt complex
defined in \cite{illusie} (though they agree for a large class of $\F_p$-algebras: see Theorem~\ref{maintheoC}). For this reason, we use the term {\em classical de~Rham--Witt complex} to refer to complex
$W \Omega^{\ast}_{R}$ constructed in \cite{illusie} (see Definition~\ref{classicaldRW} for a review).

For example, it follows from Lemma \ref{lem27} that the saturated de~Rham--Witt
complex $\WOmega^{\ast}_{R}$ agrees with the saturated de~Rham--Witt complex
$\WOmega^{\ast}_{ R^{\red} }$, where $R^{\red}$ denotes the quotient of $R$ by
its nilradical; the classical de~Rham--Witt complex does not have this property.
For a reduced example of this phenomenon, see Proposition~\ref{CuspComparedRW}. \end{warning}

For existence, we have the following:

\begin{proposition}
\label{dRWExists}
Let $R$ be an $\F_p$-algebra. Then there exists strict Dieudonn\'{e} algebra $A^{\ast}$ and a map $f\colon  R \rightarrow A^{0} / V A^{0}$ which
exhibits $A^{\ast}$ as a saturated de~Rham--Witt complex of $R$.
\end{proposition}

\begin{proof}
By virtue of Lemma \ref{lem27}, we may assume without loss of generality that $R$ is reduced. 
Let $W(R)$ denote the ring of Witt vectors of $R$, and let $\varphi\colon  W(R) \rightarrow W(R)$ be the Witt vector Frobenius.
Since $R$ is reduced, $W(R)$ is $p$-torsion-free. Using Proposition
\ref{prop42}, we can endow the de~Rham complex $\Omega^{\ast}_{ W(R) }$ with the structure of a Dieudonn\'{e} algebra.
Let $A^{\ast}$ be the completed saturation $\WSaturate( \Omega^{\ast}_{ W(R) } )$. 
Combining the universal properties described in Proposition \ref{prop38}, Proposition \ref{prop33}, 
Proposition \ref{prop39}, and Proposition \ref{prop30}, we
see that for every strict Dieudonn\'{e} algebra $B^{\ast}$, we have natural bijections
\begin{eqnarray*}
\Hom_{ \FrobAlgComplete}( A^{\ast}, B^{\ast} ) & \simeq & \Hom_{ \FrobAlgSat}(
\Saturate( \Omega^{\ast}_{ W(R) }), B^{\ast} ) \\
& \simeq & \Hom_{ \FrobAlg}( \Omega^{\ast}_{W(R)}, B^{\ast} ) \\
& \simeq & \Hom_{F}( W(R), B ) \\
& \simeq & \Hom(R, B^{0}/ V B^{0} );
\end{eqnarray*}
here $\Hom_{F}(W(R), B)$ denotes the set of ring homomorphisms from $f\colon  W(R) \rightarrow B$ satisfying $f \circ \varphi = F \circ f$.
Taking $B^{\ast} = A^{\ast}$, the image of the identity map $\id_{ A^{\ast} }$ under the composite bijection determines a ring homomorphism
$R \rightarrow A^{0} / V A^{0}$ with the desired property.
\end{proof}

Let $\CAlg_{\F_p}$ denote the category of commutative $\F_p$-algebras. Then the construction $A^{\ast} \mapsto A^{0} / V A^{0}$ determines a functor
$\FrobAlgComplete \rightarrow \CAlg_{ \F_p }$.

\begin{corollary}
\label{dRWLeftAdj}
The functor $A^{\ast} \mapsto A^{0} / V A^{0}$ described above admits a left adjoint $\WOmega^{\ast}\colon  \CAlg_{ \F_p } \rightarrow \FrobAlgComplete$, given
by the formation of saturated de~Rham--Witt complexes $R \mapsto
\WOmega^{\ast}_{R}$. \qed
\end{corollary}

\begin{remark} 
Corollary~\ref{dRWLeftAdj} can also be proved by applying the adjoint
functor theorem \cite[Th. 1.66]{AdamekRosicky}. The category $\FrobAlgComplete$ is locally presentable, and
it is a straightforward consequence of Proposition~\ref{prop28} that the
construction $A^{\ast} \mapsto A^0/VA^0$ commutes with arbitrary limits (which
are computed at the level of underlying graded abelian groups) and
filtered colimits (see Proposition~\ref{prop74}). We leave the details to the reader.
\end{remark} 

\subsection{Comparison with a Smooth Lift}

\begin{proposition}\label{prop50}
Let $R$ be a commutative ring which is $p$-torsion-free, and let $\varphi\colon  R \rightarrow R$ be a ring homomorphism satisfying
$\varphi(x) \equiv x^{p} \pmod{p}$ for $x \in R$. Let $\widehat{\Omega}^{\ast}_{R}$ be the completed de~Rham complex
of Variant \ref{var46}, and let $A^{\ast}$ be a strict Dieudonn\'{e} algebra. Then the restriction map
$$ \Hom_{ \FrobAlg}( \widehat{ \Omega}^{\ast}_{R}, A^{\ast} ) \rightarrow
\Hom( R, A^{0} / V A^{0} )$$
is bijective. In other words, every ring homomorphism $R \rightarrow A^{0} / VA^{0}$ can be lifted uniquely to a 
map of Dieudonn\'{e} algebras $\widehat{ \Omega}^{\ast}_{R} \rightarrow A^{\ast}$.
\end{proposition}

\begin{proof}
Combine the universal properties of Variant \ref{var46} and
Proposition~\ref{prop30} (noting that if $A^{\ast}$ is a strict Dieudonn\'{e} algebra,
then each of the abelian groups $A^{n}$ is $p$-adically complete).
\end{proof}

\begin{notation}\label{not53}
Let $R$ be an $\F_p$-algebra. For each $r \geq 0$, we let $\WrOmega{r}^{\ast}_{R}$ denote the
quotient $\WOmega^{\ast}_{R} / ( \im(V^{r}) + \im( dV^{r} ) )$. By construction, we have a tautological map
$e\colon  R \rightarrow \WrOmega{1}^{\ast}_{R}$.
\end{notation}

\begin{corollary}\label{cor70}
Let $R$ be a commutative ring which is $p$-torsion-free, and let $\varphi\colon  R \rightarrow R$ be a ring homomorphism satisfying
$\varphi(x) \equiv x^{p} \pmod{p}$ for $x \in R$. Then there is a unique map of Dieudonn\'{e} algebras $\mu\colon  \widehat{\Omega}^{\ast}_{R} \rightarrow \WOmega^{\ast}_{R/pR}$
for which the diagram
$$ \xymatrix{ R \ar[r] \ar[d] & \widehat{\Omega}^{0}_{R} \ar[r]^-{\mu} & \WOmega^{0}_{R/pR} \ar[d] \\
R/pR \ar[rr]^{e} & & \WrOmega{1}^{0}_{R/pR} }$$
commutes (here $e$ is the map appearing in Notation \ref{not53}). Moreover,
$\mu$ induces an isomorphism of Dieudonn\'{e} algebras $\WSaturate(
\widehat{\Omega}^{\ast}_{R} ) \xrightarrow{\sim} \WOmega^{\ast}_{R/pR}$.
\end{corollary}

\begin{proof}
The existence and uniqueness of $\mu$ follow from Proposition \ref{prop50}. To prove the last assertion, it will suffice to show that for every
strict Dieudonn\'{e} algebra $A^{\ast}$, composition with $\mu$ induces an
isomorphism
$$\theta\colon  \Hom_{ \FrobAlg}( \WOmega^{\ast}_{R/pR}, A^{\ast} )
\xrightarrow{\sim} \Hom_{\FrobAlg}( \widehat{\Omega}^{\ast}_{R}, A^{\ast}).$$
Using Proposition \ref{prop50} and the definition of the saturated de~Rham--Witt complex, we can identify $\theta$ with the evident map
$$\Hom( R/pR, A^{0} / V A^{0} ) \rightarrow \Hom(R, A^{0} / V A^{0} ).$$ This map is bijective, since $A^{0} / V A^{0}$ is an $\F_p$-algebra.
\end{proof}

\begin{theorem}\label{theo53}
Let $R$ be a commutative ring which is $p$-torsion-free, and let 
$\varphi\colon  R \rightarrow R$ be a ring homomorphism satisfying $\varphi(x) \equiv x^{p} \pmod{p}$. Suppose that
there exists a perfect ring $k$ of characteristic $p$ such that $R/pR$ is a smooth $k$-algebra. 
Then the map $\mu\colon  \widehat{ \Omega}^{\ast}_{R} \rightarrow \WOmega^{\ast}_{R/pR}$
of Corollary \ref{cor70} is a quasi-isomorphism.
\end{theorem}

\begin{proof}
By virtue of Corollary \ref{cor70}, it will suffice to show that the tautological map 
$\widehat{ \Omega}^{\ast}_{R} \rightarrow \WSaturate( \widehat{\Omega}^{\ast}_{R} )$ is
a quasi-isomorphism. This follows from Corollary \ref{cor68}.
\end{proof}

\begin{example}[Powers of $\mathbf{G}_m$]
Let $R = \mathbb{Z}[x_1^{\pm 1}, \dots, x_n^{\pm 1}]$ be a Laurent polynomial ring on variables
$x_1, \ldots, x_n$. Then the de~Rham complex $\Omega_{R}^{\ast}$ is isomorphic
to an exterior algebra over $R$ on generators $d \log x_i = \frac{dx_i}{x_i}$ for $1 \leq i \leq n$,
with differential given by $d(x_i^n) = n x_i^n d \log x_i$. Using
Proposition~\ref{prop42}, we see that there is a unique Dieudonn\'{e} algebra
structure on $\Omega_{R}^{\ast}$ which satisfies $F( x_i ) = x_i^{p}$ for $1 \leq i \leq n$;
note that we then have $F( d \log x_i ) = d \log x_i$ for $1 \leq i \leq n$.

Set $R_{\infty} = \mathbb{Z}[x_1^{\pm 1/p^\infty}, \dots, x_n^{\pm 1/p^\infty}]$, 
so that the localization $\Omega_{R}^{\ast}[ F^{-1} ]$ can be identified with
the exterior algebra over $R_{\infty}$ on generators $d \log x_i$ for $1 \leq i \leq n$,
equipped with a differential $d\colon  \Omega_{R}^{\ast}[F^{-1}] \rightarrow \Omega_{R}^{\ast}[ F^{-1} ][p^{-1}]$
given by $d( x_i^{a/p^n} ) = (a/p^n) x_i^{a/p^n}$. Using Remark~\ref{integralformsabstract}, we see that the saturation $\Saturate( \Omega_{R}^{\ast} )$
can be identified with subalgebra of $\Omega_{R}^{\ast}[ F^{-1}]$ given by those differential forms $\omega$ for which
$d \omega$ also belongs to $\Omega_{R}^{\ast}[ F^{-1} ]$: that is, which have integral coefficients when expanded in terms of
$x_i$ and $d \log x_i$. Furthermore, the saturated de~Rham--Witt complex of $\mathbb{F}_p[x_1^{\pm 1}, \dots,
x_n^{\pm 1}]$ is obtained as the completion $\mathcal{W} \mathrm{Sat}(
\Omega_R^{\ast})$. This (together with Theorem~\ref{maintheoC}) recovers the explicit description of the classical
de~Rham--Witt complex $W \Omega^{\ast}_{R/pR}$ via the \emph{integral forms}
in \cite[Sec. I.2]{illusie}. 
\end{example}

\subsection{Comparison with the de~Rham Complex}
\label{dRComp}

Let $R$ be a commutative $\F_p$-algebra, let $\WOmega^{\ast}_{R}$ be the saturated de~Rham--Witt complex of $R$, and let
$\WrOmega{1}^{\ast}_{R}$ be the quotient of $\WOmega^{\ast}_{R}$ of Notation \ref{not53}. Then
the tautological ring homomorphism $e\colon  R \rightarrow \WrOmega{1}^{0}_{R}$ admits a unique extension to a map of
differential graded algebras $\nu \colon  \Omega^{\ast}_{R} \rightarrow \WrOmega{1}^{\ast}_{R}$.

\begin{theorem}\label{theo73}
Let $R$ be a regular Noetherian $\F_p$-algebra. Then the map $\nu\colon  \Omega^{\ast}_{R} \rightarrow \WrOmega{1}^{\ast}_{R}$ is an isomorphism.
\end{theorem}

We first treat the special case where $R$ is a smooth algebra over a perfect field $k$. In fact, our argument works more generally if $k$ is a perfect ring:

\begin{proposition}\label{prop76}
Let $k$ be a perfect $\F_p$-algebra and let $R$ be a smooth algebra over $k$.
Then the map $\nu\colon  \Omega^{\ast}_{R} \rightarrow 
\WrOmega{1}^{\ast}_{R}$ is
an isomorphism.
\end{proposition}

\begin{proof}
By a result of Elkik \cite{Elkik} and (in greater generality)
Arabia \cite{Arabia}, one can find 
a smooth $W(k)$-algebra $A$ with an isomorphism $A/pA \simeq R$; see also \cite[Tag
07M8]{stacks-project} for an exposition. Let $\widehat{A}$
denote the $p$-adic completion of $A$.\footnote{Alternatively, for the following argument, one only needs to work with
$\widehat{A}$, which can be constructed more easily (by inductively lifting $R$ to
$W_n(k)$ for each $n$).} 

Using the smoothness of $A$, we see that the ring homomorphism $$A \rightarrow A/pA \simeq R \xrightarrow{x \mapsto x^{p}} R$$ can be lifted to a map
$A \rightarrow \widehat{A}$. Passing to the completion, we obtain a ring homomorphism $\varphi\colon  \widehat{A} \rightarrow \widehat{A}$ satisfying
$\varphi(x) \equiv x^{p} \pmod{p}$. Applying Corollary \ref{cor70}, we see that there is a unique morphism of Dieudonn\'{e} algebras
$\mu\colon  \widehat{ \Omega}^{\ast}_{ \widehat{A} } \rightarrow \WOmega_{R}^{\ast}$ for which the diagram
$$ \xymatrix{ \widehat{A} \ar[r]^-{\sim} \ar[d] & \widehat{\Omega}^{0}_{ \widehat{A} } \ar[r]^-{\mu} & \WOmega^{0}_{R} \ar[d] \\
\widehat{A} / p \widehat{A} \ar[r]^-{\sim} & R \ar[r]^-{e} & \WrOmega{1}^{0}_{R} }$$
commutes. From this commutativity, we see that $\nu$ can be identified with the composition
$$ \Omega^{\ast}_{R} \simeq \widehat{\Omega}^{\ast}_{\widehat{A} } / p  \widehat{\Omega}^{\ast}_{\widehat{A} } 
\xrightarrow{\mu} \WOmega^{\ast}_{R} / p \WOmega^{\ast}_{R} \rightarrow \WrOmega{1}^{\ast}_{R}.$$
We now have a commutative diagram of graded abelian groups
$$ \xymatrix{ \widehat{\Omega}^{\ast}_{ \widehat{A} } / p \widehat{\Omega}^{\ast}_{ \widehat{A} }  \ar[rr]^{\Cart} \ar[d] \ar[dr]^{ \nu } & & \mathrm{H}^{\ast}( \widehat{\Omega}^{\ast}_{ \widehat{A} } / p \widehat{\Omega}^{\ast}_{ \widehat{A} } ) \ar[d] \\
\WOmega^{\ast}_{R} / p \WOmega^{\ast}_{R} \ar[r] & \WrOmega{1}^{\ast}_{R} \ar[r] & \mathrm{H}^{\ast}( \WOmega^{\ast}_{R} / p \WOmega^{\ast}_{R} ), }$$
where the vertical maps are induced by $\mu$, and the top horizontal and bottom right horizontal maps are induced by the Frobenius on $\widehat{\Omega}^{\ast}_{ \widehat{A} }$
and $\WOmega^{\ast}_{R}$, respectively. We now observe that the top horizontal map is the Cartier isomorphism (Theorem \ref{theo71}), the right vertical map is an isomorphism
by virtue of Theorem \ref{theo53}, and the bottom right horizontal map is an isomorphism by Proposition \ref{prop11}. It follows that $\nu$ is also an isomorphism.
\end{proof}

To handle the general case, we will need the following:

\begin{proposition}\label{prop74}
The category $\FrobAlgComplete$ admits small filtered colimits, which are preserved by the functor $A^{\ast} \mapsto \WittScript_{r}(A)^{\ast} = A^{\ast} / (\im(V^{r}) + \im(dV^{r} ))$. 
\end{proposition}

\begin{proof}
We first observe that the category $\FrobAlg$ of Dieudonn\'{e} algebras admits filtered colimits, which can be computed at the level of the underlying graded abelian groups.
Moreover, the subcategory $\FrobAlgSat \subseteq \FrobAlg$ of saturated Dieudonn\'{e} algebras is closed under filtered colimits. In particular, the construction $A^{\ast} \mapsto \WittScript_{r}(A)^{\ast}$
preserves filtered colimits when regarded as a functor from the category $\FrobAlgSat$ to the category of graded abelian groups. It follows from Corollary \ref{cor69} that
the category $\FrobAlgComplete$ also admits small filtered colimits, which are computed by first taking a colimit in the larger category $\FrobAlgSat$ and then applying the completion construction
$A^{\ast} \mapsto \WittScript(A)^{\ast}$. Consequently, to show that
the restriction of $\WittScript_{r}$ to $\FrobAlgComplete$ commutes with filtered colimits, it suffices to show that the canonical map
$\WittScript_{r}( \varinjlim A_{\alpha} )^{\ast} \rightarrow \WittScript_{r}( \WittScript( \varinjlim A_{\alpha} ) )^{\ast}$ is an isomorphism for every filtered diagram $\{ A^{\ast}_{\alpha} \}$ in $\FrobAlgComplete$.
This is a special case of Proposition \ref{prop17}.
\end{proof}

\begin{remark}
In the statement and proof of Proposition \ref{prop74}, we can replace Dieudonn\'{e} algebras by Dieudonn\'{e} complexes.
\end{remark}

\begin{corollary}\label{cor75}
The construction $R \mapsto \WrOmega{r}_{R}^{\ast}$ commutes with filtered colimits (when regarded as a functor from the category of $\F_p$-algebras to the category of graded abelian groups).
\end{corollary}

\begin{proof}
Combine Proposition \ref{prop74} with the observation that the functor $R \mapsto \WOmega_{R}^{\ast}$ preserves colimits (since it is defined as the left adjoint to the functor
$A^{\ast} \mapsto A^{0} / V A^{0}$).
\end{proof}

We now prove Theorem~\ref{theo73} using Popescu's smoothing theorem. In \S \ref{dRRedux}, we will give an alternative proof using the theory of derived de~Rham(-Witt) cohomology,
which avoids the use of Popescu's theorem (see Corollary~\ref{compmain}).

\begin{proof}[Proof of Theorem \ref{theo73}]
Let $R$ be a regular Noetherian $\F_p$-algebra. Applying Popescu's smoothing theorem \cite[Tag 07GB]{stacks-project}, we can write $R$ as a filtered colimit $\varinjlim R_{\alpha}$, where
each $R_{\alpha}$ is a smooth $\F_p$-algebra. It follows from Corollary \ref{cor75} that the canonical map $\nu\colon  \Omega^{\ast}_{R} \rightarrow \WrOmega{1}^{\ast}_{R}$
is a filtered colimit of the maps $\nu_{\alpha}\colon  \Omega^{\ast}_{ R_{\alpha} }
\rightarrow \WrOmega{1}_{R_{\alpha} }^{\ast}$. It will therefore suffice to show that each of the maps
$\nu_{\alpha}$ is an isomorphism, which follows from Proposition \ref{prop76}.
\end{proof}

\begin{remark}\label{plux}
Let $R$ be a regular Noetherian $\F_p$-algebra. Composing the tautological map $\WOmega^{\ast}_{R} \rightarrow \WrOmega{1}^{\ast}_{R}$
with the inverse of the isomorphism $\nu\colon  \Omega^{\ast}_{R}
\xrightarrow{\sim}
\WrOmega{1}^{\ast}_{R}$, we obtain a surjective map of commutative differential graded algebras
$\WOmega^{\ast}_{R} \rightarrow \Omega^{\ast}_{R}$. It follows from Theorem \ref{theo73} and Corollary \ref{cor79} that the induced map
$\WOmega^{\ast}_{R} / p \WOmega^{\ast}_{R} \rightarrow \Omega^{\ast}_{R}$ is a quasi-isomorphism.
\end{remark}

\subsection{The Classical de~Rham--Witt Complex}\label{compareclass}

We now consider the relationship between the saturated de~Rham--Witt complex $\WOmega^{\ast}_{R}$ of Notation \ref{not40} and
the classical de~Rham--Witt complex $W\Omega^{\ast}_{R}$ of Bloch-Deligne-Illusie. We begin with some definitions.

\begin{definition}\label{definition.V-pro}
Let $R$ be a commutative $\F_p$-algebra. An {\it $R$-framed $V$-pro-complex} consists of the following data:
\begin{itemize}
\item[$(1)$] An inverse system 
$$ \cdots \rightarrow A_{4}^{\ast} \rightarrow A_{3}^{\ast} \rightarrow A_{2}^{\ast} \rightarrow A_1^{\ast} \rightarrow A_0^{\ast}$$
of commutative differential graded algebras. We will denote each of the transition maps in this inverse system
by $\Res\colon  A_{r+1}^{\ast} \rightarrow A_{r}^{\ast}$, and refer to them as {\it restriction maps}.
We let $A^{\ast}_{\infty}$ denote the inverse limit $\varprojlim_{r} A^{\ast}_{r}$.

\item[$(2)$] A collection of maps $V\colon  A_{r}^{\ast} \rightarrow A_{r+1}^{\ast}$ in the category of graded abelian groups, which 
we will refer to as {\it Verschiebung maps}.

\item[$(3)$] A ring homomorphism $\beta\colon  W(R) \rightarrow A^{0}_{\infty}$. For each $r \geq 0$, we let
$\beta_{r}$ denote the composite map $W(R) \xrightarrow{\beta} A^{0}_{\infty} \rightarrow A^{0}_{r}$.
\end{itemize}
These data are required to satisfy the following axioms:
\begin{itemize}
\item[$(a)$] The groups $A^{\ast}_{r}$ vanish when either $\ast < 0$ or $r = 0$.

\item[$(b)$] The Verschiebung and restriction maps are compatible with one another: that is, for every $r > 0$, we have a commutative diagram
$$ \xymatrix{ A^{\ast}_{r} \ar[r]^{V} \ar[d]^{ \Res} & A^{\ast}_{r+1} \ar[d]^{\Res} \\
A^{\ast}_{r-1} \ar[r]^{V} & A^{\ast}_{r}. }$$
It follows that the Verschiebung maps induce a homomorphism of graded abelian groups $V\colon  A^{\ast}_{\infty} \rightarrow A^{\ast}_{\infty}$, which
we will also denote by $V$.

\item[$(c)$] The map $\beta\colon  W(R) \rightarrow A^{0}_{\infty}$ is compatible with Verschiebung: that is, we have a commutative diagram
$$ \xymatrix{ W(R) \ar[r]^-{\beta } \ar[d]^{V} & A^{0}_{\infty} \ar[d]^{V} \\
W(R) \ar[r]^-{ \beta } & A^{0}_{\infty} . }$$

\item[$(d)$] The Verschiebung maps $V\colon  A^{\ast}_{r} \rightarrow A^{\ast}_{r+1}$ satisfy the identity
$V( x dy ) = V(x) dV(y)$.

\item[$(e)$] Let $\lambda$ be an element of $R$ and let $[\lambda] \in W(R)$ denote its Teichm\"{u}ller representative. 
Then, for each $x \in A^{\ast}_{r}$, we have an identity
$$ (Vx) d \beta_{r+1}([\lambda] ) = V \left( x  \beta_{r}([\lambda])^{p-1}  d
\beta_{r}([ \lambda] ) \right)$$ 
in $A^{\ast}_{r+1}$.
\end{itemize}

Let $( \{ A^{\ast}_{r} \}_{r \geq 0}, V, \beta )$ and $( \{ A'^{\ast}_{r} \}_{r \geq 0}, V', \beta' )$ be $R$-framed $V$-pro-complexes.
A {\it morphism} of $R$-framed $V$-pro-complexes from $( \{ A^{\ast}_{r} \}_{r \geq 0}, V, \beta )$ to $( \{ A'^{\ast}_{r} \}_{r \geq 0}, V', \beta' )$
is a collection of differential graded algebra homomorphisms $f_{r}\colon
A^{\ast}_{r} \rightarrow A'^{\ast}_{r}$ which commute with the restriction,
Verschiebung, and $\beta$ maps, i.e., for which the diagrams
$$ \xymatrix{ A^{\ast}_{r} \ar[r]^{f_r} \ar[d]^{\Res} & A'^{\ast}_{r} \ar[d]^{\Res} & A^{\ast}_{r} \ar[r]^{f_r} \ar[d]^{V} & A'^{\ast}_{r} \ar[d]^{V'} \\
A^{\ast}_{r-1} \ar[r]^{ f_{r-1} } & A'^{\ast}_{r-1} & A^{\ast}_{r+1}
\ar[r]^{f_{r+1}} & A'^{\ast}_{r+1} }$$
$$ \xymatrix{ & W(R) \ar[dl]_{\beta_r} \ar[dr]^{ \beta'_{r} } & \\
A_{r}^{0} \ar[rr]^{ f_r} & & A'^{0}_{r} }$$
are commutative. We let $\VPro_{R}$ denote the category whose objects are $R$-framed $V$-pro-complexes and whose morphisms
are morphisms of $R$-framed $V$-pro-complexes.
\end{definition}

\begin{remark}
In the situation of Definition \ref{definition.V-pro}, we will generally abuse terminology by simply referring to the inverse system
$\{ A^{\ast}_{r} \}_{r \geq 0}$ as an {\it $R$-framed $V$-pro-complex}; in this case, we are implicitly assuming that Verschiebung maps
$V\colon  A^{\ast}_{r} \rightarrow A^{\ast}_{r+1}$ and a map $\beta\colon  W(R) \rightarrow A^{0}_{\infty}$ have also been specified.
\end{remark}

\begin{remark}
Let $R$ be a commutative $\F_p$-algebra and let $\{ A^{\ast}_{r} \}$ be an $R$-framed $V$-pro-complex.
It follows from conditions $(a)$ and $(b)$ of Definition \ref{definition.V-pro} that, for each $r \geq 0$, the composite map
$$ A_{r}^{\ast} \xrightarrow{ V^{r} } A_{2r}^{\ast} \xrightarrow{ \Res^{r} } A_{r}^{\ast}$$
vanishes. Consequently, condition $(c)$ of Definition \ref{definition.V-pro} implies that the map
$\beta_r\colon  W(R) \rightarrow A^{0}_{r}$ annihilates the subgroup $V^{r} W(R) \subseteq W(R)$, and therefore
factors through the quotient $W_{r}(R) \simeq W(R) / V^{r} W(R)$.
\end{remark}

\begin{proposition}\label{initobj}
Let $R$ be a commutative $\F_p$-algebra. Then the category $\VPro_{R}$ has an initial object.
\end{proposition}

\begin{proof}[Proof Sketch]
For each $r \geq 0$, let $\Omega^{\ast}_{ W_r(R) }$ denote the de~Rham complex of $W_r(R)$ (relative to $\Z$).
We will say that a collection of differential graded ideals $\{ I^{\ast}_{r} \subseteq \Omega^{\ast}_{ W_r(R) } \}_{r \geq 0}$ is
{\it good} if the following conditions are satisfied:
\begin{itemize}
\item[$(i)$] Each of the restriction maps $W_{r+1}(R) \rightarrow W_{r}(R)$ determines a map of de~Rham complexes
$\Omega^{\ast}_{ W_{r+1}(R) } \rightarrow \Omega^{\ast}_{ W_r(R) }$ which carries $I_{r+1}^{\ast}$ into $I_{r}^{\ast}$,
and therefore induces a map of differential graded algebras $\Omega^{\ast}_{ W_{r+1}(R) } / I_{r+1}^{\ast} \rightarrow
\Omega^{\ast}_{ W_r(R) } / I_{r}^{\ast}$.

\item[$(ii)$] For each $r \geq 0$, there exists a group homomorphism 
$V\colon  \Omega^{\ast}_{ W_{r}(R) } / I_{r}^{\ast} \rightarrow \Omega^{\ast}_{ W_{r+1}(R) } / I_{r+1}^{\ast}$
which satisfies the identity
$$ V( x_0 dx_1 \wedge \cdots \wedge dx_n ) = V(x_0) dV(x_1) \wedge \cdots \wedge dV(x_n);$$
here we abuse notation by identifying an element of $\Omega^{\ast}_{ W_r(R) }$ with its image
in the quotient $\Omega^{\ast}_{ W_r(R) } / I_r^{\ast}$. Note that such a map $V$ is automatically unique.

\item[$(iii)$] For each $x \in \Omega^{\ast}_{ W_{r}(R) } / I_{r}^{\ast}$ and each $\lambda \in R$,
the difference $$(Vx) d[\lambda] - V( x [\lambda]^{p-1} d[\lambda])$$ belongs to $I_{r+1}^{\ast+1}$
(here we abuse notation by identifying the Teichm\"{u}ller representative $[\lambda]$ with its
image in each $\Omega^{\ast}_{W_{s}(R) } / I_s^{\ast} $). 
\end{itemize}
It is not difficult to see that the collection of good differential graded ideals is closed under intersection.
Consequently, there exists a smallest good differential graded ideal $\{ I^{\ast}_{r} \}_{r \geq 0}$. The inverse system $\{ \Omega^{\ast}_{ W_r(R) } / I_{r}^{\ast} \}_{r \geq 0}$
(together with with the Verschiebung maps defined in $(ii)$ and the evident structural map
$\beta\colon  W(R) \rightarrow \varprojlim_r \left(\Omega^{0}_{ W_r(R) } /
I_{r}^{0}\right)$) is then an initial
object of $\VPro_R$.
\end{proof}

\begin{remark}\label{belox}
In the situation of the proof of Proposition \ref{initobj}, the differential graded ideal $I_{1}^{\ast}$ vanishes:
that is, the tautological map $\Omega^{\ast}_{R} \rightarrow W_1
\Omega^{\ast}_{R}$ is always an isomorphism. In fact, the relations involved
impose no conditions for $r = 1$ (because all relations imposed for $r>1$
have vanishing image in $\Omega^{\ast}_{R}$). 
\end{remark}

\begin{definition}
\label{classicaldRW}
Let $R$ be a commutative $\F_p$-algebra and let $\{ W_{r} \Omega^{\ast}_{R} \}_{r \geq 0}$ denote
an initial object of the category of $\VPro_R$. We let $W \Omega^{\ast}_{R}$ denote the inverse limit
$\varprojlim_{r} W_{r} \Omega^{\ast}_{R}$. We will refer to $W \Omega^{\ast}_{R}$ as the {\it classical
de~Rham--Witt complex of $R$}. 
\end{definition}

\begin{remark}
The proof of Proposition \ref{initobj} provides an explicit model for the inverse system $\{ W_r \Omega^{\ast}_{R} \}$:
each $W_{r} \Omega^{\ast}_{R}$ can be regarded as a quotient of the absolute de~Rham complex $\Omega^{\ast}_{ W_r(R) }$
by a differential graded ideal $I_{r}^{\ast}$, which is dictated by the axiomatics of Definition \ref{definition.V-pro}.
\end{remark}

\begin{warning}
Our definition of the classical de~Rham--Witt complex $W \Omega^{\ast}_{R}$ differs slightly from the definition appearing
in \cite{illusie}. In \cite{illusie}, the inverse system $\{ W_r \Omega^{\ast}_{R} \}_{r \geq 0}$ is defined to be an initial
object of a category $\VPro'_{R}$, which can be identified with the full subcategory of $\VPro_{R}$
spanned by those $R$-framed $V$-pro-complexes $\{ A^{\ast}_{r} \}$ for which $\beta\colon  W(R) \rightarrow A^{0}_{\infty}$ induces isomorphisms
$W_{r}(R) \simeq A^{0}_{r}$ for $r \geq 0$. However, it is not difficult to see that the initial object of $\VPro_{R}$ belongs to
this subcategory (this follows by inspecting the proof of Proposition \ref{initobj}), so that $W \Omega^{\ast}_{R}$ is isomorphic
to the de~Rham--Witt complex which appears in \cite{illusie}.
\end{warning}

\begin{remark}\label{rem101}
Let $R$ be a commutative $\F_p$-algebra and let $W \Omega^{\ast}_{R}$ be the classical de~Rham--Witt complex of $R$.
The proof of Proposition \ref{initobj} shows that the restriction maps $W_{r+1} \Omega^{\ast}_{R} \rightarrow W_{r} \Omega^{\ast}_{R}$ are surjective. It follows
that each $W_{r} \Omega^{\ast}_{R}$ can be viewed as a quotient of $W \Omega^{\ast}_{R}$.
\end{remark}

Our next goal is to compare the saturated de~Rham--Witt complex $\WOmega^{\ast}_{R}$ with the classical
de~Rham--Witt complex $W \Omega^{\ast}_{R}$. We begin with the following observation:

\begin{proposition}\label{prop100}
Let $R$ be a commutative $\F_p$-algebra, let $A^{\ast}$ be a strict Dieudonn\'{e} algebra, and let
$\beta_1\colon  R \rightarrow A^{0} / VA^{0}$ be a ring homomorphism. Then:
\begin{itemize}
\item[$(1)$] The map $\beta_1$ admits a unique lift to a ring homomorphism $\beta\colon  W(R) \rightarrow A^{0}$
satisfying $\beta \circ F = F \circ \beta$.
\item[$(2)$] The inverse system $\{ \WittScript_{r}( A )^{\ast} \}_{r \geq 0}$ is an
$R$-framed $V$-pro-complex (when equipped with the Verschiebung maps $V\colon  \WittScript_{r}(A)^{\ast} \rightarrow \WittScript_{r+1}(A)^{\ast}$ of Remark \ref{rem13} and the map
$\beta$ described in $(1)$).
\end{itemize}
\end{proposition}

\begin{proof}
Assertion $(1)$ is the content of Proposition \ref{prop30}. To prove $(2)$, we must verify that $\beta$ and $V$ satisfy axioms $(a)$ through $(e)$ of Definition \ref{definition.V-pro}.
Axioms $(a)$ and $(b)$ are obvious. To prove $(c)$, we compute
$$F \beta( V x) = \beta( FV x) = \beta(p x) = p \beta(x) = F V \beta(x),$$
so that $\beta(Vx ) = V \beta(x)$ by virtue of the fact that the Frobenius map $F\colon  A^{0} \rightarrow A^{0}$ is injective.
The verification of $(d)$ is similar: for $x, y \in A^{\ast}$, we have
$$ FV (x dy) = (px) (dy) = (FVx) (FdVy) = F( (Vx) (dVy) );$$
invoking the injectivity of $F$ again, we obtain $V( x dy ) = (Vx) (dVy)$. To prove $(e)$, we compute
\begin{align*}
F( (Vx) d \beta([\lambda]) ) & = F(Vx) F(d \beta([\lambda]) ) \\
& =  px F(d \beta([ \lambda) ] ) 
 =  x d(F \beta([ \lambda] )) \\
& =  x d( \beta( F[\lambda]) ) 
 =  x d \beta( [\lambda]^{p} ) \\
& =  x d \beta( [ \lambda] )^{p} 
 =  p x \beta( [ \lambda] )^{p-1} d \beta( [ \lambda] ) \\
& =  FV( x \beta( [\lambda])^{p-1} d \beta([ \lambda] ).
\end{align*}
Canceling $F$, we obtain the desired identity \[ (Vx) d \beta([\lambda] ) = V( x
\beta( [\lambda])^{p-1} d \beta([\lambda] ) ). \qedhere \]
\end{proof}

Applying Proposition \ref{prop100} in the case $A^{\ast} = \WOmega^{\ast}_{R}$ and invoking the universal property of
the classical de~Rham--Witt complex, we obtain the following:

\begin{corollary}\label{makemap}
Let $R$ be a commutative $\F_p$-algebra. Then there is a unique homomorphism of differential graded algebras
$\gamma\colon  W \Omega^{\ast}_{R} \rightarrow \WOmega^{\ast}_{R}$ with the following properties:
\begin{itemize}
\item[$(i)$] For each $r \geq 0$, the composite map
$$W \Omega^{\ast}_{R} \xrightarrow{\gamma} \WOmega_{R}^{\ast} \rightarrow \WrOmega{r}_{R}^{\ast}$$
factors through $W_{r} \Omega^{\ast}_{R}$ (necessarily uniquely, by virtue of Remark \ref{rem101}); consequently,
$\gamma$ can be realized as the inverse limit of a tower of differential graded algebra homomorphisms $\gamma_{r}\colon  W_{r} \Omega^{\ast}_{R} \rightarrow \WrOmega{r}_{R}^{\ast}$.

\item[$(ii)$] The diagrams
$$ \xymatrix{ W \Omega^{\ast}_{R} \ar[d]^{\gamma} \ar[r]^{V} & W \Omega^{\ast}_{R} \ar[d]^{\gamma} & W(R) \ar[r] \ar[d] & W\Omega^{0}_{R} \ar[d]^{\gamma}  \\
\WOmega^{\ast}_{R} \ar[r]^{ V} & \WOmega^{\ast}_{R} & R \ar[r] & \WOmega^{0}_{R} / V \WOmega^{0}_{R} }$$
are commutative. \qed
\end{itemize}
\end{corollary}

We can now formulate the main result of this section:

\begin{theorem}
\label{maintheoC}
Let $R$ be a regular Noetherian $\F_p$-algebra. Then the map $\gamma\colon  W \Omega^{\ast}_{R} \rightarrow \WOmega^{\ast}_{R}$ of
Corollary \ref{makemap} is an isomorphism.
\end{theorem}

\begin{proof}
The map $\gamma$ can be realized as the inverse limit of a tower of maps $\gamma_r\colon  W_{r} \Omega^{\ast}_{R} \rightarrow \WrOmega{r}^{\ast}_{R}$; it will therefore
suffice to show that each $\gamma_r$ is an isomorphism. Since the constructions $R \mapsto W_{r} \Omega^{\ast}_{R}$ and $R \mapsto \WrOmega{r}^{\ast}_{R}$ commute
with filtered colimits, we can use Popescu's smoothing theorem to reduce to the case where $R$ is a smooth $\F_p$-algebra. Using Theorem I.2.17 and Proposition I.2.18 of
\cite{illusie}, we see that there exists a map $F\colon  W \Omega^{\ast}_{R} \rightarrow W \Omega^{\ast}_{R}$ which endows $W \Omega^{\ast}_{R}$ with the structure of
a Dieudonn\'{e} algebra and satisfies the identities $FV = VF = p$. Since $R$ is smooth over $\F_p$, Remark I.3.21.1 of \cite{illusie} guarantees that the Dieudonn\'{e} algebra $W \Omega^{\ast}_{R}$ is
saturated. Moreover, Proposition I.3.2 of \cite{illusie} (and a passage to
the limit) implies that the kernel of the projection map $W \Omega_{R}^{\ast} \rightarrow W_{r} \Omega_{R}^{\ast}$ is equal to
$\im( V^r ) + \im( d V^r )$, so that $$W \Omega^{\ast}_{R} \simeq \varprojlim
W_{r} \Omega^{\ast}_{R} \simeq \varprojlim \left(W \Omega^{\ast}_{R} / (
\im(V^{r}) + \im( dV^r) ) \right)$$
is a strict Dieudonn\'{e} algebra. For each $x \in W \Omega^{\ast}_{R}$, we have
$$ V \gamma( Fx ) = \gamma( VF x) = \gamma(px) = p \gamma(x) = V F \gamma(x).$$
Since the map $V\colon  \WOmega^{\ast}_{R} \rightarrow \WOmega^{\ast}_{R}$ is injective, it follows that $\gamma(Fx) = F \gamma(x)$: that is, $\gamma$ is a morphism of
(strict) Dieudonn\'{e} algebras. We wish to show that $\gamma_{r} = \WittScript_{r}( \gamma )$ is an isomorphism for each $r \geq 0$. By virtue of Corollary \ref{cor15}, it suffices
to prove this in the case $r = 1$. Invoking Remark \ref{belox}, we are reduced to proving that the tautological map $\Omega^{\ast}_{R} \rightarrow \WrOmega{1}^{\ast}_{R}$
is an isomorphism, which follows from Theorem \ref{theo73}.
\end{proof}

\newpage
\section{Localizations of Dieudonn\'{e} Algebras}

Our goal in this section is to introduce a global version of the saturated de~Rham--Witt complex $\WOmega^{\ast}_{R}$
of Definition \ref{def80}. Let $X$ be an $\F_p$-scheme. In \S \ref{buildscheme}, we show that there is a sheaf
of commutative differential graded algebras $\WOmega^{\ast}_{X}$ with the property that, for every
affine open subset $U \subseteq X$, the complex $\WOmega^{\ast}_{X}(U)$ can be identified
with the saturated de~Rham--Witt complex of the $\F_p$-algebra $\sheafO_{X}(U)$ (Theorem \ref{makeglob}). 
To prove the existence of this sheaf, we will need to analyze the behavior of the saturated de~Rham--Witt complex with respect to 
Zariski localization. In \S \ref{secsix}, we show that if $R$ is an $\F_p$-algebra and $s \in R$ is an element,
then the complex $\WOmega^{\ast}_{R[s^{-1}]}$ admits a relatively simple description in terms of $\WOmega^{\ast}_{R}$:
roughly speaking, it can be obtained from $\WOmega^{\ast}_{R}$ by inverting the Teichm\"{u}ller representative $[s]$ and then passing
to a suitable completion (Corollary \ref{six}). In \S \ref{orbit}, we formulate a more general version of this result, where the localization $R[s^{-1}]$
is replaced by an arbitrary {\etale} $R$-algebra $S$ (Corollary \ref{Vetalecor}). The proof requires some general facts about the behavior of the Witt vector functor
with respect to {\etale} ring homomorphisms, which we review in \S \ref{wittetalesec}.

\subsection{Localization for the Zariski Topology}\label{secsix}

Let $A^{\ast}$ be a graded-commutative ring and let $S \subseteq A^{0}$ be a multiplicatively closed subset of $A^{0}$. 
We let $A^{\ast}[S^{-1}]$ denote the (graded) ring obtained from $A^{\ast}$ by formally inverting the elements of $S$.
In this section, we study the construction $A^{\ast} \mapsto A^{\ast}[S^{-1}]$ in the case where $A^{\ast}$ is a Dieudonn\'{e} algebra.

\begin{proposition}\label{proposition.localize}
Let $A^{\ast}$ be a Dieudonn\'{e} algebra and let $S \subseteq A^{0}$ be a multiplicatively closed subset. Assume
that the Frobenius map $F\colon  A^{\ast} \rightarrow A^{\ast}$ satisfies $F(S) \subseteq S$. Then the graded ring $A^{\ast}[ S^{-1} ]$ 
inherits the structure of a Dieudonn\'{e} algebra, which is uniquely determined by the requirement that
the tautological map $A^{\ast} \rightarrow A^{\ast}[S^{-1}]$ is a morphism of Dieudonn\'{e} algebras.
\end{proposition}

\begin{remark}\label{seffer}
In the situation of Proposition \ref{proposition.localize}, the Dieudonn\'{e} algebra $A^{\ast}[S^{-1}]$ is characterized by the following universal property:
for any Dieudonn\'{e} algebra $B^{\ast}$, precomposition with the tautological map $A^{\ast} \rightarrow A^{\ast}[S^{-1}]$ induces a monomorphism of
sets
$$ \Hom_{ \FrobAlg }( A^{\ast}[ S^{-1} ], B^{\ast} ) \hookrightarrow \Hom_{ \FrobAlg }( A^{\ast}, B^{\ast} ),$$
whose image consists of those morphisms of Dieudonn\'{e} algebras $f\colon  A^{\ast} \rightarrow B^{\ast}$ with the property that
$f(s)$ is an invertible element of $B^{0}$, for each $s \in S$.
\end{remark}

\begin{proof}[Proof of Proposition \ref{proposition.localize}]
We first observe that there is a unique differential on the ring $A^{\ast}[ S^{-1} ]$ for which the map $A^{\ast} \rightarrow A^{\ast}[ S^{-1} ]$ is a morphism of differential
graded algebras, given concretely by the formula $$d( \tfrac{x}{s} ) =
\tfrac{ dx}{s} - \tfrac{x ds}{s^2}.$$ The assumption that $F(S) \subseteq S$ guarantees that there is a unique homomorphism of graded rings $F\colon  A^{\ast}[S^{-1}] \rightarrow A^{\ast}[S^{-1}]$ for which the diagram
$$ \xymatrix{ A^{\ast} \ar[d]^{F} \ar[r] & A^{\ast}[ S^{-1} ] \ar[d]^{F} \\
A^{\ast} \ar[r] & A^{\ast}[S^{-1}] }$$
is commutative, given concretely by the formula $F( \frac{x}{s} ) = \frac{ F(x)}{ F(s) }$. To complete the proof, it will suffice to show that the triple $( A^{\ast}[S^{-1}], d, F)$ is a Dieudonn\'{e} algebra:
that is, that is satisfies the axioms of Definition \ref{def20}. Axiom $(ii)$ is immediate. To verify $(i)$, we compute
\begin{align*}
dF( \tfrac{x}{s} ) & =  d\tfrac{ F(x) }{ F(s) } \\
& =  \tfrac{ dF(x)}{ F(s) } - \tfrac{ F(x) dF(s) }{ F(s)^2 } 
 =  \tfrac{ pF(dx)}{ F(s) } - \tfrac{ p F(x) F(ds) }{ F(s)^2 } \\
& =  pF( \tfrac{ dx}{s} - \tfrac{ x ds}{s^2} ) 
 =  pF( d \left( \tfrac{x}{s}\right) ). \\
\end{align*} 
To prove $(iii)$, choose any $x \in A^{0}$ and any $s \in S$. Since $A^{\ast}$ is a Dieudonn\'{e} algebra, we have
$$ F(x) s^{p} \equiv x^{p} s^{p} \equiv x^{p} F(s) \pmod{p},$$
so that we can write $F(x) s^{p} = x^{p} F(s) + py$ for some $y \in A^{0}$. It follows that
$$ F( \frac{x}{s} ) = \frac{ F(x)}{ F(s) } = \frac{ x^{p} }{s^{p} } + p \frac{y}{ s^{p} F(s) }$$
in the commutative ring $A^{0}[S^{-1}]$, so that $F( \frac{x}{s} ) \equiv ( \frac{x}{s} )^{p} \pmod{p}$.
\end{proof}

\begin{example}
Let $A^{\ast}$ be a Dieudonn\'{e} algebra and let $s$ be an element of $A^{0}$ which satisfies the equation $F(s) = s^{p}$.
Applying Proposition \ref{proposition.localize} to the set $S = \{ 1, s, s^2, \ldots \}$, we conclude that the localization
$A^{\ast}[s^{-1}]$ inherits the structure of a Dieudonn\'{e} algebra.
\end{example}

\begin{proposition}\label{proposition.localization-saturated}
Let $A^{\ast}$ be a Dieudonn\'{e} algebra, and let $s \in A^{0}$ be an element satisfying the equation $F(s) = s^{p}$.
If $A^{\ast}$ is saturated, then the localization $A^{\ast}[ s^{-1} ]$ is also saturated.
\end{proposition}

\begin{proof}
We first show that the Frobenius map $F\colon  A^{\ast}[ s^{-1} ] \rightarrow A^{\ast}[ s^{-1} ]$ is a monomorphism. 
Suppose that $F( \frac{x}{s^{k}} ) = 0$ for some $x \in A^{\ast}$. Then we have an equality $s^{n} F(x) = 0$ in $A^{\ast}$
for some $n \gg 0$. Enlarging $n$ if necessary, we can assume that $n = pm$ for some integer $m \geq 0$, so that
$$0 = s^{n} F(x) = F( s^{m} ) F(x) = F( s^{m} x ).$$
Since the Frobenius map $F$ is a monomorphism on $A^{\ast}$, it follows that $s^{m} x = 0$, so that
$\frac{x}{s^k}$ vanishes in $A^{\ast}[ s^{-1} ]$.

Now suppose that we are given an element $\frac{ y}{s^{n} } \in A^{\ast}[ S^{-1} ]$ such that
$d( \frac{y}{s^n} )$ is divisible by $p$ in $A^{\ast}[S^{-1}]$; we wish to show that $\frac{y}{s^n}$ belongs to the image of $F$.
Enlarging $n$ if necessary, we can assume that $n = pm$ for some $m \geq 0$, so that
$$ d( \frac{y}{s^n} ) = \frac{ dy}{s^{n}} - pm \frac{ y ds}{ s^{n+1} } \equiv \frac{ dy}{ s^n} \pmod{p}.$$
It follows that we can choose $k \gg 0$ such that $s^{pk} dy$ belongs to $p A^{\ast}$. 
We then have $$d( s^{pk} y) = s^{pk} dy + pk s^{pk-1} y ds \equiv s^{pk} dy \equiv 0 \pmod{p}.$$
Invoking our assumption that $A^{\ast}$ is saturated, we can write $s^{pk} y = Fz$ for some
$z \in A^{\ast}$. It follows that 
$\frac{y}{s^n} = \frac{ s^{pk} y}{ s^{p(k + m)} } = F( \frac{ z}{ s^{k + m} })$, so that $\frac{y}{s^n}$ belongs to the image of $F$ as desired.
\end{proof}

In the situation of Proposition \ref{proposition.localization-saturated}, the Dieudonn\'{e} algebra $A^{\ast}[ s^{-1} ]$ is usually
not strict, even if $A^{\ast}$ is assumed to be strict. However, the completed saturation of $A^{\ast}[ s^{-1}]$ is easy to describe, by virtue of the following result:

\begin{proposition}\label{proposition.localize-strictify}
Let $A^{\ast}$ be a saturated Dieudonn\'{e} algebra and let $s \in A^{0}$ be an element satisfying $F(s) = s^{p}$. 
Let $r \geq 0$ and let $\overline{s}$ denote the image of $s$ in the quotient $\WittScript_{r}( A )^{0}$. Then
the canonical map $\psi\colon  \WittScript_{r}(A)^{\ast}[ \overline{s}^{-1}] \rightarrow \WittScript_{r}( A[ s^{-1} ] )^{\ast}$
is an isomorphism of differential graded algebras.
\end{proposition}

\begin{proof}
The surjectivity of $\psi$ is immediate from the definitions. To prove
injectivity, we must verify the equality
$$ (V^{r} A^{\ast} + d V^{r} A^{\ast} )[ s^{-1} ] = V^{r}(A^{\ast}[ s^{-1} ]) + d V^{r}( A^{\ast}[ s^{-1} ] ).$$
Note that both sides are differential graded ideals in 
$A^{\ast}[s^{-1}]$, and the left hand side is contained in the right hand side.
To establish the reverse containment, it will suffice to show that the left hand side contains $V^{r}( A^{\ast}[ s^{-1} ] )$. This follows from the identity
\[  V^{r}(s^{-n} x) = V^{r}( F^{r}( s^{-n} ) s^{n( p^{r}-1)} x) = s^{-n} V^{r}(
s^{n(p^r-1)} x). \qedhere\]
\end{proof}

\begin{corollary}\label{six}
Let $R$ be a commutative $\F_p$-algebra, let $A^{\ast}$ be a Dieudonn\'{e} algebra, and let
$f\colon  R \rightarrow A^{0} / VA^{0}$ be a ring homomorphism which exhibits $A^{\ast}$ as
a saturated de~Rham--Witt complex of $R$ (Definition \ref{def80}). 
Let $s \in R$ be an element, and let us abuse notation by identifying the Teichm\"{u}ller representative
$[s] \in W(R)$ with its image in $A^{0}$. Set $B^{\ast} = A^{\ast}[ [s]^{-1} ]$. Then the map
$$ R[s^{-1}] \xrightarrow{f} (A^{0} / VA^{0})[ s^{-1}] \simeq B^{0} / V B^{0} =
\WittScript(B)^{0} / V \WittScript(B)^{0}$$
exhibits $\WittScript(B)^{\ast}$ as a saturated de~Rham--Witt complex of $R[s^{-1}]$.
\end{corollary}

\begin{proof}
For any strict Dieudonn\'{e} algebra $C^{\ast}$, we have a commutative diagram of sets
$$\xymatrix{
\Hom_{ \FrobAlg }( \WittScript(B)^{\ast}, C^{\ast} ) \ar[r] \ar[d] &  \Hom_{
\FrobAlg }( A^{\ast}, C^{\ast} ) \ar[d]^{\sim} \\
\Hom( R[ s^{-1} ], C^{0} / V C^{0} ) \ar[r] & \Hom( R, C^0 / V C^0 ). }$$
Note that the right vertical map is a bijection, and we wish to show that the left vertical map is also a bijection.
For this, it will suffice to show the diagram is a pullback square. Invoking the 
universal property of $B^{\ast}$ supplied by Remark \ref{seffer}, we can reformulate this assertion as follows:
\begin{itemize}
\item[$(\ast)$] Let $f\colon  A^{\ast} \rightarrow C^{\ast}$ be a morphism of Dieudonn\'{e} algebras. 
Then $f( [s] )$ is an invertible element of $C^{0}$ if and only if its image in $C^{0} / V C^{0}$ is invertible.
\end{itemize}
The ``only if'' direction is obvious. For the converse, suppose that $f([s] )$ admits an inverse in
$C^{0} / V C^{0}$, which is represented by an element $y \in C^{0}$. Then we can write
$y f([s] ) = 1 - Vz$ for some $z \in C^{0}$. Since $C^{0}$ is $V$-adically complete,
the element $f([s])$ has an inverse in $C^{0}$, given by the product $y( 1 + Vz + (Vz)^2 + \cdots )$.
\end{proof}

\subsection{The saturated de~Rham--Witt Complex of an $\F_p$-Scheme}\label{buildscheme}

We now apply the ideas of \S \ref{secsix} to globalize the construction of Definition \ref{def80}.

\begin{construction}\label{construction.global}
Let $X = (X, \sheafO_X)$ be an $\F_p$-scheme, and let $\calU_{\aff}(X)$ denote the collection of all affine open subsets of $X$.
We define a functor $\WOmega^{\ast}_{X}\colon  \calU_{\aff}(X)^{\op} \rightarrow \FrobAlg$ by the formula
$$ \WOmega^{\ast}_{X}(U) = \WOmega^{\ast}_{ \sheafO_X(U) }.$$
We regard $\WOmega^{\ast}_{X}(U)$ as a presheaf on $X$ (defined only on affine subsets of $X$) with values in the category
of Dieudonn\'{e} algebras.
\end{construction}

\begin{theorem}\label{makeglob}
Let $X$ be an $\F_p$-scheme. Then the presheaf $\WOmega^{\ast}_{X}$ of Construction \ref{construction.global} is a sheaf
(with respect to the topology on $\calU_{\aff}(X)$ given by open coverings).
\end{theorem}

\begin{proof}
We will show that, for every integer $d \geq 0$, the presheaf $\WOmega^{d}_{X}$ is a sheaf of abelian groups on $X$.
Unwinding the definitions, we can write $\WOmega^{d}_{X}$ as the inverse limit of a tower of presheaves
$$ \cdots \rightarrow \WrOmega{3}^{d}_{X} \rightarrow \WrOmega{2}^{d}_{X} \rightarrow \WrOmega{1}^{d}_{X},$$
where $\WrOmega{r}^{d}_{X}$ is given by the formula $\WrOmega{r}^{d}_{X}(U) = \WrOmega{r}^{d}_{ \sheafO_X(U) }$. 
It will therefore suffice to show that each $\WrOmega{r}^{d}_{X}$ is a sheaf of abelian groups. In other words,
we must show that for every affine open subset $U \subseteq X$ and every covering of $U$ by
affine open subsets $U_{\alpha}$, the sequence
$$ 0 \rightarrow \WrOmega{r}^{d}_{X}(U) \rightarrow \prod_{\alpha} \WrOmega{r}^{d}_{X}(U_{\alpha} )
\rightarrow \prod_{\alpha, \beta} \WrOmega{r}^{d}_{X}(U_{\alpha} \cap U_{\beta})$$
is exact. Without loss of generality, we may replace $X$ by $U$ and thereby reduce to the case where
$X = \Spec(R)$. By a standard argument, we can further reduce to the case where each $U_{\alpha}$
is the complement of the vanishing locus of some element $s_{\alpha} \in R$. In this case,
we wish to show that the sequence 
$$ 0 \rightarrow \WrOmega{r}^{d}_{R} \rightarrow \prod_{\alpha} \WrOmega{r}^{d}_{R[ s_{\alpha}^{-1} ]}
\rightarrow \prod_{\alpha, \beta} \WrOmega{r}^{d}_{R[ s_{\alpha}^{-1}, s_{\beta}^{-1} ]}$$
is exact. Setting $M = \WrOmega{r}^{d}_{R}$, we can apply Corollary \ref{six}
and Proposition \ref{proposition.localize-strictify} to rewrite this sequence as 
$$ 0 \rightarrow M \rightarrow \prod_{\alpha} M[ \overline{s}_{\alpha}^{-1} ]
\rightarrow \prod_{\alpha, \beta} M[ \overline{s}_{\alpha}^{-1}, \overline{s}_{\beta}^{-1} ],$$
where $\overline{s}_{\alpha}$ denotes the image in $W_r(R)$ of the Teichm\"{u}ller representative
$[s_{\alpha}] \in W(R)$. The desired result now follows from the observation that the elements
$\overline{s}_{\alpha}$ generate the unit ideal in $W_r(R)$.
\end{proof}

It follows from Theorem \ref{makeglob} that for any $\F_p$-scheme $X$, the presheaf
$\WOmega^{\ast}_{X}$ can be extended uniquely to a sheaf of Dieudonn\'{e} algebras on
the collection of all open subsets of $X$. We will denote this sheaf also by
$\WOmega^{\ast}_{X}$ and refer to it as the {\it saturated de~Rham--Witt complex of $X$}.

\begin{remark}\label{billings}
The proof of Theorem \ref{makeglob} shows that each $\WrOmega{r}^{d}_{X}$
can be regarded as a {\em quasi-coherent} sheaf on the $(\Z / p^{r} \Z)$-scheme
$(X, W_r( \sheafO_X) )$; here $W_r( \sheafO_X )$ denotes the sheaf of commutative
rings on $X$ given on affine open sets $U$ by the formula $W_r( \sheafO_X)(U) = W_r( \sheafO_X(U) )$. 
\end{remark}

\begin{proposition}\label{snippet}
Let $R$ be a commutative $\F_p$-algebra and set $X = \Spec(R)$.
Then, for every integer $d$, the canonical map
$\WOmega_{R}^{d} \rightarrow \mathrm{H}^{0}( X; \WOmega^{d}_{X} )$ is an isomorphism
and the cohomology groups $\mathrm{H}^{n}(X; \WOmega^{d}_X )$ vanish for $n > 0$.
\end{proposition}

\begin{proof}
Let $\sheafF^{\bullet}$ denote the homotopy limit of the diagram
$$ \cdots \rightarrow \WrOmega{3}^{d}_{X} \rightarrow \WrOmega{2}^{d}_{X} \rightarrow \WrOmega{1}^{d}_{X},$$
in the derived category of abelian sheaves on $X$. For every affine open subset $U \subseteq X$,
the hypercohomology $\RGamma( U; \sheafF^{\bullet} )$ can be identified with the homotopy limit of the diagram
$$ \cdots \rightarrow \RGamma(U; \WrOmega{3}^{d}_{X}) \rightarrow \RGamma(U; \WrOmega{2}^{d}_{X}) \rightarrow \RGamma(U; \WrOmega{1}^{d}_{X})$$
in the derived category of abelian groups. It follows from Remark \ref{billings} that each $\WrOmega{r}^{d}_{X}$ can be regarded as a quasi-coherent sheaf on $(X, W_r( \sheafO_X ) )$, so
we can identify the preceding diagram with the tower of abelian groups
$$ \cdots \rightarrow \WrOmega{3}^{d}_{ \sheafO_X(U) } \rightarrow \WrOmega{2}^{d}_{ \sheafO_X(U) } \rightarrow \WrOmega{1}^{d}_{ \sheafO_X(U) }.$$
This diagram has surjective transition maps, so its homotopy limit can be identified with the abelian group $\WOmega^{d}_{ \sheafO_X(U) }$ (regarded
as an abelian group, concentrated in degree zero). It follows that $\sheafF^{\bullet}$ can be identified with $\WOmega^{d}_{X}$ (regarded as a chain complex concentrated in degree zero),
so the preceding calculation gives $\RGamma(U; \WOmega^{d}_{X} ) \simeq \WOmega^{d}_{ \sheafO_X(U) }$ for each affine open subset $U \subseteq X$.
Proposition \ref{snippet} now follows by taking $U = X$.
\end{proof}
\subsection{Localization for the \'{E}tale Topology}\label{orbit}

Let $R$ be a commutative $\F_p$-algebra. In \S \ref{secsix}, we showed that
the saturated de~Rham--Witt complex $\WOmega^{\ast}_{R[s^{-1}]}$ of any localization
$R[s^{-1}]$ can be described explicitly in terms of $\WOmega^{\ast}_{R}$ (Corollary \ref{six}). In this section,
we will formulate a generalization of this result which applies to any \'{e}tale $R$-algebra $S$ (Corollary \ref{Vetalecor}).

Let $(A^{\ast}, d)$ be a commutative differential graded algebra. In what follows, we will use
the term {\it $A^{\ast}$-algebra} to refer to a commutative differential graded algebra $B^{\ast}$ equipped with a
morphism of differential graded algebras $A^{\ast} \rightarrow B^{\ast}$.

\begin{definition}\label{etaledga}
Let $f\colon  A^{\ast} \to B^{\ast}$ be a map of commutative differential
graded algebras. We will say that $f$ is \emph{\'etale} (or that $B^{\ast}$ is
an \emph{\'etale $A^{\ast}$-algebra})
if the ring homomorphism $f\colon  A^0 \to B^0$ is \'etale and the map of graded algebras $A^{\ast}
\otimes_{A^0} B^0 \to B^{\ast}$ is an isomorphism. 
\end{definition} 

If $A^{\ast}$ is a commutative differential graded algebra, then every \'{e}tale $A^{0}$-algebra can be extended to an \'{e}tale $A^{\ast}$-algebra:

\begin{proposition} 
\label{compareetale}
Let $A^{\ast}$ be a commutative differential graded algebra. Then:
\begin{itemize}
\item[$(1)$] For every \'{e}tale $A^{0}$-algebra $R$, there exists an
\'{e}tale $A^{\ast}$-algebra $B^{\ast}$ and an isomorphism of $A^{0}$-algebras $R \simeq B^{0}$.

\item[$(2)$] Let $B^{\ast}$ be an \'{e}tale $A^{\ast}$-algebra. Then, for any $A^{\ast}$-algebra $C^{\ast}$, the canonical map
$$ \Hom_{ A^{\ast}}( B^{\ast}, C^{\ast} ) \rightarrow \Hom_{ A^{0} }( B^0, C^0)$$
is bijective.

\item[$(3)$] The construction $B^{\ast} \mapsto B^{0}$ induces an equivalence from the category of \'{e}tale $A^{\ast}$-algebras to the
category of \'{e}tale $A^{0}$-algebras.
\end{itemize}
\end{proposition} 

\begin{proof} 
Let $R$ be an \'{e}tale $A^{0}$-algebra. Then the canonical map
$R \otimes_{ A^{0} } \Omega^{1}_{ A^{0} } \rightarrow \Omega^{1}_{R}$ is an isomorphism.
It follows that the de~Rham complex $\Omega^{\ast}_{R}$ is given, as a graded ring, by the
tensor product $R \otimes_{A^{0}} \Omega^{\ast}_{A^{0}}$. Extending scalars along the map
$\Omega^{\ast}_{A^{0}} \rightarrow A^{\ast}$, we obtain an isomorphism of graded rings
$$ R \otimes_{ A^{0} } A^{\ast} \simeq \Omega^{\ast}_{R} \otimes_{ \Omega^{\ast}_{A^0} } A^{\ast}.$$
Setting $B^{\ast} = \Omega^{\ast}_{R} \otimes_{ \Omega^{\ast}_{A^0} } A^{\ast}$, we conclude
that $B^{\ast}$ is an \'{e}tale $A^{\ast}$-algebra with $B^{0} \simeq R$, which proves $(1)$.
Moreover, $B^{\ast}$ has the universal property described in assertion $(2)$. If $B'^{\ast}$ is any other \'{e}tale $A^{\ast}$-algebra
equipped with an isomorphism $\alpha\colon  R \simeq B'^{0}$, then $\alpha$ extends uniquely to a map of $A^{\ast}$-algebras $\overline{\alpha}\colon  B^{\ast} \rightarrow B'^{\ast}$,
which is automatically an isomorphism (since it is an isomorphism in degree zero and the domain and codomain of $\overline{\alpha}$ are both \'{e}tale over $A^{\ast}$).
It follows that $B'^{\ast}$ also has the universal property of assertion $(2)$. Assertion $(3)$ is an immediate consequence of $(1)$ and $(2)$.
\end{proof} 

We now formulate an analogue of Proposition \ref{compareetale} in the setting
of strict Dieudonn\'{e} algebras.

\begin{definition} 
\label{def:Vetale}
Let $f\colon  A^{\ast} \to B^{\ast}$ be a morphism of strict Dieudonn\'e algebras. We will say that 
$f$ is \emph{$V$-adically \'etale} if, for each $n$, $\mathcal{W}_n(f)\colon 
\mathcal{W}_n(A)^{\ast} \to \mathcal{W}_n(B)^{\ast}$ is \'etale 
as a morphism of commutative differential graded algebras (Definition~\ref{etaledga}).
\end{definition} 

\begin{theorem} 
\label{Vetale}
Let $A^{\ast}$ be a strict Dieudonn\'e algebra. Then:
\begin{itemize}
\item[$(1)$] For every \'{e}tale $A^{0} / VA^{0}$-algebra $R$, there exists an
$V$-adically \'{e}tale morphism of strict Dieudonn\'{e} algebras $A^{\ast} \rightarrow B^{\ast}$ and an isomorphism of
$A^{0}/VA^{0}$-algebras $B^{0}/VB^{0} \simeq R$.

\item[$(2)$] Let $f\colon  A^{\ast} \rightarrow B^{\ast}$ be a $V$-adically
\'{e}tale morphism of strict Dieudonn\'{e} algebras.
Then, for every morphism of strict Dieudonn\'{e} algebras $A^{\ast} \rightarrow C^{\ast}$, the canonical map
$$ \Hom_{ \FrobAlg_{A^{\ast}} }( B^{\ast}, C^{\ast} ) \rightarrow \Hom_{ A^{0}/VA^{0} }( B^0 / VB^{0}, C^0/VC^{0} )$$
is bijective.

\item[$(3)$] The construction $B^{\ast} \mapsto B^{0} / VB^0$ induces an equivalence from the category of $V$-adically \'{e}tale
strict Dieudonn\'e algebras over $A^{\ast}$ to the category of \'etale $A^0/VA^0$-algebras. 
\end{itemize}
\end{theorem}

We will give the proof of Theorem \ref{Vetale} in \S \ref{pushetale}. Note that it
implies an analogue of Corollary~\ref{six}: 

\begin{corollary}\label{Vetalecor}
Let $R \to S$ be an \'etale map of $\mathbb{F}_p$-algebras. 
Then the map $\mathcal{W} \Omega_R^{\ast} \to \mathcal{W} \Omega_S^{\ast}$ of strict Dieudonn\'e
complexes is $V$-adically \'etale. Furthermore, for each $n$, we have an
isomorphism \[ \mathcal{W}_n \Omega_R^{\ast} \otimes_{W_n(R)} W_n(S) \simeq \mathcal{W}_n
\Omega_S^{\ast}. \]
\end{corollary}

\begin{proof} 
Let $A^{\ast}$ be a strict Dieudonn\'{e} algebra equipped with a map $R \to A^{0}/VA^{0}$ which
exhibits $A^{\ast}$ as a saturated de~Rham--Witt complex of $R$. Set $S' = S \otimes_{R} (A^{0} / VA^{0})$. 
By Theorem~\ref{Vetale}, there exists a $V$-adically \'etale strict Dieudonn\'e algebra $B^{\ast}$ over
$A^{\ast}$ such that $B^{0}/VB^{0} \simeq S'$.
For any strict Dieudonn\'{e} algebra $C^{\ast}$, we have bijections
\begin{eqnarray*}
\Hom_{ \FrobAlg}( B^{\ast}, C^{\ast} ) & \simeq & \Hom_{ \FrobAlg}( A^{\ast},
C^{\ast} ) \times_{ \Hom( A^{0} / VA^{0}\ , C^{0} / VC^{0} ) }
\Hom( S', C^{0} / VC^{0} ) \\
& \simeq &  \Hom_{ \FrobAlg}( A^{\ast}, C^{\ast} ) \times_{ \Hom( R\ , C^{0} / VC^{0} ) }
\Hom( S, C^{0} / VC^{0} ) \\
& \simeq & \Hom(S, C^{0} / VC^{0} ), \end{eqnarray*}
so that the canonical map $S \rightarrow B^{0} / VB^{0}$ exhibits $B^{\ast}$ as a saturated de~Rham--Witt complex of $S$.
For each $n \geq 0$, we have canonical maps
$$\mathcal{W}_n(A)^{\ast} \otimes_{W_n(R)} W_n(S)
\xrightarrow{\alpha} \mathcal{W}_n(A)^{\ast} \otimes_{W_n( A^{0}/VA^{0}) } W_n(S') \xrightarrow{\beta} \mathcal{W}_n(B)^{\ast}.$$
Here $\beta$ is an isomorphism because $B^{\ast}$ is $V$-adically \'{e}tale over $A^{\ast}$, and $\alpha$ is an isomorphism
by virtue of Theorem \ref{wittetale}. It follows that $\beta \circ \alpha$ is an isomorphism.
\end{proof}

Let us now formulate the \'etale analog of Theorem~\ref{makeglob}.

\begin{construction}\label{construction.global.etale}
For any scheme $X$, let $\calU_{\aff,\mathet}(X)$ denote the category of affine schemes $U$ equipped with an \'{e}tale map $U \to X$. When $X$ is an $\F_p$-scheme, the assignment
\[ \WOmega^{\ast}_{X_{\mathet}}(U \to X) = \WOmega^{\ast}_{ \sheafO_U(U) },\]
defines a $\FrobAlg$-valued presheaf  $\W\Omega^{\ast}_{X_{\mathet}}(-)$ on $\mathcal{U}_{\aff,\mathet}(X)$.
\end{construction}

\begin{theorem}
\label{makeglobetale}
Let $X$ be an $\F_p$-scheme. Then the presheaf $\WOmega^{\ast}_{X_{\mathet}}$ of Construction \ref{construction.global.etale} is a sheaf with respect to the \'etale topology on $\calU_{\aff,\mathet}(X)$.
\end{theorem}

\begin{proof}
Let us begin by recalling a general fact. If $\Spec(A)$ is an affine scheme and
$M$ is an $A$-module, write $(\widetilde{M})_{\mathet}$ for the presheaf on $\mathcal{U}_{\aff,\mathet}(\Spec(A))$ determined by the formula $(\Spec(B) \to \Spec(A)) \mapsto M \otimes_A B$. It is a basic fact in descent theory that $(\widetilde{M})_{\mathet}$ is an \'etale sheaf. 

To prove Theorem \ref{makeglobetale}, it will suffice (as in the proof of Theorem~\ref{makeglob}) to show that the presheaf $\W_n\Omega^i_{X_{\mathet}}(-)$ determined by $(U \to X) \mapsto \W_n \Omega^i_{\mathcal{O}(U)}$ is a sheaf of abelian groups on $\mathcal{U}_{\aff,\mathet}(X)$ for each $i,n \geq 0$. To prove this, we may assume without loss of generality that $X \simeq \mathrm{Spec}(R)$ is affine. In this case, the functor $S \mapsto W_n(S)$ identifies the category of \'etale $R$-algebras with category of \'etale $W_n(R)$-algebras: this follows by combining Theorem~\ref{wittetale} below with the topological invariance of the \'etale site (since the restriction map $W_n(R) \to W_1(R) \simeq R$ is surjective with nilpotent kernel). Under the resulting equivalence $\mathcal{U}_{\aff,\mathet}(\Spec(R)) \simeq \mathcal{U}_{\aff,\mathet}(\Spec(W_n(R)))$, it follows from Corollary~\ref{Vetalecor} that the presheaf $\W_n \Omega^i_{X_{\mathet}}(-)$ is isomorphic to the sheaf $\widetilde{M}_{\mathet}$
associated to the $W_n(R)$-module $M = \W_n \Omega^{i}_{R}$.
\end{proof}
\subsection{Digression: Witt Vectors and \'{E}tale Morphisms}\label{wittetalesec}

To formulate and prove the \'etale localization property of the saturated
de~Rham--Witt complex  (Theorem~\ref{Vetale} above), we will make use of the following result:

\begin{theorem}\label{wittetale} 
Let $f\colon  A \to B$ be an \'etale morphism of commutative rings.
Then, for every integer $n \geq 0$, the induced map $W_n(A) \to W_n(B)$ is also \'{e}tale.
Moreover, the diagram
\[ \xymatrix{
W_n(A) \ar[d]^R \ar[r] &  W_n(B) \ar[d]^R \\
W_{n-1}(A) \ar[r] &  W_{n-1}(B) 
} \]
is a pushout square of commutative rings.
Finally, if $A \to A'$ is any map, then the induced square
of commutative rings
\[ \xymatrix{
W_n(A) \ar[d] \ar[r] &  W_n(A') \ar[d]  \\
W_n(B) \ar[r] &  W_n(A' \otimes_A B)
}\]
is a pushout. 
\end{theorem} 

Theorem \ref{wittetale} (in various forms) has been proved by Langer--Zink \cite[Prop. A.8, A.11]{LZ}, van der
Kallen \cite[Th. 2.4]{vdkdesc}, and Borger \cite[Th. 9.2]{Borgerbasic}. We will need only the special case where
$A$ and $B$ are $\F_p$-algebras which appears in the original work of Illusie \cite[Prop. 1.5.8]{illusie}. 

\begin{proof}[Proof of Theorem \ref{wittetale} for $\F_p$-algebras]
Note that the last assertion follows from the first two in view of the
topological invariance of the \'etale site, since $W_n(A) \to A$ is a nilpotent
thickening. 
For the first two assertions, we proceed by induction on $n$, the case $n = 1$ being trivial. To carry out the inductive step, let us assume that the map $W_n(A) \rightarrow W_n(B)$ is \'{e}tale and that the diagram
$$ \xymatrix{ W_n(A) \ar[r] \ar[d]^{R^{n-1}} & W_n(B) \ar[d]^{R^{n-1}} \\
A \ar[r] & B }$$
is a pushout square. Since $A$ is an $\F_p$-algebra, we can regard $W_{n+1}(A)$ as a square-zero extension of $W_n(A)$. Invoking
the topological invariance of the \'{e}tale site, we see that $W_n(f)\colon  W_n(A) \rightarrow W_n(B)$ can be lifted to an \'{e}tale map
$\overline{f}\colon  W_{n+1}(A) \rightarrow \widetilde{B}$. Moreover, the infinitesimal lifting property of \'{e}tale morphisms guarantees
that there is a unique $W_{n+1}(A)$-algebra map $\rho\colon  \widetilde{B} \rightarrow W_{n+1}(B)$ lifting the identity map $\id\colon  W_n(B) \rightarrow W_n(B)$. We have
a commutative diagram
\[ \xymatrix{ 0 \ar[r] & V^n A \ar[d] \ar[r] & W_{n+1}(A) \ar[r]^-{R} \ar[d] & W_n(A) \ar[d] \ar[r] & 0 \\	
		0 \ar[r] & V^n A \otimes_{W_{n+1}(A)} \widetilde{B} \ar[d]^{\rho_0} \ar[r] & \widetilde{B} \ar[r] \ar[d]^{\rho} & W_n(B) \ar[r] \ar@{=}[d] & 0 \\			
		0 \ar[r] & V^n B \ar[r] & W_{n+1}(B) \ar[r]^-R & W_n(B) \ar[r] & 0}\]
where each row is a short exact sequence, and the middle row is obtained from
the top row by extending scalars along $\overline{f}$.
To complete the proof, it will suffice to show that $\rho$ is an isomorphism of commutative rings, or equivalently that $\rho_0$ is an isomorphism
of abelian groups. Note that the action of $W_{n+1}(A)$ on the ideal $V^n A$ factors through the restriction map $R^{n}\colon  W_{n+1}(A) \rightarrow A$ (and
that $A$ acts on $V^n A \simeq A$ via the iterated Frobenius map $\varphi_{A}^{n}\colon  A \rightarrow A$). 
Moreover, the map $\rho_0$ factors as a composition
$$ V^n A \otimes_{ W_{n+1}(A)} \widetilde{B} \simeq V^n A \otimes_{ W_n(A)} W_n(B) \xrightarrow{\alpha} V^n A \otimes_{A} B \xrightarrow{\beta} V^n B,$$
where $\alpha$ is an isomorphism by virtue of our inductive hypothesis and $\beta$ is an isomorphism because
the diagram
\[ \xymatrix{ A \ar[r]^-{\varphi_{A}^{n}} \ar[d] & A \ar[d] \\
		  B \ar[r]^-{\varphi_{B}^{n}} & B }\]
is a pushout square (since $A \to B$ is \'etale; see \cite[Tag 0EBS]{stacks-project}). 
It follows that $\rho_0$ is an isomorphism, as desired.
\end{proof}

\begin{remark}\label{olose}
Let $f\colon  A \to B$ be an \'etale morphism of $\F_p$-algebras. Then, for every pair of integers $n,k \geq 0$, the diagram
$$ \xymatrix{ W_n(A) \ar[r]^{ W_n(f) } \ar[d]^{ F^k } & W_n(B) \ar[d]^{ F^k} \\
W_n(A) \ar[r]^{ W_n(f) } & W_n(B) }$$
is a pushout square of commutative rings. To prove this, we note that the induced map
$$ \theta\colon  W_n(A) \otimes_{ W_n(A)} W_n(B) \rightarrow W_n(B)$$
is a morphism of \'{e}tale $W_n(A)$-algebras. Consequently, it will suffice to show that $\theta$
is an isomorphism after extending scalars along the restriction map $W_n(A) \rightarrow A$ (since
$\ker( W_n(A) \rightarrow A)$ is nilpotent), which allows us to reduce to the case $n=1$.
\end{remark}

\begin{definition} 
Let $A$ be an $\mathbb{F}_p$-algebra and let $M$ be a $W(A)$-module. We will say
that $M$ is \emph{nilpotent} if $M$ is annihilated by the ideal $V^n W(A)
\subseteq W(A)$ for some $n$, i.e., if $M$ is actually a $W_n(A)$-module for some
$n$. 
\end{definition} 

\begin{definition} 
Let $A \to B$ be an \'etale map of $\mathbb{F}_p$-algebras. If $M$ is a nilpotent $W(A)$-module,
so that $M$ is a $W_n(A)$-module for some $n$, 
then we write $M_B$ for the $W(B)$-module $ M \otimes_{W_n(A)} W_n(B)$. 
By Theorem~\ref{wittetale}, the construction of $M_B$ does not depend on
the choice of $n$.
\end{definition} 

\begin{remark}\label{obos}
The construction $M \mapsto M_B$ is left adjoint
to the forgetful functor from nilpotent $W(B)$-modules to nilpotent
$W(A)$-modules. In addition, the functor $M \mapsto M_B$ is exact because each of the maps $W_n(A) \to W_n(B)$
is \'etale (and therefore flat). 
\end{remark}

Let $A$ be an $\F_p$-algebra and let $M$ be a $W(A)$-module. For each non-negative integer $k$, we let $M_{(k)}$ denote the $W(A)$-module
obtained by restriction of scalars of $M$ along the Witt vector Frobenius map $F^k\colon  W(A) \to W(A)$. 

\begin{remark}
If $M$ is a nilpotent $W(A)$-module, then $M_{(k)}$ is also nilpotent. 
In fact, if $M$ is a $W_n(A)$-module for some integer $n$, then 
$M_{(k)}$ is also a $W_n(A)$-module.
\end{remark}

\begin{proposition} 
\label{frobtwistbase}
Given an \'etale map $A \to B$ of $\mathbb{F}_p$-algebras and a nilpotent
$W(A)$-module $M$, we have $(M_B)_{(k)} \simeq (M_{(k)})_B$ for each integer $k$. 
\end{proposition} 

\begin{proof} 
Choose an integer $n \gg 0$ for which $M$ can be regarded as a $W_n(A)$-module.
The desired result then follows from the fact that the diagram
$$ \xymatrix{ W_n(A) \ar[r]^{ W_n(f) } \ar[d]^{ F^k } & W_n(B) \ar[d]^{ F^k} \\
W_n(A) \ar[r]^{ W_n(f) } & W_n(B) }$$
is a pushout (see Remark \ref{olose}).
\end{proof} 

\begin{definition}\label{bost}
Let $M, N$ be nilpotent $W(A)$-modules. We will say that a homomorphism $f\colon  M \to N$ of
abelian groups is $(F^k, F^{\ell})$-linear if 
for all $x \in W(A)$ and $m \in M$, we have
$$f( (F^k x)  m )  = (F^{\ell} x) f(x) \in N.$$
Equivalently, $f$ defines a map of $W(A)$-modules $M_{(k)} \to N_{(\ell)}$. 
\end{definition} 

\begin{corollary} 
\label{extendmapsFaFb}
Fix an \'etale map $A \to B$ of $\F_p$-algebras. 
Let $M$ be a nilpotent $W(A)$-module and let $N$ be a nilpotent $W(B)$-module,
considered as a nilpotent $W(A)$-module via restriction of scalars. 
Then any $(F^k, F^{\ell})$-linear map $f\colon  M \to N$ extends uniquely to a $(F^{k},
F^{\ell})$-linear map $M_{B} \to N$. 
\end{corollary} 

\begin{proof} 
Apply Proposition~\ref{frobtwistbase}. 
\end{proof} 

\begin{remark}
\label{sespreserved}
In the situation of Corollary~\ref{extendmapsFaFb}, suppose
that we are given nilpotent $W(A)$-modules $M$, $N$, and $P$, together
with an $(F^a, F^b)$-lienar map $f\colon  M \to N$ and an $(F^c, F^d)$-linear map
$g\colon  N \to P$. Suppose that the sequence $M \stackrel{f}{\to} N \stackrel{g}{\to} P$ is exact in the category of
abelian groups. Then $M_{B} \stackrel{f_B}{\to} N_B \stackrel{g_B}{\to} P_B$ is
exact as well. 
To see this, by replacing $(a,b)$ by $(a+k, b+k)$ and $(c,d)$ by $(c+l, d+l)$
for some $k, l \geq 0$, we can assume that $b = c$. 
We then have an exact sequence of nilpotent $W(A)$-modules $M_{(a)} \to
N_{(b)} \to P_{(d)}$, so that extending scalars to $W(B)$ yields an exact sequence
$(M_B)_{(a)} \to (N_B)_{(b)} \to (P_B)_{(d)}$ by virtue of Remark \ref{obos}.
\end{remark}

\subsection{The Proof of Theorem~\ref{Vetale}}\label{pushetale}

We begin by proving the second assertion of Theorem~\ref{Vetale}.

\begin{proposition}\label{Vetaleprop}
Let $A^{\ast}$ be a strict Dieudonn\'{e} algebra, let $B^{\ast}$ be a strict Dieudonn\'e algebra which is $V$-adically \'etale over
$A^{\ast}$, and let $C^{\ast}$ be another strict Dieudonn\'e
algebra over $A^{\ast}$. Then the restriction map
\begin{equation} 
\label{etalevcompl}
\Hom_{\FrobAlg_{A^{\ast}}} (B^{\ast}, C^{\ast}) \simeq \Hom_{A^0/VA^0}( B^0/VB^0,
C^0/VC^0).
\end{equation} 
is bijective.
\end{proposition}

\begin{proof}
Let $\overline{f}\colon  B^{0} / V B^{0} \rightarrow C^{0} / V C^{0}$ be a morphism of 
$A^{0} / V A^{0}$-algebras; we wish to show that $\overline{f}$ can be lifted
uniquely to a morphism $f\colon  B^{\ast} \rightarrow C^{\ast}$ of strict Dieudonn\'{e} algebras over $A^{\ast}$.
For each $n \geq 0$, we have a commutative diagram
$$ \xymatrix{ A^{0} /V^{n} A^{0} \ar[r] \ar[d] & C^{0} / V^{n} C^{0} \ar[d] \\
B^0 / V^{n} B^{0} \ar@{-->}[ur]^{f^{0}_{n} } \ar[r]^{ \overline{f} } & C^{0} / V C^{0}, }$$
where the left vertical map is {\etale} (by virtue of Proposition~\ref{prop28} and Theorem~\ref{wittetale})
and the right vertical map is a surjection with nilpotent kernel. It follows that there is a unique ring homomorphism $f^{0}_{n}\colon  B^{0} /V^{n} B^{0} \rightarrow C^{0} / V^{n} C^{0}$ as indicated
which makes the diagram commute. Using Proposition~\ref{compareetale}, we see that each $f^{0}_{n}$ can be extended uniquely to a map of $\W_{n}(A)^{\ast}$-modules
$f_n\colon  \W_n(B)^{\ast} \rightarrow \W_n(C)^{\ast}$. Passing to the inverse limit over $n$, we obtain a map of differential graded algebras $f\colon  B^{\ast} \rightarrow C^{\ast}$. We will complete
the proof by showing that $f$ is a morphism of strict Dieudonn\'{e} algebras: that is, that it commutes with the Frobenius. For this, it suffices to prove the commutativity of the diagram
$$\xymatrix{
\mathcal{W}_n(B)^{\ast} \ar[d]^{F} \ar[r]^{f_n} &  \mathcal{W}_{n}(C)^{\ast}
\ar[d]^{F}  \\
\mathcal{W}_{n-1}(B)^{\ast} \ar[r]^{f_{n-1}} &  \mathcal{W}_{n-1}(C)^{\ast}
}$$
for each $n \geq 0$. Invoking Proposition~\ref{compareetale} again, we are reduced to proving the commutativity of the diagram
of the left square in the diagram of $A^{0} / V^{n} A^{0}$-algebras
$$ \xymatrix{ B^{0} / V^{n} B^{0} \ar[r]^{ f_{n}^{0} } \ar[d]^{F} & C^{0} / V^n C^{0} \ar[r] \ar[d]^{F} & C^{0} / V C^{0} \ar[d]^{F} \\
B^{0} / V^{n-1} B^{0} \ar[r]^{ f_{n-1}^{0} } & C^{0} / V^{n-1} C^{0} \ar[r] & C^{0} / V C^{0}. }$$
Note that the right square in this diagram commutes, and that the horizontal maps on the right are surjections with nilpotent kernel.
Since $B^{0} / V^{n} B^{0}$ is {\etale} over $A^{0} / V^{n} A^{0}$, we are reduced to showing that the outer rectangle commutes. This follows from
the commutativity of the diagram
\[    \xymatrix{ B^{0} / V^{n} B^{0} \ar[r] \ar[d]^{F} & B^{0} / VB^{0} \ar[r]^{\overline{f}} \ar[d]^{F} & C^{0} / V C^{0} \ar[d]^{F} \\
B^{0} / V^{n-1} B^{0} \ar[r] & B^{0} / VB^{0} \ar[r]^{\overline{f}} & C^{0} / V
C^{0}. } \qedhere \]
\end{proof}

To prove the first assertion of Theorem \ref{Vetale}, we will need the following:

\begin{lemma} 
\label{froblinear}
Let $R$ be a commutative ring.
Let $M$ be a $W_n(R)$-module such that $p^k M = 0$ and let $d\colon  W_n(R) \to M$ be
a derivation. Then the composite $d \circ F^k\colon  W_{n+k}(R) \to M$ vanishes. 
\end{lemma} 

\begin{proof} 
Let $x$ be an element of $W_{n+k}(R)$; we wish to show that $d( F^{k} x)$ vanishes.
Without loss of generality, we may assume that $x = V^{m} [a]$, where $[a] \in W_{n+k-m}(R)$
denotes the image of the Teichm\"{u}ller representative of some $a \in R$.
If $m \geq k$, then $F^k(x) = p^k V^{m-k}([a])$. If $m \leq k$, then
$F^{k} x = p^{m} [a]^{ p^{k-m} }$, so that $d( F^k x ) = p^{k} [a]^{p^{k-m}-1} d[a]$.
In either case, the desired result follows from our assumption that $p^k M = 0$.
\end{proof} 

\begin{example}\label{example:diffetale}
Let $A \rightarrow B$ be an \'{e}tale morphism of $\mathbb{F}_p$-algebras.
Suppose we are given a commutative differential graded algebra $R^{\ast}$ and
a map of rings $W(A) \rightarrow R^{0}$ which factors through $W_n(A)$ for some $n$.
Let $R^{\ast}_B$ denote the tensor product $R^{\ast} \otimes_{W_n(A)} W_n(B)$, formed
in the category of graded commutative rings. It follows from Lemma \ref{froblinear}
that the differential $d\colon  R^{\ast} \to R^{\ast}$ is $(F^n,
F^n)$-linear, in the sense of Definition \ref{bost}. It follows from
Corollary~\ref{extendmapsFaFb} that there is a unique
differential $d$ on the graded ring $R^{\ast}_B$
which is $(F^n, F^n)$-linear as a map of $W(B)$-modules and compatible with the differential
on $R^{\ast}$. Note that $R^{\ast}_{B}$ is {\etale} when regarded as a commutative differential graded algebra over
$R^{\ast}$ (in the sense of Definition \ref{etaledga}).
\end{example} 

\begin{proof}[Proof of Theorem~\ref{Vetale}]
Assertion $(2)$ of Theorem~\ref{Vetale} follows from Proposition \ref{Vetaleprop}, and assertion $(3)$ follows formally from $(1)$ and $(2)$
(as in the proof of Proposition \ref{compareetale}). We will prove $(1)$. Let $A^{\ast}$ be a strict Dieudonn\'{e} algebra, set $R = A^{0} / V A^{0}$, and let
$S$ be an {\etale} $R$-algebra; we will construct a strict Dieudonn\'{e} algebra $B^{\ast}$ which is
$V$-adically {\etale} over $A^{\ast}$ such that $B^{0} / V B^{0}$ is isomorphic to $S$ (as an $R$-algebra). 

For each $n \geq 0$, the commutative ring $W_n(S)$ is {\etale} over
$\mathcal{W}_n(A)^{0}$ (Theorem \ref{wittetale} and Proposition~\ref{prop28}). 
Using part $(1)$ of Proposition \ref{compareetale}, we can choose a commutative differential graded algebra
$B^{\ast}_{n}$ which is {\etale} over $\mathcal{W}_n(A)^{\ast}$ and an isomorphism $B^{0}_{n} \simeq W_n(S)$
of $\mathcal{W}_n(A)^{0}$-algebras. Using part $(2)$ of Proposition \ref{compareetale}, we see that the Frobenius
and restriction maps $R, F\colon  W_n(S) \rightarrow W_{n-1}(S)$ extend uniquely to maps of commutative
differential graded algebras $R, F\colon  B^{\ast}_{n} \rightarrow B^{\ast}_{n-1}$ for which the diagrams
$$ \xymatrix{ \mathcal{W}_{n}(A)^{\ast} \ar[d]^R \ar[r] & B^{\ast}_{n} \ar[d]^{R} & \mathcal{W}_n(A)^{\ast} \ar[d]^{F} \ar[r] & B^{\ast}_{n} \ar[d]^{F} \\
\mathcal{W}_{n-1}(A)^{\ast} \ar[r] & B^{\ast}_{n-1} & \mathcal{W}_{n-1}(A)^{\ast} \ar[r] & B^{\ast}_{n-1} }$$
commute. Let $B^{\ast}$ denote the inverse limit $\varprojlim B^{\ast}_{n}$ in the category of commutative differential graded algebras,
so that the Frobenius maps $F\colon  B^{\ast}_{n} \rightarrow B^{\ast}_{n-1}$ assemble
to a map of graded algebras $F\colon  B^{\ast} \rightarrow B^{\ast}$. We will show that $(B^{\ast}, F)$ is
a strict Dieudonn\'{e} algebra.

Note that conditions $(ii)$ and $(iii)$ of Definition \ref{def20} are automatically satisfied by $B^{\ast}$ (since $B^{0} \simeq W(S)$ by construction).
It will therefore suffice to show that $F$ exhibits $B^{\ast}$ as a strict
Dieudonn\'{e} complex. By virtue of Corollary~\ref{cor102},
it will suffice to show that we can define Verschiebung maps $V\colon  B^{\ast}_{n} \rightarrow B^{\ast}_{n+1}$ which endow
$\{ B^{\ast}_{n} \}_{n \geq 0}$ with the structure of a strict Dieudonn\'{e} tower, in the sense of Definition \ref{stricTD}.
Note that the Verschiebung map $V\colon  \mathcal{W}_n(A)^{\ast} \rightarrow \mathcal{W}_{n+1}(A)^{\ast}$ is
an $(F^1, F^0)$-linear map of nilpotent $W(R)$-modules. Using Corollary \ref{extendmapsFaFb}, we see
that $V$ admits an essentially unique extension to a map $V\colon  B^{\ast}_{n} \rightarrow B^{\ast}_{n+1}$
which is $(F^1, F^0)$-linear as a map of $W(S)$-modules. To complete the proof,
it will suffice to show that the Frobenius, Verschiebung, and restriction maps on $\{ B^{\ast}_{n} \}_{n \geq 0}$ satisfy axioms
$(1)$ through $(8)$ appearing in Definition \ref{stricTD}. We consider axioms
$(6)$ through $(8)$ (the first five are formal consequences of
Corollary~\ref{extendmapsFaFb}, and left to the reader):
\begin{itemize}
\item[$(6)$] For each $n \geq 0$, the sequence $B^{\ast}_{n+1} \xrightarrow{F} B^{\ast}_{n} \xrightarrow{d} B^{\ast+1}_{n} / p B^{\ast+1}_{n}$
is exact. This follows from applying Remark \ref{sespreserved} to the sequence
$$ \mathcal{W}_{n+1}(A)^{\ast} \xrightarrow{F} \mathcal{W}_n(A)^{\ast} \xrightarrow{d} \mathcal{W}_{n}(A)^{\ast+1} / p \mathcal{W}_{n}(A)^{\ast+1}$$
(which is exact by virtue of Proposition \ref{buildtowerfromsat}); note that the maps in this sequence are $(F^0, F^1)$-linear and $(F^n, F^n)$-linear, respectively
(Example~\ref{example:diffetale}).

\item[$(7)$] For each $n \geq 0$, we have an exact sequence
$$ B^{\ast}_{n+1}[p] \xrightarrow{\id} B^{\ast}_{n+1} \xrightarrow{R} B^{\ast}_{n}.$$
This follows from from the exactness of the sequence
$$ \mathcal{W}_{n+1}(A)^{\ast}[p] \xrightarrow{\id} \mathcal{W}_{n+1}(A)^{\ast} \xrightarrow{R} \mathcal{W}_{n}(A)^{\ast}$$
by extending scalars along the {\etale} map $W_{n+1}(R) \rightarrow W_{n+1}(S)$.

\item[$(8)$] For each $n \geq 0$, we haven an exact sequence
$$ B^{\ast}_{1} \oplus B^{\ast-1}_{1} \xrightarrow{(V^n, dV^{n})} B^{\ast}_{n+1} \xrightarrow{R} B^{\ast}_{n}.$$
This follows from applying Remark \ref{sespreserved} to the sequence
$$ \mathcal{W}_1(A)^{\ast} \oplus \mathcal{W}_{1}(A)^{\ast-1} \xrightarrow{(V^{n}, dV^{n})} \mathcal{W}_{n+1}(A)^{\ast} \xrightarrow{R} \mathcal{W}_{n}(A)^{\ast}$$
(which is exact by virtue of Proposition \ref{buildtowerfromsat}); note that the maps in this sequence are
$(F^{2n}, F^{n})$-linear and $(F^0, F^0)$-linear, respectively (again by Example~\ref{example:diffetale}).
\end{itemize}
\end{proof} 
\newpage

\section{The Case of a Cusp}
\label{sec:Seminormal}

Let $R$ be a commutative $\F_p$-algebra. If $R$ is smooth over a perfect field $k$ of characteristic $p$, then Theorem \ref{maintheoC} supplies a canonical isomorphism
from the classical de~Rham--Witt complex $W \Omega_{R}^{\ast}$ of \cite{illusie} to the saturated de~Rham--Witt complex $\WOmega_{R}^{\ast}$ of Definition~\ref{def80}.
It follows that we can regard $\WOmega_{R}^{\ast}$ as a representative, in the
derived category of abelian groups, for the crystalline cochain complex
$\RGamma_{\crys}(\Spec(R))$ of the affine scheme $\Spec(R)$ (we will give another proof of this result in \S~\ref{dRtocryscomp:sec}).

In this section, we study the effect of introducing a mild singularity. Let
$R = \F_p[ x, y] / (x^2 - y^3)$ be the ring of functions of an affine curve with a single cusp.
Our main result, which we prove in \S \ref{sec7sub2}, asserts that the canonical map from $R$ to its normalization $\widetilde{R} \simeq \F_p[t]$ induces
an isomorphism of saturated de~Rham--Witt complexes $\WOmega_{R}^{\ast} \simeq \WOmega_{\widetilde{R}}^{\ast}$ (Proposition \ref{CuspdRW}). 
From this, we will deduce several consequences:
\begin{itemize}
\item The comparison map $\gamma\colon  W \Omega_{R}^{\ast} \rightarrow \WOmega_{R}^{\ast}$ is not an isomorphism when $R = \F_p[x,y] / (x^2-y^3)$
(Proposition \ref{loost}).
\item The cochain complex $\WOmega_{R}^{\ast}$ is not isomorphic to $\RGamma_{\crys}( \Spec(R) )$ as an object of the derived category of abelian
groups (Proposition~\ref{ols}). 
\end{itemize}

Recall that the ring $R = \F_p[ x, y] / (x^2 - y^3)$ studied above is, in some sense, the universal example of a commutative ring that fails to be seminormal. Using our calculation of $\W\Omega_R^*$, we will prove the following:

\begin{itemize}
\item Let $S$ be any commutative $\F_p$-algebra, and let $S \to S^{\mathrm{sn}}$ be the seminormalization of $S$. Then the induced map $\W\Omega^*_S \to \W\Omega^*_{S^{\mathrm{sn}}}$ is an isomorphism (Corollary~\ref{dRWsn1}). Moreover, the unit map 
$$S \rightarrow \WOmega^{0}_{S} / V \WOmega^{0}_{S}$$ exhibits
$\WOmega^{0}_{S} / V \WOmega^{0}_{S}$ as the seminormalization $S^{\mathrm{sn}}$ of $S$ (Theorem \ref{dRWsn2}). 
\end{itemize}

This result is inspired by a result in complex algebraic geometry: if
$\underline{\Omega}^*_X$ is the Deligne--du Bois complex of a complex algebraic
variety $X$, then $\mathcal{H}^0(\underline{\Omega}^0_X)$ is identified with the structure sheaf of the seminormalization of $X$ \cite[Proposition 7.8]{KovacsSchwede2011}.

\subsection{Digression: The de~Rham Complex of a Graded Ring}\label{sec7sub1}

 Suppose that $R = \bigoplus_{d \in \Z} R_{d}$ is a graded ring.
Then the de~Rham complex of $R$ admits a bigrading $\Omega^{\ast}_{R} \simeq \bigoplus_{d,n} (\Omega^{n}_{R})_{d}$, characterized
by the requirement that for every sequence of homogeneous elements $x_0, x_1, \ldots, x_n \in R$ of degrees $d_0, d_1, \ldots, d_n \in \Z$,
the differential form
$$ x_0 (dx_1 \wedge dx_2 \wedge \cdots \wedge dx_n)$$
belongs to the summand $(\Omega^{n}_{R})_{d_0 + \cdots + d_n} \subseteq \Omega^{n}_R$.

\begin{remark}
Let $R = \bigoplus_{d \in \Z} R_d$ be a graded ring which is $p$-torsion free
and let $\varphi\colon  R \rightarrow R$ be a ring homomorphism satisfying $\phi(x) \equiv x^{p} \pmod{p}$ for $x \in R$, so that
the de~Rham complex $\Omega^{\ast}_{R}$ inherits the structure of a Dieudonn\'{e} algebra (Proposition \ref{prop42}).
Suppose that, for every homogeneous element $x \in R_{d}$, we have $\varphi(x) \in R_{pd}$. Then,
for every homogeneous differential form $\omega \in (\Omega^{\ast}_{R})_{d}$, we have $F(\omega) \in (\Omega^{\ast}_{R})_{pd}$.
\end{remark}

\begin{proposition}\label{gradedsatiso}
Let $f\colon  R \to R'$ be a homomorphism of non-negatively graded,
$p$-torsion-free rings equipped with
ring homomorphisms $\varphi\colon  R \to R$ and $\varphi'\colon  R' \to R'$
which lift the Frobenius endomorphisms on $R/pR$ and $R'/pR'$, respectively. Assume that:
\begin{itemize}
\item The diagram of ring homomorphisms
$$ \xymatrix{ R \ar[r]^{f} \ar[d]^{\varphi} & R' \ar[d]^{\varphi'} \\
R \ar[r]^{f} & R' }$$
is commutative.
\item We have $\varphi( R_d ) \subseteq R_{pd}$ and $\varphi'( R'_{d} ) \subseteq R'_{pd}$.
\item The ring homomorphism $f_0\colon  R_0 \rightarrow R'_0$ is an isomorphism.
\item There exists an integer $N \gg 0$ such that $f$ induces an isomorphism $R_{d} \rightarrow R'_{d}$ for $d \geq N$.
\end{itemize}
Then the map of de~Rham complexes $\Omega^{\ast}_{R} \rightarrow \Omega^{\ast}_{R'}$ induces
an isomorphism 
\[ \Saturate( \Omega_R^{\ast}) \to \Saturate( \Omega_{R'}^{\ast})\]
of saturated Dieudonn\'e algebras. 
\end{proposition} 

\begin{proof} 
Let $(\Omega_R^{\ast})_{\mathrm{tors}} \subseteq 
\Omega_R^{\ast}$ and $(\Omega_{R'}^{\ast})_{\mathrm{tors}} \subseteq
\Omega_{R'}^{\ast}$ denote the submodules of $p$-power torsion elements,
and let us write $F$ for the Frobenius map on the Dieudonn\'{e} algebras $\Omega^{\ast}_{R}$ and
$\Omega^{\ast}_{R'}$. By virtue of the description of the saturation given in Remark~\ref{integralformsabstract},
it will suffice to show that the natural map
$$ \theta\colon   \left(  \Omega_R^{\ast}/(\Omega_R^{\ast})_{\mathrm{tors}}\right)[1/F]  
\to \left(  \Omega_{R'}^{\ast}/(\Omega_{R'}^{\ast})_{\mathrm{tors}}\right)[1/F]$$
is an isomorphism of graded rings.

Choose an integer $r \geq 0$ such that $p^{r} \geq N$. Then the map $f$ induces an isomorphism
$R_{d} \simeq R'_{d}$ whenever $d$ is divisible by $p^r$. It follows that there is a unique map
$\psi\colon  R' \rightarrow R$ for which the diagram
\[ \xymatrix@R=50pt@C=50pt{
R \ar[d]^{\varphi^r} \ar[r]^{f} &  R' \ar[d]^{\varphi'^r}  \ar[ld]^\psi \\
R \ar[r]_{f} &  R'
}\]
is commutative. Consequently, the kernel of the map
$$\Omega_R^{\ast}/(\Omega_R^{\ast})_{\mathrm{tors}} \to 
\Omega_{R'}^{\ast}/(\Omega_{R'}^{\ast})_{\mathrm{tors}}$$
is annihilated by $F^r$, since $p^{rd} F^r  = (\varphi^r)^{\ast}$ on
$\Omega_R^{d}$. It follows that $\theta$ is injective.

Note that, since $f$ induces an isomorphism $R[1/F] \xrightarrow{\sim} R'[1/F]$, we
can identify $\theta$ with
the natural map
$$((R' \otimes_{R} \Omega_{R}^{\ast}) / (R' \otimes_{R} \Omega_{R}^{\ast})_{\mathrm{tors} })[ 1/F] \to
(\Omega_{R'}^{\ast}/(\Omega_{R'}^{\ast})_{\mathrm{tors}})[1/F].$$
We will complete the proof by showing that this map is surjective. Note that the codomain
of $\theta$ is generated, as an algebra over $R'[1/F]$, by elements of degree $1$.
We are therefore reduced to showing that the natural map
$$ (R' \otimes_R \Omega^1_R)[1/F] \to \Omega^1_{R'}[1/F]$$
is surjective. For this, it will suffice to show that $f$ induces a surjection
$\rho\colon  (R' \otimes_{R} \Omega^{1}_{R})_{d} \rightarrow (\Omega^{1}_{R'})_{d}$ for
$d = 0$ (which is clear) and for $d \geq 2N$. In the latter case, we note that
$(\Omega^{1}_{R'})_{d}$ is generated by elements of the form $x(dy)$
for homogeneous elements $x, y \in R'$, where either $x$ or $y$ has degree $\geq N$.
Writing $x dy = d(xy) - y(dx)$, we can reduce to the case where $y$ has degree $\geq N$
and therefore belongs to the image of $f$, so that $x(dy)$ belongs to the image of $\rho$ as desired.
\end{proof} 

\subsection{The Saturated de~Rham--Witt Complex of a Cusp}\label{sec7sub2}

We now apply Proposition \ref{gradedsatiso} to study the saturated de~Rham--Witt complex of a cusp.

\begin{proposition}\label{CuspdRW}
Let $R \subseteq \F_p[ t ]$ be the subring generated by $t^2$ and $t^3$. Then the inclusion
$R \hookrightarrow \F_p[t]$ induces an isomorphism of saturated de~Rham--Witt complexes
$$ \WOmega^{\ast}_{R} \rightarrow \WOmega^{\ast}_{ \F_p[t] }.$$
\end{proposition}

\begin{proof}
Let $\widetilde{R}$ denote the $\Z_p$-subalgebra of the polynomial ring $\Z_p[t]$ generated by $t^2$ and $t^3$.
Then $\widetilde{R}$ is a $p$-torsion free ring satisfying $R \simeq \widetilde{R} / p \widetilde{R}$.
Let us regard $\Z_p[t]$ as a graded ring in which the element $t$ is homogeneous of degree $1$, and let
$\varphi\colon  \Z_p[t] \rightarrow \Z_p[t]$ be the $\Z_p$-algebra homomorphism given by $\varphi(t) = t^{p}$.
Then $\widetilde{R}$ is a graded subring of $\Z_p[t]$ satisfying $\varphi( \widetilde{R} ) \subseteq \widetilde{R}$, and
the inclusion map $\widetilde{R} \hookrightarrow \Z_p[t]$ is an isomorphism in degrees $\neq 1$. Applying Proposition \ref{gradedsatiso},
we deduce that the associated map of de~Rham complexes $\Omega^{\ast}_{ \widetilde{R} } \rightarrow \Omega^{\ast}_{ \Z_p[t] }$
induces an isomorphism of saturated Dieudonn\'{e} algebras $\Saturate(
\Omega^{\ast}_{ \widetilde{R} } ) \xrightarrow{\sim} \Saturate( \Omega^{\ast}_{ \Z_p[t] } )$. 
By virtue of Corollary \ref{cor70}, we have a commutative diagram of Dieudonn\'{e} algebras
$$ \xymatrix{ \Saturate( \Omega^{\ast}_{ \widetilde{R} } ) \ar[r] \ar[d] & \Saturate( \Omega^{\ast}_{ \Z_p[t] } ) \ar[d] \\
\WOmega^{\ast}_{R} \ar[r] & \WOmega^{\ast}_{ \F_p[t] }, }$$
where the vertical maps exhibit $\WOmega^{\ast}_{R}$ and $\WOmega^{\ast}_{ \F_p[t] }$ as the completions of
$\Saturate( \Omega^{\ast}_{ \widetilde{R} } )$ and $\Saturate( \Omega^{\ast}_{ \Z_p[t] } )$, respectively.
It follows that the lower horizontal map is also an isomorphism.
\end{proof}

\begin{remark}
The inclusion map $\widetilde{R} \hookrightarrow \Z_p[t]$ appearing in the proof of Proposition \ref{CuspdRW} does not
induce an isomorphism of de~Rham complexes $\rho\colon  \Omega^{\ast}_{ \widetilde{R} } \rightarrow \Omega^{\ast}_{ \Z_p[t] }$: in fact, the differential
form $dt \in \Omega^{1}_{\Z_p[t]}$ does not belong to the image of $\rho$. Nevertheless, Proposition \ref{gradedsatiso} guarantees that 
$F^{n}( dt )$ belongs to the image of $\rho$ for some integer $n$. In fact, we can be more explicit. Set $x = t^3$ and $y = t^2$, so that
$x$ and $y$ belong to $\widetilde{R}$. For $p \geq 5$, we have
$F(dt ) = t^{p-1} dt = \frac{1}{2} x y^{(p-5)/2} dy$. For $p = 3$, we have $F^2 (dt ) = t^{8} dt= \frac{1}{2} x y^2 dy$. For $p = 2$, we have
$F^3(dt) = t^7 dt = \frac{1}{3} x y dx$.
\end{remark}

\subsection{The Classical de~Rham--Witt Complex of a Cusp}\label{sec7sub3}

Let $R$ be a commutative $\F_p$-algebra. In \S \ref{compareclass}, we constructed a comparison map
$W \Omega_{R}^{\ast} \rightarrow \WOmega_{R}^{\ast}$ (Corollary \ref{makemap}) and showed that it is an isomorphism in the case where
$R$ is smooth over a perfect field (Theorem \ref{maintheoC}). Proposition \ref{CuspComparedRW} shows that it is not an isomorphism in general, even for reduced rings:

\begin{proposition}\label{loost}
\label{CuspComparedRW}
Let $R \subseteq \F_p[t]$ be the subalgebra generated by $t^2$ and $t^3$. Then the comparison map 
$$ c_{R}\colon  W \Omega_{R}^{\ast} \rightarrow \WOmega_{R}^{\ast}$$
of Corollary \ref{makemap} is not an isomorphism.
\end{proposition}

\begin{proof}
Set $x = t^3$ and $y = t^2$, so that we can identify $R$ with the quotient ring $\F_p[x,y] / ( x^2 - y^3 )$. The module of K\"{a}hler differentials
$\Omega^{1}_{R}$ is generated by $dx$ and $dy$, subject only to the relation $2x(dx) = 3 y^2(dy)$. It follows that $dx$ and $dy$ have linearly
independent images in the $\F_p$-vector space $R/(x,y) \otimes_{R} \Omega^{1}_{R}$, so that the differential form $dx \wedge dy \in \Omega^2_{R}$
is nonzero. Since the natural map $W \Omega_{R}^{\ast} \rightarrow \Omega_{R}^{\ast}$ is surjective, we must have $W \Omega_{R}^{2} \neq 0$.
On the other hand, Proposition \ref{CuspdRW} supplies an isomorphism $\WOmega_{R}^{\ast} \simeq \WOmega_{ \F_p[t] }^{\ast}$, so that
we can use Corollary \ref{cor70} to identify $\WOmega_{R}^{\ast}$ with the completed saturation of the completed de~Rham complex $\widehat{\Omega}^{\ast}_{ \Z_p[t] }$.
It follows that $\WOmega_{R}^{2}$ vanishes, so that the map $c_{R}\colon  W \Omega_{R}^{\ast} \rightarrow \WOmega_{R}^{\ast}$ cannot be an isomorphism in degree $2$.
\end{proof}

\begin{remark}
The comparison map $c_{R}\colon  W \Omega_{R}^{\ast} \rightarrow \WOmega_{R}^{\ast}$ need not even be a quasi-isomorphism.
Note that we have a commutative diagram
$$ \xymatrix{ W \Omega_{R}^{\ast} \ar[r]^{c_R} \ar[d]^{c'} & \WOmega_{R}^{\ast} \ar[d] \\
W \Omega_{ \F_p[t] }^{\ast} \ar[r] & \WOmega_{ \F_p[t] }^{\ast}, }$$
where the right vertical map is a (quasi-)isomorphism by virtue of Proposition \ref{CuspdRW}. Consequently, if $c_{R}$ is a quasi-isomorphism,
then the map $c'$ is also a quasi-isomorphism, and therefore induces an isomorphism
$W\Omega_{R}^* \otimes_\Z^L \F_p \rightarrow W \Omega_{\F_p[t]}^* \otimes_{\Z}^L \F_p$ in the derived category of abelian groups.
Using the commutativity of the diagram
\[ \xymatrix{
		   W\Omega_{R}^* \otimes_\Z^L \F_p \ar[r] \ar[d] & W \Omega_{\F_p[t]}^* \otimes_{\Z}^L \F_p \ar[d] \\ 
		   \Omega_{R}^* \ar[r] & \Omega_{\F_p[t]}^*, }\]  
and the fact that the right vertical map is an isomorphism (since $\F_p[t]$ is a smooth $\F_p$-algebra), it would follow that
that the map of de~Rham complexes $\Omega_{R}^{\ast} \rightarrow \Omega_{ \F_p[t]}^{\ast}$ induces a surjection on cohomology.
In particular, if $c_{R}$ were a quasi-isomorphism, then the $0$-cocycle $t^{p}
\in \Omega^{0}_{ \F_p[t] }$ could be lifted to a $0$-cocycle in $\Omega_{R}^{\ast}$: that is, the
differential $d(t^p)$ would vanish as an element of $\Omega^{1}_{R}$. This is not true for $p \leq 7$. To see this, write $x = t^3$ and $y = t^2$ as in Proposition \ref{loost},
so that $t^{p} = x^{m} y^{n}$ for some $0 \leq m, n < p$. For $p \leq 7$, the canonical map
$$\Big(R \cdot dx \oplus R \cdot dy\Big) \rightarrow  \Big(R \cdot dx \oplus R \cdot dy\Big)/ R \cdot (2x \cdot dx - 3y^2 \cdot dy) \simeq \Omega^{1}_{R}$$
is an isomorphism in degree $p$ (with respect to the grading of \S \ref{sec7sub1}), so that $d(t^{p}) = m t^{p-3} \cdot dx + n t^{p-2} \cdot dy$ does not vanish.
\end{remark}

\subsection{The Crystalline Cohomology of a Cusp}\label{sec7sub4}

Let $R$ be an $\F_p$-algebra. If $R$ is smooth over a perfect field, then the saturated de~Rham--Witt complex $\WOmega_{R}^{\ast}$ agrees with the classical
de~Rham--Witt complex $W \Omega_{R}^{\ast}$, and can therefore be identified
with the crystalline cohomology $\RGamma_{\crys}(\Spec(R))$ (as an object of the derived category).
This observation does not extend to the non-smooth case:

\begin{proposition}\label{ols}
Let $R \subseteq \F_p[t]$ be the subalgebra generated by $t^2$ and $t^3$. Then the cochain complex
$\RGamma_{\crys}(\Spec(R)) \otimes^{L}_{ \Z_p } \F_p$ has nonvanishing cohomology in degree $2$.
In particular, there does not exist a quasi-isomorphism $\RGamma_{\crys}( \Spec(R) ) \simeq \WOmega_{R}^{\ast}$.
\end{proposition}
\begin{remark}
We will see in \S \ref{sec:sddRW} that, for any $\F_p$-algebra $R$ which can be realized as a local complete intersection over a perfect field $k$, there
is a natural comparison map $\RGamma_{\crys}(\Spec(R)) \to \W\Omega^*_R$ in the
derived category $D(\Z_p)$. By virtue of Proposition~\ref{ols}, this map
cannot be an isomorphism in general.
\end{remark}

\begin{proof}[Proof of Proposition \ref{ols}]
Since $R \simeq \F_p[x,y] / (x^2 - y^3)$ is a local complete intersection over $\F_p$, we can identify $\RGamma_{\crys}(\Spec(R)) \otimes_{\Z_p}^L \F_p$ with the derived de~Rham complex 
$L\Omega_{R}$ (see \cite[Theorem 1.5]{BhattTorsion}). Because $R$ admits a
flat lift to $\Z/(p^2)$ together with a lift of Frobenius, there is a natural
identification (as in loc.~cit.)
\[ \bigoplus_i ({\bigwedge}^i L_{R/\F_p})[-i] \simeq L\Omega_R \]
where ${\bigwedge}^{i} L_{ R/ \F_p }$ denotes the $i$th derived exterior power of the cotangent complex $L_{R/\F_p}$. In particular,
the cohomology group $\mathrm{H}^{2}( L \Omega_{R} )$ contains the classical exterior power
$\mathrm{H}^{0}( {\bigwedge}^{2} L_{R/\F_p}) \simeq \Omega^{2}_{R}$ as a direct summand. This direct summand is nonzero (since $dx \wedge dy$ is a nonvanishing element of
$\Omega^{2}_{R}$, as noted in the proof of Proposition \ref{loost}).

For the last claim, 
suppose that there exists a quasi-isomorphism $\RGamma_{\crys}(\Spec(R) ) \simeq \WOmega_{R}^{\ast}$. Reducing modulo $p$ and
using Proposition \ref{CuspdRW} and Remark \ref{plux}, we obtain quasi-isomorphisms
\begin{eqnarray*}
\RGamma_{\crys}(\Spec(R)) \otimes^{L}_{ \Z_p } \F_p & \simeq & \WOmega_{R}^{\ast} \otimes^{L}_{\Z_p} \F_p \\
& \simeq & \WOmega_{R}^{\ast} / p \WOmega_{R}^{\ast} \\
& \simeq & \WOmega_{ \F_p[t] }^{\ast} / p \WOmega_{\F_p[t] }^{\ast} \\
& \simeq & \Omega^{\ast}_{ \F_p[t] }. \end{eqnarray*}
This contradicts Proposition \ref{ols}, since $\Omega^{\ast}_{ \F_p[t] }$ is concentrated in cohomological degrees $\leq 1$. 
\end{proof}

\subsection{Seminormality}\label{sec7sub5}

Recall that a ring $R$ is {\em seminormal} if it is reduced and satisfies the following condition: given $x,y \in R$ with $x^2 = y^3$, we can find $t \in R$ with $x = t^3$ and $y= t^2$ (see \cite{SwanSeminormal}). 

\begin{proposition}\label{older}
Let $A^{\ast}$ be a strict Dieudonn\'{e} algebra. Then the commutative ring $A^{0} / V A^{0}$ is seminormal.
\end{proposition}

\begin{proof}
Lemma \ref{lem27} asserts that $A^{0} / V A^{0}$ is reduced. Suppose we are given elements $x,y \in A^{0} / V A^{0}$ satisfying $x^2 = y^3$, classified
by an $\F_p$-algebra homomorphism $u\colon  R = \F_p[ x,y] / (x^2-y^3) \rightarrow A^{0} / V A^{0}$. Then we can identify $f$ with a map of strict Dieudonn\'{e} algebras
$v\colon  \WOmega^{\ast}_{R} \rightarrow A^{\ast}$. By virtue of Proposition \ref{CuspdRW}, the map $v$ factors uniquely as a composition
$$ \WOmega^{\ast}_{R} \simeq \WOmega^{\ast}_{\F_p[t]} \xrightarrow{v'} A^{\ast}.$$
It follows that $u$ factors as a composition $R \hookrightarrow \F_p[t] \xrightarrow{u'} A^{0} / V A^{0}$: that is, there exists an element
$t \in A^{0} / VA^{0}$ satisfying $x = t^3$ and $y = t^2$.
\end{proof}

Any commutative ring $R$ admits a universal map $R \to R^{\mathrm{sn}}$ to a seminormal ring by \cite[Theorem 4.1]{SwanSeminormal}; we refer to $R^{\mathrm{sn}}$ as the {\em seminormalization} of $R$.

\begin{corollary}
\label{dRWsn1}
Let $f\colon  R \to S$ be a map of $\F_p$-algebras which induces an isomorphism of seminormalizations $R^{\mathrm{sn}} \rightarrow S^{\mathrm{sn}}$. Then
the induced map $\rho\colon  \WOmega_{R}^{\ast} \rightarrow \WOmega_{S}^{\ast}$ is an isomorphism of strict Dieudonn\'{e} algebras.
\end{corollary}

\begin{proof}
For any strict Dieudonn\'{e} algebra $A^{\ast}$, we have a commutative diagram
$$ \xymatrix{ \Hom( S^{\mathrm{sn}}, A^{0} / VA^{0}) \ar[r] \ar[d] & \Hom(R^{\mathrm{sn}}, A^{0} / VA^{0} ) \ar[d] \\
\Hom(S, A^{0} / V A^{0} ) \ar[r] \ar[d]^{\sim} & \Hom( R, A^{0} / V A^{0} ) \ar[d]^{\sim} \\
\Hom_{ \FrobAlg}( \WOmega_{S}^{\ast}, A^{\ast} ) \ar[r]  & \Hom_{ \FrobAlg}( \WOmega_{R}^{\ast}, A^{\ast} ) }$$
where the top vertical maps are bijective by Proposition \ref{older} and the top horizontal map is bijective by virtue of our
assumption that $f$ induces an isomorphism of seminormalizations. It follows that the bottom horizontal map is also bijective.
\end{proof}

We will show that Corollary \ref{dRWsn1} is, in some sense, optimal:

\begin{theorem}\label{dRWsn2}
Let $R$ be a commutative $\F_p$-algebra. Then the unit map $R \rightarrow \WOmega^{0}_{R} / V \WOmega^{0}_{R}$ exhibits
$\WOmega^{0}_{R} / V \WOmega^{0}_{R}$ as a seminormalization of $R$.
\end{theorem}

\begin{remark}
Recall that we have a pair of adjoint functors
\[ \xymatrix{ \CAlg_{ \F_p} \ar@<0.6ex>[rrr]^-{R \mapsto \WOmega^{\ast}_{R} } & & & \FrobAlgComplete \ar@<0.6ex>[lll]^-{ A^{0} / V A^{0}  \mapsfrom A^{\ast} } }\]
(see Corollary \ref{dRWLeftAdj}). It follows from formal categorical considerations that this adjunction restricts to an equivalence
of categories $\calC \simeq \calD$, where $\calC \subseteq \CAlg_{\F_p}$ is the full subcategory spanned by those commutative
$\F_p$-algebras $R$ for which the unit map $R \rightarrow \WOmega^{0}_{R} / V \WOmega^{0}_{R}$ is an isomorphism,
and $\calD \subseteq \FrobAlgComplete$ is the full subcategory spanned by those strict Dieudonn\'{e} algebras $A^{\ast}$ for which the counit map
$\WOmega^{\ast}_{ A^{0} / V A^{0} } \rightarrow A^{\ast}$ is an isomorphism. Using Theorem \ref{dRWsn2}, we see that a commutative
$\F_p$-algebra belongs to $\calC$ if and only if it is seminormal. It then follows from Proposition \ref{older} that
the functor $A^{\ast} \mapsto A^{0} / V A^{0}$ takes values in $\calC$, and from Corollary \ref{dRWsn1} that
the functor $\WOmega^{\ast}$ takes values in $\calD$. In particular, the inclusion functor $\calC \hookrightarrow \CAlg_{\F_p}$
admits a left adjoint (given by $R \mapsto \WOmega^{0}_{R} / V \WOmega^{0}_{R} \simeq R^{\mathrm{sn}}$), and
the inclusion functor $\calD \hookrightarrow \FrobAlgComplete$ admits a right adjoint (given by
$A^{\ast} \mapsto \WOmega^{\ast}_{A^{0} / V A^{0} }$). In other words, the adjunction between $\CAlg_{\F_p}$
and $\FrobAlgComplete$ is {\it idempotent}.
\end{remark}

\begin{remark}
Theorem~\ref{dRWsn2} gives a clean and intrinsic description of the seminormalization $R^{\mathrm{sn}}$ of an $\F_p$-algebra $R$ in terms of the
ring of Witt vectors $W(R)$. Assume that $R$ is reduced, so that $W(R)$ is $p$-torsion-free.
For each $n \geq 0$, let $W(R)^{(n)}$ denote the subring of $W(R)$ consisting of those elements $f$ which satisfy
$$ df \in p^{n}  \big(\Omega^1_{W(R)}/\mathrm{torsion}\big).$$
Then the Witt vector Frobenius $F$ carries $W(R)^{(n)}$ into $W(R)^{(n+1)}$, and the Verschiebung $V$ carries $W(R)^{(n+1)}$ into $W(R)^{(n)}$. 
In particular, the colimit 
\[ W(R)^{(\infty)} := \varinjlim_n \Big(W(R)^{(0)} \xrightarrow{F} W(R)^{(1)} \xrightarrow{F} W(R)^{(2)} \xrightarrow{F} ...\Big)\]
comes equipped with an endomorphism $V\colon W(R)^{(\infty)} \to W(R)^{(\infty)}$. Unwinding definitions, we have $W(R)^{(\infty)} \simeq \Saturate(\Omega^*_{W(R)})^0$, so that we have an isomorphism $R^{\mathrm{sn}} \simeq W(R)^{(\infty)}/VW(R)^{(\infty)}$ by Theorem~\ref{dRWsn2} and 
Remark~\ref{rem:satissubgroup}. 
\end{remark}

\subsection{The Proof of Theorem \ref{dRWsn2}}\label{sec7sub6}

We begin by introducing a (temporary) bit of notation.

\begin{notation}
For every commutative $\F_p$-algebra $R$, we let $\Psi(R)$ denote the commutative $\F_p$-algebra $\WOmega^{0}_{R} / V \WOmega^{0}_{R}$,
so that the definition of the saturated de~Rham--Witt complex supplies a comparison map $u_{R}\colon  R \rightarrow \Psi(R)$.
\end{notation}

It follows from Proposition \ref{older} that for any commutative $\F_p$-algebra $R$, the $R$-algebra $\Psi(R)$ is seminormal. We wish to show that
$u_{R}$ exhibits $\Psi(R)$ as a seminormalization of $R$. The proof proceeds in several steps.

\begin{remark}\label{elwose1}
By virtue of Corollary~\ref{cor75}, the functor $R \mapsto \Psi(R)$ commutes with filtered colimits.
\end{remark}

\begin{remark}\label{elwose2}
Let $k$ be a perfect $\F_p$-algebra. Then, for any smooth $k$-algebra $R$, the unit map $u_{R}\colon  R \rightarrow \Psi(R)$ is an isomorphism
(Proposition \ref{prop76}). In particular, if $R$ is perfect, then the map $u_R\colon  R \rightarrow \Psi(R)$ is an isomorphism.
\end{remark}

\begin{lemma}\label{elwose}
Let $R$ be a commutative $\F_p$-algebra and let $\varphi\colon  R \rightarrow R$ be the Frobenius endomorphism.
Then $\Psi( \varphi)\colon  \Psi(R) \rightarrow \Psi(R)$ is the Frobenius endomorphism of $\Psi(R)$.
\end{lemma}

\begin{proof}
For every map of Dieudonn\'{e} algebras $f\colon  \WOmega_{R}^{\ast} \rightarrow \WOmega_{R}^{\ast}$,
let $f_0\colon  \Psi(R) \rightarrow \Psi(R)$ be the induced map. From the universal property of the saturated de~Rham--Witt complex $\WOmega_{R}^{\ast}$, we see that
there is at most one map of Dieudonn\'{e} algebras $f\colon  \WOmega_{R}^{\ast} \rightarrow \WOmega_{R}^{\ast}$
for which the diagram
$$ \xymatrix{ R \ar[r]^-{\varphi} \ar[d]^{u_R} & R \ar[d]^{ u_R } \\
\Psi(R) \ar[r]^-{f_0} & \Psi(R)}$$
commutes, and $\Psi( \varphi )$ is then given by $f_0$. Consequently, to show that $\Psi(\varphi)$ coincides with the Frobenius map on $\Psi(R)$, it will
suffice to construct a map of Dieudonn\'{e} algebras $f\colon  \WOmega_{R}^{\ast} \rightarrow \WOmega_{R}^{\ast}$ which induces
the Frobenius on $\Psi(R)$. We conclude by observing that such a map exists, given by $f(x) = p^{n} F(x)$ for $x \in \WOmega^{n}_{R}$. 
\end{proof}

\begin{corollary}\label{elwosecor}
Let $R$ be any $\F_p$-algebra. Then the canonical map $u_R\colon  R \rightarrow \Psi(R)$ induces an isomorphism of perfections
$R^{1 / p^{\infty} } \rightarrow \Psi(R)^{1/ p ^{\infty} }$.
\end{corollary}

\begin{proof}
Combine Lemma \ref{elwose} with Remarks \ref{elwose1} and \ref{elwose2}.
\end{proof}

\begin{corollary}\label{elwosecor2}
Let $R$ be a commutative $\F_p$-algebra and let $S$ be a reduced $R$-algebra. If there
exists an $R$-algebra homomorphism $f\colon  \Psi(R) \rightarrow S$, then $f$ is unique.
\end{corollary}

\begin{proof}
Suppose that we have a pair of $R$-algebra homomorphisms $f,g\colon  \Psi(R) \rightarrow S$. Passing
to perfections, we obtain a pair of $R^{1/p^{\infty}}$-algebra homomorphisms $f^{1/p^{\infty}}, g^{1/p^{\infty}}\colon  \Psi(R)^{1/p^{\infty}} \rightarrow S^{1/p^{\infty}}$.
It follows from Corollary \ref{elwosecor} that $f^{1/p^{\infty}} = g^{1/p^{\infty}}$, so that $f$ and $g$ agree after composition with the map $S \rightarrow S^{1/p^{\infty}}$.
Since $S$ is reduced, this map is injective; it follows that $f = g$.
\end{proof}

\begin{proposition}\label{techprop1}
Let $R$ be a commutative $\F_p$-algebra and let $k$ be a field of characteristic $p$. Then every ring homomorphism $f\colon  R \to k$ factors uniquely
through $u_{R}\colon  R \rightarrow \Psi(R)$.
\end{proposition}

\begin{proof}
We have a commutative diagram
$$ \xymatrix{ R \ar[r]^-{u_{R} } \ar[d]^-{f} & \Psi(R) \ar[d] \\
k \ar[r]^-{ u_k} & \Psi(k). }$$
Since $k$ can be written as a filtered colimit of smooth $\F_p$-algebras, the map $u_k$ is an isomorphism (Remarks \ref{elwose1} and \ref{elwose2}).
It follows that $f$ factors through $u_R$; the uniqueness of the factorization follows from Corollary \ref{elwosecor2}.
\end{proof}

\begin{proposition}\label{techprop2}
Let $R$ be a finitely generated reduced $\F_p$-algebra, let $K(R)$ denote the total ring of fractions of $R$ (that is, the product of the residue fields
of the generic points of $\Spec(R)$), and let $\overline{R}$ be the integral closure of $R$ in $K(R)$. Then the canonical map
$R \rightarrow \overline{R}$ factors uniquely through $u_{R}$, and the
resulting map $\Psi(R) \to
\overline{R}$ is injective.
\end{proposition}

\begin{proof}
Let $x$ be a generic point of $\Spec(R)$. Using Proposition \ref{techprop1}, we see that there exists
an $R$-algebra homomorphism $f_{x}\colon  \Psi(R) \rightarrow \kappa(x)$. Let $A(x)$ denote the integral closure of $R$ in
$\kappa(x)$. Then $A(x)$ is an integrally closed Noetherian domain, and can therefore be written as an intersection
$\bigcap_{\mathfrak{p}} A(x)_{\mathfrak{p}}$ where $\mathfrak{p}$ ranges over the height $1$ primes of $A(x)$.
Note that each $A(x)_{\mathfrak{p}}$ is regular and can therefore be realized as a filtered colimit of smooth $\F_p$-algebras.
We have a commutative diagram
$$ \xymatrix{ R \ar[r]^-{ u_R} \ar[d] & \Psi(R) \ar[d] \\
A(x)_{\mathfrak{p} } \ar[r]^-{ u_{ A(x)_{\mathfrak{p}}} } & \Psi( A(x)_{\mathfrak{p} }) }$$
where the bottom horizontal map is an isomorphism (Remarks \ref{elwose1} and \ref{elwose2}).
It follows that there exists an $R$-algebra homomorphism $g\colon  \Psi(R) \rightarrow A(x)_{\mathfrak{p}}$.
Using Corollary \ref{elwosecor2}, we see that the composition
$\Psi(R) \xrightarrow{g} A(x)_{\mathfrak{p}} \hookrightarrow \kappa(x)$ coincides with $f_x$.
It follows that $f_x$ factors through the valuation ring $A(x)_{\mathfrak{p}}$. Allowing $\mathfrak{p}$
to vary, we conclude that $f_x$ factors through the subalgebra $A(x) \subseteq \kappa(x)$.
Forming a product as $x$ varies, we obtain an $R$-algebra homomorphism $\Psi(R) \rightarrow \overline{R}$;
by virtue of Corollary \ref{elwosecor2}, this homomorphism is unique.
Injectivity follows by comparison with the perfect case and
Corollary~\ref{elwosecor}. 
\end{proof}

\begin{corollary}\label{techprop3}
Let $R$ be a commutative $\F_p$-algebra. Then the unit map $u_{R}\colon  R \rightarrow \Psi(R)$ is integral.
\end{corollary}

\begin{proof}
Combine Proposition \ref{techprop2} with Remark \ref{elwose1}.
\end{proof}

\begin{proof}[Proof of Theorem \ref{dRWsn2}]
Let $R$ be a commutative $\F_p$-algebra; we wish to show that the canonical map $R \rightarrow \Psi(R)$ exhibits
$\Psi(R)$ as a seminormalization of $R$. By virtue of Corollary \ref{dRWsn1}, we can replace $R$ by $R^{\mathrm{sn}}$ and thereby
reduce to the case where $R$ is seminormal. In particular, $R$ is reduced, so the map $u$ exhibits $R$ as a subring of $\Psi(R)$
(Corollary \ref{elwosecor}). It follows from Corollary \ref{techprop3} and Proposition \ref{techprop1} that
the map $u_{R}\colon  R \rightarrow \Psi(R)$ is {\it subnormal}, in the sense of \cite{SwanSeminormal}: that is it is an integral ring homomorphism which induces
a bijection $\Spec( \Psi(R) )(k) \rightarrow \Spec(R)(k)$ for every field $k$. Since $R$ is seminormal, it follows that the map $u_{R}$ is an isomorphism
(\cite[Lemma 2.6]{SwanSeminormal}).
\end{proof}

\newpage

\part{Complements and Applications}
\section{Homological Algebra}
\label{derivedcatsec}

\normalfont

For every cochain complex $M^{\ast}$ of $p$-torsion-free abelian groups, let $(\eta_p M)^{\ast}$ denote the subcomplex of
$M^{\ast}[ 1/ p ]$ introduced in Construction \ref{etaconstruction}. Recall that endowing $M^{\ast}$ with the structure of
a Dieudonn\'{e} complex is equivalent to choosing a chain map $\alpha\colon  M^{\ast} \rightarrow (\eta_p M)^{\ast}$ (Remark \ref{rem2}),
which is an isomorphism if and only if the Dieudonn\'{e} complex $M^{\ast}$ is saturated. We can reformulate this observation
as follows: the category $\FrobCompSat$ of saturated Dieudonn\'{e} complexes can be identified with the category of
{\em fixed points} for the functor $\eta_p$ on $p$-torsion-free cochain complexes (see Definition \ref{fixe} and Example \ref{fixedpointsat}).

In this section, we will study a variant of this construction, where we replace the category of cochain complexes $\coch$ by
the {\it derived category} $D(\Z)$ obtained by formally inverting quasi-isomorphisms. The construction $M^{\ast} \mapsto (\eta_p M)^{\ast}$
induces an endofunctor of the derived category $L \eta_p\colon  D(\Z)
\rightarrow D(\Z )$, which was originally introduced in the work of Berthelot--Ogus (\cite{BO78}).
Every cochain complex $M^{\ast}$ can be regarded as an object of the derived category $D(\Z)$, and every saturated Dieudonn\'{e} complex
can be regarded as an object of the category $D(\Z)^{L \eta_p}$ of fixed points for the action of $L \eta_p$ on $D(\Z)$. The main result of
this section is that this construction restricts to an equivalence of categories 
$$ \{ \text{Strict Dieudonn\'{e} Complexes} \} \simeq \{ \text{$p$-Complete Fixed Points of $L \eta_p$} \}$$
(see Theorem \ref{fixedpointDZ} for a precise statement; we also prove an $\infty$-categorical refinement as Theorem \ref{1category}). In other words, we show that a $p$-complete object $M$
of the derived category $D(\Z)$ which is equipped with an isomorphism $\alpha\colon  M \simeq L \eta_p M$ (in the derived category)
has a canonical representative $M^{\ast}$ at the level of cochain complexes (which is determined up to isomorphism by the requirement that
$\alpha$ should lift to an isomorphism of cochain complexes $M^{\ast} \simeq (\eta_p M)^{\ast}$). In particular,
the de~Rham--Witt complex $W\Omega^{\ast}_{R}$ of a smooth algebra $R$ over a perfect field $k$ can be functorially
reconstructed from the crystalline cochain complex $R \Gamma_{ \crys }( \Spec(R) )$, together with the isomorphism
$R \Gamma_{ \crys }( \Spec(R) ) \simeq L \eta_p R \Gamma_{ \crys }(\Spec(R) )$ induced by the Frobenius on $\Spec(R)$.
In \S \ref{cryscompsec}, we will apply this observation to give a simple construction of the crystalline comparison
isomorphism for the cohomology theory $A \Omega$ of Bhatt-Morrow-Scholze \cite{BMS}. 

\subsection{$p$-Complete Objects of the Derived Category}\label{derivedsub1}
\label{S:derivedreview}
We begin with some recollections on the derived category of abelian groups. For
a modern textbook reference, we refer the reader to \cite[Ch.~13--14]{KS}.

\begin{definition}\label{defderived}
The derived category $D( \Z)$ of the integers can be defined in two equivalent ways: 
\begin{itemize}
\item[$(1)$] Let $\coch$ denote the category of cochain complexes of abelian groups. Then
$D(\Z)$ is the localization $\coch[W^{-1} ]$, where $W$ is the collection of all quasi-isomorphisms in $\coch$.
In other words, $D(\Z)$ is obtained from $\coch$ by formally adjoining inverses of quasi-isomorphisms.

\item[$(2)$] Let $\cochfr \subseteq \coch$ denote the full subcategory spanned by the cochain complexes
of {\em free} abelian groups. Then $D( \Z )$ is equivalent to the homotopy category of $\cochfr$.
That is, given a pair of objects $X^{\ast}, Y^{\ast} \in \cochfr$, we can identify
$\Hom_{D(\Z)}(X, Y)$ with the abelian group of chain homotopy classes of maps of cochain complexes $X^{\ast}
\to Y^{\ast}$ (in fact, this holds more generally if $X^{\ast}$ is a cochain complex of free abelian groups and $Y^{\ast}$ is arbitrary).
\end{itemize}

Using description $(2)$, we see that $D( \Z)$ admits a symmetric monoidal structure, given by the usual
tensor product of cochain complexes (of free abelian groups). We will denote the underlying tensor product functor by
$$ \lotimes\colon  D(\Z) \times D(\Z) \rightarrow D(\Z).$$
\end{definition} 

\begin{variant}\label{mariant}
In version $(1)$ of Definition \ref{defderived}, one can replace the category $\coch$ of all cochain complexes with
various subcategories. For example, $D(\mathbb{Z})$ can be obtained from $\cochfr$ by formally inverting all quasi-isomorphisms. 
It will be helpful for us to know that $D(\mathbb{Z})$ can also be obtained from
the intermediate subcategory $\cochtfr \subseteq \coch$ consisting of cochain
complexes of \emph{$p$-torsion-free} abelian
groups (again by formally inverting quasi-isomorphisms).
\end{variant}

\begin{notation}
In what follows, it will be convenient to distinguish between a cochain complex of abelian groups $M^{\ast} \in \coch$ and its image in the derived category $D(\Z)$. We will
typically denote the latter simply by $M$ (that is, by omitting the superscript $\ast$).
\end{notation}

\begin{remark} 
In description $(2)$ of Definition \ref{defderived}, we have implicitly used the fact that the ring $\Z$ of integers has finite projective dimension.
If we were to replace $\Z$ by a more general ring $R$ (not assumed to be of finite projective dimension), then we should also replace
$\cochfr$ by the category of \emph{$K$-projective} complexes in the sense of \cite{Spalt}. See also \cite[Ch. 2]{Hovey} for a presentation as a model category. 
\end{remark}

Next, we recall the notion of a \emph{$p$-complete} object in
$D(\mathbb{Z})$, as in \cite{DG02} and going back to ideas of Bousfield \cite{Bousfield}. 

\begin{definition}\label{definition.dhp}
Let $X$ be an object of $D(\Z)$. We will say that $X$ is \emph{$p$-complete} if for every object $Y \in D(\mathbb{Z})$
such that $p\colon  Y \to Y$ is an isomorphism in $D(\mathbb{Z})$ (i.e., such that $Y$ arises via restriction of scalars from $D( \mathbb{Z}[1/p])$), we have
$\Hom_{D(\mathbb{Z})}(Y, X) = 0$. We let $\dhp \subseteq D( \mathbb{Z})$ be the full subcategory spanned by the
$p$-complete objects. 
\end{definition} 

\begin{remark}[Derived $p$-Completion]\label{dercomp}
The inclusion functor $\dhp \hookrightarrow D(\Z)$ admits a left adjoint, which we will denote by $X \mapsto \widehat{X}_{p}$
and refer to as the {\it derived $p$-completion functor}. Concretely, if $X^{\ast}$ is a $p$-torsion-free cochain complex
representing an object $X \in D(\Z)$, then the derived $p$-completion $\widehat{X}_{p}$ is represented by the cochain complex
$\widehat{X}^{\ast}_{p}$, obtained from $X^{\ast}$ by levelwise $p$-adic
completion; this follows because $\widehat{X}_p$ can be obtained as the
homotopy limit of the diagram of quotient complexes $\{ X^{\ast}/p^n X^{\ast} \}_{n \geq 0}$.
\end{remark}

\begin{definition}\label{scab}
Let $X$ be an abelian group. We say that $X$ is \emph{derived $p$-complete} if, when considered
as an object of $D(\mathbb{Z})$ concentrated in a single degree, it is $p$-complete. 
\end{definition} 

\begin{remark}
Let $X$ be an abelian group. Then $X$ is derived $p$-complete if and only if
$$ \Hom_{\Z}( \Z[1/p], X) \simeq 0 \simeq \Ext_{\Z}( \Z[1/p], X ).$$
In other words, $X$ is derived $p$-complete if and only if every short exact sequence
$0 \rightarrow X \rightarrow M \rightarrow \Z[1/p] \rightarrow 0$ admits a unique splitting.
\end{remark}

\begin{remark}\label{cose2}
Let $X$ be an object of $D(\mathbb{\Z})$. Since the category of abelian
groups has projective dimension $1$, $X$ can be written (noncanonically) as a product
$\prod_{n \in \Z} \mathrm{H}^{n}(X)[ -n ]$. It follows that $X$ is $p$-complete (in the sense of Definition \ref{definition.dhp}) if and only if
each cohomology group $\mathrm{H}^{n}(X)$ is derived $p$-complete (in the sense of Definition \ref{scab}).
See \cite[Prop. 5.2]{DG02}. 
\end{remark}

\begin{remark}
Let $M$ be an abelian group. We say that $M$ is {\it $p$-adically separated} if the canonical map $\rho\colon  M \rightarrow \varprojlim M / p^{n} M$
is injective, and {\it $p$-adically complete} if $\rho$ is an isomorphism. Then an abelian group $M$ is $p$-adically complete if and only if
it is both $p$-adically separated and derived $p$-complete (see \cite{schenzel} or \cite[Prop. 3.4.2]{proetale}). In particular, every $p$-adically complete abelian group is derived $p$-complete.
However, the converse is not true.
\end{remark}

\begin{remark}\label{cose}
The category of derived $p$-complete abelian groups is an abelian category, stable under the formation of kernels, cokernels and extensions in
the larger category of all abelian groups. For a detailed discussion of the theory of derived $p$-complete (also called
\emph{$L$-complete}) abelian groups, we refer the reader to \cite[Appendix A]{HoveyStrickland}.
\end{remark}

\begin{example}\label{sez}
Let $X^{\ast}$ be a cochain complex. Suppose that for each $n \in \Z$, the abelian group
$X^{n}$ is derived $p$-complete. Then each cohomology group $\mathrm{H}^{n}(X)$ is derived
$p$-complete (Remark \ref{cose}), so that $X \in D(\Z)$ is $p$-complete (Remark \ref{cose2}). In particular, if each $X^{n}$ is a $p$-adically complete abelian group, then $X$ is a $p$-complete object of the derived category $D(\Z)$.
\end{example}

\begin{proposition}\label{prop-free}
Let $M$ be an abelian group. Then $M$ is isomorphic to the $p$-adic completion
of a free abelian group if and only if $M$ is derived $p$-complete and $p$-torsion free.
\end{proposition}

\begin{proof}
The ``only if'' direction is immediate. For the converse, choose a system of elements $\{ x_i \}_{i \in I}$ in $M$ which form a basis for $M / p M$ as a vector space over $\F_p$.
Let $F$ denote the free abelian group on the generators $\{ x_i \}_{i \in I}$, so that we have a canonical map
$u\colon  F \rightarrow M$ which induces an isomorphism $F/pF \rightarrow M/pM$. If $M$ is derived $p$-complete and $p$-torsion free, then $u$
exhibits $M$ as the $p$-adic completion of $F$.
\end{proof}

\begin{definition}\label{def-pro-free}
We will say that an abelian group $M$ is \emph{pro-free} if it satisfies the
equivalent conditions of Proposition \ref{prop-free}.
\end{definition} 

\begin{example}\label{subgroupfree}
Let $M$ be a pro-free abelian group and let $M_0 \subseteq M$ be a subgroup.
If $M_0$ is derived $p$-complete, then it is also pro-free.
\end{example} 

We close this section by showing that the $p$-complete derived category $\dhp$ can be realized as the homotopy category of chain complexes of pro-free abelian groups.
Closely related results appear in  \cite{rezk, valenzuela} (where a model category
structure is produced).

\newcommand{\profree}{\mathrm{pro}\text{-}\mathrm{free}}
\begin{proposition}\label{pcompleteDZ}
Let $\coch^{\profree} \subseteq \coch$ denote the full subcategory whose objects are cochain complexes $X^{\ast}$ of pro-free abelian groups.
Then the construction $X^{\ast} \rightarrow X$ determines a functor
$\coch^{\profree} \rightarrow \dhp$ which exhibits $\dhp$
as the homotopy category of $\coch^{\profree}$. In other words:
\begin{itemize}
\item[$(a)$] An object of $D(\Z)$ is $p$-complete if and only if it is isomorphic (in $D(\Z)$) to a cochain complex of pro-free abelian groups.
\item[$(b)$] For every pair of objects $X^{\ast}, Y^{\ast} \in
\coch^{\profree}$, the canonical map
$$ \Hom_{ \coch}( X^{\ast}, Y^{\ast} ) \rightarrow \Hom_{ D(\Z) }( X, Y )$$
is a surjection, and two chain maps $f,g\colon  X^{\ast} \rightarrow Y^{\ast}$ have the same image in $\Hom_{D(\Z)}( X, Y)$ if and only if
they are chain homotopic.
\end{itemize}
\end{proposition} 

\begin{lemma} 
\label{complcomplex}
Let $X^{\ast}$ be a cochain complex of 
$p$-torsion-free abelian groups. Suppose that the underlying object $X \in D(\mathbb{Z})$ is
$p$-complete. Let $\widehat{X}^{\ast}_p$ denote the (levelwise) $p$-completion of the cochain
complex $X^{\ast}$. Then $X^{\ast} \to \widehat{X}^{\ast}_p$ is a
quasi-isomorphism. 
\end{lemma} 

\begin{proof} 
The map of cochain complexes $X^{\ast} / p X^{\ast} \rightarrow
\widehat{X}^{\ast}_{p} / p \widehat{X}^{\ast}_{p}$ is an isomorphism. 
Since $X^{\ast}, \widehat{X}^{\ast}_p$ are  $p$-torsion-free, it follows that 
$X^{\ast} \to \widehat{X}^{\ast}_p$ becomes a  quasi-isomorphism after taking the derived tensor
product with $\Z /p \Z$. Since both complexes define $p$-complete objects in $D(\Z)$, it
follows that $X^{\ast} \to \widehat{X}^{\ast}_p$ is a quasi-isomorphism as
desired.  
\end{proof} 

\begin{lemma}\label{rast}
Let $X^{\ast}$ be a cochain complex of pro-free abelian groups. Then $X^{\ast}$ splits (noncanonically) as a direct sum
$\bigoplus_{n \in \Z} X(n)^{\ast}$, where each $X(n)^{\ast}$ is a cochain complex of pro-free abelian groups concentrated
in degrees $n$ and $n-1$.
\end{lemma}

\begin{proof}
For every integer $n$, let $B^{n} \subseteq X^{n}$ denote the image of the differential $d\colon  X^{n-1} \rightarrow X^{n}$.
Then $B^{n}$ is derived $p$-complete (Remark \ref{cose}). Since $X^{n}$ is pro-free, it follows that $B^{n}$ is also
pro-free (Example \ref{subgroupfree}). It follows that the surjection $d\colon  X^{n-1} \rightarrow B^{n}$ splits: 
that is, we can choose a subgroup $Y^{n-1} \subseteq X^{n-1}$ for which the differential $d$ on $X^{\ast}$
restricts to an isomorphism $Y^{n-1} \simeq B^{n}$. Let $X(n)^{\ast}$ denote the subcomplex of
$X^{\ast}$ given by $Y^{n-1} \oplus \ker(d\colon  X^{n} \rightarrow X^{n+1})$. It is then easy to verify that
these subcomplexes determine a splitting $X^{\ast} \simeq \bigoplus_{n \in \Z} X(n)^{\ast}$.
\end{proof}

\begin{proof}[Proof of Proposition \ref{pcompleteDZ}]
We first prove $(a)$. The ``if'' direction follows from Example \ref{sez}. To prove the converse, let $X$ be an object of $\dhp$; we wish to
show that $X$ belongs to the essential image of the functor
$\coch^{\profree} \rightarrow \dhp$.
Without loss of generality, we may assume that $X$ is represented by a cochain complex $X^{\ast}$ of free abelian groups. 
Let $\widehat{X}^{\ast}_{p}$ be the $p$-adic completion of $X^{\ast}$. Lemma \ref{complcomplex} implies that
$X$ is isomorphic to $\widehat{X}_{p}$ as an object of $D(\Z)$, and therefore belongs to the essential image of
the functor $\coch^{\profree} \rightarrow D(\Z)$.

We now prove $(b)$. Let $X^{\ast}$ be a cochain complex of pro-free abelian
groups; we wish to show that for every cochain complex of pro-free abelian groups $Y^{\ast}$, 
the canonical map
$$ \rho\colon  \Hom_{ \mathrm{h} \coch}( X^{\ast}, Y^{\ast} ) \rightarrow \Hom_{ D(\Z) }( X, Y)$$
is bijective; here $\mathrm{h} \coch$ denotes the homotopy category of cochain complexes.
Using Lemma \ref{rast}, we can reduce to the case where $X^{\ast}$ is concentrated in degrees
$n$ and $n-1$, for some integer $n$. Using the (levelwise split) exact sequence of cochain complexes
$$ 0 \rightarrow X^{n}[-n] \rightarrow X^{\ast} \rightarrow X^{n-1}[-n+1] \rightarrow 0,$$
we can further reduce to the case where $X^{\ast}$ is concentrated in a single degree $n$.
In this case, we can write $X^{\ast}$ as the $p$-adic completion of a chain complex
$\widetilde{X}^{\ast}$ of free abelian groups. Then $\rho$ can be identified with
the restriction map
$$ \Hom_{ \mathrm{h} \coch}( X^{\ast}, Y^{\ast} ) \rightarrow \Hom_{ \mathrm{h} \coch}( \widetilde{X}^{\ast}, Y^{\ast} ),$$
which is bijective by virtue of the fact that each term of $Y^{\ast}$ is $p$-adically complete.
\end{proof} 

\subsection{The Functor $L\eta_p$}
\label{Leta}

For every cochain complex $M^{\ast}$ of $p$-torsion-free abelian groups, we let $\eta_p M^{\ast}$ denote the subcomplex of
$M^{\ast}[1/p]$ defined in Construction \ref{etaconstruction}, given by $(\eta_p M)^n = \left\{x \in p^n M^n\colon  dx \in
p^{n+1} M^{n+1}\right\}$. The construction $M^{\ast} \mapsto \eta_p M^{\ast}$ determines a functor
$$ \eta_p\colon  \cochtfr \rightarrow \cochtfr.$$
In this section, we discuss a corresponding operation on the derived category $D(\Z)$, which we will denote by $L \eta_p$. 
For a more extensive discussion, we refer the reader to \cite[Sec. 6]{BMS}. 

We begin with an elementary property of the functor $\eta_p$:

\begin{proposition} 
\label{etap_quism}
For every cochain complex $(M^{\ast}, d)$ of $p$-torsion-free abelian groups,
there is a canonical isomorphism of graded abelian groups $$\mathrm{H}^*( \eta_p M)
\simeq \mathrm{H}^*(M)/ \mathrm{H}^*(M)[p],$$ where $\mathrm{H}^*(M)[p]$ denotes the
$p$-torsion subgroup of $\mathrm{H}^*(M)$
\end{proposition} 

\begin{proof} 
For every integer $n$, the construction $z \mapsto p^{n} z$ determines a bijection
from the set of cocycles in $M^{n}$ to the set of cocycles in $(\eta_p M)^{n}$.
Note that if $z = dy$ is a boundary in $M^{n}$, then $p^{n} z = d ( p^{n} y )$ is a boundary
in $(\eta_p M)^{n}$; it follows that the construction $z \mapsto p^{n} z$
induces a surjection of cohomology groups $\mathrm{H}^n(M) \to \mathrm{H}^n(\eta_p M)$. 
By definition, the kernel of this map consists of those cohomology classes $[z]$
which are represented by cocycles $z \in M^{n}$ for which $p^{n} z = d (p^{n-1} x)$ for
some $x \in M^{n-1}$, which is equivalent to the requirement that $p [z] = 0$
in $\mathrm{H}^{n}(M)$.
\end{proof} 

\begin{corollary}\label{qism}
Let $f\colon  M^{\ast} \rightarrow N^{\ast}$ be a map between $p$-torsion-free cochain complexes of
abelian groups. If $f$ is a quasi-isomorphism, then the induced map $(\eta_p
M)^{\ast} \rightarrow (\eta_p N)^{\ast}$ is also a quasi-isomorphism. \qed
\end{corollary}

Combining Corollary \ref{qism} with Variant \ref{mariant}, we see that $\eta_p$ induces an endofunctor of the derived category $D(\Z)$:

\begin{corollary}\label{qism2}
There is an essentially unique functor $L \eta_p\colon  D(\Z) \rightarrow D(\Z)$ for which the diagram of categories
$$ \xymatrix{ \cochtfr \ar[r]^{ \eta_p } \ar[d] & \cochtfr \ar[d] \\
D(\Z) \ar[r]^{ L \eta_p } & D(\Z) }$$
commutes up to natural isomorphism. \qed
\end{corollary}

Given an object $X \in D(\Z)$, we have $\mathrm{H}^*(L \eta_p X) \simeq \mathrm{H}^*(X)/\mathrm{H}^*(X)[p]$
by Proposition~\ref{etap_quism}. In particular, $L \eta_p$ is an operator on $D( \Z)$ which reduces the $p$-torsion by a small amount. 

We close this section with the following (see \cite[Lem. 6.19]{BMS}):

\begin{proposition}\label{meta}
The functor $L \eta_p$ commutes with the derived $p$-completion functor of Remark \ref{dercomp}. More precisely,
if $M \rightarrow N$ is a morphism in $D(\Z)$ which exhibits $N$ as a derived $p$-completion of $N$, then
the induced map $L \eta_p M \rightarrow L \eta_p N$ exhibits $L \eta_p N$ as a derived $p$-completion of $L \eta_p M$.
In particular, the functor $L \eta_p$ carries $p$-complete objects of $D(\Z)$ to $p$-complete objects of $D(\Z)$.
\end{proposition} 

Our proof will make use of the following:

\begin{lemma} 
\label{leftexactnesscompletion}
Suppose we are given an exact sequence of $p$-torsion-free abelian groups
\[ 0 \to A \stackrel{f}{\to}B \stackrel{g}{\to} C   \]
and for which the image of $g$ contains $p^n C$ for some $n \gg 0$. 
Then the sequence of pro-free abelian groups
\[ 0 \to \widehat{A}_p \xrightarrow{\widehat{f}_p} \widehat{B}_p
\xrightarrow{\widehat{g}_p} \widehat{C}_p \]
is also exact.
\end{lemma} 
\begin{proof} 
Let $C' \subseteq C$ be the image of $g$, so that the quotient $C/C'$ is annihilated by $p^{n}$ and is therefore (derived) $p$-complete.
Using the exactness of derived $p$-completion, we obtain exact sequences
\[ 0 \to \widehat{A}_p \to \widehat{B}_p \to \widehat{C}'_p \to 0 \]
\[ 0 \to \widehat{C}'_{p} \to \widehat{C}_p \to C/C' \to 0, \]
which can be spliced to obtain the desired sequence $0 \to \widehat{A}_p \to \widehat{B}_p \to \widehat{C}_p$.
\end{proof} 

\begin{proof}[Proof of Proposition \ref{meta}]
Let $M$ be an object of $D(\Z)$, which we can represent by a cochain complex $M^{\ast}$ of free abelian groups.
Let $\widehat{M}^{\ast}_{p}$ denote the cochain complex obtained from $M^{\ast}$ by levelwise $p$-completion,
so that the canonical map $M^{\ast} \rightarrow \widehat{M}^{\ast}_{p}$ exhibits $\widehat{M}_{p}$ as a derived $p$-completion of $M$ in $D(\Z)$.
We wish to show that the induced map $L \eta_p M \rightarrow L \eta_{p} \widehat{M}^{\ast}_{p}$ exhibits $L \eta_p \widehat{M}^{\ast}_{p}$
as a derived $p$-completion of $L \eta_p M$. Since each term of the cochain complex $(\eta_p M)^{\ast}$ is $p$-torsion-free (even free, since subgroups of free abelian
groups are free), it suffices to observe that
that the canonical map $(\eta_p M)^{\ast} \rightarrow (\eta_{p} \widehat{M}_{p})^{\ast}$ exhibits each $(\eta_{p} \widehat{M}_{p})^{n}$
as the $p$-adic completion of $(\eta_p M)^{n}$. This follows by applying Lemma~\ref{leftexactnesscompletion} to the exact sequence
$$ 0 \rightarrow (\eta_p M)^{n} \xrightarrow{x \mapsto (x,dx)} p^{n} M^{n} \oplus p^{n+1} M^{n+1} \xrightarrow{ (y,z) \mapsto dy-z} p^{n} M^{n+1}.$$
\end{proof}

It follows from Proposition~\ref{meta} that $L \eta_p$ restricts to an endofunctor of the subcategory $\widehat{D(\Z)}_p \subseteq D(\Z)$.

\subsection{Fixed Points of $L \eta_p$: $1$--Categorical Version}
\label{Leta1cat}

In this section, we explain that the category $\FrobCompComplete$ of strict Dieudonn\'{e} complexes
can be realized as the {\it fixed points} for the operator $L \eta_p$ on the $p$-complete derived
category $\dhp$ (Theorem~\ref{fixedpointDZ}). In particular, every fixed point
of $L \eta_p$ of $\dhp$ admits a canonical representative in the category
$\coch$ of cochain complexes. 

\begin{definition}\label{fixe}
Let $\mathcal{C}$ be a category and $T\colon  \mathcal{C} \to \mathcal{C}$ be an endofunctor. The \emph{fixed point category} $\mathcal{C}^T$ is defined as follows:
\begin{itemize}
\item The objects of $\calC^{T}$ are pairs $(X, \varphi)$, where $X$ is an object of $\calC$ and $\varphi\colon  X \simeq T(X)$ is an isomorphism.

\item A morphism from $(X, \varphi)$ to $(X', \varphi' )$ in $\calC^{T}$ is a morphism $f\colon  X \rightarrow X'$ in the category $\calC$ with the property that
the diagram
$$ \xymatrix{ X \ar[d]^-{f} \ar[r]^-{ \varphi}_-{\sim} & T(X) \ar[d]^{ T(f)} \\
X' \ar[r]^-{\varphi' }_-{\sim} & T(X') }$$
commutes. 
\end{itemize}
\end{definition} 

\begin{remark}[Functoriality]\label{funcfixed}
Let $\calC$ and $\calC'$ be categories equipped with endofunctors $T\colon  \calC \rightarrow \calC$ and $T'\colon  \calC' \rightarrow \calC'$.
Suppose we are given a functor $U\colon  \calC \rightarrow \calC'$ for which the diagram
$$ \xymatrix{ \calC \ar[r]^{T} \ar[d]^{U} & \calC \ar[d]^{U} \\
\calC' \ar[r]^{T'} & \calC' }$$
commutes up to (specified) isomorphism. Then $U$ induces a functor of fixed point categories $\calC^{T} \rightarrow \calC'^{T'}$, given
on objects by the construction $(C, \varphi) \mapsto ( U(C), U( \varphi ) )$.
\end{remark}

\begin{example}\label{fixedpointsat}
Let $\cochtfr$ denote the category of $p$-torsion-free cochain complexes, and let
$\eta_p\colon  \cochtfr \rightarrow \cochtfr$ be the functor of Construction \ref{etaconstruction}.
If $M^{\ast}$ is a $p$-torsion-free cochain complex, then endowing $M^{\ast}$ with the structure of
a Dieudonn\'{e} complex $(M^{\ast}, F)$ is equivalent to supplying a map of cochain complexes $\alpha_{F}\colon  M^{\ast} \rightarrow \eta_p M^{\ast}$ (Remark \ref{rem2}).
Moreover, the Dieudonn\'{e} complex $( M^{\ast}, F )$ is saturated if and only if the map $\alpha_{F}$ is an isomorphism. It follows that
we can identify the category $\FrobCompSat$ of saturated Dieudonn\'{e} complexes with the fixed point category
$(\cochtfr)^{\eta_p}$ of Definition \ref{fixe}).
\end{example}

Combining the identification of Example \ref{fixedpointsat} with Remark \ref{funcfixed}, we obtain a natural map $\FrobCompSat \simeq ( \cochtfr )^{\eta_p} \rightarrow D(\Z)^{ L \eta_p }$.

\begin{theorem}\label{fixedpointDZ}
The map $\FrobCompSat\rightarrow D(\Z)^{L \eta_p}$ constructed above restricts to an equivalence of categories
$$\FrobCompComplete \xrightarrow{\sim} \widehat{ D(\Z) }^{L \eta_p}_{p}.$$
\end{theorem} 

The proof of Theorem \ref{fixedpointDZ} will require some preliminaries.

\begin{definition} 
Let $X^{\ast}, Y^{\ast}$ be $p$-torsion free Dieudonn\'e complexes. Let $f\colon  X^{\ast} \to
Y^{\ast}$ be a map of the underlying cochain complexes. We say that $f$ is
\emph{weakly $F$-compatible} if the diagram of cochain complexes
\[ \xymatrix{
X^{\ast} \ar[d]^f \ar[r]^{\alpha_F} &  \eta_p X^{\ast} \ar[d]^{\eta_p(f)} \\
Y^{\ast}  \ar[r]^{\alpha_F} &  \eta_p Y^{\ast}
}\]
commutes up to chain homotopy. Equivalently, if $Y^{\ast}$ is saturated (so
that $\alpha_F$ is an isomorphism), we require that the map $F^{-1} \circ f
\circ F = \alpha_F^{-1} \circ \eta_p(f) \circ \alpha_F $ of cochain complexes
$X^{\ast} \to Y^{\ast}$ should be chain homotopic to $f$. 
\end{definition} 

The main ingredient in the proof of Theorem \ref{fixedpointDZ} is the following:

\begin{proposition} 
\label{Phifullfaithful}
Let $X^{\ast}$ and $Y^{\ast}$ be Dieudonn\'{e} complexes and let
$f\colon  X^{\ast} \to Y^{\ast}$ be a map which is weakly $F$-compatible.
If $Y^{\ast}$ is strict, then there is a unique map of Dieudonn\'{e} complexes
$\overline{f}\colon  X^{\ast} \to Y^{\ast}$ which is chain homotopic to $f$.
\end{proposition} 

\begin{proof} 
Let $u\colon  X^{\ast} \to Y^{\ast -1}$ be an arbitrary map of graded abelian groups.
Then, using $FdV = d$, we have
\begin{align*} F^{-1} \circ (du + ud)  \circ F  =d(VuF) + (VuF) d. 
\end{align*}

Write $F^{-1} \circ f \circ F = f + dh + hd$ for some map $h\colon  X^{\ast} \to
Y^{\ast -1}$. 
We now set $u = \sum_{n=0}^\infty V^n h F^n$ and $\overline{f} = f + du + ud$. 
It follows that 
\begin{align*} F^{-1} \circ \overline{f} \circ F  - \overline{f} & = 
dh + hd + d(VuF) + (VuF) d - (du + ud )\\
& = dh + hd + \sum_{n=0}^\infty (dV^{n+1} h F^{n+1} + V^{n+1} h F^{n+1} d ) - 
(\sum_{n=0}^\infty d V^n h F^n + V^n h F^n  d ) \\
& = 0.
\end{align*}
It follows that $\overline{f}\colon  X^{\ast} \to Y^{\ast}$ is a map of Dieudonn\'e complexes which
is chain homotopic to $f$ as a map of chain complexes.

To establish uniqueness, suppose that $g\colon  X^{\ast} \to Y^{\ast}$ is a map of Dieudonn\'e
complexes which is chain homotopic to zero; we can then write $g = dh+hd$ for
some map $h\colon  X^{\ast} \rightarrow Y^{\ast-1}$. For each $r \geq 0$, we have
\begin{eqnarray*}
g & = & F^{-r} \circ g \circ F^{r} \\
& = & F^{-r} (dh + hd) F^r \\
& = & d V^{r} h F^{r} + V^r h F^r d \\
& \in & d V^{r} Y^{\ast} + V^{r} Y^{\ast}.
\end{eqnarray*}
It follows that the composite map $X^{\ast} \xrightarrow{g} Y^{\ast} \rightarrow \W_{r}(Y)^{\ast}$ vanishes for
each $r \geq 0$, so that $g$ vanishes by virtue of our assumption that $Y^{\ast}$ is complete.
\end{proof} 

\begin{proof}[Proof of Theorem \ref{fixedpointDZ}]
Let $X^{\ast}$ and $Y^{\ast}$ be strict Dieudonn\'{e} complexes, and let $X$ and $Y$
denote their images in $D(\Z)$. Since $X^{\ast}$ and $Y^{\ast}$ are levelwise pro-free (in the sense of Definition \ref{def-pro-free}),
we can identify $\Hom_{ D(\Z) }(X, Y)$ with the set of chain homotopy equivalence classes of maps
from $X^{\ast}$ to $Y^{\ast}$ (Proposition \ref{pcompleteDZ}). It follows that
we can identify $\Hom_{ D(\Z)^{L \eta_p} }(X, Y)$ with the set of chain homotopy equivalence classes of
maps $f\colon  X^{\ast} \rightarrow Y^{\ast}$ which are weakly $F$-compatible. Using Proposition
\ref{Phifullfaithful}, we deduce that the canonical map
$$ \Hom_{ \FrobCompComplete}( X^{\ast}, Y^{\ast}) \rightarrow \Hom_{ D(\Z)^{L \eta_p} }( X, Y)$$
is bijective.

To complete the proof, we must show that every  object $(X, \varphi) \in
D(\Z)^{L \eta_p}$ with $X$ $p$-complete
is isomorphic to a strict Dieudonn\'{e} complex. Without loss of generality, we may assume that
$X$ is represented by a cochain complex $X^{\ast}$ of free abelian groups. In this case, we can choose
a quasi-isomorphism of cochain complexes $\alpha\colon  X^{\ast} \rightarrow \eta_p X^{\ast}$ representing the
isomorphism $\varphi\colon  X \simeq L \eta_p X$.
By virtue of Remark \ref{rem2}, this choice endows $X^{\ast}$ with the structure of a Dieudonn\'{e} complex. 
Using Corollary \ref{qism}, we see that each of the transition maps in the diagram
$$ X^{\ast} \xrightarrow{ \alpha } (\eta_p X)^{\ast} \xrightarrow{ \eta_p( \alpha ) } (\eta_p \eta_p X)^{\ast} \xrightarrow{ \eta_p( \eta_p( \alpha ) )} ( \eta_p \eta_p \eta_p X)^{\ast} \rightarrow \cdots $$
is a quasi-isomorphism. It follows that the natural map $X^{\ast} \rightarrow
\Saturate(X^{\ast})$ (which is the colimit of the above sequence) is a quasi-isomorphism. Since $X$ is a $p$-complete object
of $D(\Z)$, it follows from Corollary \ref{cor66} that the canonical map $\Saturate(X^{\ast}) \rightarrow \WSaturate(X^{\ast})$ is also a quasi-isomorphism.
In particular, $X$ is isomorphic to $\WSaturate(X^{\ast})$ as an object of the fixed point category $\widehat{D( \Z)}^{L \eta_p}_{p}$.
\end{proof}

\subsection{Fixed Points of $L \eta_p$: $\infty$--Categorical Version}
\label{LetaInfCat}

From a homotopy-theoretic point of view, the definition of the fixed-point category
$D(\Z)^{L \eta_p}$ has an unnatural feature. Concretely, a morphism in $D(\Z)^{L \eta_p}$
can be represented by a diagram of cochain complexes of free abelian groups
$$ \xymatrix{ X^{\ast} \ar[d]^{f} \ar[r]^{ \varphi} & (\eta_p X)^{\ast} \ar[d]^{ \eta_p(f)} \\
Y^{\ast} \ar[r]^{\varphi' } & (\eta_p Y)^{\ast} }$$
which is required to commute up to chain homotopy, but the chain homotopy itself need not be specified. Typically, constructions which
allow ambiguities of this nature will give rise to categories which are badly behaved. One can avoid this problem by contemplating
a {\em homotopy coherent} variant of the fixed point construction, where the datum of a morphism from $(X^{\ast}, \varphi)$ to $(Y^{\ast}, \varphi')$
is required to also supply a homotopy $h\colon  X^{\ast} \rightarrow (\eta_p Y)^{\ast-1}$ which witnesses the homotopy-commutativity of the diagram above. 
Our goal in this section is to show that, if we restrict our attention to the {\em $p$-complete} derived category $\dhp$, then this modification is
unnecessary: the fixed point category $\widehat{ D(\Z) }^{L \eta_p}_{p}$ given
by Definition \ref{fixe} agrees with its homotopy-coherent
refinement.\footnote{This  refinement will be used in
Section~\ref{sec:crystallinecomp} below.}
To give a precise formulation of this statement, it will be convenient to use
the theory of $\infty$-categories (see \cite{Lur09} for a general treatment,
and \cite[Sec.~1.3]{HA} for a treatment of the derived $\infty$-category). We
will also use freely the theory of stable $\infty$-categories \cite[Ch.~1]{HA}. 

\begin{notation}
Given an $\infty$-category $\mathcal{D}$ and pair of morphisms $f,g\colon  X \rightarrow Y$ in $\mathcal{D}$, we
define an {\it equalizer} of $f$ and $g$ to be a limit of the diagram $X \rightrightarrows Y$, indexed by
the $\infty$-category $\Delta^1 \cup_{\partial \Delta^1} \Delta^1$.
\end{notation}

\begin{example}
\label{spectraequalizer}
Given spectra $X, Y$ and a pair of maps $f, g\colon  X \rightarrow Y$, the equalizer
$\mathrm{eq}(f,g)$ can be identified with the fiber $\mathrm{fib}( f- g)$.
\end{example}

\begin{definition}\label{fixeinfinity}
Let $\mathcal{C}$ be an $\infty$-category and let $T\colon  \mathcal{C} \to
\mathcal{C}$ be a functor. We define the $\infty$-category $\calC^T$
of {\it fixed points of $T$} to be the equalizer of the diagram
$$\mathrm{id}_{\calC}, T\colon  \mathcal{C} \rightrightarrows \mathcal{C},$$
formed in the $\infty$-category of $\infty$-categories.
\end{definition}


We refer to \cite[Sec.~II.1]{NS17} for a detailed treatment of fixed point
$\infty$-categories. In fact, \emph{loc.~cit.}~ treats \emph{lax} fixed point
$\infty$-categories; fixed points 
form a full subcategory of the lax fixed points (see also
Definition~\ref{fixeinfinity2} below). The $\infty$-category $\mathcal{C}^T$ can also be represented as the
pullback
\[ \xymatrix@R=50pt@C=50pt{
\mathcal{C}^{T} \ar[d]  \ar[r] &  \mathcal{C}
\ar[d]^{(\mathrm{id}_{\mathcal{C}}, T)} \ar[d]  \\
\mathcal{C} \ar[r]^{(\mathrm{id}_{\mathcal{C}}, \mathrm{id}_{\mathcal{C}})} &  \mathcal{C} \times \mathcal{C}
}.\]

\begin{remark}\label{ooso}
By construction, an object of $\mathcal{C}^T$ is represented by a pair $(X,
\varphi_X)$ where $X \in \mathcal{C}$ and $\varphi_X\colon  X \simeq TX$ is an equivalence. 
The space of maps between two objects 
$(X, \varphi_X), (Y,  \varphi_Y)$ can be described concretely as the (homotopy) equalizer of the pair of maps
$$\Hom_{\mathcal{C}}(X, Y) \rightrightarrows \Hom_{\mathcal{C}}(X,
TY),$$
given by $f \mapsto \phi_Y \circ f$ and $
f \mapsto F(f) \circ \phi_X$, respectively (see \cite[Prop. II.1.5(ii)]{NS17},
which proves this for the \emph{lax} equalizer, which contains the equalizer
as a full subcategory). 
\end{remark}

\newcommand{\DD}{\mathcal{D}}
\newcommand{\ddhp}{\widehat{\mathcal{D}(\mathbb{Z}_p)}}

\begin{notation}\label{bulish}
We let $\DD(\mathbb{Z})$ denote the derived $\infty$-category of $\Z$-modules. This $\infty$-category has many equivalent descriptions, of which
we single out the following:
\begin{itemize}
\item[$(1)$] The $\infty$-category $\DD(\mathbb{Z})$ can be obtained from the ordinary category
$\cochtfr$ of $p$-torsion-free cochain complexes by formally adjoining (in the
$\infty$-categorical sense) inverses of quasi-isomorphisms (here we can also replace
the category $\cochtfr$ by the larger category $\coch$ of all cochain complexes, or the smaller category $\cochfr$ of cochain complexes of free abelian groups).
Compare with \cite[Sec.~1.3.4]{HA}. 
\item[$(2)$] The $\infty$-category $\DD(\Z)$ can be realized as the {\it differential graded nerve} (in the sense of \cite[Tag 00PL]{kerodon}) of
$\cochfr$, regarded as a differential graded category (that is, a category enriched in cochain complexes).
\end{itemize}

The $\infty$-category $\DD(\Z)$ can be regarded as an ``enhancement'' of the derived category $D(\Z)$ of abelian groups. In particular,
there is a forgetful functor $\DD(\Z) \rightarrow D(\Z)$, which exhibits $D(\Z)$ as the homotopy category of $\DD(\Z)$. We let
$\ddhp \subseteq \DD(\Z)$ denote the inverse image of the subcategory $\dhp \subseteq D(\Z)$ (that is, the full subcategory of
$\DD(\Z)$ spanned by the $p$-complete objects).
\end{notation}

\begin{notation}\label{makeeta}
It follows from Corollary \ref{qism} (and the first characterization of $\DD(\Z)$ given in Notation \ref{bulish}) that
there is an essentially unique functor $L \eta_p\colon  \DD(\Z) \rightarrow \DD(\Z)$ for which the diagram
$$ \xymatrix{ \cochtfr \ar[r]^{ \eta_p } \ar[d] & \cochtfr \ar[d] \\
\DD(\Z) \ar[r]^{ L \eta_p} & \DD(\Z) }$$
commutes up to equivalence. Note that after passing to homotopy categories,
the functor $L \eta_p\colon  \DD(\Z) \rightarrow \DD(\Z)$ reduces to the functor
$L \eta_p\colon  D(\Z) \rightarrow D(\Z)$ of Corollary \ref{qism2}. Moreover,
the functor $L \eta_p$ carries the full subcategory $\ddhp \subseteq \DD(\Z)$
into itself (Proposition~\ref{meta}).
\end{notation}

We can now formulate our main result:

\begin{theorem}\label{1category}
The forgetful functor $\ddhp \rightarrow \dhp$ induces
an equivalence of $\infty$-categories
$$ \widehat{ \calD(\Z_p)}^{L \eta_p} \rightarrow \widehat{D(\Z_p)}^{L \eta_p}.$$
In particular, $\widehat{ \calD(\Z)}^{L \eta_p}_{p}$ is (equivalent to) an ordinary category.
\end{theorem}

\begin{corollary}\label{1categorycor}
The construction $M^{\ast} \mapsto M$ induces an
equivalence of $\infty$-categories
$$ \FrobCompComplete \rightarrow  \widehat{ \calD(\Z)}^{L \eta_p}_{p}.$$
\end{corollary}

\begin{proof}
Combine Theorems \ref{1category} and \ref{fixedpointDZ}.
\end{proof}

\subsection{The Proof of Theorem \ref{1category}}\label{derivedsub6}

To prove Theorem \ref{1category}, we will need a mechanism for computing spaces of morphisms in
the fixed point $\infty$-category $\calD(\Z)^{L \eta_p}$. For this, it will be convenient to work with
a certain model for $\calD(\Z)$ at the level of differential graded categories.

\newcommand{\DDD}{\mathcal{D}_{\mathrm{dg}}}

\begin{definition} 
We let $\DDD(\mathbb{Z})$ be the \emph{differential graded category} whose objects are the cochain
complexes of free abelian groups. Given objects $X^{\ast}, Y^{\ast} \in \DD(\Z)$, we let
$\Hom_{\DDD(\mathbb{Z})}(X^{\ast}, Y^{\ast})$ denote the truncation
$\tau^{\leq 0} [X^{\ast}, Y^{\ast}]$, where 
$[X^{\ast}, Y^{\ast}]$ denotes the usual cochain complex of maps from $X^{\ast}$ to $Y^{\ast}$.
More concretely, we have
\[ 
\Hom_{\DDD(\mathbb{Z})}(X^{\ast}, Y^{\ast})^r \simeq \begin{cases} 
0 & \text{ if } r > 0 \\
\Hom_{\cochfr}(X^{\ast}, Y^{\ast}) & \text{ if } r = 0 \\
\prod_{n \in \mathbb{Z}} \Hom(X^n, Y^{n+r}) & \text{ if } r< 0
 \end{cases} 
.\]
\end{definition} 

Note that we can identify the derived $\infty$-category $\DD(\Z)$ with the differential
graded nerve \cite[Tag 00PL]{kerodon} of $\DDD(\Z)$. This is a slight variant of the second description appearing in Notation \ref{bulish}; 
note that replacing the mapping complexes $[X^{\ast}, Y^{\ast}]$ by their truncations $\tau^{\leq 0} [ X^{\ast}, Y^{\ast} ]$
has no effect on the differential graded nerve.

\begin{construction}
We define a functor $\eta_p\colon  \DDD(\mathbb{Z}) \to \DDD(\mathbb{Z})$ of
differential graded categories defined as follows: 
\begin{enumerate}
\item On objects, the functor $\eta_p$ is given by Construction \ref{etaconstruction}.

\item For every pair of complexes $X^{\ast}, Y^{\ast} \in \DDD(\Z)$, we associate a map of chain complexes
$$\rho\colon  \Hom_{\DDD(\mathbb{Z})}(X^{\ast}, Y^{\ast}) \to \Hom_{\DDD(\mathbb{Z})}( (\eta_p X)^{\ast}, (\eta_p Y)^{\ast}).$$ 
The map $\rho$ vanishes in positive degrees, and is given in degree zero by the map
$\Hom_{ \cochtfr}( X^{\ast}, Y^{\ast} ) \rightarrow \Hom_{ \cochtfr}( (\eta_p X)^{\ast}, (\eta_p Y)^{\ast} )$ induced
by the functoriality of Construction \ref{etaconstruction}. In degrees $r < 0$, it is given by the construction
$$ \rho( \{ f_i\colon  X^{i} \to Y^{i+r} \}_{i \in \Z} ) = \{ f'_{i}\colon
(\eta_p X)^{i} \rightarrow (\eta_p Y)^{i+r}  \}_{i \in \Z},$$
where $f'_{i}$ is given by the restriction of the map $f_i[1/p]\colon  X^i[1/p] \to Y^{i+r}[1/p]$.
This is well-defined by virtue of the observation that 
$$ (\eta_p X)^{i} \subseteq p^{i} X^{i} \subseteq f_i[1/p]^{-1} (p^i Y^{i+r}) \subseteq f_i[1/p]^{-1} ( \eta_p Y)^{i+r}.$$
\end{enumerate}

Passing to differential graded nerves, we obtain a functor of $\infty$-categories $L \eta_p\colon  \DD(\Z) \rightarrow \DD(\Z)$.
By construction, the diagram
$$ \xymatrix{ \cochtfr \ar[r]^{ \eta_p } \ar[d] & \cochtfr \ar[d] \\
\DD(\Z) \ar[r]^{ L \eta_p} & \DD(\Z) }$$
commutes up to homotopy, so that we recover the functor $L \eta_p$ described in Notation \ref{makeeta} (up to canonical equivalence).
\end{construction}

The crucial feature of $\eta_p$ needed for the proof of Theorem \ref{1category} is the following divisibility property:

\begin{proposition} 
\label{divisibilitydg}
Let $X^{\ast}$ and $Y^{\ast}$ be cochain complexes of free abelian groups and let $r$ be a positive integer.
Then:
\begin{itemize}
\item[$(1)$] The canonical map 
$$ u\colon  \Hom_{\DDD(\mathbb{Z})} (X^{\ast}, Y^{\ast})^{-r} \to 
 \Hom_{\DDD(\mathbb{Z})} (\eta_p X^{\ast}, \eta_p Y^{\ast})^{-r}$$ 
is divisible by $p^{r-1}$.
\item[$(2)$] The restriction of $u$ to $(-r)$-cocycles
$$u_0\colon  Z^{-r} \Hom_{\DDD(\mathbb{Z})} (X^{\ast}, Y^{\ast}) \to 
Z^{-r} \Hom_{\DDD(\mathbb{Z})} (\eta_p X^{\ast}, \eta_p Y^{\ast})$$
is divisible by $p^{r}$. 

\item[$(3)$] Let
$$v\colon  \mathrm{H}^{-r}\Hom_{\DDD(\mathbb{Z})} (X^{\ast}, Y^{\ast}) \to \mathrm{H}^{-r} \Hom_{\DDD(\mathbb{Z})} (\eta_p X^{\ast}, \eta_p Y^{\ast})$$
be the map induced by $u$. Then we can write $v = p^{r} v'$
for some $$v'\colon   \mathrm{H}^{-r}\Hom_{\DDD(\mathbb{Z})} (X^{\ast},
Y^{\ast}) \to \mathrm{H}^{-r} \Hom_{\DDD(\mathbb{Z})} (\eta_p X^{\ast}, \eta_p
Y^{\ast}).$$
\end{itemize}
\end{proposition}

\begin{proof} 
By definition, a $(-r)$-cochain of $\Hom_{\DDD(\mathbb{Z})} (X^{\ast}, Y^{\ast})$ is given
by a system of maps $\{ f_n\colon  X^{n} \to Y^{n-r} \}_{n \in \Z}$. Let us abuse notation by identifying each $f_n$ with
its extension to a map $X^{n}[1/p] \rightarrow Y^{n-r}[1/p]$. Assertion $(1)$ follows
from the inclusions
$$f_n( (\eta_p X)^{n} ) \subseteq f_n( p^{n} X^{n} ) \subseteq p^n Y^{n-r} \subseteq p^{r-1} ( \eta_p Y)^{n-r}.$$
If $\{ f_n \}_{n \in \Z}$ is a cocycle, then we have $d f_n = (-1)^r f_{n+1} d$ for each $n$. For $x \in (\eta_p X)^{n}$, it follows
that 
$$ d f_n(x) = \pm f_{n+1}(dx) \in (p^{n+1} Y^{n+1-r}) \cap
\ker(d) \subseteq p^{r} (\eta_p Y)^{n+1-r}$$
so that $f_n(x) \in p^{r} (\eta_p Y)^{n-r}$, which proves $(2)$.

We now prove $(3)$. For each cocycle $z$ in a cochain complex $M^{\ast}$, let $[z]$ denote its image
in the cohomology $\mathrm{H}^{\ast}( M )$. We claim that there is a unique map
$$v'\colon   \mathrm{H}^{-r}\Hom_{\DDD(\mathbb{Z})} (X^{\ast}, Y^{\ast}) \to \mathrm{H}^{-r} \Hom_{\DDD(\mathbb{Z})} (\eta_p X^{\ast}, \eta_p Y^{\ast})$$
satisfying $v'( [z] ) = [ p^{-r} u(z) ]$ for each $z \in Z^{-r} \Hom_{\DDD(\mathbb{Z})} (X^{\ast}, Y^{\ast})$. To show that
this construction is well-defined, it suffices to show that if $z$ and $z'$ are cohomologous $(-r)$-cocycles
of $\Hom_{\DDD(\mathbb{Z})} (X^{\ast}, Y^{\ast})$, then $p^{-r} u(z)$ and $p^{-r} u(z')$ are cohomologous $r$-cocycles of $\Hom_{\DDD(\mathbb{Z})} (\eta_p X^{\ast}, \eta_p Y^{\ast})$.
Writing $z = z' + dw$ for $w \in \mathrm{H}^{-r-1}\Hom_{\DDD(\mathbb{Z})} (X^{\ast}, Y^{\ast})$, we are reduced to proving
that $u(dw) = d u(w)$ is divisible by $p^{r}$. This is clear, since $u(w)$ is divisible by $p^r$ (by virtue of $(1)$). 
\end{proof} 

The next result now follows directly from Proposition~\ref{divisibilitydg} and
the construction of $\DD(\mathbb{Z})$ as a homotopy coherent nerve. 

\begin{corollary}
\label{divhomotopy}
Let $X, Y \in \DD(\mathbb{Z})$. Then, for $r \geq 0$, the canonical map
$$ v\colon   \pi_r  \Hom_{\DD(\mathbb{Z})}(X, Y) \to \pi_r \Hom_{\DD(\mathbb{Z})}(L
\eta_p X, L \eta_p Y) $$
is divisible by $p^r$. That is, there exists a map $$v'\colon
\pi_r \Hom_{\DD(\mathbb{Z})}(X, Y) \to \pi_r \Hom_{\DD(\mathbb{Z})}(L
\eta_p X, L \eta_p Y)$$ such that $v = p^{r} v'$. \qed
\end{corollary}

\begin{remark} 
The functor $L \eta_p\colon  \DD(\Z) \rightarrow \DD(\Z)$ preserves zero objects,
so there is a canonical natural transformation
$f\colon  \Sigma \circ L \eta_p \rightarrow L \eta_p \circ \Sigma$, where $\Sigma\colon  \DD(\Z) \rightarrow \DD(\Z)$ denotes
the suspension equivalence. One can show that the natural transformation $f$ is divisible by $p$,
which gives another proof of Corollary \ref{divhomotopy}. 
\end{remark} 

We will deduce Theorem \ref{1category} by combining Corollary \ref{divhomotopy} with the following:

\begin{lemma} 
\label{autdercompl}
Let $A$ be a derived $p$-complete abelian group and let $f, g\colon  A \to A$ be group homomorphisms.
Suppose $f$ is an automorphism of $A$. Then $f + pg$ is an automorphism of $A$. 
\end{lemma} 

\begin{proof} 
Since $A$ is derived $p$-complete, it suffices to show that $f + pg$
induces an automorphism of $A \otimes^L \Z/p\Z \in D(\Z)$. This follows
from the observation that $f$ and $f+pg$ induce the same
endomorphism of $A \otimes^L \Z/p\Z \in D(\Z)$.
\end{proof} 

\begin{proof}[Proof of Theorem \ref{1category}]
It is clear that the forgetful functor
$$ \widehat{ \calD(\Z)}^{L \eta_p}_{p} \rightarrow \widehat{D(\Z)}^{L \eta_p}_{p}.$$
is essentially surjective. We will show that it is fully faithful.
Fix objects $(X, \varphi_X)$ and $(Y, \varphi_Y) \in \widehat{\mathcal{}D(\Z)}^{L \eta_p}_{p}$,
so that $\varphi_X\colon  X \simeq L \eta_p X$ and $\varphi_{Y}\colon  Y \simeq L \eta_p Y$
are equivalences in $\calD(\Z)$. Using Remark \ref{ooso}, we see that
the mapping space 
$$\Hom_{\widehat{ \calD(\Z)}^{L \eta_p}_{p}}( (X, \varphi_X), (Y, \varphi_Y))$$
can be realized as the homotopy fiber
\[  \mathrm{fib}(f-g\colon  \Hom_{\DD(\mathbb{Z})}( X, Y) \to
\Hom_{\DD(\mathbb{Z})}(X, L \eta_p
Y)),
\]
where
$f$ is given by composition with $\varphi_Y$ and $g$ is 
the composite of $L \eta_p$ and precomposition with $\varphi_X$. 
Note that $f$ is a homotopy equivalence. Combining Lemma~\ref{autdercompl} with
Corollary~\ref{divhomotopy}, we find that $f-g$ induces an isomorphism on 
$\pi_r \Hom_{\DD(\mathbb{Z})}( X, Y)$ for $r > 0$, since $\Hom_{\DD(\Z)} (X,
Y)$ is $p$-complete. 
It follows that
the homotopy fiber of $f-g$ can be identified with the discrete space
given by the kernel of the map
$$ \pi_0(f) - \pi_0(g)\colon  \pi_0 \Hom_{\DD(\mathbb{Z})}( X, Y) \to
\pi_0 \Hom_{\DD(\mathbb{Z})}(X, L \eta_p
Y),$$
which is the set of homomorphisms from 
$(X, \varphi_X)$ to $(Y, \varphi_Y)$ in the ordinary category
$\widehat{D(\Z)}^{L \eta_p}_{p}$.
\end{proof}

\subsection{Tensor Products of Strict Dieudonn\'{e} Complexes}\label{derivedsub4}

The derived $\infty$-category $\calD(\Z)$ admits a symmetric monoidal structure, with underlying tensor product
$$ \otimes^{L}\colon  \calD(\Z) \times \calD(\Z) \rightarrow \calD(\Z ).$$
It is not difficult to see that for any pair of objects $M, N \in \calD(\Z)$, the canonical map
$$ M \otimes^{L} N \rightarrow \widehat{M}_{p} \otimes^{L} \widehat{N}_{p}$$
induces an isomorphism after derived $p$-completion. It follows that there is an essentially
unique symmetric monoidal structure on the $p$-complete derived $\infty$-category $\ddhp$ for which
the derived $p$-completion functor $\calD(\Z) \rightarrow \ddhp$ is symmetric monoidal. We will denote
the underlying tensor product on $\ddhp$ by $\widehat{\otimes}^{L}\colon  \ddhp \times \ddhp \rightarrow \ddhp$;
it is given concretely by the formula
$$ M \widehat{\otimes}^{L} N = \widehat{ M \otimes^{L} N }_{p}.$$

\begin{proposition} 
\label{LetaSymMon}
The functor $L \eta_p\colon  \calD(\Z) \rightarrow \calD(\Z)$ admits a symmetric monoidal structure.
In particular, there are canonical isomorphisms 
$$(L \eta_p)(X) \otimes^{L} (L \eta_p)(Y) \simeq (L \eta_p)( X \otimes^{L} Y )$$
in $\calD(\Z)$.
\end{proposition} 

\begin{proof} 
We first note that $X^{\ast}$ and $Y^{\ast}$ are torsion-free cochain complexes of abelian groups, then the
canonical isomorphism
$$ X^{\ast}[1/p] \otimes Y^{\ast}[1/p] \simeq (X^{\ast} \otimes Y^{\ast})[1/p]$$
carries $(\eta_p X)^{\ast} \otimes (\eta_p Y)^{\ast}$ into $\eta_p( X^{\ast} \otimes Y^{\ast} )$. 
It follows that the functor $\eta_p$ can be regarded as a \emph{lax} symmetric monoidal functor
from the category of torsion-free cochain complexes to itself. Beware that this functor is not symmetric monoidal: the induced map
$\alpha\colon  (\eta_p X)^{\ast} \otimes (\eta_p Y)^{\ast}$ into $\eta_p( X^{\ast} \otimes Y^{\ast} )$
is usually not an isomorphism of cochain complexes. The content of Proposition \ref{LetaSymMon} is
that $\alpha$ is nevertheless a {\em quasi-isomorphism}: that is, it induces an isomorphism
$L \eta_p X \otimes^{L} L \eta_p Y \rightarrow L \eta_p (X \otimes^{L} Y)$ in the $\infty$-category $\calD(\Z)$.
To prove this, we can use the fact that $X$ and $Y$ split (noncanonically) as a
direct sum of their cohomology groups (Remark~\ref{cose2}).
to reduce to the case where $X$ and $Y$ are abelian groups, regarded as cochain complexes concentrated in degree zero.
Writing $X$ and $Y$ as filtered colimits of finitely generated abelian groups,
and using that $L \eta_p$ commutes with filtered colimits, we may further assume that $X$ and $Y$ are finitely
generated. Decomposing $X$ and $Y$ into direct summands, we can assume that both $X$ and $Y$ are cyclic groups which are either free or of prime power order.

We now proceed by direct calculation (as in \cite[Prop. 6.8]{BMS}). We will assume that $X \simeq \Z / p^{m} \Z$ and $Y \simeq \Z / p^{n} \Z$ for some integers
$m, n > 0$ (the remaining cases are easy and left to the reader. Then $X$ is represented by the two-term chain complex $X^{\ast} = (\Z e_{-1} \rightarrow  \Z e_0)$ and $Y^{\ast} = (\Z f_{-1} \rightarrow {\Z} f_0)$,
respectively (both complexes are concentrated in degrees $\{ 0, -1\}$, with $d(e_{-1}) = p^m e_0, d(f_{-1}) = p^n f_0$). 
Then  $X^{\ast} \otimes Y^{\ast}$ is the complex with basis vectors 
$e_0 \otimes f_0, e_{-1} \otimes f_{0},  e_0 \otimes f_{-1}, e_{-1} \otimes
f_{-1}$. 
It is not difficult to see that 
$\eta_p (X^{\ast} \otimes Y^{\ast})$ is the complex with basis vectors $
e_0 \otimes f_0, \frac{1}{p} e_{-1} \otimes f_0, \frac{1}{p} e_0 \otimes f_{-1},
\frac{1}{p^2} e_{-1} \otimes f_{-1}$, and that this agrees (at the level of
cochain complexes) with 
$\eta_p X^{\ast} \otimes \eta_p Y^{\ast}$. 
\end{proof} 

\begin{corollary}\label{LetaSymMon2}
The functor $L \eta_p\colon  \ddhp \rightarrow \ddhp$ admits a symmetric monoidal structure.
In particular, there are canonical isomorphisms 
$$(L \eta_p)(X) \widehat{\otimes}^{L} (L \eta_p)(Y) \simeq (L \eta_p)( X \widehat{\otimes}^{L} Y )$$
in the $\infty$-category $\ddhp$.
\end{corollary}

\begin{proof}
Combine Proposition \ref{LetaSymMon} with Proposition \ref{meta}.
\end{proof}

Let $\WSaturate\colon  \FrobComp \rightarrow \FrobCompComplete$ be the completed saturation functor of Notation \ref{completesat}.
In what follows, we will abuse notation by using Corollary \ref{1categorycor} to identify $\WSaturate$ with a functor
from $\FrobComp$ to the fixed point $\infty$-category $\widehat{ \calD(\Z_p) }^{L \eta_p}$. Note that since $L \eta_p$
is a symmetric monoidal functor from $\dhp$ to itself (Corollary \ref{LetaSymMon2}), the fixed point $\infty$-category
$\widehat{\calD(\Z_p) }^{L \eta_p}$ inherits a symmetric monoidal structure.

\begin{proposition}\label{solik}
The functor $\WSaturate\colon  \FrobComp \rightarrow \widehat{ \calD(\Z_p)}^{L \eta_p}$ is symmetric monoidal, where
we regard $\FrobComp$ as equipped with the symmetric monoidal structure given by the tensor product of
Dieudonn\'{e} complexes (Remark \ref{tprop}).
\end{proposition}

\begin{proof}
Let $\FrobComp^{\mathrm{tf}}$ denote the category of torsion-free Dieudonn\'{e} complexes
and let $L\colon  \FrobComp \rightarrow \FrobComp^{\mathrm{tf}}$ be a left adjoint to the inclusion (given by 
$LM^{\ast} = M^{\ast} / \{ \text{torsion} \}$). Then the functor $L$ is symmetric monoidal, and
$\WSaturate$ factors as a composition 
$$ \FrobComp \xrightarrow{ L} \FrobComp^{ \mathrm{tf} } \xrightarrow{ \WSaturate }  \widehat{ D(\Z) }^{L \eta_p}_{p}.$$
It will therefore suffice to construct a symmetric monoidal structure on the functor $\WSaturate|_{ \FrobComp^{\mathrm{tf} } }$. 

For every pair of torsion-free Dieudonn\'{e} complexes $M^{\ast}$ and
$N^{\ast}$, there are natural maps
$$ \rho_n\colon  (\eta_p ^{n} M)^{\ast} \otimes (\eta_p^{n} N)^{\ast} \rightarrow (\eta_p^{n} (M \otimes N))^{\ast},$$
determined by the natural inclusion of both in $M^{\ast}[1/p] \otimes
N^{\ast}[1/p]$. 
Passing to the limit over $n$, we obtain comparison maps
$$ \rho\colon  \Saturate(M^{\ast}) \otimes \Saturate(N^{\ast}) \rightarrow \Saturate( M^{\ast} \otimes N^{\ast} )$$
which allow us to regard $\Saturate$ as a {\em lax} symmetric monoidal functor from the
category $\FrobComp^{\mathrm{tf}}$ to $\FrobCompSat \simeq (\cochtfr)^{\eta_p}$. 
Composing with the (symmetric monoidal) functors $(\cochtfr)^{\eta_p} \rightarrow \calD(\Z)^{L \eta_p} \rightarrow \widehat{ \calD(\Z_p) }^{L \eta_p}$,
we obtain a lax symmetric monoidal structure on the composite functor 
$\WSaturate\colon  \FrobComp^{ \mathrm{tf} } \rightarrow \widehat{ \calD(\Z_p) }^{L \eta_p}$. We will complete the proof
by showing that this lax symmetric monoidal structure is actually a symmetric monoidal structure. For this,
it will suffice to show that the comparison map $\rho$ is a quasi-isomorphism, for every pair of objects
$M^{\ast}, N^{\ast} \in \FrobComp^{\mathrm{tf}}$. In fact, we claim that
each $\rho_n$ is an isomorphism: this follows from Proposition \ref{LetaSymMon}, using induction on $n$. 
\end{proof}

\begin{remark}\label{sublet}
By virtue of Corollary \ref{1categorycor}, there is an essentially unique symmetric monoidal structure on
the category $\FrobCompComplete$ for which the functor 
$$\FrobCompComplete \xrightarrow{\sim} \widehat{ \calD(\Z_p) }^{L \eta_p}$$
is symmetric monoidal. Let us denote the underlying tensor product functor by
$$ \otimes_{ \mathrm{str} }\colon   \FrobCompComplete  \times \FrobCompComplete \rightarrow \FrobCompComplete.$$
Proposition \ref{solik} can be understood as giving an explicit description of this symmetric monoidal
structure: it is uniquely determined by the requirement that the completed saturation functor $\FrobComp \rightarrow \FrobCompComplete$
should be symmetric monoidal. In particular, for every pair of strict Dieudonn\'{e} complexes $M^{\ast}$ and $N^{\ast}$, we have
a canonical isomorphism
$$ M^{\ast} \otimes_{ \mathrm{str} } N^{\ast} \simeq \WSaturate( M^{\ast} \otimes N^{\ast} ),$$
where $M^{\ast} \otimes N^{\ast}$ is the usual tensor product of Dieudonn\'{e} complexes (Remark \ref{tprop}).
\end{remark}

Proposition \ref{solik} has a number of consequences.

\begin{example}\label{toko}
Let $R$ be a commutative ring which is $p$-complete and $p$-torsion free. Then we can regard $R$ as a commutative ring object of
the $\infty$-category $\ddhp$. The $\infty$-category $\Mod_{R}( \ddhp )$ can then be identified with the $p$-complete
derived $\infty$-category $\widehat{ \calD(R) }$: that is, with the full subcategory of $\calD(R)$ spanned by the $p$-complete objects.

Let $\sigma\colon  R \rightarrow R$ be an automorphism, and let $\sigma_{\ast}\colon  \widehat{ \calD(R) } \rightarrow \widehat{ \calD(R) }$ be
the equivalence of $\infty$-categories given by restriction of scalars along $\sigma$. We let $L \eta_p^{\sigma}\colon  \widehat{ \calD(R) } \rightarrow \widehat{ \calD(R) }$
denote the composite functor $\sigma_{\ast} \circ L \eta_p$. Unwinding the definitions, we have an equivalence of $\infty$-categories
$$ \widehat{\calD(R)}^{ L \eta_{p}^{\sigma}} \simeq \Mod_{R}( \widehat{\calD(\Z_p)}^{L \eta_p} ),$$
where we regard $R$ as a (commutative algebra) object of the fixed point $\infty$-category $\widehat{\calD(\Z_p)}^{L \eta_p}$ via the isomorphism
$\sigma\colon  R \simeq \sigma_{\ast} R$. Combining this observation with Remark \ref{sublet}, we see that
$\widehat{\calD(R)}^{ L \eta_{p}^{\sigma}}$ can be identified with the ordinary category $\Mod_{R}( \FrobCompComplete )$ of $R$-modules in the category
of strict Dieudonn\'{e} complexes. Concretely, the objects of $\Mod_{R}( \FrobCompComplete )$ are given by triples $(M^{\ast}, d, F)$, where
$(M^{\ast}, d)$ is a cochain complex of $R$-modules and $F\colon  M^{\ast} \rightarrow M^{\ast}$ is a map of graded abelian groups satisfying the identities
$$ d(Fx) = p F(dx),  \quad \quad F( ax ) = \sigma(a) F(x) \text{ for $a \in R$},$$
where the underlying Dieudonn\'{e} complex (obtained by forgetting the $R$-module structure) is strict.

We can summarize the situation as follows: given a $p$-complete object $M$ of the derived $\infty$-category $\calD(R)$ equipped
with an isomorphism $\alpha\colon  M \simeq \sigma_{\ast} L \eta_p(M)$, we can find a canonical representative of $M$ as a cochain complex
$M^{\ast}$ of ($p$-torsion free) $R$-modules, for which $\alpha$ can be lifted to an isomorphism of cochain complexes
$M^{\ast} \simeq \sigma_{\ast} \eta_p(M^{\ast})$.
\end{example}

\begin{remark}
In practice, we will be most interested in the special case of Example \ref{toko} where $R = W(k)$ is the ring of Witt vectors of a perfect field $k$,
and $\sigma\colon  W(k) \rightarrow W(k)$ is the Witt vector Frobenius map. In this case, we will denote the functor $L \eta_{p}^{\sigma}$ by
$L \eta_{p}^{\varphi}$.
\end{remark}

\begin{example}\label{torch}
Combining Theorem \ref{1category}, Theorem \ref{fixedpointDZ}, and Remark \ref{sublet}, we see that the following are equivalent:
\begin{itemize}
\item The category of commutative algebra objects $A^{\ast}$ of $\FrobComp$ which are strict as Dieudonn\'{e} algebras.
\item The category of commutative algebra objects $A^{\ast}$ of $\FrobCompComplete$ (with respect to
the tensor product of Remark \ref{sublet}).
\item The category of commutative algebra objects of the fixed point category $\widehat{ D(\Z_p) }^{L \eta_p}$
(or, equivalently, the fixed points of $L \eta_p$ on the category of commutative algebra objects of $\dhp$).
\item The $\infty$-category of commutative algebra objects of the fixed point $\infty$-category
$\widehat{ \calD(\Z_p) }^{L \eta_p}$ (or, equivalently, the fixed points of $L \eta_p$ on the $\infty$-category of commutative algebra objects of $\ddhp$).
\end{itemize}
By virtue of Remark \ref{algdesc}, each of these categories contains the category
$\FrobAlgComplete$ of strict Dieudonn\'{e} algebras as a full subcategory.

In particular, if $A$ is a $p$-complete $E_{\infty}$-algebra over $\Z$ for which there exists an isomorphism of $E_{\infty}$-algebras
$A \simeq L \eta_p (A)$, then $A$ has a canonical representative by a commutative differential graded algebra over $\Z_p$.
\end{example}

\begin{remark} 
In view of the Nikolaus--Scholze description \cite{NS17} of the homotopy theory
$\mathrm{CycSp}$ of \emph{cyclotomic spectra}, we can view the category $\FrobCompComplete$ as a
toy analog of $\mathrm{CycSp}$ (although this is perhaps more appropriate for
the \emph{lax} fixed points of $L \eta_p$ discussed in section~\ref{saturateddRW}). The connection between the de~Rham--Witt
complex and cyclotomic spectra goes back to the work of Hesselholt \cite{He96}.
\end{remark}

\newpage
\section{The Nygaard Filtration}

Let $\widetilde{R}$ be a flat $\Z_p$-algebra whose reduction $R = \widetilde{R} / p \widetilde{R}$ is smooth over $\F_p$, and let $\widehat{\Omega}^*_{\widetilde{R}}$
be the completed de~Rham complex of $\widetilde{R}$ (Variant~\ref{var46}). Then $\widehat{\Omega}^{\ast}_{\widetilde{R}}$ admits subcomplexes
\begin{equation} \label{Nygaardnaive} \calN^k \widehat{\Omega}^*_{\widetilde{R}} := (p^k \widehat{\Omega}{^0}_{\widetilde{R}} \to p^{k-1}
\widehat{\Omega}^1_{\widetilde{R}} \xrightarrow{d} ... \xrightarrow{d} p
\widehat{\Omega}^{k-1}_{\widetilde{R}} \xrightarrow{d}
\widehat{\Omega}^k_{\widetilde{R}} \xrightarrow{d}
\widehat{\Omega}^{k+1}_{\widetilde{R}} \xrightarrow{d} \cdots). \end{equation}
Allowing $k$ to vary, we obtain a descending filtration $\{\calN^k \widehat{\Omega}^*_{\widetilde{R}}\}_{k \geq 0}$ of $ \widehat{\Omega}^*_{\widetilde{R}}$.
This filtration has the following features:
\begin{itemize}
\item[$(1)$] Let $\varphi\colon  \widetilde{R} \to \widetilde{R}$ be any lift of the Frobenius map on $R$. Then the pullback map
$\varphi^*\colon  \widehat{\Omega}^*_{\widetilde{R}} \to  \widehat{\Omega}^*_{\widetilde{R}}$ carries $\calN^k  \widehat{\Omega}^*_{\widetilde{R}}$ into $p^k \widehat{\Omega}_{\widetilde{R}}^*$ (since $\varphi^*$ is divisible by $p^i$ on $i$-forms).
\item[$(2)$] Let $\varphi\colon  \widetilde{R} \to \widetilde{R}$ be any lift of the Frobenius map on $R$. Then the maps $\calN^i \widehat{\Omega}^*_{\widetilde{R}} \xrightarrow{\varphi^*}
p^i \widehat{\Omega}^*_{\widetilde{R}}$ determined by $(1)$ for $i=k,k+1$ induce a
quasi-isomorphism
\[ \gr^k_{\calN}  \widehat{\Omega}^*_{\widetilde{R}} \to \tau^{\leq k} \big(p^k  \widehat{\Omega}^*_{\widetilde{R}}/p^{k+1}  \widehat{\Omega}^*_{\widetilde{R}}) \simeq \tau^{\leq k} \Omega^*_R,\] 
where the second isomorphism is obtained by dividing by $p^k$. 
\end{itemize}

The filtered complex $\{\calN^k \widehat{\Omega}^*_{\widetilde{R}}\}_{k \geq 0}$ of $ \widehat{\Omega}^*_{\widetilde{R}}$ depends on the $\Z_p$-algebra $\widetilde{R}$, and not only on its mod $p$ reduction $R = \widetilde{R} / p \widetilde{R}$. Nevertheless, it was observed by Nygaard \cite[\S 1]{Nygaard} (and elaborated by Illusie-Raynaud \cite[\S III.3]{IllusieRaynaud}) that 
the de~Rham--Witt complex $W\Omega^{\ast}_{R}$ admits an analogous filtration $\{\calN^k W\Omega^*_R\}_{k \geq 0}$, which depends only on $R$. Moreover, if
$\varphi\colon  \widetilde{R} \rightarrow \widetilde{R}$ is a lift of Frobenius, then the quasi-isomorphism $\widehat{\Omega}^*_{\widetilde{R}} \simeq W\Omega^*_R$ of Theorem~\ref{theo53}
is actually a {\em filtered} quasi-isomorphism. Consequently, when viewed as an object of the filtered derived category, $\{\calN^k  \widehat{\Omega}^*_{\widetilde{R}}\}_{k \geq 0}$
depends only on $R$.

In this section, we construct a {\it Nygaard filtration} $\{ \calN^{k} M^{\ast} \}_{k \geq 0}$ on any saturated Dieudonn\'e complex $M^*$ (Construction~\ref{NygaardDef}). In \S \ref{nygaardsec2}, we show that this filtration always satisfies analogues of properties $(1)$ and $(2)$ above (Proposition~\ref{Nygaard}). In the special case where
$M^{\ast}$ is the saturation of a Dieudonn\'{e} complex $N^{\ast}$ of Cartier type, we show that the inclusion $N^{\ast} \hookrightarrow M^{\ast}$ is a filtered quasi-isomorphism,
where $M^{\ast}$ is equipped with the Nygaard filtration and $N^{\ast}$ is equipped with an analogue of the filtration \eqref{Nygaardnaive} (Proposition~\ref{NaiveNygaard}). 
In \S \ref{nygaardsec4}, we use the ideas of \S \ref{derivedcatsec} to give an alternative description of the Nygaard filtration, using the language of filtered derived categories.

\subsection{The Nygaard Filtration of a Saturated Dieudonn\'{e} Complex}

Throughout this section, we let $(M^\ast,d,F)$ denote a saturated Dieudonn\'e complex.

\begin{construction}[The Nygaard filtration]
\label{NygaardDef}
For every pair of integers $i$ and $k$, we define
$$ \calN^{k} M^{i} = \begin{cases} p^{k-i-1} V M^{i} & \text{ if } i < k \\
M^{i} & \text{ if } i \geq k. \end{cases}$$
Note that for $i < k$ and $x = p^{k-i} Vy \in \calN^{k} M^{i-1}$, we have
$$ dx  =  d( p^{k-i} Vy ) 
 =  p^{k-i} d(Vy)  
 =  p^{k-i-1} V(dy) 
 \in  \calN^{k} M^{i}.$$ 
It follows that 
\[ \calN^k M^\ast := \cdots \xrightarrow{d} p V M^{k-2}  \xrightarrow{d} VM^{k-1} \xrightarrow{d} M^k \xrightarrow{d} M^{k+1} ... .\]
is a subcomplex of $M^{\ast}$. Note that $\calN^{k+1} M^\ast \subseteq \calN^k M^\ast$, so we can view $\{\calN^k M^\ast\}_{k \in \Z}$ as a descending filtration on $M^\ast$.
We will refer to this filtration as the {\it Nygaard filtration} on $M^\ast$, and write 
\[ \gr^k_{\calN} M^\ast := \calN^k M^\ast/\calN^{k+1} M^\ast\]
for the corresponding graded pieces.
\end{construction}

\begin{remark}
Note that $p^{k-i} M^i \subseteq \calN^k M^i \subseteq p^{k-i-1} M^i$ for all $i < k$: the second containment is obvious, while the first follows from the equality $p^{k-i} M^i = p^{k-i-1} V(FM^i)$. It follows that $M^\ast$ complete for the Nygaard filtration (that is, the map $M^{\ast} \rightarrow \varprojlim M^{\ast} / \calN^{k} M^{\ast}$ is an isomorphism) if and only if each $M^i$ is $p$-adically complete.
\end{remark}

\begin{example}
Let $R$ be a perfect ring of characteristic $p$ and let $M = W(R)$, which we regard as a Dieudonn\'{e} complex concentrated in degree zero (Example~\ref{ex24}). Then $\calN^k W(R) = p^k W(R)$ for $k \geq 0$, since $F$ is an automorphism of $W(R)$ and $V = F^{-1} p$.
\end{example}

\begin{remark}
Let $R$ be a smooth algebra over a perfect field $k$ of characteristic $p$, so that the de~Rham--Witt complex $W\Omega^*_R$ is a saturated Dieudonn\'e complex. The image of $W\Omega^*_R$ in the derived category $D(\Z_p)$ can be described as the absolute crystalline cohomology of $\Spec(R)$: that is, it is given by
$\RGamma( (\Spec(R))_{\crys}, \sheafO_{\crys})$ where $\sheafO_{\crys}$ is the
structure sheaf on the crystalline site of $\Spec(R)$ (relative to $(\spec
(\mathbb{Z}_p), (p))$). The Nygaard filtration $\calN^k W\Omega^*_R$ also admits a crystalline interpretation for $0 \leq k < p$, as explained by Langer-Zink in \cite[Theorem 4.6]{LZ07}: the image of $\calN^k W\Omega^*_R$ in the derived category $D(\Z_p)$ is identified with $\RGamma( (\Spec(R))_{\crys}, \mathcal{I}_{\crys}^{[k]})$, where $\mathcal{I}_{\crys}^{[k]} \subseteq \sheafO_{\crys}$ is the $k$-th divided power of the universal pd-ideal sheaf $\mathcal{I}_{\crys} \subseteq \sheafO_{\crys}$ on the crystalline site. This description fails for $k \geq p$: for example, taking $p=2$ and $R = \F_2$ perfect gives a counterexample as $(2^{[k]}) = (2)$ as ideals in $W(\F_2)$ when $k$ is a power of $2$, while $\calN^k W(\F_2) = (2^k)$ for all $k \geq 0$. We are not aware of a crystalline description of the Nygaard filtration for $k \geq p$.
\end{remark}

\subsection{The Nygaard Filtration of a Completion}\label{nygaardsec2}

Let $(M^{\ast},F)$ be a Dieudonn\'{e} complex, where $M^{\ast}$ is
$p$-torsion-free. As explained in Remark \ref{rem2}, we can identify $F$ with a map
of chain complexes $\alpha_{F}\colon  M^{\ast} \to \eta_{p} M^{\ast}$. In what follows, we will identify $\eta_p M^{\ast}$ with a subcomplex of $M^{\ast}[ p^{-1} ]$, so that
we can also regard $\alpha_F$ as a map from $M^{\ast}$ to $M^{\ast}[ p^{-1} ]$ (given by $\alpha_F(x) = p^{i} F(x)$ for $x \in M^{i}$). 

\begin{proposition} 
\label{Nygaard}
Let $M^{\ast}$ be a saturated Dieudonn\'{e} complex and let $k$ be an integer. Then:
\begin{enumerate}
\item The terms of the complex $\gr^k_{\calN} M^\ast$ vanish in degrees $> k$. 

\item We have $\calN^{k} M^{\ast} = \alpha_{F}^{-1}( p^{k} M^{\ast}  \cap
\eta_p M^{\ast})$. In particular,
$\alpha_{F}$ determines a map of cochain complexes $\calN^{k} M^{\ast}
\rightarrow p^{k} M^{\ast} \cap \eta_p M^{\ast}$.

\item The map $\alpha_F$ induces an isomorphism of cochain complexes
$$\gr^k_{\calN} M^\ast \to \tau^{\leq k} \big(p^k M^\ast/p^{k+1} M^\ast\big).$$

\item The map $x \mapsto p^{-k} \alpha_F(x)$ induces an isomorphism $\gr^k_{\calN} M^* \simeq \tau^{\leq k} \big(M^*/pM^*\big)$. 
\end{enumerate}
\end{proposition}

\begin{proof}
Assertion $(1)$ follows from the observation that $\calN^{k} M^{i} = M^i$ for $i \geq k$, and assertion $(4)$ is a formal consequence of $(3)$.
We now prove $(2)$. We first claim that $\alpha_{M}( \calN^{k} M^{i} ) \subseteq p^{k} M^{i}$. By definition, we have
$\alpha_F(x) = p^{i} F(x)$ for $x \in M^{i}$, so the inclusion is trivial when $i \geq k$. For $i < k$, we observe that
\[ \alpha_F(p^{k-i-1} VM^i) = p^i F p^{k-i-1} VM^i = p^{k-1} FV M^i = p^k M^i.\]
To complete the proof of $(2)$, we must show that if $x \in M^{i}$ satisfies $\alpha_F(x) \in p^{k} M^{i}$, then $x$ belongs to $\calN^{k} M^{i}$. This is vacuous for
$i \geq k$, so we may assume that $i < k$. Our hypothesis gives
$$p^{i} F(x) = p^{k} y = (p^{i} F)( p^{k-i-1} Vy )$$
for some $y \in M^{i}$. Since $p^{i} F$ is an injection, we obtain $x = p^{k-i-1} Vy \in \calN^{k} M^{i}$, as desired.

We now prove $(3)$. In degrees $i < k$, the map under consideration is given by
\[ p^{k-i-1} VM^i/p^{k-i} VM^i \xrightarrow{p^i F} p^k M^i/p^{k+1} M^i.\]
It is thus sufficient to check that $p^i F$ gives a bijection $p^{k-i-1} VM^i \to p^k M^i$ for $i < k$. Since $M^{\ast}$ has no $p$-torsion, this reduces to checking that $F$ gives a bijection $VM^i \to pM^i$, which is clear: $F$ is injective and $FV = p$.

In degree $k$, the map under consideration is given by
\[ M^k/VM^k \xrightarrow{p^k F}  \Big(p^k M^k \cap d^{-1}(p^{k+1} M^{k+1})\Big)/p^{k+1} M^k.\]
To show this is an isomorphism, we may divide by $p^k$ to reduce to checking that the right vertical map in the diagram
$$ \xymatrix{ 0 \ar[r] & V M^{k} \ar[r] \ar[d]^{F} & M^{k} \ar[r] \ar[d]^{F} & M^{k} / V M^{k} \ar[r] \ar[d]^{F} & 0 \\
0 \ar[r] & p M^{k} \ar[r] & d^{-1}(p M^{k+1} ) \ar[r] & d^{-1}( p M^{k+1} ) / p M^{k} \ar[r] & 0 }$$
is an isomorphism. This is clear, since the rows are exact and the maps
$$ F\colon  V M^{k} \rightarrow pM^{k} \quad \quad F\colon  M^{k} \rightarrow d^{-1}( pM^{k+1} )$$
are isomorphisms (by virtue of our assumption that $M^{\ast}$ is saturated).
\end{proof}

\begin{corollary}\label{Nycor}
Let $f\colon  M^{\ast} \rightarrow N^{\ast}$ be a map of saturated Dieudonn\'{e} complexes. Suppose
that $f$ induces a quasi-isomorphism $M^{\ast} / p M^{\ast} \rightarrow N^{\ast} / p N^{\ast}$.
Then, for every integer $k$, the induced map
$$ M^{\ast} / \calN^{k} M^{\ast} \rightarrow N^{\ast} / \calN^{k} N^{\ast}$$
is a quasi-isomorphism.
\end{corollary}

\begin{proof}
For each $k \in \Z$, Proposition \ref{Nygaard} supplies a commutative diagram of cochain complexes
$$ \xymatrix{ \gr^{k}_{\calN} M^{\ast} \ar[r] \ar[d] & \gr^{k}_{\calN} N^{\ast} \ar[d] \\
\tau^{\leq k} (M^{\ast} / p M^{\ast} ) \ar[r] & \tau^{\leq k} ( N^{\ast} / p N^{\ast} ) }$$
where the vertical maps are isomorphisms. It follows that $f$ induces a quasi-isomorphism
$\gr^{k}_{\calN} M^{\ast} \rightarrow \gr_{\calN} N^{\ast}$. Proceeding by induction, we
deduce that the maps
$$ \calN^{k'} M^{\ast} / \calN^{k} M^{\ast} \rightarrow \calN^{k'} N^{\ast} / \calN^{k} N^{\ast}$$
are quasi-isomorphisms for each $k' \leq k$. The desired result now follows by passing to the direct limit
$k' \rightarrow - \infty$.
\end{proof}

\begin{corollary}\label{Nycor2}
Let $M^{\ast}$ be a saturated Dieudonn\'{e} complex. Then, for every integer $k$,
the canonical map $M^{\ast} / \calN^{k} M^{\ast} \rightarrow \WittScript(M)^{\ast} / \calN^{k} \WittScript(M)^{\ast}$
is a quasi-isomorphism.
\end{corollary}

\begin{proof}
By virtue of Corollary \ref{Nycor}, it will suffice to show that the map
$M^{\ast} / p M^{\ast} \rightarrow \WittScript(M)^{\ast} / p \WittScript(M)^{\ast}$ is a quasi-isomorphism.
This follows from Proposition \ref{prop17} and Corollary \ref{cor79}.
\end{proof}

\begin{corollary}\label{Nycor3}
Let $f\colon  M^{\ast} \rightarrow N^{\ast}$ be a quasi-isomorphism of saturated Dieudonn\'{e} complexes.
Then $f$ is a filtered quasi-isomorphism: that is, for every integer $k$, the induced map
$\calN^{k} M^{\ast} \rightarrow \calN^{k} N^{\ast}$ is a quasi-isomorphism.
\end{corollary}

\begin{proof}
If $f$ is a quasi-isomorphism, then the induced map $M^{\ast} / p M^{\ast} \rightarrow N^{\ast} / p N^{\ast}$
is also a quasi-isomorphism. We have a commutative diagram of short exact sequences
$$ \xymatrix{ 0 \ar[r] & \calN^{k} M^{\ast} \ar[d] \ar[r] & M^{\ast} \ar[r] \ar[d]^{f} & M^{\ast} / \calN^{k} M^{\ast} \ar[r] \ar[d] & 0 \\
0 \ar[r] & \calN^{k} N^{\ast} \ar[r] & N^{\ast} \ar[r] & N^{\ast} / \calN^{k} N^{\ast} \ar[r] & 0 }$$
where the vertical map in the middle is a quasi-isomorphism by assumption and the vertical map on the right is
a quasi-isomorphism by Corollary \ref{Nycor}, so that the vertical map on the left is also a quasi-isomorphism.
\end{proof}

\subsection{Dieudonn\'{e} Complexes of Cartier Type}\label{nygaardsec3}

We now consider a variant of the Nygaard filtration which can be defined for {\em any} cochain complex.

\begin{construction}
\label{NygaardCartier}
Let $M^{\ast}$ be any cochain complex of abelian groups. For every integer $k$, we set
$\mathcal{N}_u^k M^{\ast} \subseteq M^{\ast}$ denote the subcomplex given by
\[ \calN^k_u M^{\ast}:= \cdots \to p^{3} M^{k-3} \to p^2 M^{k-2} \to p M^{k-1} \to M^k \to M^{k+1} \to \cdots.\]
This defines a descending filtration $\{\calN^k_u M^{\ast}\}_{k \in \Z}$ of $M^*$ in the category of cochain complexes of abelian groups.
\end{construction}

\begin{remark}\label{spanor}
Let $M^{\ast}$ be a saturated Dieudonn\'{e} complex. Then, for every integer $k$,
we have $\calN^{k}_{u} M^{\ast} \subseteq \calN^{k} M^{\ast}$.
\end{remark}

Let $M^{\ast}$ be a Dieudonn\'{e} complex and let $\Saturate( M^{\ast} )$ denote its saturation. Using Remark \ref{spanor}, we obtain maps of filtered complexes
$$ \{ \calN^{k}_{u} M^{\ast} \}_{k \in \Z} \rightarrow \{ \calN^{k}_u \Saturate(M^{\ast}) \}_{k \in \Z} \rightarrow \{ \calN^{k} \Saturate(M^{\ast} ) \}_{k \in \Z}.$$

\begin{proposition} 
\label{NaiveNygaard}
Let $M^{\ast}$ be a Dieudonn\'{e} complex of Cartier type (Definition~\ref{def:Cartiertype}).
Then, for every integer $k$, the canonical map
$$ M^{\ast} / \calN^{k}_{u} M^{\ast} \rightarrow \Saturate(M^{\ast}) / \calN^{k} \Saturate(M^{\ast} )$$
is a quasi-isomorphism.
\end{proposition} 

\begin{proof}
For every pair of integers $i$ and $k$, we have canonical isomorphisms
$$ \mathrm{H}^{i}( \gr^{k}_{\calN_u} M^{\ast} ) \simeq (\gr^{k}_{\calN_u} M)^{i} \simeq \begin{cases} M^{i} / p M^{i} & \text{ if } i \leq k \\
0 & \text{ otherwise. } \end{cases}$$
Using Proposition \ref{Nygaard}, we obtain isomorphisms $$\mathrm{H}^{i}( \gr_{\calN}^{k} \Saturate(M^{\ast}) ) \simeq
\begin{cases} \mathrm{H}^{i}( M^{\ast} / p M^{\ast}) & \text{ if } i \leq k \\
0 & \text{ otherwise. } \end{cases}$$
Under these isomorphisms, we can identify the canonical map
$$ \mathrm{H}^{i}(  \gr^{k}_{\calN_u} M^{\ast} ) \rightarrow \mathrm{H}^{i}( \gr_{\calN}^{k} \Saturate(M^{\ast}) )$$
with the composition
$$ M^{i} / p M^{i} \xrightarrow{F} \mathrm{H}^{i}( M^{\ast} / p M^{\ast} ) \rightarrow \mathrm{H}^{i}( \Saturate(M)^{\ast} / p \Saturate(M)^{\ast} )$$
for $i \leq k$, and with the zero map otherwise. Here the first map is a quasi-isomorphism by virtue of our assumption that $M^{\ast}$ is of Cartier type,
and the second by virtue of Theorem \ref{theo60}.

Proceeding by induction, we deduce that for each $k' \leq k$, the induced map
$$ \calN^{k'}_{u} M^{\ast} / \calN^{k}_{u} M^{\ast} \rightarrow \calN^{k'} \Saturate(M^{\ast} ) / \calN^{k} \Saturate(M^{\ast} )$$
is a quasi-isomorphism. Passing to the direct limit as $k' \rightarrow - \infty$, we obtain the desired quasi-isomorphism
$M^{\ast} / \calN^{k}_{u} M^{\ast} \rightarrow \Saturate(M^{\ast}) / \calN^{k} \Saturate(M^{\ast})$.
\end{proof}

\begin{corollary}\label{Nycor4}
Let $M^{\ast}$ be a Dieudonn\'{e} complex of Cartier type. Then, for every integer $k$, the canonical map
$$ M^{\ast} / \calN^{k}_{u} M^{\ast} \rightarrow \WSaturate(M^{\ast}) / \calN^{k} \WSaturate(M^{\ast} )$$
is a quasi-isomorphism.
\end{corollary}

\begin{proof}
Combine Proposition \ref{NaiveNygaard} with Corollary \ref{Nycor2}.
\end{proof}

\begin{corollary}\label{Nycor5}
Let $M^{\ast}$ be a Dieudonn\'{e} complex of Cartier type, and suppose that each of the abelian groups $M^{i}$ is $p$-adically complete.
Then, for every integer $k$, the canonical map $\calN^{k}_{u} M^{\ast} \rightarrow \calN^{k} \WSaturate(M^{\ast})$ is a quasi-isomorphism.
In other words, the map $M^{\ast} \rightarrow \WSaturate(M^{\ast} )$ is a filtered quasi-isomorphism, where we equip
$\WSaturate(M^{\ast})$ with the Nygaard filtration of Construction \ref{NygaardDef} and $M^{\ast}$ with the filtration of Construction \ref{NygaardCartier}.
\end{corollary}

\begin{proof}
Combine Corollary \ref{Nycor4} with Corollary \ref{cor67}.
\end{proof}

\begin{example}
Let $\widetilde{R}$ be a flat $\Z_p$-algebra for which the reduction $R = \widetilde{R} / p \widetilde{R}$ is smooth over $\F_p$, and let $\varphi\colon  \widetilde{R} \rightarrow \widetilde{R}$ be a lift
of the Frobenius map on $R$. Then the completed de~Rham complex $\widehat{\Omega}^{\ast}_{ \widetilde{R} }$ is a Dieudonn\'{e} complex of Cartier type (Corollary \ref{cor68}),
and the saturated de~Rham--Witt complex $\WOmega^{\ast}_{R}$ can be identified with the completed saturation $\WSaturate(  \widehat{\Omega}^{\ast}_{ \widetilde{R} } )$.
It follows from Corollary \ref{Nycor5} that the canonical map $\rho\colon  \calN^{k}_{u} \widehat{\Omega}^{\ast}_{ \widetilde{R} }
\rightarrow \calN^{k} \WOmega^{\ast}_{R}$ is a quasi-isomorphism for every integer $k$.
\end{example}

\subsection{The Nygaard Filtration and $L\eta_p$ via Filtered Derived Categories}\label{nygaardsec4}

Let $M^{\ast}$ be a saturated Dieudonn\'{e} complex and let $\{ \calN^{k} M^{\ast} \}_{k \in \Z}$ be its Nygaard filtration. The goal of this section is to show that,
as an object of the filtered derived category $DF(\Z)$ (see Construction \ref{DFLeta1} below), the filtered complex $\{ \calN^{k} M^{\ast} \}_{k \in \Z}$ depends only
on the image $M \in D(\Z)$ of $M^{\ast}$ in the derived category $D(\Z)$, together with the isomorphism $M \simeq L \eta_p M$ induced by
the Frobenius operator on $M^{\ast}$. In other words, it depends only on the structure of $M$ as an object of the fixed point category
$D(\Z)^{L \eta_p }$ studied in \S \ref{derivedcatsec}.

\begin{construction}[The Filtered Derived Category]\label{DFLeta1}
Let us regard the set of integers $\Z$ as a linearly ordered set. For any $\infty$-category $\calC$, we can identify functors $\Z^{\op} \rightarrow \calC$ with
diagrams
$$ \cdots \rightarrow C^2 \rightarrow C^1 \rightarrow C^0 \rightarrow C^{-1} \rightarrow C^{-2} \rightarrow \cdots$$
The collection of all such diagrams $\calC$ can be organized into an $\infty$-category $\Fun( \Z^{\op}, \calC )$.
We define the {\it filtered derived category $DF(\Z)$} to be the homotopy category of the $\infty$-category $\Fun( \Z^{\op}, \DD(\Z) )$. Here
$\DD(\Z)$ denotes the derived $\infty$-category of abelian groups (Notation \ref{bulish}). 
\end{construction}

\begin{example}\label{sipard}
Let $\{ \mathcal{F}^{k} M^{\ast} \}_{k \in \Z}$ be a filtered cochain complex of abelian groups, which we can identify with a diagram
$$ \cdots \hookrightarrow \mathcal{F}^{2} M^{\ast}  \hookrightarrow \mathcal{F}^{1} M^{\ast} \hookrightarrow \mathcal{F}^{0} M^{\ast} \hookrightarrow \mathcal{F}^{-1} M^{\ast} \hookrightarrow \mathcal{F}^{-2} M^{\ast} \hookrightarrow \cdots$$
in the category $\coch$ of cochain complexes. Applying the functor $\coch \rightarrow \DD(\Z)$, we obtain an object of the filtered derived category $DF(\Z)$. This construction determines a functor from the category of filtered cochain complexes to the filtered derived category $DF(\Z)$.
\end{example}

\begin{remark}
Let $f \colon M^{\ast} \rightarrow N^{\ast}$ be a morphism of filtered cochain complexes. We will say that
$f$ is a {\it levelwise quasi-isomorphism} if, for every integer $k$, the induced map $\mathcal{F}^{k} M^{\ast} \rightarrow \mathcal{F}^{k} N^{\ast}$ is a quasi-isomorphism.
It is not difficult to show that the category $DF(\Z)$ of Construction \ref{DFLeta1} can be obtained from the category of filtered cochain complexes
by formally inverting the levelwise quasi-isomorphisms.
\end{remark}

\begin{warning}
The terminology of Construction \ref{DFLeta1} is not standard. Many authors use the term {\it filtered derived category} to refer to the category $\widehat{DF}(\Z)$ obtained by
localizing the category of filtered cochain complexes with respect to the
collection of {\it filtered quasi-isomorphisms}: that is, the collection of maps
$f \colon M^{\ast} \to N^{\ast}$ with the property
that, for every integer $k$, the induced map of associated graded complexes $\gr^{k}(M^{\ast}) \rightarrow \gr^{k}( N^{\ast} )$ is a quasi-isomorphism. Every levelwise quasi-isomorphism is a filtered quasi-isomorphism, but the converse is not necessarily true (for example, if $M^{\ast}$ is a cochain complex equipped with a {\em constant} filtration $\mathcal{F}^{k} M^{\ast} = M^{\ast}$, then the projection map $M^{\ast} \rightarrow 0$ is always
a filtered quasi-isomorphism, but usually not a levelwise quasi-isomorphism). The resulting category $\widehat{DF}(\Z)$ can be identified with a full subcategory of our filtered derived category $DF(\Z)$,
spanned by those functors $C^{\bullet} \colon \Z^{\op} \rightarrow \DD(\Z)$
which are {\it complete} in the sense that the homotopy limit $\varprojlim_{n}
C^{n}$ vanishes (see \cite[Sec.~5]{BMS2} for a more detailed discussion).
\end{warning}

\begin{remark}[The Underlying Complex]\label{remund}
For each object $\overrightarrow{C} = \{ C^{k} \}_{k \in \Z} \in DF(\Z)$, we let $C$ denote the direct limit $\varinjlim_{k \to -\infty} C^{k}$.
We will refer to $C$ as the {\it underlying complex} of $\overrightarrow{C}$.
The construction $\overrightarrow{C} \mapsto C^{- \infty}$ determines a functor of $\infty$-categories
$\Fun( \Z^{\op}, \DD(\Z) ) \rightarrow \DD(\Z )$, hence also a functor of ordinary categories $DF(\Z) \rightarrow D(\Z)$. 

Note that if $\overrightarrow{C}$ arises from a filtered cochain complex $\{ \mathcal{F}^{k} M^{\ast} \}_{k \in \Z}$ as in
Example \ref{sipard}, then $C$ is the image in the derived category of the usual direct limit
$\varinjlim_{k \to -\infty} \mathcal{F}^{k} M^{\ast}$.
\end{remark}

\begin{remark}[The Associated Graded]
For each object $\overrightarrow{C} = \{ C^{k} \}_{k \in \Z} \in DF(\Z)$ and each integer $n$, we let $\gr^{n}( \overrightarrow{C} )$ denote the cofiber of the map
$C^{n+1} \rightarrow C^{n}$. The construction $\overrightarrow{C} \mapsto \gr^{n}(\overrightarrow{C})$ determines a functor of $\infty$-categories
$\Fun( \Z^{\op}, \DD(\Z) ) \rightarrow \DD(\Z )$, hence also a functor of ordinary categories $DF(\Z) \rightarrow D(\Z)$.

Note that if $\overrightarrow{C}$ arises from a filtered cochain complex $\{ \mathcal{F}^{k} M^{\ast} \}_{k \in \Z}$ as in
Example \ref{sipard}, then $\gr^{n}( \overrightarrow{C})$ is the image in the derived category of the quotient
$\mathcal{F}^{n} M^{\ast} / \mathcal{F}^{n+1} M^{\ast}$.
\end{remark}

\begin{example}\label{sosho}
Let $M^{\ast}$ be a cochain complex of $p$-torsion-free abelian groups. Then we can equip the localization $M^{\ast}[1/p]$ with the $p$-adic filtration, given by the diagram
$$ \cdots \hookrightarrow p^2 M^{\ast} \hookrightarrow p M^{\ast} \hookrightarrow M^{\ast} \hookrightarrow p^{-1} M^{\ast} \hookrightarrow p^{-2} M^{\ast} \hookrightarrow \cdots.$$
This construction determines a filtered cochain complex $\{ p^{k} M^{\ast} \}_{k \in \Z}$, hence an object of the filtered derived category $DF(\Z)$. Note
that the construction $M^{\ast} \mapsto \{ p^{k} M^{\ast} \}_{k \in \Z}$ carries quasi-isomorphisms in $\cochtfr$ to equivalences in $DF(\Z)$.
We therefore obtain a functor of derived categories $D(\Z) \rightarrow DF(\Z)$, which we will denote by $M \mapsto p^{\bullet} M$. 
\end{example}

\begin{remark}[The Beilinson t-Structure]\label{rembeil}
The filtered derived category $DF(\Z)$ arises as the homotopy category of a stable $\infty$-category $\Fun( \Z^{\op}, \DD(\Z) )$, and therefore inherits the structure of a triangulated category.
This triangulated category admits a t-structure $( DF^{\leq 0}(\Z), DF^{\geq
0}(\Z))$ which can be described as follows:
\begin{itemize}
\item An object $\overrightarrow{C} \in DF(\Z)$ belongs to $DF^{\leq 0}(\Z)$ if and only if, for every integer $k$, the object $\gr^{k}( \overrightarrow{C} )$ belongs to $D^{\leq k}(\Z)$: that
is, the cohomology groups $\mathrm{H}^{n}( \gr^{k}(\overrightarrow{C} ) )$ vanish for $n > k$.
\item An object $\overrightarrow{C} = \{ C^k \}_{k \in \Z} \in DF(\Z)$ belongs to $DF^{\geq 0}(\Z)$ if and only if, for every integer $k$, we have $C^{k} \in D^{\geq k}(\Z)$: that
is, the cohomology groups $\mathrm{H}^{n}( C^{k} )$ vanish for $n < k$.
\end{itemize}
We will refer to $( DF^{\leq 0}(\Z), DF^{\geq 0}(\Z))$ as the {\it Beilinson t-structure on $DF(\Z)$}. The heart of this t-structure can be identified with the abelian category
$\coch$ of cochain complexes of abelian groups.
For details, see \cite[Sec. 5.1]{BMS2}; the $t$-structure goes back to
\cite[Appendix A]{Beilinson}. 
\end{remark}

\begin{example}\label{sup}
Let $\{ \mathcal{F}^{k} M^{\ast} \}_{k \in \Z}$ be a filtered cochain complex, which we identify with its image in the filtered derived category. Let
$\tau^{\leq 0}_{B} \left( \{ \mathcal{F}^{k} M^{\ast} \}_{k \in \Z} \right)$ denote its truncation with respect to the Beilinson t-structure. Then
$\tau^{\leq 0}_{B} \left( \{ \mathcal{F}^{k} M^{\ast} \}_{k \in \Z} \right)$ can be represented explicitly by the filtered subcomplex $\{ \mathcal{F}'^{k} M^{\ast} \}_{k \in \Z}$ described by the formula
$$ \mathcal{F}'^{k} M^n = \begin{cases} \mathcal{F}^{k} M^n & \text{ if } n < k, \\
\{ x \in \mathcal{F}^{n} M^{n}: dx \in \mathcal{F}^{n+1} M^{n+1} \} & \text{ if } n \geq k. \end{cases}$$
To prove this, it suffices to observe that each quotient $\mathcal{F}'^{k}
M^{\ast} / \mathcal{F}'^{k+1} M^{\ast}$ can be identified with $\tau^{\leq k} (
\mathrm{gr}^k M^{\ast})$
(so that $\{ \mathcal{F}'^{k} M^{\ast} \}_{k \in \Z}$ belongs to $DF^{\leq 0}(\Z)$), and that each quotient $\mathcal{F}^{k} M^{\ast} / \mathcal{F}'^{k} M^{\ast}$
belongs to $D^{> k}(\Z)$ (so that $\{ \mathcal{F}^{k} M^{\ast} / \mathcal{F}'^{k} M^{\ast} \}_{k \in \Z}$ belongs to $DF^{>0}(\Z)$).

In particular, the underlying cochain complex of $\tau^{\leq 0}_{B} \left(  \{
\mathcal{F}^{k} M^{\ast} \}_{k \in \Z} \right)$ can be identified with
$N^{\ast}$, where $N^{k} = \{ x \in \mathcal{F}^{k} M^{k}: dx \in \mathcal{F}^{k+1} M^{k+1} \}$.
\end{example}

\begin{proposition}\label{makegaard}
The functor $L \eta_p\colon  D(\Z) \rightarrow D(\Z)$ of Corollary \ref{qism2} is isomorphic to the composition
$$ D(\Z) \xrightarrow{M \mapsto p^{\bullet} M} DF(\Z) \xrightarrow{ \tau^{\leq 0}_{B} } DF(\Z) \xrightarrow{ \overrightarrow{M} \mapsto M} D(\Z).$$
Here the first functor is given by the formation of $p$-adic filtrations (Example \ref{sosho}), the second by
truncation for the Beilinson t-structure (Remark \ref{rembeil}), and the third is given by passing to the underlying cochain complex
(Remark \ref{remund}).
\end{proposition}

\begin{proof}
Let $M^{\ast}$ be a cochain complex of $p$-torsion-free abelian groups. Using the analysis of Example \ref{sup}, we see that
the underlying cochain complex of $\tau^{\leq 0}_{B}( p^{\bullet} M^{\ast} )$ can be identified with 
$(\eta_p M)^{\ast} = \{ x \in p^{\ast} M^{\ast}: dx \in p^{\ast +1} M^{\ast+1} \}$.
\end{proof}

It follows from Proposition \ref{makegaard} that, for any object $M$ of the derived category $D(\Z)$, the object $L \eta_p(M)$ comes equipped with a canonical filtration:
more precisely, it can be lifted to an object of the filtered derived category $DF(\Z)$, given by $\tau^{\leq 0}_{B} (p^{\bullet} M)$. In particular, if $M$ is equipped with
an isomorphism $\varphi\colon  M \simeq L \eta_p(M)$, then we can lift $M$ itself to an object of $DF(\Z)$. In the case where $M$ arises from a saturated Dieudonn\'{e} complex
$M^{\ast}$, we can model this lift explicitly using the Nygaard filtration of Construction \ref{NygaardDef}:

\begin{proposition}
Let $M^{\ast}$ be a saturated Dieudonn\'{e} complex, so that the Frobenius map $F$ induces an isomorphism $\alpha_F\colon  M^{\ast} \simeq (\eta_p M)^{\ast}$ (Remark \ref{rem2}).
Let us abuse notation by identifying $\alpha_F$ with an isomorphism $M \simeq L \eta_p(M)$ in the derived category $D(\Z)$. Then $\alpha_F$
can be lifted to an isomorphism $\{ \mathcal{N}^{k} M^{\ast} \}_{k \in \Z} \simeq \tau^{\leq 0}_{B} ( p^{\bullet} M)$ in the filtered derived category $DF(\Z)$.
\end{proposition}

\begin{proof}
Using Example \ref{sup}, we see that $\tau^{\leq 0}_{B} (p^{\bullet} M)$ can be represented explicitly by the filtered cochain complex $\{ \mathcal{N}'^{k} (\eta_p M)^{\ast} \}_{k \in \Z}$
given by the formula
$$ \mathcal{N}'^{k} (\eta_p M)^{n} = \begin{cases} p^k M^n & \text{ if } n < k \\
\{ x \in p^n M^{n}: dx \in p^{n+1} M^{n+1} \} & \text{ if } n \geq k. \end{cases}$$
It now suffices to observe that each $\mathcal{N}^{k} M^{n}$ can be described as the inverse image of $\mathcal{N}'^{k} (\eta_p M)^{n}$ under the isomorphism
$(p^{n}F)\colon  M^{n} \xrightarrow{\sim} (\eta_p M)^{n}$ determined by $\alpha_F$.
\end{proof}

\newpage
\section{The Derived de Rham-Witt Complex}
\label{sec:sddRW}

In this section, we provide an alternative description of the saturated de
Rham--Witt complex $\W \Omega_R^{\ast}$ of an $\mathbb{F}_p$-algebra $R$ using the {\em derived de~Rham--Witt complex} (cf. \cite{BhattdR} and
\cite[Ch. VIII]{Ill2}). Our main result (Theorem~\ref{sddRWdRW}) can be summarized informally as follows: the derived de~Rham--Witt complex $LW\Omega_R$
of an $\F_p$-algebra $R$ admits a natural ``divided Frobenius'' map $\alpha_R\colon LW\Omega_R \to L\eta_p(LW\Omega_R)$, and inverting the map $\alpha_{R}$
yields the saturated de~Rham--Witt complex $\W \Omega^*_R$ studied in this paper. 

We give two applications of the preceding description. First, we prove an analog
of the Berthelot--Ogus isogeny theorem \cite{BO83} for the saturated de~Rham--Witt complex without any regularity constraints: if $R$ is an $\mathbf{F}_p$-algebra of embedding dimension $\leq d$, then the Frobenius endomorphism of $\W\Omega^*_R$ factors multiplication by $p^d$ (Corollary~\ref{FrobSatdRWIsogeny}); in particular, it is a $p$-isogeny. Second, we give an alternative proof that the comparison map
$\Omega_{R}^{\ast} \rightarrow \W_1 \Omega_{R}^{\ast}$ is an isomorphism for a regular Noetherian $\F_p$-algebra $R$ (Theorem \ref{theo73}), which avoids the 
use of Popescu's theorem. 

\subsection{Lax Fixed Points}
\label{saturateddRW} 

We begin by introducing a variant of Definition \ref{fixeinfinity}.

\begin{definition}\label{fixeinfinity2}
Let $\mathcal{C}$ be an $\infty$-category and let $T \colon \mathcal{C} \to
\mathcal{C}$ be a functor. Form a pullback diagram
$$ \xymatrix{ \calC^{T}_{\mathrm{lax} } \ar[r] \ar[d] & \Fun( \Delta^1, \calC ) \ar[d] \\
\calC \ar[r]^-{ (\id, T)} & \calC \times \calC. }$$

Then $\calC^{T}_{\mathrm{lax}}$ is an $\infty$-category, whose objects
can be identified with pairs $(C, \varphi)$ where $C$ is an object of $\calC$
and $\varphi\colon  C \rightarrow TC$ is a morphism in $\calC$. We will refer to
$\calC^{T}_{\mathrm{lax}}$ as the {\it $\infty$-category of lax fixed points of
$T$ on $\calC$}. When $\mathcal{C}$ is an ordinary 1-category, then so is
$\mathcal{C}^T_{\mathrm{lax}}$, given by the above description. 
\end{definition}

\begin{example}\label{CFrobComp}
Let $\FrobComp^{\mathrm{tf}} \subseteq \FrobComp$ be the subcategory spanned
by the $p$-torsion-free Dieudonn\'e complexes. Then $\FrobComp^{\mathrm{tf} }$ can be identified
with the category of lax fixed points for $\eta_p$ on the category $\cochtfr$ of
$p$-torsion-free cochain complexes of abelian groups
(see Remark \ref{rem2}). 
\end{example}

\begin{construction}
\label{LaxToStrict}
Let $\calC$ be an $\infty$-category and let $T\colon  \calC \rightarrow \calC$ be a functor. Then we can regard the fixed point
$\infty$-category $\calC^{T}$ of Definition \ref{fixeinfinity} as a full subcategory of the lax fixed point $\infty$-category
$\calC^{T}_{\mathrm{lax}}$ of Definition \ref{fixeinfinity2}. If $\calC$ admits filtered colimits which are preserved by the functor
$T$, then the inclusion $\calC^{T} \hookrightarrow \calC^{T}_{\mathrm{lax}}$ admits a left adjoint, which carries
a pair $(C, \varphi)$ to the colimit of the diagram
$$ C \xrightarrow{ \varphi} TC \xrightarrow{T \varphi} T^2 C \xrightarrow{ T^2 \varphi} T^3 C \rightarrow \cdots$$
Compare the proof of \cite[Prop. II.5.3]{NS17}. 
\end{construction}

\begin{example}
In the case where $\calC = \cochtfr$ and $T$ is the functor $\eta_p$, the functor
$$ \FrobComp^{\mathrm{tf}} = (\cochtfr)^{\eta_p}_{ \mathrm{lax} } \rightarrow (\cochtfr)^{\eta_p} = \FrobCompSat$$
given in Construction \ref{LaxToStrict} agrees with the saturation functor on Dieudonn\'{e} complexes introduced in \S \ref{substrict}.
\end{example}

\begin{notation}\label{satcomp}
Let $\widehat{\calD(\Z)}^{L \eta_p}_{\mathrm{lax}}$ denote the lax fixed points for $L \eta_p$, regarded
as a functor from the $p$-complete derived $\infty$-category $\ddhp$ to itself. Note that
the functor $L \eta_p\colon  \ddhp \rightarrow \ddhp$ preserves filtered colimits, since it commutes
with filtered colimits on the larger $\infty$-category $\calD(\Z)$ and with the operation of $p$-completion (Proposition \ref{meta}).
Applying Construction \ref{LaxToStrict}, we see that the inclusion
$\widehat{\calD(\Z_p)}^{L \eta_p} \hookrightarrow \widehat{\calD(\Z)}^{L \eta_p}_{\mathrm{lax}}$ admits a left adjoint,
which we will denote by
$$ \widehat{\Saturate}\colon  \widehat{\calD(\Z)}^{L \eta_p}_{\mathrm{lax}} \rightarrow \widehat{\calD(\Z_p)}^{L \eta_p}.$$
\end{notation}

\begin{remark}\label{obo}
The $\infty$-category $\widehat{\calD(\Z)}^{L \eta_p}_{\mathrm{lax}}$ can be described as the 
\emph{lax equalizer} of the pair of functors $(\mathrm{id}, L \eta_p)\colon  \ddhp \to \ddhp$, in the sense of \cite[Def. II.1.4]{NS17}. 
It follows that $\widehat{\calD(\Z)}^{L \eta_p}_{\mathrm{lax}}$ is a presentable $\infty$-category and that
the forgetful functor $\widehat{\calD(\Z)}^{L \eta_p}_{\mathrm{lax}} \rightarrow
\widehat{\calD(\Z_p)}$ preserves small colimits, cf. \cite[Prop.
II.1.5(iv)]{NS17}.
\end{remark}

\begin{remark}
Let $M^{\ast}$ be a torsion-free Dieudonn\'{e} complex having image $M \in \calD(\Z)$. Then we can regard
$M$ as a lax fixed point for the functor $L \eta_p\colon  \calD(\Z) \rightarrow \calD(\Z)$, so that its derived $p$-completion
$\widehat{M}$ is a lax fixed point for the restriction $L \eta_p\colon  \widehat{ \calD(\Z_p)} \rightarrow \widehat{ \calD(\Z_p)}$.
Using Corollary \ref{derivedcompleteDC}, we see that the saturation functor $\widehat{\Saturate}$ of
Construction \ref{LaxToStrict} carries $\widehat{M}$ to the completed saturation $\WSaturate( M^{\ast} )$ of Notation \ref{completesat}.
In other words, the diagram of $\infty$-categories
$$ \xymatrix{ \FrobComp^{\mathrm{tf}} \ar[r]^{\WSaturate} \ar[d]^-{M^{\ast}
\mapsto \widehat{M}} & \FrobCompComplete \ar[d]^{\sim} \\
\widehat{\calD(\Z)}^{L \eta_p}_{\mathrm{lax}} \ar[r]^-{\widehat{\Saturate}} & \widehat{\calD(\Z_p)}^{L \eta_p} }$$
commutes up to canonical isomorphism; note here that we use 
that $\WSaturate(\cdot)$ is insensitive to $p$-adic completion. Here the right vertical map is an equivalence by virtue of Corollary \ref{1categorycor}.
\end{remark}

\subsection{Digression: Nonabelian Derived Functors}
 
We now give a brief review of the theory of nonabelian derived functors in the setting of $\infty$-categories, following \cite[\S 5.5.9]{Lur09}. For the sake of concreteness,
we will confine our attention to functors which are defined on the $\infty$-category of simplicial commutative $\F_p$-algebras.

\begin{notation}
Let $\SCRf$ denote the $\infty$-category of simplicial commutative $\F_p$-algebras. More precisely, we define
$\SCRf$ to be the $\infty$-category obtained from the ordinary category of
simplicial commutative $\F_p$-algebras by formally inverting quasi-isomorphisms. For a detailed discussion of the resulting homotopy theory (using the language of model categories), we refer the reader to \cite[II.4]{Quillen}. 
\end{notation}

Let $\CAlg_{\F_p}^{\mathrm{poly}}$ denote the ordinary category whose objects are $\F_p$-algebras of the form $\F_p[ x_1, \ldots, x_n ]$. 
In what follows, we will identify $\CAlg_{ \F_p}^{ \mathrm{poly} }$ with a full subcategory of the $\infty$-category $\SCRf$.
By virtue of \cite[Prop.~5.5.8.15 and Cor.~5.5.9.3]{Lur09} (which applies since
$\CAlg_{\F_p}^{\mathrm{poly}}$ admits finite coproducts), the $\infty$-category $\SCRf$ can be characterized by a universal property:

\begin{proposition}\label{uprop}
Let $\calD$ be any $\infty$-category which admits small sifted
colimits (equivalently, filtered colimits and geometric realizations of simplicial objects) and let $\Fun'( \SCRf, \calD )$ denote
the full subcategory of $\Fun( \SCRf, \calD )$ spanned by those functors which preserve small sifted colimits. 
Then composition with the inclusion functor $\CAlg_{ \F_p}^{\mathrm{poly} } \hookrightarrow \SCRf$ induces an equivalence of $\infty$-categories
\[ \pushQED{\qed} \Fun'( \SCRf, \calD ) \rightarrow \Fun( \CAlg_{ \F_p }^{\mathrm{poly} },
\calD ).   \qedhere \popQED \] 
\end{proposition}

We can summarize Proposition \ref{uprop} more informally as follows: any functor $F_0 \colon   \CAlg_{ \F_p }^{\mathrm{poly} } \rightarrow \calD$
admits an essentially unique extension to a functor $F\colon  \SCRf \rightarrow \calD$ which preserves small sifted colimits. In this case,
we will refer to $F$ as the {\it nonabelian derived functor} of $F_0$.

\begin{example} 
Let $\calD = \calD( \F_p )$ be the derived $\infty$-category of $\F_p$ and let
$F_0\colon  \CAlg_{ \F_p}^{\mathrm{poly} } \rightarrow \mathcal{D}(\F_p)$ be the functor given by
$F_0(R) = \Omega^{1}_{R/\F_p}$. Then the nonabelian derived functor
$F\colon  \SCRf \rightarrow \calD$ carries a simplicial commutative $\F_p$-algebra
$R$ to the {\it cotangent complex} $L_{R/\F_p}$ (regarded as an object of $\mathcal{D}(\F_p)$
via restriction of scalars).
\end{example} 

\begin{remark}\label{sebix}
Let $\calD$ be an $\infty$-category which admits small sifted colimits, let $f\colon  \CAlg_{\F_p}^{\mathrm{poly}} \rightarrow \calD$
be a functor, and let $F\colon  \SCRf \rightarrow \calD$ be its nonabelian derived functor. Then $F$ can be characterized
as the {\it left Kan extension} of the functor $f$ (see \cite[Sec.
4.3.2]{Lur09}), in view of the $\infty$-categorical Yoneda theory, cf.
\cite[Lem.~5.1.5.5, Thm. 5.1.5.6, and Prop. 5.5.8.10]{Lur09}. It follows that if $\calC \subseteq \SCRf$ is a full subcategory containing
$\CAlg_{\F_p}^{\mathrm{poly}}$ and $G\colon  \calC \rightarrow \calD$ is any functor, then the restriction map
$$ \Hom_{ \Fun( \calC, \calD) }( F|_{\calC}, G) \rightarrow \Hom_{ \Fun( \CAlg_{\F_p}^{\mathrm{poly}}, \calD) }( f, G|_{ \CAlg_{\F_p}^{\mathrm{poly}}} )$$
is a homotopy equivalence. We will be primarily interested in the special case where $\calC = \CAlg_{\F_p}$ is the category of commutative $\F_p$-algebras.
\end{remark}

We now restrict our attention to the case of primary interest to us.

\begin{construction}[The derived de~Rham--Witt complex]\label{ddRW}
For each polynomial algebra $R \in \CAlg_{ \F_p }^{\mathrm{poly} }$, let us abuse notation by identifying
the de~Rham--Witt complex $W \Omega^{\ast}_{R} \simeq \WOmega^{\ast}_{R}$ with its image
in the ($p$-complete) derived $\infty$-category $\widehat{ \calD(\Z_p) }$. The construction
$R \mapsto W \Omega^{\ast}_{R}$ determines a functor $f\colon  \CAlg_{\F_p}^{\mathrm{poly}} \rightarrow 
\ddhp$, which admits a nonabelian derived functor $F\colon  \SCRf \rightarrow \ddhp$. Given
a simplicial commutative $\F_p$-algebra $R$, we will denote its image under $F$
by $L W\Omega_{R}$, and refer to it as the {\it derived de~Rham--Witt complex of
$R$}. We refer the reader to \cite[Sec. 8]{BMS2} for a more detailed discussion.

Note that for $R \in \CAlg_{ \F_p}^{\mathrm{poly} }$, we can regard $W \Omega^{\ast}_{R}$ as a Dieudonn\'{e} complex.
In particular, it can be regarded as a lax fixed point for the functor $L \eta_p$ on $\ddhp$; that is, we can promote $f$ to a
functor $\widetilde{f}\colon  \CAlg_{\F_p}^{\mathrm{poly}} \rightarrow 
\widehat{\calD(\Z)}^{L \eta_p}_{\mathrm{lax}}$. Applying Proposition \ref{uprop}, we see that $\widetilde{f}$ admits a nonabelian
left derived functor $\widetilde{F}\colon  \SCRf \rightarrow  \widehat{\calD(\Z)}^{L \eta_p}_{\mathrm{lax}}$.
Since the forgetful functor $\widehat{\calD(\Z)}^{L \eta_p}_{\mathrm{lax}} \rightarrow \widehat{\calD(\Z)}_p$ preserves small colimits (Remark \ref{obo}),
we can regard $\widetilde{F}$ as a lift of the functor $F$. Given
a simplicial commutative $\F_p$-algebra $R$, we will abuse notation by denoting its image under the functor $\widetilde{F}$ also by
$L W \Omega_{R}$. In other words, we regard $\widetilde{F}$ as supplying a canonical map
$\alpha_{R}\colon  LW\Omega_R \to L\eta_p(LW\Omega_R)$, for each simplicial commutative $\F_p$-algebra $R$. 
\end{construction}

\begin{variant}[The derived de~Rham complex]\label{deriveddeRham}
For any $R \in \SCRf$, write 
\[ L\Omega_{R} := LW\Omega_R \otimes_{\Z_p}^L \F_p \in \mathcal{D}(\F_p).\]
We will refer to $L \Omega_{R}$ as the {\it derived de~Rham complex} of $R$. Using Theorem~\ref{theo73} and Corollary~\ref{cor79}, 
we see that the construction $R \mapsto L \Omega_{R}$ is the nonabelian derived functor of the usual de~Rham complex functor
$\Omega^*_{(-)}\colon  \CAlg_{\F_p}^{\mathrm{poly}} \to \mathcal{D}(\F_p)$.
\end{variant}

\newcommand{\filc}{\mathrm{Fil}^{\mathrm{conj}}}
\begin{remark}[The Conjugate Filtration]\label{ddr} 
For every polynomial algebra $R \in \CAlg_{\F_p}^{ \mathrm{poly} }$, let
$\tau^{\leq n} \Omega_{R}^{\ast}$ denote the $n$th truncation of the
de~Rham complex $\Omega_{R}^{\ast}$, which we regard as an object of
the derived $\infty$-category $\calD( \F_p )$. The construction
$R \mapsto \tau^{\leq n} \Omega_{R}^{\ast}$ admits a nonabelian 
left derived functor
$$ \SCRf \rightarrow \calD( \F_p ) \quad \quad R \mapsto \filc_n L \Omega_{R}.$$
Then we can regard $\{ \filc_n L  \Omega_{R} \}_{n \geq 0}$ as an exhaustive
filtration of the derived de~Rham complex $L \Omega_{R}$, which we refer
to as the {\it conjugate filtration}. Using the Cartier isomorphism for $R \in \CAlg_{\F_p}^{\mathrm{poly}}$, 
we obtain canonical isomorphisms
$$ \gr^{n} \left( \filc_{\ast}L \Omega_{R} \right) \simeq (\bigwedge^{n} L_{ R^{(1)} / \F_p})[-n]$$
for all $R \in \SCRf$; here $R^{(1)} \to R$ denotes the Frobenius endomorphism of $R$.
\end{remark}

\begin{remark}\label{ddr2}
For any commutative $\F_p$-algebra $R$, there is a canonical map
$L\Omega_{R} \to \Omega^*_{R/\F_p}$, which is uniquely determined by the
requirement that it depends functorially on $R$ and is the identity when $R$ is
a polynomial algebra over $\F_p$ (Remark~\ref{sebix}).
Using the conjugate filtration of Remark \ref{ddr}, one can show that this map is a quasi-isomorphism when
$R$ is a smooth algebra over a perfect field $k$ of characteristic $p$ (see
\cite[Cor. 3.10]{BhattdR}). Indeed, 
we observe that for a smooth $k$-algebra $R$, we have the
(increasing, exhaustive) Postnikov filtration
on 
$\Omega^{\ast}_{R/\mathbb{F}_p}$ whose associated graded is given by
$\bigwedge^i \Omega^1_{R^{(1)}/\mathbb{F}_p}[-i]$. Since each of the functors
$$ \{ \text{Smooth $k$-algebras} \} \rightarrow \DD(\Z) \quad \quad R \mapsto \bigwedge^{i} \Omega^{1}_{R^{(1)}/\mathbb{F}_p}[-i]$$
is a left Kan extension of its restriction to the subcategory of polynomial algebras over $k$, it follows that the functor
$$ \{ \text{Smooth $k$-algebras} \} \rightarrow \DD(\Z) \quad \quad R \mapsto \Omega_{R}^{\ast}$$
is also a left Kan extension of its restriction to the subcategory of of polynomial algebras over $k$.
\end{remark}

\subsection{Saturated Derived Crystalline Cohomology}

If $R$ is a polynomial algebra over $\F_p$, then the de~Rham--Witt complex $W \Omega^{\ast}_{R}$ is a {\em strict} Dieudonn\'{e} algebra.
In particular, we can regard $W \Omega^{\ast}_{R}$ as a fixed point (rather than a lax fixed point) for the endofunctor $L \eta_p\colon  \ddhp \rightarrow \ddhp$.
In general, the derived de~Rham--Witt complex $L W \Omega_{R}$ of Construction \ref{ddRW} need not have this property: the canonical map
$\alpha_{R}\colon  L W \Omega_{R} \rightarrow L \eta_p ( L W \Omega_{R} )$ need not be an isomorphism, since the functor
$L \eta_p$ does not commute with sifted colimits. This can be corrected by passing to the saturation 
$\widehat{\Saturate}( L W \Omega_{R} )$ in the sense of Notation \ref{satcomp}. This saturation admits a more concrete description:

\begin{theorem}\label{sddRWdRW}
Let $R$ be a simplicial commutative $\F_p$-algebra. Then there is a canonical map
\[ \tau_R\colon  LW\Omega_R \to \W\Omega^*_{\pi_0(R)}\]
which exhibits the saturated de~Rham--Witt complex $\WOmega^{\ast}_{ \pi_0(R)}$ as a saturation of
$L W \Omega_{R}$, in the sense of Notation \ref{satcomp}.
\end{theorem}

\begin{remark}\label{labelremark}
Concretely, Theorem \ref{sddRWdRW} asserts that the saturated de~Rham--Witt complex $\WOmega^{\ast}_{\pi_0(R)}$ can be computed as the $p$-completed direct limit of the diagram
\[ LW\Omega_R \xrightarrow{\alpha_R} L\eta_p(LW\Omega_R) \xrightarrow{L\eta_p(\alpha_R)} L\eta_{p^2}(LW\Omega_R) \to \cdots.\]
\end{remark}

The proof of Theorem \ref{sddRWdRW} will make use of the following:

\begin{lemma} 
\label{colimitsAlg}
The category $\FrobAlgComplete$ of strict Dieudonn\'{e} algebras admits small colimits, and the forgetful functor
$\FrobAlgComplete \to \FrobCompComplete$ commutes with sifted colimits. 
\end{lemma} 

\begin{proof} 
By virtue of Corollary \ref{oloter}, the colimit of a diagram $\{ A^{\ast}_{\alpha} \}$ in the
category $\FrobAlgComplete$ can be computed by applying the completed saturation functor
$\WSaturate$ to the colimit $\varinjlim A^{\ast}_{\alpha}$, formed in the larger category
$\FrobAlg$ of all Dieudonn\'{e} algebras. It will therefore suffice to observe that
the forgetful functor $\FrobAlg \rightarrow \FrobComp$ (as a forgetful
functor from a category of algebras) preserves sifted colimits.
\end{proof} 

\begin{proof}[Proof of Theorem \ref{sddRWdRW}]
We wish to prove that the diagram of $\infty$-categories
$$ \xymatrix{ \SCRf \ar[r]^{\pi_0} \ar[d]^{L W\Omega_{\bullet}} & \CAlg_{ \F_p } \ar[r]^{ \WOmega^{\ast}_{\bullet} } & \FrobAlgComplete \ar[d] \\
\widehat{\calD(\Z)}^{L \eta_p}_{\mathrm{lax} } \ar[r]^{ \widehat{\Saturate} }&
\widehat{\calD(\Z)}^{L \eta_p} \ar[r]^-{\sim} & \FrobCompComplete }$$
commutes (up to natural isomorphism), where the lower right horizontal map is an inverse to the equivalence of Corollary \ref{1categorycor}.
Note that the horizontal maps in this diagram admit right adjoints, and therefore preserve all small colimits. The left vertical map preserves sifted
colimits by construction, and the right vertical map preserves sifted colimits by Lemma \ref{colimitsAlg}. Using the universal property
of $\SCRf$ (Proposition \ref{uprop}), we are reduced to showing that the analogous diagram commutes when we replace $\SCRf$ with by the full subcategory
$\CAlg_{\F_p}^{\mathrm{poly}} \subseteq \SCRf$. This follows from the observation that the canonical maps
$$ \widehat{\Saturate}( L W \Omega_{R} ) \leftarrow L W \Omega_{R} \rightarrow \WOmega^{\ast}_{R}$$
are isomorphisms when $R$ is a polynomial algebra over $\F_p$ (on the left because 
the divided Frobenius $W \Omega^{\ast}_{R} \rightarrow \eta_p W \Omega^{\ast}_{R}$ is an isomorphism, and
on the right because of Theorem \ref{maintheoC}).
\end{proof}

\begin{remark}
For any commutative $\F_p$-algebra $R$, the derived de~Rham--Witt complex $LW\Omega_R$ provides a lift to $\Z_p$ of the derived de~Rham complex $L\Omega_R$. In particular, if $R$ is not a local complete intersection, then $LW\Omega_R$ potentially has cohomology in infinitely many negative (cohomological) degrees; for instance, this holds if $R$ itself admits a lift to $\Z_p$ with a lift of Frobenius (by \cite[Theorem 1.5]{BhattTorsion}). By virtue of Remark \ref{labelremark}, we can identify the saturated de~Rham--Witt complex $\WOmega^{\ast}_{R}$ with the colimit of the diagram
\[ LW\Omega_R \xrightarrow{\alpha_R} L\eta_p(LW\Omega_R) \xrightarrow{L\eta_p(\alpha_R)} L\eta_{p^2}(LW\Omega_R) \to \cdots,\]
computed in the $\infty$-category $\ddhp$. It follows that the formation of this direct limit has the effect of killing the cohomology of $L W \Omega_{R}$ in negative degrees.
\end{remark}


Using the description of saturated de~Rham--Witt complexes via derived de~Rham
cohomology, we can prove an analog of the Berthelot--Ogus isogeny theorem \cite[Theorem 1.6]{BO83} in our setting. 

\begin{corollary}
\label{FrobSatdRWIsogeny}
Let $R$ be an $\mathbf{F}_p$-algebra. Then the endomorphism $\varphi^*\colon \W \Omega^*_R \to \W\Omega^*_R$ induced by the Frobenius endomorphism $\varphi\colon R \to R$ is a $p$-isogeny: that is,
it induces an isomorphism after inverting the prime $p$. Moreover, if the $R$-module $\Omega^1_R$ can be locally generated by $\leq d$ elements,  then multiplication by $p^d$ factors through $\varphi^*$.
\end{corollary}

The proof of Corollary \ref{FrobSatdRWIsogeny} will require a few preliminaries.

\begin{lemma}[Quillen's connectivity estimate] 
\label{Quillenlemma}
Let $R$ be a commutative ring, and let $M \in D(R)^{\leq -1}$. Then
$\bigwedge^i M \in D(R)^{\leq -i}$ for each $i \geq 0$.
\end{lemma} 
\begin{proof} 
See \cite[Cor. 7.40]{Quillennotes} for this as well as additional connectivity
assertions on derived symmetric and exterior powers. 
Alternatively, the result follows from 
Illusie's formula $\bigwedge^i (N[1]) \simeq (\Gamma^i N)[i]$ for $\Gamma^i$
the $i$th divided power functor, cf. \cite[Ch. I, Sec. 4.3.2]{Ill1} and the
fact that, by construction, $\Gamma^i$ is defined as a functor $D(R)^{\leq
0} \to D(R)^{\leq 0}$. 
\end{proof} 

\begin{lemma}\label{lemma:concentration}
Let $R$ be a commutative $\F_p$-algebra and suppose that the module of K\"{a}hler differentials $\Omega^{1}_{R/\F_p}$ is locally generated by $\leq d$ elements, for some $d \geq 0$.
Then:
\begin{itemize}
\item[$(1)$] The derived de~Rham complex $L \Omega_{R}$ belongs to $D( \F_p)^{\leq d}$ (that is, its cohomologies are concentrated in degrees $\leq d$).
\item[$(2)$] The derived de~Rham--Witt complex $LW\Omega_R$ belongs to $D( \Z )^{\leq d}$ (that is, its cohomologies are concentrated in degrees $\leq d$).
\item[$(3)$] The saturated de~Rham--Witt complex $\W\Omega^{\ast}_{R}$ is concentrated in degrees $\leq d$.
\end{itemize}
\end{lemma}

\begin{proof}
We first prove $(1)$. Using the conjugate filtration (Remark~\ref{ddr}), we are reduced to checking that the derived exterior powers $\bigwedge^i_{R} L_{R/\mathbf{F}_p}$ belong to $D(R)^{\leq d-i}$ for all $i \geq 0$. 
This assertion is local on $\Spec(R)$, so we may assume that there exists a sequence of elements $x_1, \ldots, x_d$ which generate $\Omega^{1}_{R/\F_p}$ as an $R$-module.
The $x_{i}$ then determine a map $f\colon R^{d} \rightarrow L_{ R/\F_p }$ in the derived category
$D(R)$. Let $K$ denote the cone of $f$, so that $K$ belongs to $D( R)^{\leq -1}$. Using the distinguished triangle
\[ R^{d} \to L_{R/ \F_p } \to K \to R^{d}[1], \]
we see that $\bigwedge^{i}_{R} L_{R/\F_p}$ admits a finite filtration 
(a left Kan extension of the usual filtration on the exterior power of the
middle term of a short exact sequence) whose successive quotients have the form
$\bigwedge^{a}_{R}( R^{d} ) \otimes_{R}^{L} \bigwedge^{i-a}_{R}( K )$ for $0
\leq a \leq i$. Note that each $\bigwedge^{a}_{R}(R^{d})$ is a projective
$R$-module of finite rank which vanishes for $a > d$. It will therefore suffice to show that $\bigwedge^{i-a}_{R}(K)$ belongs to $D(R)^{\leq d-i}$ for $a \leq d$, which
is a special case of Lemma~\ref{Quillenlemma}.

Using $(1)$ and induction on $n$, we deduce that each derived tensor product $(\Z / p^{n} \Z) \otimes^{L}_{\Z} LW\Omega_{R}$ is concentrated in degrees $\leq d$
and that the transition maps $\HH^{d}( (\Z / p^{n} \Z) \otimes^{L}_{\Z} LW\Omega_{R} ) \rightarrow \HH^{d}( (\Z / p^{n-1} \Z) \otimes^{L}_{\Z} LW\Omega_{R}$ are surjective.
Passing to the homotopy limit over $n$, we deduce that $LW \Omega_{R}$ belongs to $D(\Z)^{\leq d}$. This proves $(2)$.

Using Theorem \ref{sddRWdRW} (and Remark \ref{labelremark}), we see that as an object of the derived category $D(\Z)$, the
saturated de~Rham--Witt complex $\W\Omega^{\ast}_{R}$ is given by the $p$-completed homotopy direct limit of the diagram
\[ LW\Omega_R \xrightarrow{\alpha_R} L\eta_p(LW\Omega_R) \xrightarrow{L\eta_p(\alpha_R)} L\eta_{p^2}(LW\Omega_R) \to \cdots.\]
Since each of the functors $L\eta_{p^k}$ carries $D(\Z)^{\leq d}$ into itself, it follows from $(2)$ that the {\em cohomologies}
of the complex $\WOmega^{\ast}_{R}$ are concentrated in degrees $\leq d$. In particular, the cohomology groups
$\HH^{i}( \W \Omega_{R}^{\ast} / p^{r} )$ vanish for $i > d$ (and any nonnegative integer $r$). Using the isomorphism
$\W_{r} \Omega^{\ast}_{R} \simeq \HH^{\ast}( \W \Omega_{R}^{\ast} / p^{r} )$ of Proposition \ref{prop11}, we deduce
that each of the chain complexes $\W_r \Omega^{\ast}_{R}$ is concentrated in cohomological degrees $\leq d$.
Assertion $(3)$ now follows by passing to the inverse limit over $r$.
\end{proof}

\begin{proof}[Proof of Corollary \ref{FrobSatdRWIsogeny}]
For any Dieudonn\'e algebra $(A^*,d,F)$, let $\varphi_{A^*}\colon A^* \to A^*$ denote the map of cochain complexes determined by $\varphi_{A^*} = p^n F$ in degree $n$ for all $n$: that is, the composition $A^* \xrightarrow{\alpha_F} \eta_p A^* \xrightarrow{i_A} A^*$ where $i_A$ is the natural inclusion. If $A^*$ is a saturated Dieudonn\'e algebra, then $\varphi_{A^*}[1/p]$ is an isomorphism: the map $\alpha_F$ is an isomorphism because $A^{\ast}$ is saturated, and the map $i_A[1/p]$ is always an isomorphism. If, in addition, we have $A^i = 0$ for $i > d$, then $\varphi_{A^*}$ divides $p^d$: the subcomplex $p^d A^*$ is contained inside $i_A(\eta_p A^*) = \varphi_{A^*}(A^*) \subseteq A^*$. Applying the preceding observation in the case $A^* = \W\Omega^*_R$, we deduce that $\varphi^*$ is a $p$-isogeny (note that $\varphi_{\W\Omega^*_R} = \varphi^*$ by the universal property of the saturated de~Rham--Witt complex). We complete the proof of Corollary \ref{FrobSatdRWIsogeny} by observing that if the module of K\"{a}hler differentials $\Omega^{1}_{R/\F_p}$ is locally generated by
$d$ elements, then $\W\Omega^{\ast}_{R}$ is concentrated in cohomological degrees $\leq d$ by virtue of Lemma \ref{lemma:concentration}.
\end{proof}

\begin{remark}
In contrast with Corollary~\ref{FrobSatdRWIsogeny}, the  endomorphism  
\[ \varphi^*\colon \RGamma_{\crys}(\Spec(R)) \to \RGamma_{\crys}(\Spec(R))\]
induced by the Frobenius map $\varphi\colon R \to R$ is generally not an $p$-isogeny if $R$ is not regular. This phenomenon can already be seen when $R = \mathbf{F}_p[x]/(x^2)$ if $p$ is odd. 
Let $D$ denote the $p$-adically completed divided power envelope of the natural surjection $\mathbf{Z}_p[x] \to \mathbf{F}_p[x]/(x^2)$ (where the divided powers are required to be compatible with those on $(p)$),
given explicitly by the formula
\begin{align*}
D &\simeq \mathbf{Z}_p\left[x, \{\frac{x^{2k}}{k!}\}_{k \geq 1}\right]^{\wedge} \\
  &\simeq \big({\bigoplus_{m \geq 0}} \ \mathbf{Z}_p \cdot \frac{x^{2m}}{m!}\big)^\wedge \oplus \big({\bigoplus_{n \geq 0}} \ \mathbf{Z}_p \cdot \frac{x^{2n+1}}{n!} \big)^\wedge,\end{align*}
where the indicated completions are taken with respect to the $p$-adic topology. Standard results in crystalline cohomology
(see \cite[Theorem 7.23]{BO78}) allow us to identify $\RGamma_{\crys}(\Spec(R))$ with the de~Rham complex
\[ A^* := \Big(D \xrightarrow{d} D \cdot dx\Big). \]
Under this identification, the Frobenius endomorphism of $\RGamma_{\crys}(\Spec(R))$ is induced by the map $x \mapsto x^p$ of $D$. A simple calculation then gives an isomorphism
\[ \HH^1(A^*) \simeq \big({\bigoplus_{n \geq 0}} \ \mathbf{Z}_p/(2n+1) \cdot \frac{x^{2n}}{n!} \cdot dx\big)^{\wedge} \oplus \big({\bigoplus_{m \geq 0}} \ \mathbf{Z}_p/(2) \cdot \frac{x^{2m+1}}{m!} \cdot dx\big)^{\wedge}, \]
where again the indicated completions are with respect to the $p$-adic topology. When $p$ is odd, the second summand above vanishes, but the first contains non-torsion elements: for example, it contains the non-torsion $\mathbf{Z}_p$-module
\[ \big(\bigoplus_{k \geq 0} \ \mathbf{Z}_p/(p^k)\big)^\wedge \] as a direct summand. In particular, the localization $\HH^1_{\crys}(\Spec(R))[\frac{1}{p}]$ is nonzero. On the other hand, the Frobenius map $\varphi\colon \mathbf{F}_p[x]/(x^2) \to \mathbf{F}_p[x]/(x^2)$ factor through the inclusion map $\mathbf{F}_p \hookrightarrow \mathbf{F}_p[x]/(x^2)$, so the induced map $\varphi^*\colon  \HH^1_{\crys}(\Spec(R)) \to \HH^1_{\crys}(\Spec(R))$ vanishes. 
In particular, the map $\varphi^*[1/p]$ is not an isomorphism.
\end{remark}

\begin{remark}
Fix a perfect field $k$ of characteristic $p$, and let $K = W(k)[1/p]$. The construction $R \mapsto \W \Omega^{\ast}_{R}[ 1/ p ]$ determines a functor from category
of finite type $k$-algebras to the derived $\infty$-category $\calD( K )$ of
the field $K$. This presheaf is a sheaf for the \'etale topology
(Theorem~\ref{makeglobetale}) and carries universal homeomorphisms to
quasi-isomorphisms (this follows from Corollary~\ref{FrobSatdRWIsogeny} since universal
homeomorphisms between perfect $\F_p$-schemes are isomorphisms; see \cite[Lemma 3.8]{BhattScholzeProjectivity}). It is tempting to guess that this functor is a sheaf for the $h$-topology. However, we do not even know if it satisfies fppf descent.
\end{remark}

\subsection{Comparison with the de~Rham complex}

\newcommand{\cl}{\colon}
Let $R$ be an $\F_p$-algebra. Recall that the map $R \to \W_1 \Omega^0_R$
extends uniquely to a map 
\[ \nu \cl  \Omega_R^{\ast} \to \W_1 \Omega_R^{\ast} \]
of differential graded algebras. In \S \ref{dRComp}, we showed that $\nu$ is an isomorphism when
$R$ is a regular Noetherian $\F_p$-algebra (Theorem~\ref{theo73}).
Our strategy was to give a direct proof for smooth $\F_p$-algebras, and to extend the result to arbitrary regular Noetherian $\F_p$-algebras using Popescu's smoothing theorem.

In this section, we give a more general criterion for that $\nu$ to be an isomorphism (Theorem~\ref{critnuiso}). We will apply this criterion in \S\ref{dRRedux} to give a new proof that
$\nu$ is an isomorphism in the case where $R$ is regular and Noetherian (which avoids the use of Popescu's theorem, 
at the cost of using ``derived'' technology), and also in the case where $R$ admits a $p$-basis (Theorem \ref{theorem:p-basis}).

Recall that the  derived de~Rham complex $R \mapsto L \Omega_{R}$ is defined as the left Kan extension of the usual de~Rham complex
$R \mapsto \Omega_{R/\F_p}^{\ast}$, where the latter is considered as a functor from the category $\CAlg_{\F_p}^{\mathrm{poly}}$ to
the $\infty$-category $\calD( \F_p )$ (see Variant \ref{deriveddeRham}). In particular, for any $\F_p$-algebra $R$, we have a tautological comparison map
$\rho_{R} \colon L \Omega_{R} \rightarrow \Omega_{R/\F_p}^{\ast}$ in the derived $\infty$-category $\calD( \F_p )$, which is an isomorphism when $R$ is a polynomial ring (Remark \ref{ddr2}).

\begin{theorem}\label{critnuiso}
Let $R$ be a commutative $\mathbb{F}_p$-algebra. Suppose that: 
\begin{enumerate}
\item The comparison map $\rho_R \colon L \Omega_R \to \Omega_R^{\ast}$ is
an isomorphism in $\mathcal{D}( \F_p)$.  
\item The map $\Cart \colon \Omega_R^{\ast} \to \mathrm{H}^*( \Omega_R^{\ast})$ 
(of Proposition~\ref{cartmapexist})
is an
isomorphism. 
\end{enumerate}
Then the map $\nu \cl  \Omega_R^{\ast} \to \W_1 \Omega_R^{\ast}$ is an isomorphism. 
\end{theorem} 


The proof of Theorem \ref{critnuiso} will require some auxiliary constructions. First, we need to introduce a derived version of the Cartier map
of Proposition~\ref{cartmapexist}. 

\begin{construction}[Derived Cartier Operator]\label{construction:derived-cartier}
Let $R$ be a commutative $\F_p$-algebra with derived de~Rham complex $L \Omega_{R}$ and
derived de~Rham--Witt complex $LW\Omega_{R}$. The identification $L \Omega_{R} \simeq \F_p \otimes_{\Z}^{L} L W\Omega_{R}$
determines a Bockstein operator on the cohomology ring $\mathrm{H}^{\ast}( L \Omega_{R} )$, which we will denote by $\beta$.

When $R$ is a polynomial algebra over $\F_p$, we recall the Cartier map of
Proposition~\ref{cartmapexist},
which yields a map of differential graded algebras
$$\Cart \cl \Omega^{\ast}_R \to \mathrm{H}^{\ast}( L \Omega_R).$$
We claim that $\Cart$ intertwines de Rham
differential on the left-hand side with the Bockstein differential on the
right-hand side. To prove this, we can reduce to the case where $R = \F_p[x]$ is a polynomial ring on one generator.
In this case, $\Cart(x) \in H^0( L \Omega_R)$ is represented by the class
$[x^p]$, which is carried under the Bockstein differential to $[x^{p-1} dx] =
\Cart( dx)$. 

Returning to the case where $R$ is any polynomial algebra over $\F_p$, we obtain a morphism of differential graded algebras
$$ \Cart \cl ( \Omega^{\ast}_{R}, d) \rightarrow ( \mathrm{H}^{\ast}( L \Omega_{R}), \beta ).$$
Note that, when regarded as a functor (of ordinary categories) from commutative $\F_p$-algebras to differential graded algebras, the functor
$R \mapsto \Omega^{\ast}_{R}$ commutes with sifted colimits and is therefore a left Kan extension of its restriction to polynomial algebras.
It follows that for any commutative $\F_p$-algebra $R$, there is a canonical map of differential graded algebras
$$ \epsilon \cl  ( \Omega^{\ast}_{R}, d) \rightarrow ( \mathrm{H}^{\ast}( L \Omega_{R}), \beta ),$$
which is uniquely determined by the requirement that it coincides with the Cartier map when $R$ is a polynomial algebra, and depends functorially on $R$
(Remark~\ref{sebix}). Alternatively, one can construct $\epsilon$ by invoking the universal property of the de~Rham complex. 
\end{construction}

\begin{construction}[The map $\mu$]\label{construction:mu}
Let $R$ be a commutative $\F_p$-algebra. According to Remark~\ref{labelremark}, the saturated
de~Rham--Witt complex $\W \Omega_{R}^{\ast}$ can be identified, as an object of the derived $\infty$-category $\calD(\Z)$, with the $p$-completed
direct limit of the diagram
\begin{equation} \label{dRWfiltcolimit} L W \Omega_R \xrightarrow{\alpha_R} L \eta_p ( L W
\Omega_R ) \to L \eta_{p^2} ( L W
\Omega_R ) \to \dots ,    \end{equation}
In particular, we have a canonical map $L W \Omega_{R} \rightarrow \W \Omega_{R}^{\ast}$ in $\calD(\Z)$.
Reducing modulo $p$ and taking cohomology (using Proposition~\ref{prop11}), we obtain a map of differential
graded algebras
$$ \mu \cl (\mathrm{H}^*(L \Omega_R), \beta)  \to 
(\mathrm{H}^*( \W \Omega_R^\ast / p), \beta) \simeq
(\W_1 \Omega_R^\ast, d).$$ 
Note that, if $\epsilon \cl   ( \Omega^{\ast}_{R}, d) \rightarrow ( \mathrm{H}^{\ast}( L \Omega_{R}), \beta )$ is the map
of Construction \ref{construction:derived-cartier}, then the composition $\mu \circ \epsilon$ agrees
with the comparison map $\nu \cl  ( \Omega^{\ast}_{R}, d) \rightarrow ( \W_1 \Omega_{R}^{\ast}, d)$
introduced in \S \ref{dRComp} (this is immediate from the construction of $\epsilon$ in the case where $R$ is a polynomial ring,
and follows in general by Remark~\ref{sebix}).
\end{construction}

\begin{remark}\label{muaniso}
Let $R$ be a commutative $\F_p$-algebra and suppose that the map $\alpha_{R} \cl  L W \Omega_R \to L \eta_p  (L W \Omega_R)$ is an
isomorphism in $\mathcal{D}( \mathbb{Z})$. Then each of the transition maps in the direct system \eqref{dRWfiltcolimit} is an isomorphism,
so the comparison map $\mu \cl  L W \Omega_{R} \rightarrow \W \Omega^{\ast}_{R}$ is an isomorphism in the derived category.
\end{remark}

\begin{proposition}\label{proposition:diagram}
Let $R$ be a commutative $\F_p$-algebra. Then we have a commutative diagram
\begin{equation} \label{epsdiag} \xymatrix{
L W \Omega_R \ar[d]  \ar[r]^-{\alpha_R} & L\eta_p (L W \Omega_R ) \ar[d] \\
L \Omega_R \ar[r]^-{\epsilon \circ \rho_R} &  (\mathrm{H}^*( L \Omega_R), \beta )
},\end{equation}
in the derived $\infty$-category $\calD(\Z)$, where the bottom row is obtained by reducing the top row modulo $p$. 
\end{proposition} 

\begin{proof} 
When $R$ is a polynomial algebra over $\F_p$, we have a commutative diagram
\begin{equation} \label{epsdiag} \xymatrix{
W \Omega_R \ar[d]  \ar[r]^-{\alpha_R} & \eta_p (W \Omega_R ) \ar[d] \\
\Omega_R \ar[r] &  (\mathrm{H}^*( \Omega_R), \beta )
},\end{equation}
in the ordinary category of cochain complexes, which yields a commutative diagram in the $\infty$-category $\calD(\Z)$ depending functorially on
$R \in \CAlg_{\F_p}^{\mathrm{poly}}$. The general result follows by passing to left Kan extension along the inclusion
$\CAlg_{\F_p}^{\mathrm{poly}} \hookrightarrow \SCRf$ (see Remark~\ref{sebix}).
\end{proof} 

\begin{proof}[Proof of Theorem~\ref{critnuiso}] 
Let $R$ be a commutative $\F_p$-algebra for which the maps $\rho_R \cl L \Omega_R \to \Omega_R^{\ast}$
and $\Cart \cl  \Omega_R^{\ast} \to \mathrm{H}^*( \Omega_R^{\ast})$ are isomorphisms.
We wish to show that the map $\nu \cl  \Omega_R^{\ast} \to \W_1 \Omega_R^{\ast}$ is an isomorphism.
Let $\epsilon \cl  \Omega^{\ast}_{R} \rightarrow \mathrm{H}^{\ast}( L \Omega_R)$ be as in Construction \ref{construction:derived-cartier}.
By construction, the composition
$$ \Omega^{\ast}_{R} \xrightarrow{\epsilon} \mathrm{H}^{\ast}( L \Omega_R) \xrightarrow{ \mathrm{H}^{\ast}(\rho_{R}) } \mathrm{H}^{\ast}( \Omega_{R}^{\ast} )$$
coincides with the Cartier map $\Cart$ (this reduces to the case of
polynomial algebras). Since $\Cart$ and $\rho_{R}$ are isomorphisms, it follows that $\epsilon$ is an isomorphism (in the ordinary category of differential graded algebras).
Combining this with our assumption that $\rho_{R}$ is an isomorphism, we conclude that the composite map
$$ L \Omega_R \xrightarrow{ \rho_{R} } \Omega^{\ast}_{R} \xrightarrow{\epsilon}
( \mathrm{H}^*( L \Omega_R), \beta),$$
is an isomorphism in the derived $\infty$-category $\calD(R)$. Using Proposition \ref{proposition:diagram} (together with the observation that
$L W \Omega_R$ and $\eta_{p}( L W \Omega_{R} )$ are derived $p$-complete), we conclude that the map
$\alpha_{R} \cl  L W \Omega_R \to L \eta_p( L W \Omega_R)$ is an isomorphism in the derived $\infty$-category $\mathcal{D}(\Z)$. 
Applying Remark \ref{muaniso}, we see that comparison map $L W \Omega_{R} \rightarrow \W \Omega^{\ast}_{R}$ is an isomorphism in $\calD(\Z)$.
Reducing modulo $p$ and passing to cohomology, we conclude that the map  
$\mu \cl  (\mathrm{H}^*(L \Omega_R), \beta) \rightarrow (\W_1 \Omega_R^\ast, d)$ is an isomorphism of differential graded algebras,
so that the composition $\nu = \mu \circ \epsilon$ is also an isomorphism.
\end{proof} 

Let $R$ be a commutative ring. We will say that an object of the derived category $D(R)$ is {\it flat} if it is isomorphic to a flat $R$-module, regarded as a cochain complex concentrated in degree zero.
If $R$ is an $\F_p$-algebra for which the cotangent complex $L_{ R/ \F_p}$ is flat, then the first hypothesis of Theorem \ref{critnuiso} is automatic. In particular, we obtain a criterion
which does not reference the theory of derived de~Rham cohomology.

\begin{proposition}\label{ddRHodgeComplete}
Let $R$ be an $\F_p$-algebra. Suppose that the map $\Cart \cl  \Omega_R^{\ast} \to \mathrm{H}^*( \Omega_R^{\ast})$
is an isomorphism and that the cotangent complex $L_{R/\mathbb{F}_p}$ is flat. 
Then:
\begin{itemize}
\item[$(1)$] The comparison map $\rho_{R}\colon  L \Omega_{R} \rightarrow \Omega^{\ast}_{R}$ is an isomorphism in $\calD(\Z)$.
\item[$(2)$] The map $\nu \cl \Omega_R^{\ast} \to \W_1 \Omega_R^\ast$ is an
isomorphism of differential graded algebras.
\end{itemize}
\end{proposition}

\begin{proof}
We will prove $(1)$; assertion $(2)$ then follows from the criterion of Theorem \ref{critnuiso}.
For each $n \geq 0$, Remark \ref{sebix} supplies a comparison map
$$ \rho_{R}^{\leq n}\colon  \filc_n L \Omega_{R} \rightarrow \tau^{\leq n} \Omega^{\ast}_{R} $$
in the $\infty$-category $\calD(\Z)$, where the domain of $\rho_{R}^{n}$ is the $n$th stage of the
conjugate filtration of Remark \ref{ddr}. Note that $\rho_{R}$ can be identified with a filtered colimit
of the maps $\rho_{R}^{\leq n}$. It will therefore suffice to show that each $\rho_{R}^{\leq n}$ is
an isomorphism.

Proceeding by induction on $n$, we are reduced to showing that the induced map of filtration quotients
$$ \rho_{R}^{=n} \cl  ({\bigwedge}^{n} L_{R/\F_p})[-n] \rightarrow \mathrm{H}^{n}( \Omega^{\ast}_{R} ) [-n]$$
is an isomorphism in $\calD(\F_p)$; here $\bigwedge^{n}$ denotes the nonabelian derived functor of
the $n$th exterior power. The cotangent complex $L_{R/\mathbb{F}_p}$ is flat by assumption, so Lazard's theorem yields an isomorphism
$\bigwedge^{n} L_{R/\F_p} \simeq \Omega^{n}_{R}$. Under this identification, the map $\rho_{R}^{=n}$ corresponds
to the Cartier map of Proposition \ref{cartmapexist} (in degree $n$), which is
an isomorphism by assumption. 
\end{proof}

\begin{warning}
Assertion $(1)$ of Proposition \ref{ddRHodgeComplete} is not a formal consequence of the flatness of the cotangent complex $L_{R / \F_p}$. 
For example, let $R$ be a nonreduced $\F_p$-algebra for which the cotangent
complex $L_{ R/ \F_p}$ vanishes (Gabber has shown that such algebras exist;
see \cite{Bhatt_Gabber}). In this case, one can use the conjugate filtration of Remark \ref{ddr} to identify the comparison map $L \Omega_{R} \rightarrow \Omega^{\ast}_{R}$ with the Frobenius morphism $R^{(1)} \to R$, which is not an isomorphism when $R$ is not reduced.
\end{warning}

\begin{remark} 
It is known that the above hypotheses are satisfied for valuation rings over
$\F_p$, by results of Gabber--Romero and Gabber; see \cite[Appendix A]{KST} for
an account. Such rings are not known to be ind-smooth.  
\end{remark}

\subsection{The de~Rham comparison for regular $\F_p$-algebras, redux}
\label{sec9sub6}
\label{dRRedux}

The purpose of this subsection is to reprove Theorem~\ref{theo73}, which asserts that the
comparison map $\nu \cl \Omega_R^{\ast} \to \W_1 \Omega_R^{\ast}$ is an
isomorphism when $R$ is a regular Noetherian $\F_p$-algebra. We will prove this by showing
that $R$ satisfies the hypotheses of Proposition~\ref{ddRHodgeComplete}. We also use this
strategy to prove that $\nu$ is an isomorphism in the case where $R$ is a (possibly non-Noetherian) $\F_p$-algebra which admits a $p$-basis (see Definition \ref{definition:p-basis}).

\begin{proposition}
\label{RegCCFlat}
Let $R$ be a regular Noetherian $\F_p$-algebra. Then the cotangent complex $L_{R/\F_p}$ is flat.
\end{proposition}

If $R$ is smooth over $\mathbb{F}_p$, then $L_{R/\mathbb{F}_p}
\simeq \Omega^1_{R/\mathbb{F}_p}$ is flat. Consequently,
Proposition~\ref{RegCCFlat} follows from Popescu's theorem (which implies that every regular
Noetherian $\mathbb{F}_p$-algebra is a filtered colimit of smooth $\F_p$-algebras). However, we give a more direct argument below,
based on the following flatness criterion (see \cite[Prop. 5.3F]{AvFox} for a more general result):

\begin{lemma}
\label{FlatCrit}
Let $R$ be a Noetherian ring and fix $M \in D^{\leq 0}(R)$. Then $M$ is flat
if and only if $M \otimes_{R}^{L} k$ belongs to $D^{\geq 0}(R)$, for each residue field $k$ of $R$.
\end{lemma}

\begin{proof}
Assuming the condition on $M$, we will prove the following:
\begin{itemize}
\item[$(\ast)$] For every $R$-module $N$, the derived tensor product $M \otimes_{R}^{L} N$ belongs to the heart of $D(R)$.
\end{itemize}
Applying $(\ast)$ in the case $N = R$, we deduce that $M$ belongs to the heart of $D(R)$, in which case $(\ast)$ implies the flatness of $M$.
Since the construction $N \mapsto \mathrm{H}^{\ast}( M \otimes_{R}^{L} N )$ commutes with filtered colimits, it will suffice to prove
assertion $(\ast)$ in the case where $N$ is a finitely generated $R$-module. Proceeding by Noetherian induction, we may assume
that $(\ast)$ is satisfied for every finitely generated $R$-module $N'$ with $\mathrm{supp}(N') \subsetneq \mathrm{supp}(N)$ (as subsets of $\Spec(R)$).
Writing $N$ as a finite extension, we can reduce to the case where $N \simeq R / \mathfrak{p}$ for some prime ideal $\mathfrak{p} \subseteq R$. 
Let $k$ denote the fraction field of $R / \mathfrak{p}$. We have an exact sequence of $R$-modules
$$ 0 \rightarrow N \rightarrow k \rightarrow N' \rightarrow 0,$$
where $N'$ can be written as a filtered colimit of finitely generated $R$-modules whose support is a proper subset of $\mathrm{\supp}(N)$.
We therefore obtain a distinguished triangle
$$ M \otimes_{R}^{L} N \rightarrow M \otimes_{R}^{L} k \rightarrow M \otimes_{R}^{L} N' \rightarrow (M \otimes_{R}^{L} N)[1],$$
where $M \otimes_{R}^{L} N'$ belongs to $D^{\geq 0}(R)$ (by our inductive hypothesis), and
$M \otimes_{R}^{L} k$ belongs to $D^{\geq 0}(R)$ (by assumption). It follows that $M \otimes_{R}^{L} N$ also belongs to
$D^{\geq 0}(R)$. Since $D^{\leq 0}(R)$ is closed under derived tensor products, the object
$M \otimes_{R}^{L} N$ belongs to the heart $D^{\leq 0}(R) \cap D^{\geq 0}(R)$ of $D(R)$.
\end{proof}

For later reference, we record the following consequence of Lemma \ref{FlatCrit}:

\begin{corollary}\label{skez}
Let $R$ be a Noetherian ring and let $f\colon  M \to N$ be a homomorphism of flat $R$-modules. Assume that,
for each residue field $k$ of $R$, the map $f_{k}\colon  M \otimes_{R} k \rightarrow N \otimes_{R} k$ is a monomorphism.
Then $f$ is a monomorphism and the cokernel $\coker(f)$ is flat. If each $f_{k}$ is an isomorphism, then $f$ is an isomorphism.
\end{corollary}

\begin{proof}
The first assertion follows by applying Lemma \ref{FlatCrit} to the cone $\mathrm{cn}(f)$ of $f$ (formed in the derived category $D(R)$), and
the second assertion from applying Lemma \ref{FlatCrit} to the shift $\mathrm{cn}(f)[1]$.
\end{proof}

\begin{proof}[Proof of Proposition \ref{RegCCFlat}]
Let $R$ be a regular Noetherian $\F_p$-algebra. By virtue of Lemma~\ref{FlatCrit}, it will suffice to show that
$L_{R/\F_p} \otimes_{R}^{L} k$ belongs to $D^{\geq 0}(R)$ for each residue field $k$ of $R$. 
As the formation of the cotangent complex commutes with localization, we may assume $R$ is a regular local ring with maximal ideal $\mathfrak{m} \subseteq R$ and residue field
$k \simeq R / \mathfrak{m}$. Writing the field $k$ as a filtered colimit of
smooth $\F_p$-algebras (using generic smoothness applied to any finitely generated extension of the perfect field $\F_p$), we see that
$L_{k/\F_p}$ belongs to the heart of $D(k)$. Since $R$ is regular, the kernel of the map $R \to k$ is generated by a regular sequence, so that
$L_{k/R}[-1] \simeq \mathfrak{m}/\mathfrak{m}^2$ also belongs to the heart of $D(k)$. 
The transitivity triangle for $\F_p \to R \to k$ gives a distinguished triangle \[ L_{k/R}[-1] \to L_{R/\F_p} \otimes_R^L k \to L_{k/\F_p} \to L_{k/R} .\]
which shows that $L_{R/\F_p} \otimes_{R}^{L} k$ also belongs to the heart of $D(k)$.
\end{proof}

\begin{corollary}\label{saget}
Let $R$ be a regular Noetherian $\F_p$-algebra. Then, for each $n \geq 0$, the module $\Omega^{n}_{R}$ is flat over $R$.
\end{corollary}

\begin{proof}
For $n = 1$, this follows from Proposition \ref{RegCCFlat}. The general case then follows from the definition
$\Omega^{n}_{R} = \bigwedge^{n} ( \Omega^{1}_{R} )$ and Lazard's theorem.
\end{proof}

\begin{corollary}
\label{CCComp}
Let $R$ be a regular Noetherian $\F_p$-algebra. Assume that $R$ is local with maximal ideal $\mathfrak{m}$, and let $\widehat{R}$ denote the completion of $R$.
Let $N$ be an $R$-module on which the action of $\mathfrak{m}$ is locally
nilpotent, i.e., every element is annihilated by a power of $\mathfrak{m}$ (so that $N$ can also be regarded as a $\widehat{R}$-module). 
Then the canonical map
$$ \rho_m\colon  \Omega^{m}_{R} \otimes_{R} N \rightarrow \Omega^{m}_{ \widehat{R} } \otimes_{ \widehat{R} } N$$
is an isomorphism for each $m \geq 0$.
\end{corollary}

\begin{proof}
Writing $N$ as a union of finitely generated submodules, we can reduce to the case where $N$ is finitely generated and therefore Artinian. Proceeding by induction
on the length of $N$, we can further reduce to the case where $N = k$ is the residue field of $R$. In this case, we can identify $\rho_m$ with the 
$m$th exterior power of $\rho_1$ (in the category of vector spaces over $k$). Using Proposition \ref{RegCCFlat}, we can identify
$\rho_1$ with the canonical map $L_{R/ \F_p} \otimes_{R}^{L} k \rightarrow L_{
\widehat{R} / \F_p} \otimes_{ \widehat{R} } k$.
Using the commutative diagram of transitivity sequences
$$ \xymatrix{ L_{R / \F_p} \otimes_{R}^{L} k \ar[d]^{\rho_1} \ar[r] & L_{k/\F_p} \ar[d]^{\id} \ar[r] & L_{ k / R} \ar[d]^{\rho'} \\
L_{ \widehat{R} / \F_p} \otimes_{ \widehat{R} }^{L} k \ar[r] & L_{ k / \F_p} \ar[r] & L_{k/ \widehat{R} }, }$$
we are reduced to showing that the map $\rho'$ is an isomorphism. This follows from the observation that the map
$R \to \widehat{R}$ induces an isomorphism $\mathfrak{m} / \mathfrak{m}^2 \simeq \widehat{\mathfrak{m} } / \widehat{ \mathfrak{m} }^2$, where
$\widehat{ \mathfrak{m} }$ denotes the maximal ideal of $\widehat{R}$.
\end{proof}


Let $R$ be a commutative $\F_p$-algebra and let $\Cart\colon  \Omega_{R}^{\ast} \rightarrow \mathrm{H}^{\ast}( \Omega^{\ast}_{R} )$
be the Cartier map (Proposition \ref{cartmapexist}). It follows from Theorem \ref{theo71} that
$\Cart$ is an isomorphism when $R$ is a smooth algebra over a perfect field $k$. Our next goal is to prove the following
more general result:

\begin{theorem}\label{avoidpop}
Let $R$ be a regular Noetherian $\F_p$-algebra. Then the Cartier map
$$ \Cart\colon  \Omega_{R}^{\ast} \rightarrow \mathrm{H}^{\ast}( \Omega^{\ast}_{R} )$$
is an isomorphism of graded rings.
\end{theorem}

\begin{remark}
As with Proposition~\ref{RegCCFlat}, Theorem \ref{avoidpop} is an immediate consequence of Popescu's theorem (together with the classical Cartier isomorphism).
However, we will give a more direct argument which avoids the use of Popescu's theorem.
\end{remark}

The proof of Theorem \ref{avoidpop} will require some preliminaries.

\begin{notation}
Let $R$ be an $\F_p$-algebra. To avoid confusion, we let $R^{(1)}$ denote the same commutative ring $R$, where
we view $R$ as an $R^{(1)}$-module via the Frobenius map $R^{(1)} \to R$.

For each $n \geq 0$, we let $B^{n} \Omega^{\ast}_{R}$ denote the $n$-coboundaries in the de~Rham complex
$\Omega^{\ast}_{R}$: that is, the image of the de~Rham differential $d\colon  \Omega^{n-1}_{R} \rightarrow \Omega^{n}_{R}$ (by convention,
we set $B^{0} \Omega^{\ast}_{R} = (0)$). We will regard the cohomology group $\mathrm{H}^{n}( \Omega^{\ast}_{R} )$
as a subgroup of the quotient $\Omega^{n}_{R} / B^{n} \Omega^{\ast}_{R}$, and we define
$Q^{n+1} \Omega^{\ast}_{R}$ to be the cokernel of the composite map
$$ \Omega^{n}_{ R^{(1)} } \xrightarrow{\Cart} \mathrm{H}^{n}( \Omega^{\ast}_{R} ) \hookrightarrow
\Omega^{n}_{R} / B^{n} \Omega^{\ast}_{R}.$$
By construction, we have short exact sequences of $R^{(1)}$-modules

\begin{equation}
\label{CC1smooth}
\Omega^{n}_{R^{(1)} } \xrightarrow{ \Cart} \Omega^{n}_{R} / B^{n} \Omega^{\ast}_R
\rightarrow Q^{n+1} \Omega^{\ast}_{R} \rightarrow 0 ,\end{equation}

\begin{equation}
\label{CC2smooth}
Q^{n+1} \Omega^{\ast}_{ R } \xrightarrow{d} \Omega^{n+1}_{R} \rightarrow \Omega^{n+1}_{R} / B^{n+1} \Omega^{\ast}_{R} \rightarrow 0.
\end{equation}

Note that the sequence (\ref{CC1smooth}) is exact on the left if and only if
the Cartier map $\Cart\colon  \Omega^{n}_{ R^{(1) } } \to \mathrm{H}^{n}( \Omega^{\ast}_{R} )$ is injective,
and that the sequence (\ref{CC2smooth}) is exact on the left if and only if
the Cartier map $\Cart\colon  \Omega^{n}_{R^{(1)}} \to \mathrm{H}^{n}( \Omega^{\ast}_{R} )$ is surjective
(in which case the de~Rham differential induces an isomorphism $Q^{n+1} \Omega^{\ast}_{R} \simeq B^{n+1} \Omega^{\ast}_{R}$).
\end{notation}

\begin{definition}\label{universalcartier}
Let $R$ be an $\F_p$-algebra. We will say that $R$ {\it has a universal Cartier isomorphism} 
if it satisfies the following pair of conditions, for every integer $n \geq 0$:
\begin{itemize}
\item[$(a_n)$] The sequence
$$0 \rightarrow \Omega^{n}_{R^{(1)} } \xrightarrow{ \Cart} \Omega^{n}_{R} / B^{n} \Omega^{\ast}_R
\rightarrow Q^{n+1} \Omega^{\ast}_{R} \rightarrow 0$$
is exact. Moreover, $Q^{n+1} \Omega^{\ast}_{R}$ is a flat $R^{(1)}$-module.

\item[$(b_n)$] The sequence 
$$ 0 \rightarrow Q^{n+1} \Omega^{\ast}_{R} \xrightarrow{d} \Omega^{n+1}_{R} \rightarrow \Omega^{n+1}_{R} / B^{n+1} \Omega^{\ast}_{R}  \rightarrow 0$$
is exact. Moreover, the quotient $\Omega^{n+1}_{R} / B^{n+1} \Omega^{\ast}_{R}$ is a flat $R^{(1)}$-module.
\end{itemize}
\end{definition}

\begin{example}\label{polynomialring}
Let $k$ be a perfect ring of characteristic $p$ and let $R = k[x_1, \ldots, x_m]$ be a polynomial algebra over $k$. Then $R$ has a universal Cartier isomorphism.
In fact, we know that $R$ has a Cartier isomorphism, so the desired sequences
\eqref{CC1smooth} and \eqref{CC2smooth} are exact. 
It will therefore suffice to prove the following assertion for each $n \geq 0$:
\begin{itemize}
\item[$(S_n)$] The groups $Q^{n+1} \Omega_R^{\ast}$ and $\Omega^{n+1}_R/B^{n+1} \Omega R^{\ast}$ are flat
$R^{(1)}$-modules.
\end{itemize}
Note that $(S_n)$ is automatic for $n \gg 0$ (since the modules $Q^{n+1}
\Omega_R^{\ast}$ and $\Omega^{n+1}_R/B^{n+1} \Omega_R^{\ast}$ both vanish).
It will therefore suffice to show that if condition $(S_{n+1})$ is satisfied for some $n \geq 0$, then $(S_n)$ is also satisfied. This is clear:
the exact sequence of \eqref{CC1smooth} guarantees that $\Omega_R^{n} /B^n
\Omega_R^{\ast}$ is a flat $R^{(1)}$-module, and the exact sequence of \eqref{CC2smooth} guarantees that $Q^{n} \Omega_{R}^{\ast}$
is a flat $R^{(1)}$-module.
\end{example}

\begin{proposition} 
\label{transferunivcart}
Let $R \to R'$ be a morphism of $\F_p$-algebras. Suppose that: 
\begin{enumerate}
\item $R$ has a universal Cartier isomorphism.  
\item The diagram
of $\F_p$-algebras
\[ \xymatrix{
R^{(1)} \ar[d]  \ar[r] &  R  \ar[d]  \\
R'^{(1)} \ar[r] &  R'
}\]
is a pushout square. 
\end{enumerate}
Then $R'$ has a universal Cartier isomorphism. 
\end{proposition} 
\begin{proof} 
The identifications $\Omega^1_R = \Omega^1_{R/R^{(1)}}$ and 
$\Omega^1_{R'} = \Omega^1_{R'/R'^{(1)}}$ yield an isomorphism $ \Omega_{R'}^1 \simeq \Omega_R^1 \otimes_{R^{(1)}} R'^{(1)}$, which extends
to an isomorphism of commutative differential graded algebras
\[ \Omega_{R'}^{\ast} \simeq\Omega^{\ast}_R \otimes_{R^{(1)}} R'^{(1)}  .\]
Since $Q^{n+1} \Omega^{\ast}_{R}$ and $\Omega^{n+1}_{R} / B^{n+1} \Omega^{\ast}_{R}$ are flat over
$R^{(1)}$, the exact sequences $(a_n)$ and $(b_n)$ of Definition~\ref{universalcartier} 
remain exact after extending scalars along the map $R^{(1)} \to R'^{(1)}$, yielding analogous sequences for
the de Rham complex $\Omega_{R'}^{\ast}$.
\end{proof}

\begin{corollary}\label{effis}
Let $k$ be a perfect ring of characteristic $p$ and let $R$ be a smooth $k$-algebra. Then $R$ has a universal
Cartier isomorphism.
\end{corollary}

\begin{proof}
The assertion is Zariski local on $\Spec(R)$, so we may assume without loss of generality that there exists
an \'{e}tale ring homomorphism $k[x_1, \ldots, x_n] \rightarrow R$. In this case, the desired result follows from
Example \ref{polynomialring} and Proposition \ref{transferunivcart}.
\end{proof}


\begin{corollary}\label{effiscor}
Let $k$ be a field of characteristic $p$. Then the polynomial ring $k[t_1, \ldots, t_d]$ has a universal Cartier isomorphism.
\end{corollary}

\begin{proof}
Writing $k$ as a union of finitely generated subfields, we can reduce to the case where $k$ is finitely generated over $\F_p$.
In this case, $k[t_1, \ldots, t_d]$ is a localization of a smooth $\F_p$-algebra, so the desired result follows from Corollary \ref{effis}.
\end{proof}

In good cases, the conditions of Definition \ref{universalcartier} can be tested pointwise.

\begin{definition}\label{universalcartieratx}
Let $R$ be an $\F_p$-algebra and let $k$ denote the residue field of $R$ at a point $x \in \Spec(R)$.
We will say that $R$ {\it has a universal Cartier isomorphism at $x$} if it satisfies the following
pair of conditions, for every integer $n \geq 0$:
\begin{itemize}
\item[$(a_{n,x})$] The sequence of vector spaces
$$0 \rightarrow \Omega^{n}_{R^{(1)} } \otimes_{ R^{(1)} } k^{(1)}
\xrightarrow{ \Cart} (\Omega^{n}_{R} / B^{n} \Omega^{\ast}_R) \otimes_{ R^{(1)}} k^{(1)}
\rightarrow Q^{n+1} \Omega^{\ast}_{R} \otimes_{ R^{(1)}} k^{(1)} \rightarrow 0$$
is exact. 

\item[$(b_{n,x})$] The sequence of vector spaces
$$ 0 \rightarrow (Q^{n+1} \Omega^{\ast}_{R} ) \otimes_{ R^{(1)}} k^{(1)} \xrightarrow{d}
 \Omega^{n+1}_{R} \otimes_{ R^{(1)}} k^{(1)}  \rightarrow (\Omega^{n+1}_{R} / B^{n+1} \Omega^{\ast}_{R}) \otimes_{ R^{(1)} } k^{(1)}  \rightarrow 0$$
is exact.
\end{itemize}
\end{definition}

\begin{proposition}\label{kanster}
Let $R$ be a regular Noetherian $\F_p$-algebra. Then $R$ has a universal Cartier isomorphism if and only if it has a universal Cartier isomorphism at
$x$, for each $x \in \Spec(R)$.
\end{proposition}

\begin{proof}
The ``only if'' direction is immediate. For the converse, suppose that $R$ has a universal Cartier isomorphism at $x$, for each $x \in \Spec(R)$. We will show that $R$ satisfies
conditions $(a_n)$ and $(b_n)$ of Definition \ref{universalcartier}. The proof proceeds by induction on $n$. To prove $(a_n)$, we note that
the inductive hypothesis (together with Corollary \ref{saget}) implies that the Cartier map $\Cart\colon  \Omega^{n}_{R^{(1)}} \rightarrow \Omega^{n}_{R} / B^{n} \Omega^{n}_{R}$
can be regarded as a morphism of flat $R^{(1)}$-modules. Consequently, assertion $(a_n)$ follows from assertions $(a_{n,x})$ for $x \in \Spec(R)$
by virtue of Corollary \ref{skez}. To prove $(b_n)$, we observe that $\Omega^{n+1}_{R}$ is flat as an $R$-module (Corollary \ref{saget}), hence also
as an $R^{(1)}$-module (since the regularity of $R$ guarantees that the
Frobenius morphism $R^{(1)} \to R$ is flat; see \cite{Kunz69}). Using $(a_n)$, we see that the map
$Q^{n+1} \Omega^{\ast}_{R} \rightarrow \Omega^{n+1}_{R}$ is a morphism of flat $R^{(1)}$-modules,
so that assertion $(b_n)$ follows from $(b_{n,x} )$ for $x \in \Spec(R)$ (again by virtue of Corollary \ref{skez}).
\end{proof}

\begin{remark}\label{slip1}
Let $R$ be a commutative $\F_p$-algebra and let $x$ be a point of $\Spec(R)$, corresponding to a prime ideal $\mathfrak{p} \subseteq R$.
Then $R$ has a universal Cartier isomorphism at $x$ if and only if the local ring $R_{\mathfrak{p}}$ has a universal Cartier isomorphism
at $x$ (where we abuse notation by identifying $x$ with the closed point of $\Spec( R_{\mathfrak{p}} )$).
\end{remark}

\begin{remark}\label{slip2}
Let $R$ be a regular Noetherian local $\F_p$-algebra and let $\widehat{R}$ be its completion. Let $x$ denote the closed point of $\Spec( \widehat{R} )$, which
we identify with its image in $\Spec(R)$. Using Corollary \ref{CCComp}, we see
that if $M$ is an $R^{(1)}$-module on which the action of the maximal ideal $\mathfrak{m}^{(1)} \subset R^{(1)}$ is locally nilpotent, 
then $\Omega_R^m \otimes_{R^{(1)}} M \simeq \Omega^m_{\widehat{R}}
\otimes_{\widehat{R}^{(1)}} M$. 
Therefore, $R$ has a universal Cartier isomorphism at $x$ if and only if $\widehat{R}$ has a universal Cartier isomorphism
at $x$.
\end{remark}

Theorem \ref{avoidpop} is a consequence of the following slightly more refined statement:

\begin{theorem}\label{sep}
Let $R$ be a regular Noetherian $\F_p$-algebra. Then $R$ has a universal Cartier isomorphism.
\end{theorem}

\begin{proof}
By virtue of Proposition \ref{kanster}, it will suffice to show that $R$ has a universal Cartier isomorphism at
$x$, for each point $x \in \Spec(R)$. Using Remark \ref{slip1}, we can reduce to the case where $R$ is local
and $x$ is the closed point of $\Spec(R)$. By virtue of Remark \ref{slip2}, we can replace $R$ by its completion
$\widehat{R}$. Let $k$ denote the residue field of $R$, so that the Cohen structure theorem supplies an isomorphism
$R \simeq k \llbracket t_1,..,t_d \rrbracket$. Set $S = k[t_1, \ldots, t_d]$, and let $y \in \Spec(S)$ be
the point corresponding to the maximal ideal $(t_1, \ldots, t_d) \subseteq S$. Using Remarks \ref{slip1} and \ref{slip2}
again, we see that $R$ has a universal Cartier isomorphism at $x$ if and only if $S$ has a universal Cartier isomorphism at $y$.
Using Proposition \ref{kanster}, we are reduced to showing that $S$ has a universal Cartier isomorphism, which follows
from Corollary \ref{effiscor}.
\end{proof}

We can now reprove Theorem~\ref{theo73}. 

\begin{corollary} 
\label{compmain}
Let $R$ be a regular Noetherian $\F_p$-algebra. Then the map $\nu \cl
\Omega_R^{\ast} \to \W_1 \Omega_R^\ast$ is an isomorphism. 
\end{corollary} 

\begin{proof} 
Combine Proposition~\ref{ddRHodgeComplete}, Proposition~\ref{RegCCFlat}, and
Theorem~\ref{avoidpop}. 
\end{proof} 

We conclude this section by applying Proposition~\ref{ddRHodgeComplete} to another class of $\F_p$-algebras.

\begin{definition}\label{definition:p-basis}
Let $R$ be a commutative $\F_p$-algebra. We say that a collection of elements $\{ x_i \}_{i \in I}$ of $R$
is a {\it $p$-basis} if the products $\prod_{i \in I} x_i^{d_i}$ freely generate $R$ as a module over $R^{(1)}$,
where $\{ d_i \}_{i \in I}$ ranges over all functions $I \rightarrow \{0, 1, \ldots, p-1\}$ which vanish on all but finitely many elements of $I$.
\end{definition} 

\begin{theorem}\label{theorem:p-basis}
Let $R$ be an $\F_p$-algebra which admits a $p$-basis. Then the map $\nu \cl 
\Omega_R^{\ast} \to \W_1 \Omega_R^\ast$ is an isomorphism. 
\end{theorem} 

\begin{proof} 
Let $\{ x_i \}_{i \in I}$ be a collection of elements of $R$, classified by a ring homomorphism
$f \colon \F_p[ \{ X_i \}_{i \in I} ] \rightarrow R$. Writing $A = \F_p[ \{ X_i \}_{i \in I} ]$ as a filtered colimit of finitely generated polynomial rings, we
see that $A$ has a universal Cartier isomorphism (Example \ref{polynomialring}). If $\{ x_i \}_{i \in I}$ is a $p$-basis for $R$, then
the homomorphism $f$ satisfies the hypotheses of Proposition \ref{transferunivcart}. It follows that $R$ has a universal Cartier isomorphism.
Moreover, the cotangent complex $L_{R/\F_p}$ is a flat object of $D(R)$ by \cite[Lemma 1.1.2]{deJong}. Applying the criterion of Proposition~\ref{ddRHodgeComplete}, we deduce
that the map $\nu \colon \Omega^{\ast}_{R} \to \W_1 \Omega^{\ast}_{R}$ is an isomorphism.
\end{proof}

\newpage
\section{Comparison with Crystalline Cohomology}
\label{dRtocryscomp:sec}
\subsection{Introduction}
Let $k$ be a perfect field of characteristic $p >0$ and let $X$ be a smooth $k$-scheme. We consider the following two invariants of $X$: 

\begin{enumerate}
\item The de Rham--Witt cohomology $\mathrm{H}^{\ast}(X, W \Omega_X^{\ast})$, defined as the cohomology of $X$ with coefficients in the de Rham--Witt complex $W \Omega_{X}^{\ast}$.
By virtue of Theorem~\ref{maintheoC}, we can replace $W \Omega_{X}^{\ast}$ with our saturated de Rham--Witt complex $\WOmega^{\ast}_{X}$.

\item The crystalline cohomology $\mathrm{H}^{\ast}_{\crys}(X)$, defined as the cohomology of the structure sheaf on the crystalline site of $X$ (see \cite{Berthelot74, BO78}) relative to 
the divided power ring $(\Z_{p}, (p) )$.
\end{enumerate}

These cohomology rings are canonically isomorphic, by virtue of the following result of Illusie (\cite[Th{\'e}or{\`e}me II.1.4]{illusie}):

\begin{theorem}[The de Rham--Witt to crystalline comparison]
\label{dRWtocrys}
Let $k$ be a perfect field and let $X$ be a smooth $k$-scheme. Then there is a canonical isomorphism of cohomology rings
$\mathrm{H}^{\ast}_{\crys}(X) \simeq \mathrm{H}^{\ast}(X, W\Omega_{X}^{\ast})$, depending functorially on $X$.
\end{theorem} 

Our goal in this section is to give a new proof of Theorem~\ref{dRWtocrys}. Our argument will not make use of any details of the construction of the (saturated) de Rham-Witt complex $\WOmega^{\ast}_{X}$.
Instead, we will deduce Theorem \ref{dRWtocrys} formally from the fact that there is a comparison map $\WOmega^{\ast}_{X} \rightarrow \Omega^{\ast}_{X}$ (depending functorially on $X$)
which induces a quasi-isomorphism $\WOmega^{\ast}_{X} / p \WOmega^{\ast}_{X} \rightarrow \Omega^{\ast}_{X}$ (see Remark \ref{plux}). 

To carry out the details, it will be convenient to restrict our attention to affine schemes. Let $\CAlg_{\F_p}^{\mathrm{reg}}$ denote the category of regular Noetherian $\F_p$-algebras
and let $\Ab$ denote the category of abelian groups. We regard the construction $R \mapsto \WOmega^{\ast}_{R}$ as a chain complex taking values in the abelian category of functors
$\Fun( \CAlg_{\F_p}^{\mathrm{reg}}, \Ab)$. We denote this chain complex of functors by $\WOmega^{\ast}_{(-)}$, and we write $\WOmega_{(-)}$ for its image in the derived category
$D( \Fun( \CAlg_{\F_p}^{\mathrm{reg}}, \Ab) )$. Note that $D( \Fun( \CAlg_{\F_p}^{\mathrm{reg}}, \Ab) )$ has a symmetric monoidal structure (given by the left derived tensor product $\otimes^{L}$),
and that $\WOmega_{(-)}$ can be regarded as a commutative algebra object of $D(
\Fun( \CAlg_{\F_p}^{\mathrm{reg}}, \Ab) )$ (inherited from the multiplication
on the saturated de Rham--Witt complex
$\WOmega_{R}^{\ast}$ for $R \in \CAlg_{\F_p}^{\mathrm{reg} }$). Similarly, the construction of crystalline cohomology determines a commutative algebra object
$R \Gamma_{\crys}( - )$ of the derived category $D( \Fun(
\CAlg_{\F_p}^{\mathrm{reg}}, \Ab) )$. Both of these commutative algebra objects have the property that, after reduction modulo $p$,
they are equivalent to the de Rham complex functor $R \mapsto \Omega_{R}$, regarded as a commutative algebra object of the derived category $D( \Fun( \CAlg_{\F_p}^{\mathrm{reg}}, \Vect_{\F_p} ) )$.
To show that these functors are equivalent, it will suffice to prove the following:

\begin{theorem}\label{theorem:uniqueness-of-crys}
Let $A(-)$ be a commutative algebra object of the derived category $D( \Fun(
\CAlg_{\F_p}^{\mathrm{reg}}, \Ab))$ and let $u_0 \colon \Omega_{(-)} \xrightarrow{\sim} \F_p \otimes^{L} A(-)$
be an isomorphism of commutative algebra objects of the derived category $D( \Fun( \CAlg_{\F_p}^{\mathrm{reg}}, \Vect_{\F_p} ) )$. If $A(-)$ is $p$-complete,
then $u_0$ can be lifted uniquely to an isomorphism $u \colon R \Gamma_{\crys}(-)
\xrightarrow{\sim} A(-)$ of commutative algebra objects of $D( \Fun(
\CAlg_{\F_p}^{\mathrm{reg}}, \Ab))$.
\end{theorem}

\begin{remark}
The derived category $D( \Fun( \CAlg_{\F_p}^{\mathrm{reg}}, \Ab) )$ can be regarded as the homotopy category of the $\infty$-category
$$\calD( \Fun( \CAlg_{\F_p}^{\mathrm{reg}}, \Ab) ) \simeq \Fun( \CAlg_{\F_p}^{\mathrm{reg}}, \calD( \Z) )$$ of functors from the ordinary category
$\CAlg_{\F_p}^{\mathrm{reg}}$ to the derived $\infty$-category $\calD( \Z )$ of abelian groups. In what follows, we will often invoke this identification implicitly.
\end{remark}

\begin{proof}[Proof of Theorem \ref{dRWtocrys} from Theorem \ref{theorem:uniqueness-of-crys}]
For every regular Noetherian $\F_p$-algebra $R$, Remark \ref{plux} supplies a quasi-isomorphism of chain complexes
$$\WOmega^{\ast}_{R} / p \WOmega^{\ast}_{R} \rightarrow \Omega^{\ast}_{R}.$$ These quasi-isomorphisms depend functorially on $R$, and
therefore the inverse maps define an isomorphism $u_0 \colon \Omega_{(-)} \rightarrow \F_p \otimes^{L} \WOmega_{(-)}$ in the derived category
$D( \Fun( \CAlg_{\F_p}^{\mathrm{reg}}, \Vect_{\F_p} ) )$. By virtue of Theorem \ref{theorem:uniqueness-of-crys}, the isomorphism $u_0$
lifts uniquely to an isomorphism $u \colon R \Gamma_{\crys}(-) \simeq \WOmega_{(-)}$ in the derived category
$D( \Fun( \CAlg_{\F_p}^{\mathrm{reg}}, \Ab) )$.

Let us identify $u$ with an isomorphism in the $\infty$-category $\Fun( \CAlg_{\F_p}^{\mathrm{reg}}, \calD( \Z) )$. For any regular $\F_p$-scheme,
we have canonical isomorphisms
\begin{eqnarray*}
R \Gamma_{\crys}(X) & \simeq & \varprojlim_{U} R \Gamma_{\crys}(U) \\
& \simeq & \varprojlim_{U} \WOmega_{\mathcal{O}_{X}(U)} \\
& \simeq & R \Gamma(X, \WOmega_{X}^{\ast} ),
\end{eqnarray*}
where the limits are indexed by the category of affine open subsets $U \subseteq X$ and formed in the $\infty$-category $\calD( \Z )$.
Passing to cohomology rings, we obtain the desired isomorphism $\mathrm{H}^{\ast}_{\crys}(X) \simeq \mathrm{H}^{\ast}(X, W\Omega_{X}^{\ast})$.
\end{proof}

We will carry out the proof of Theorem \ref{theorem:uniqueness-of-crys} in three steps. The first is to produce the natural transformation
$u \colon R \Gamma_{\crys}(-) \rightarrow A(-)$ which appears in the statement of Theorem \ref{theorem:uniqueness-of-crys}. 
Our construction is based on the syntomic ideas of \cite{FontaineMessing}, \cite{FontaineJannsen} and \cite{BMS2}, and will be given in
in \S \ref{subsection:make-comparison-map}. In \S \ref{subsection:check-comparison-iso}, we prove that $u$ is an isomorphism by showing that it is compatible
with the given isomorphism $u_0 \colon \Omega_{(-)} \xrightarrow{\sim} \F_p \otimes^{L} A(-)$, by virtue of a suitable rigidity property of the de Rham functor
$R \mapsto \Omega_{R}$ (Proposition \ref{proposition:check-comparison-iso}). In \S \ref{subsection:unique-comparison}, we show that the isomorphism $u$ is uniquely determined. Note that, for the purpose of proving Theorem~\ref{dRWtocrys}, this last step is not necessary.

\subsection{Construction of the Comparison Map}\label{subsection:make-comparison-map}

Throughout this section, we fix a $p$-complete commutative algebra object $A(-)$ of $D( \Fun( \CAlg_{\F_p}^{\mathrm{reg}}, \Ab) )$
and an isomorphism of commutative algebras $u_0 \colon \Omega_{(-)} \xrightarrow{\sim} \F_p \otimes^{L} A(-)$ in the derived category
$D( \Fun( \CAlg_{\F_p}^{\mathrm{reg}}, \Vect_{\F_p}) )$. We will identify $A(-)$ with a functor of $\infty$-categories $\CAlg_{\F_p}^{\mathrm{reg}} \rightarrow \calD(\Z)$.
Our assumption that $A(-)$ is $p$-complete guarantees that this functor takes values in the full subcategory $\widehat{\calD( \Z_p )}$ spanned by the
derived $p$-complete objects of $\calD(\Z)$ (Definition \ref{definition.dhp}).

For every regular Noetherian $\F_p$-algebra $R$, let us write $A(R)$ for the value of $A(-)$ on $R$, and $\epsilon_{R}$ for the composite map
$$ A(R) \rightarrow \F_p \otimes^{L} A(R) \xrightarrow{u_0^{-1}} \Omega_{R} \rightarrow R;$$
we can then regard $R \mapsto \epsilon_{R}$ as a natural transformation $\epsilon$ from the functor $A(-)$ to the forgetful functor $\CAlg_{\F_p}^{\mathrm{reg}} \rightarrow \Ab \hookrightarrow \calD(\Z)$.
Similarly, for every regular Noetherian $\F_p$-algebra $R$, let us write $\epsilon^{\crys}_{R}$ for the composite map
$$ R \Gamma_{\crys}( \Spec(R) ) \rightarrow \F_p \otimes^{L} \RGamma_{\crys}(\Spec(R)) \simeq \Omega_{R} \rightarrow R.$$
More concretely, $\epsilon^{\crys}_{R} \colon R \Gamma_{\crys}( \Spec(R) ) \rightarrow R$ is the map given by evaluation on $\Spec(R)$ (regarded as an object of its own crystalline site).
Our goal is to prove the following result:

\begin{proposition}\label{proposition:make-comparison-map}
There exists a natural transformation $u \colon R\Gamma_{\crys}(-) \rightarrow A(-)$ which is compatible with the commutative
algebra structures on $R \Gamma_{\crys}(-)$ and $A(-)$ (as objects of the derived category $D( \Fun( \CAlg_{\F_p}^{\mathrm{reg}}, \Ab) )$)
and for which the diagram
$$ \xymatrix{ R \Gamma_{\crys}(-) \ar[rr]^-{u} \ar[dr]^-{ \epsilon^{\crys} } & & A(-) \ar[dl]_-{ \epsilon} \\
& \id & }$$
commutes.
\end{proposition}

\begin{remark}
The statement of Proposition \ref{proposition:make-comparison-map} does not directly reference the isomorphism
$u_0 \colon \Omega_{(-)} \simeq \F_p \otimes^{L} A(-)$; it mentions only the natural transformation $\epsilon$ obtained from $u_0$.
Consequently, it is not {\it a priori} clear that the natural transformation $u$ of Proposition \ref{proposition:make-comparison-map} is compatible with $u_0$:
we will verify this in \S \ref{subsection:check-comparison-iso}.
\end{remark}

Our first step is to extend the functor $A(-)$ to a larger category of rings.

\begin{definition}[\cite{BMS2}] 
\label{qsyndef}
Let $R$ be a commutative $\F_p$-algebra. We say that $R$ is {\it quasisyntomic} if
the relative cotangent complex $L_{R/\F_p} \in D(R)$ has $\mathrm{Tor}$-amplitude in $[-1, 0]$. 
Let $\CAlg_{\F_p}^{\mathrm{qs}}$ denote the full subcategory of $\CAlg_{\F_p}$ spanned by the
quasisyntomic $\F_p$-algebras. Note that $\CAlg_{\F_p}^{\mathrm{qs}}$ contains the category
$\CAlg_{\F_p}^{\mathrm{reg}}$ of regular Noetherian $\F_p$-algebras (Proposition~\ref{RegCCFlat}).
\end{definition} 

\begin{lemma}\label{lemma:polynomial-extension}
The functor $A \colon \CAlg_{\F_p}^{\mathrm{reg}} \rightarrow \widehat{\calD(\Z_p)}$ is a left Kan extension of
its restriction to the full subcategory $\CAlg^{\mathrm{poly}}_{\mathbb{F}_p}
\subset \CAlg_{\F_p}^{\mathrm{reg}}$ spanned by the finitely generated polynomial algebras over $\F_p$.
\end{lemma}

\begin{proof}
Since the extension-of-scalars functor $\widehat{ \calD(\Z_p)} \rightarrow \calD( \F_p )$ is conservative
and preserves small colimits, it will suffice to prove that the functor $R \mapsto \F_p \otimes^{L} A(R) \simeq \Omega_{R}$
is a left Kan extension of its restriction to $\CAlg^{\mathrm{poly}}_{\mathbb{F}_p}$. In other words, it suffices to show that if
$R$ is a regular Noetherian $\F_p$-algebra, then the comparison map $L \Omega_{R} \rightarrow \Omega_{R}$
is an equivalence. This follows from the criterion of Proposition~\ref{ddRHodgeComplete} (together with Proposition~\ref{RegCCFlat} and
Theorem~\ref{avoidpop}).
\end{proof}

\begin{lemma}\label{lemma:make-lke}
The functor $A \colon \CAlg_{\F_p}^{\mathrm{reg}} \rightarrow \widehat{ \calD( \Z_p ) }$ admits a left Kan extension
$\overline{A} \colon \CAlg_{\F_p}^{\mathrm{qs}} \rightarrow \widehat{ \calD( \Z_p ) }$.
\end{lemma}

\begin{proof}
By virtue of Lemma \ref{lemma:polynomial-extension}, it will suffice to show that the restriction $A_0 = A|_{ \CAlg^{\mathrm{poly}}_{\mathbb{F}_p} }$
admits a left Kan extension to the category $\CAlg_{\F_p}^{\mathrm{qs}}$. This is clear, since the category $\CAlg^{\mathrm{poly}}_{\mathbb{F}_p}$ is small and the $\infty$-category $\widehat{ \calD(\Z_p) }$ admits small colimits.
\end{proof}

\begin{remark}\label{remark:describe-lke}
The natural isomorphism of functors $u_0 \colon \Omega_{(-)} \simeq \F_p \otimes^{L} A(-)$ admits an essentially unique extension to a natural isomorphism
$\overline{u}_0 \colon L \Omega_{(-)} \rightarrow \F_p \otimes^{L} \overline{A}(-)$ of functors from $\CAlg^{\mathrm{qs}}_{\F_p}$ to $\widehat{ \calD(\Z_p) }$. Here
$L \Omega_{(-)}$ denotes the derived de Rham complex functor of Variant \ref{deriveddeRham} (restricted to quasisyntomic $\F_p$-algebras).

In particular, for every quasisyntomic $\F_p$-algebra $R$, the reduction $\F_p \otimes^{L} \overline{A}(R)$ can be equipped with the conjugate filtration of
Remark \ref{ddr}. This is an exhaustive increasing filtration, whose associated graded is given on objects by the formula
$$ \gr^{n}( \F_p \otimes^{L} \overline{A}(R) ) \simeq \gr^{n} \left( \filc_{\ast}L \Omega_{R} \right) \simeq (\bigwedge^{n} L_{ R^{(1)} / \F_p})[-n].$$
\end{remark}

In the situation of Lemma \ref{lemma:make-lke} and Remark \ref{remark:describe-lke}, there is no need to restrict our attention to quasisyntomic $\F_p$-algebras: one can also contemplate
the left Kan extension of the functor $A(-)$ to the category of {\em all} $\F_p$-algebras (or even simplicial $\F_p$-algebras). However, the extension to quasisyntomic $\F_p$-algebras is particularly convenient
because of the following:

\begin{lemma}\label{lemma:cohomologically-nonnegative}
Let $R$ be a quasisyntomic $\F_p$-algebra. Then the cohomology groups $\mathrm{H}^{n}( \overline{A}(R) )$ vanish for $n < 0$. Moreover, the cohomology group $\mathrm{H}^{0}( \overline{A}(R) )$ is
$p$-torsion-free.
\end{lemma}

\begin{proof}
It will suffice to show that the cohomology of the reduction $\F_p \otimes^{L} \overline{A}(R) \simeq L \Omega_{R}$ is concentrated in nonnegative cohomological degrees. This follows
by inspecting the conjugate filtration of Remark \ref{remark:describe-lke} (our assumption that $R$ is quasisyntomic guarantees that each of the exterior powers
$(\bigwedge^{n} L_{ R^{(1)} / \F_p})[-n]$ is concentrated in nonnegative cohomological degrees).
\end{proof}

We now restrict to a subcategory of $\CAlg_{\F_p}^{\mathrm{qs}}$ where the functor $\overline{A}(-)$ is even more well-behaved.

\begin{definition}
Let $R$ be a commutative $\F_p$-algebra. We will say that $R$ is {\it quasiregular semiperfect} if it is quasisyntomic
and the Frobenius map $\varphi \colon R \rightarrow R$ is surjective. We let $\CAlg_{\F_p}^{\mathrm{qrsp}}$ denote the full subcategory of
$\CAlg_{\F_p}^{\mathrm{qs}}$ spanned by the quasiregular semiperfect $\F_p$-algebras.
\end{definition}

\begin{remark}\label{rasp}
Let $R$ be a quasiregular semiperfect $\F_p$-algebra. Then the shifted cotangent complex $L_{R/\F_p}[-1] \in D(R)$ is a flat $R$-module,
concentrated in degree $0$. It follows that each of the shifted exterior powers $(\bigwedge^{n} L_{ R / \F_p})[-n]$ is also a flat $R$-module
(which can be identified with the divided power module $\Gamma^{n}_{R}( L_{R/\F_p}[-1] )$: see \cite[Prop. 4.3.2.1]{Ill1}).
Using the conjugate filtration of Remark \ref{remark:describe-lke}, we conclude that $\F_p \otimes^{L} \overline{A}(R) \simeq L \Omega_{R}$
is concentrated in cohomological degree zero. By derived $p$-adic
completeness, it follows that $\overline{A}(R)$ is also concentrated in cohomological degree zero,
and can therefore be identified with an ordinary commutative ring (which is $p$-adically complete and $p$-torsion-free), and the derived de Rham complex
$L \Omega_{R}$ can be identified with the quotient ring $\overline{A}(R) / p \overline{A}(R)$.
\end{remark}

Our next goal is to show that the functor $\overline{A}(-)$ can be recovered from its values on quasiregular semiperfect $\F_p$-algebras.

\begin{lemma}\label{lemma:right-kan-from-qrsp}
The functor $\overline{A} \colon \CAlg_{\F_p}^{\mathrm{qs}} \rightarrow \widehat{\calD(\Z_p)}$ is a right Kan extension of its
restriction to the subcategory $\CAlg_{\F_p}^{\mathrm{qrsp}} \subset \CAlg_{\F_p}^{\mathrm{qs}}$.
\end{lemma}

\begin{proof}
Let $R$ be a quasisyntomic $\F_p$-algebra. Choose a collection of elements $\{ x_i \}_{i \in I}$ of $R$ which generate
$R$ as an algebra over $\F_p$, and let $R^{0}$ denote the tensor product $R \otimes_{ \F_p[ \{ x_i \} ] } \F_p[ \{ x_i \} ]_{\perf}$
(here $\F_p[ \{ x_i \} ]$ denotes the polynomial algebra on generators $\{ x_i \}_{i \in I}$). Let $R^{\bullet}$ denote the cosimplicial
$R$-algebra given by the tensor powers of $R^{0}$ over $R$. Note that $R^{0}$ is a weakly initial object in the category of
$R$-algebras which are quasiregular semiperfect: that is, any ring homomorphism from $R$ to a quasiregular semiperfect $\F_p$-algebra
$S$ factors (not necessarily uniquely) through $R^{0}$. Consequently, to show that the functor $\overline{A}$ is a right Kan extension of its
restriction $\overline{A}|_{ \CAlg_{\F_p}^{\mathrm{qrsp}} }$ at $R$, it will suffice to show that the comparison map
$\overline{A}(R) \rightarrow \Tot( \overline{A}( R^{\bullet} ) )$ is an
isomorphism; here $\Tot$ denotes the homotopy limit over the simplex category
$\Delta$. To establish this, we are free to reduce modulo $p$
and (by virtue of the coconnectivity supplied by Lemma \ref{lemma:cohomologically-nonnegative}) it will suffice to verify the analogous property
for each successive quotient of the conjugate filtration of Remark \ref{remark:describe-lke}. In other words, we are reduced to the problem of showing
that for each $n \geq 0$, the canonical map $\bigwedge^{n} L_{R/\F_p} \rightarrow \Tot( \bigwedge^{n} L_{ R^{\bullet} / \F_p})$ is an isomorphism
in the $\infty$-category $\calD( R )$, which follows from \cite[Theorem 3.1]{BMS2}.
\end{proof}

For verifying certain identities in derived de Rham cohomology, it will be
convenient to use automatic gradings that are inherited in certain cases. 
We now quickly review how a semisimple automorphism yields such a grading. 

\begin{definition} 
Let $V$ be a (possibly infinite-dimensional) vector space over an
algebraically closed field $k$ equipped with an
endomorphism $g \colon V \to V$. We say that the pair $(V, g)$ is \emph{semisimple} if, when $V$ is
regarded as a $k[u]$-module with $u$ acting as $g$, then $V$ is isomorphic to a
direct sum of simple modules $k[u]/(u-\lambda)$, for $\lambda \in k$. 
\end{definition} 

\begin{proposition} 
\label{semisimplepair}
\begin{enumerate}
\item Every semisimple pair $(V, g)$ decomposes  
uniquely as a direct sum $V \simeq \bigoplus_{\lambda \in k} V_{\lambda}$, where
$V_\lambda$ is the $\lambda$-eigenspace of $g$; any map $(V, g) \to (V', g')$
of semisimple pairs respects the decomposition. 
\item 
Given a short exact sequence of pairs $0 \to (V', g') \to (V, g) \to (V'', g'')
\to 0$ with $(V', g')$ and $(V'', g'')$ semisimple and such that the
eigenvalues of $g', g''$ are disjoint, we have that $(V, g)$ is
semisimple. 
\end{enumerate}
\end{proposition} 
\begin{proof} 
Part (1) follows from the definition. 
Part (2) follows because $\mathrm{Ext}^1_{k[u]} (V'', V') = 0$; indeed, this reduces by
taking direct sums to the case where $V'' = V''_\lambda$ has a single
eigenvalue $\lambda$, and then $u - \lambda$ acts trivially on $V''$ but as an
isomorphism on $V'$. 
\end{proof} 

We will need the following result on the compatibility of derived de Rham cohomology with the Frobenius operator:

\begin{lemma}\label{lemma:compatibility-of-frobenius}
Let $R$ be a quasiregular semiperfect $\F_p$-algebra, and let $L \Omega_{\varphi_{R}}$ denote the endomorphism of
$L \Omega_{R}$ induced by the Frobenius on $R$. Then $L \Omega_{\varphi_R}$ coincides with the Frobenius
endomorphism of $L \Omega_{R}$ (where we regard $L \Omega_{R}$ as an
$\mathbb{F}_p$-algebra as in Remark \ref{rasp}).
\end{lemma}

\begin{proof}
Choose a surjection of commutative $\F_p$-algebras $q \colon S \twoheadrightarrow R$, where $S$ is perfect
(for example, we can take $S$ to be the inverse limit of the tower $\cdots \rightarrow R \xrightarrow{\varphi_{R}} R \xrightarrow{ \varphi_{R} } R$).
Let $I \subseteq S$ denote the kernel of $q$. Choose a collection of elements $\{ x_{\alpha} \in I \}$ whose images generate the quotient $I/I^2$ as a module over $R$.
Let $\overline{R}$ denote the tensor product $\bigotimes_{\alpha} S[ x_{\alpha} ]_{\perf} / (x_{\alpha} )$, formed in the category of $S$-algebras, so that
we have a commutative diagram of quasiregular semiperfect $\F_p$-algebras
$$ \xymatrix{ \bigotimes_{\alpha} S[ x_{\alpha} ]_{\mathrm{perf}} \ar[r] \ar[d] & \overline{R} \ar[d] \\
S \ar[r] & R. }$$
By construction, the right vertical map is a surjection and induces a surjection on shifted cotangent complexes $L_{ \overline{R} / \F_p}[-1] \twoheadrightarrow
L_{R/\F_p}[-1] \simeq I/I^2$. It follows that the map $L \Omega_{\overline{R}} \rightarrow L \Omega_{R}$ induces a surjection of associated graded rings
$\gr^{\ast}( L \Omega_{\overline{R}} ) \rightarrow \gr^{\ast}( L \Omega_{R} )$ (with respect to the conjugate filtration of Remark \ref{remark:describe-lke})
and is therefore a surjection. We may therefore replace $R$ by $\overline{R}$ and thereby reduce to the case where $R$ is given by
a (possibly infinite) tensor product, over a perfect $\F_p$-algebra $S$, of algebras having the form $S[ x]_{\perf} / (x)$. In this situation,
the derived de Rham complex also factors as the tensor product (over $S$) of algebras of the form $L \Omega_{ S[x]_{\perf} / (x) }$.
We may therefore assume without loss of generality that $R = S[ x]_{\perf} / (x)$, where $S$ is perfect. Moreover, we can further arrange
(by extending scalars on $S$ if necessary) that $S$ contains an algebraically closed field $k$ which contains an element $t$ which is not algebraic over $\F_p$.
Let $\widetilde{\tau} \colon S[x]_{\perf} \rightarrow S[x]_{\perf}$ be the automorphism given by $\widetilde{\tau}(x) = tx$,
so that $\widetilde{\tau}$ descends to an automorphism $\tau \colon R \rightarrow R$.

We have two endomorphisms of the algebra $L \Omega_R$ given by 
$L \Omega_{\varphi_R}$ and the internal Frobenius of $L\Omega_R$ (as an
$\F_p$-algebra), which we need to prove are equal. 
We can regard these as $k$-linear maps 
$L \Omega_R^{(1)} \to L \Omega_R$ where $L \Omega_R^{(1)}$ denotes the Frobenius
twist. 
Note that $L \Omega_R$ admits the functorial, increasing, and exhaustive conjugate filtration
of $k$-vector spaces
with $\mathrm{gr}^i L \Omega_{R} \simeq (\bigwedge^i L_{R/k})[-i]^{(1)}$. 
In particular, 
$\mathrm{gr}^* L \Omega_R $ is identified with the divided power algebra $R \left \langle \overline{x}\right\rangle$ 
for $\overline{x}  \in L_{R/\mathbb{F}_p}[-1]^{(1)} =
L_{R/S[x]_{\mathrm{perf}}}^{(1)}[-1] \simeq  ( (x)/(x)^2)^{(1)}$ the image of the class
$x$. 
On the associated graded of the conjugate filtration, it is easy to see that
both the map induced by the Frobenius $\varphi \colon R \to R$ and the internal
Frobenius agree (in fact, both are zero in degrees $>0$ with respect to the
grading as $\overline{x}$ admits divided powers). 

The automorphism $\tau$ acts on $L \Omega_R$ by functoriality and preserves the
conjugate filtration.  Moreover, for each $i \geq 0$, the automorphism $\tau$ acts semisimply on 
$\mathrm{gr}^i L \Omega_R$, and 
the eigenvalues of $\tau$ on $\mathrm{gr}^i L \Omega_R$ have the form
$t^{p(i + \alpha)}$, for $\alpha \in \mathbb{Z}[1/p]$ with $0 \leq \alpha < 1$
(corresponding to the eigenspace spanned by $\gamma_i(\overline{x})
x^{\alpha}$). It follows that the eigenvalues are disjoint for distinct $i$, and Proposition~\ref{semisimplepair} implies that $\tau$ induces a semisimple automorphism of both $L \Omega_R$ and $L \Omega_R^{(1)}$. 
It follows that the conjugate filtrations of $L \Omega_R^{(1)}$ and $L \Omega_R$ admit canonical splittings (which are preserved by the Frobenius morphisms $L \Omega_{\varphi_{R}}$ and $\varphi_{ L \Omega_{R} }$).
Since $L \Omega_{ \varphi_{R} }$ and $\varphi_{ L \Omega_{R} }$ agree at the associated graded level, they are equal.
\end{proof}

\begin{remark}\label{remark:compatfrob2}
For every commutative $\F_p$-algebra $R$, the endomorphism $L \Omega_{\varphi_{R}}$ of $L \Omega_{R}$ is given by the composition
$L \Omega_R \xrightarrow{\epsilon^{\mathrm{dR}}} R \xrightarrow{\iota} L \Omega_R$,
where $\iota$ is the inclusion of the first stage of the conjugate filtration and $\epsilon^{\mathrm{dR}}$ is the projection map: this follows by Kan extension from the case where $R$ is a smooth $\F_p$-algebra,
in which case it can be verified at the level of chain complexes. For later use, we remark that for $R$ quasiregular semiperfect, the complex $L\Omega_R$ is concentrated in degree $0$ and the map $\iota$ is injective: this follows as each associated graded piece $\bigwedge^n L_{R^{(1)}/\mathbf{F}_p}[-n]$ of the conjugate filtration (Remark~\ref{ddr}) on $L\Omega_R$ is concentrated in degree $0$ for such $R$.
\end{remark}

Note that the natural transformation of functors $\epsilon \colon A(-) \rightarrow \id_{ \CAlg_{\F_p}^{\mathrm{reg}} }$ admits an essentially unique extension to a natural transformation
$\overline{A}(-) \rightarrow \id_{ \CAlg_{\F_p}^{\mathrm{qs} } }$. By a slight abuse of notation, we will denote the value of this natural transformation on a quasisyntomic $\F_p$-algebra $R$
by $\epsilon_{R} \colon \overline{A}(R) \rightarrow R$.

\begin{lemma}\label{lemma:make-divided-powers}
Let $R$ be a quasiregular semiperfect $\F_p$-algebra. Then the map $\epsilon_{R}
\colon\overline{A}(R) \rightarrow R$
is surjective, and the ideal $I = \ker( \epsilon_{R} )$ has divided powers (which are necessarily unique, since the ring $\overline{A}(R)$ is $p$-torsion-free).
\end{lemma}

\begin{proof}
Note that $\epsilon_{R} \colon \overline{A}(R) \rightarrow R$ induces the homomorphism
of $\F_p$-algebras $\epsilon^{\mathrm{dR}} \colon \overline{A}(R) / p \overline{A}(R) \rightarrow R$,
whose restriction to the first step of the conjugate filtration is the Frobenius
map $\varphi_{R} \colon R \rightarrow R$. Since $R$ is semiperfect, the map $\varphi_{R}$
is surjective. It follows that $\epsilon_{R}$ is surjective as well.

Let $\varphi \colon \overline{A}(R) \rightarrow \overline{A}(R)$ denote the map
induced by $\varphi_{R}$. It follows from Lemma
\ref{lemma:compatibility-of-frobenius}
that $\varphi$ is a lift of the Frobenius on the $\F_p$-algebra $\overline{A}(R)/p \overline{A}(R)$, and therefore endows $\overline{A}(R)$ with the structure of a $\delta$-ring (Definition \ref{lambdapring}).
We have a commutative diagram
$$ \xymatrix@R=50pt@C=50pt{ \overline{A}(R) \ar[rr]^{\varphi} \ar[d] \ar[dr]^-{ \epsilon_R} & & \overline{A}(R) \ar[d] \\
\overline{A}(R) / p \overline{A}(R) \ar[r]^-{\epsilon^{\mathrm{dR}}} & R \ar[r]^-{\iota} & \overline{A}(R) / p \overline{A}(R), }$$
in which $\epsilon^{\mathrm{dR}}, \iota$ are defined as in Remark
\ref{remark:compatfrob2} (using the identification of $\overline{A}(R)/p
\overline{A}(R)$ with the derived de Rham complex $L \Omega_{R}$). Since
$\iota$ is injective, it follows
that 	an element $x \in \overline{A}(R)$ belongs to the ideal $I$ if and only
if $\varphi(x)$ is divisible by $p$. If this condition is satisfied,
then \cite[Lemma 2.34]{BhattScholzePrisms} guarantees that $x$ has divided powers in
$\overline{A}(R)$. Moreover, for each $n \geq 0$, the calculation
$$ \varphi( \tfrac{ x^n }{n!} ) = \tfrac{ \varphi(x)^{n} }{n!} =
\tfrac{p^{n}}{n!} (\tfrac{\varphi(x)}{p})^n \equiv 0 \pmod{p}$$
shows that the divided power $\frac{x^n}{n!}$ also belongs to the ideal $I$.
\end{proof}

Note that the functor $R \Gamma_{\crys}(-) \colon \CAlg_{\F_p}^{\mathrm{reg}} \rightarrow \calD(\Z)$ extends naturally to the category of 
all $\F_p$-algebras, as the cohomology of the structure sheaf on the
crystalline site. 
In what follows, we will abuse notation by denoting this extension also by $R \Gamma_{\crys}(-)$.
Although we will not need this fact, $R \Gamma_{\crys}$ when restricted to
$\CAlg_{\F_p}^{\mathrm{qs}}$ is a left Kan extension of its restriction to
regular (or even finitely generated polynomial) $\F_p$-algebras. 
 Proposition \ref{proposition:make-comparison-map} is an immediate consequence of the following more refined statement:

\begin{proposition}\label{proposition:make-comparison-map2}
There exists a natural transformation $\overline{u} \colon R\Gamma_{\crys}(-) \rightarrow \overline{A}(-)$ which is compatible with the commutative
algebra structures on $R \Gamma_{\crys}(-)$ and $\overline{A}(-)$ (as objects of the derived category $D( \Fun( \CAlg_{\F_p}^{\mathrm{qs}}, \Ab) )$)
and for which the diagram
\begin{equation} \label{crystrianglediag} \xymatrix{ R \Gamma_{\crys}(-) \ar[rr]^-{\overline{u}} \ar[dr]_-{ \epsilon^{\crys} } & & \overline{A}(-) \ar[dl]^-{ \epsilon} \\
& \id & } \end{equation}
commutes.
\end{proposition}

\begin{proof}[Proof of Proposition \ref{proposition:make-comparison-map2}]
Let $R$ be a quasiregular semiperfect $\F_p$-algebra. It follows from Lemma \ref{lemma:make-divided-powers} that each of the schemes $\Spec( \overline{A}(R) / p^{n} \overline{A}(R) )$
can be regarded as an object of the crystalline site of $\Spec(R)$. We therefore
obtain a comparison map $\overline{u}_{R} \colon R \Gamma_{\crys}( \Spec(R) )
\rightarrow \varprojlim_{n} \overline{A}(R) / p^{n} \overline{A}(R) \simeq
\overline{A}(R)$, depending functorially on $R$ and making
\eqref{crystrianglediag} commute. By virtue of Lemma \ref{lemma:right-kan-from-qrsp}, this construction admits an essentially unique extension to a natural transformation
$\overline{u} \colon R \Gamma_{\crys}( \Spec(-) ) \rightarrow \overline{A}(-)$
of functors from $\CAlg_{\F_p}^{\mathrm{qs} }$ to
$\widehat{\calD(\Z_p)}$, making
\eqref{crystrianglediag} commute. 
\end{proof}

\subsection{Endomorphisms of the de Rham Functor}\label{subsection:check-comparison-iso}

Let $A(-)$ be a $p$-complete commutative algebra object of the derived category $D( \Fun( \CAlg_{\F_p}^{\mathrm{reg}}, \Ab) )$, and let
$u_0 \colon \Omega_{(-)} \simeq \F_p \otimes^{L} A(-)$ be an isomorphism of commutative algebra objects in the derived category
$D( \Fun( \CAlg_{\F_p}^{\mathrm{reg}}, \Vect_{\F_p} ) )$. In \S \ref{subsection:make-comparison-map}, we constructed a comparison map
$u \colon R \Gamma_{\crys}(-) \rightarrow A(-)$. Reducing modulo $p$, we obtain a map
$$ \theta \colon \Omega_{(-)} \simeq \F_p \otimes^{L} R\Gamma_{\crys}(-) \xrightarrow{u} \F_p \otimes^{L} A(-) \xrightarrow{u_0^{-1}} \Omega_{(-)}.$$
By construction, this map fits into a commutative diagram
$$ \xymatrix{ \Omega_{(-)} \ar[rr]^{\theta} \ar[dr]_-{ \epsilon^{\dR} } & & \Omega_{(-)} \ar[dl]^-{ \epsilon^{\dR} } \\
& \id, & }$$
where $\epsilon^{\dR}$ denotes the natural transformation carrying each regular
$\F_p$-algebra $R$ to the map $\epsilon^{\dR}_{R} \colon \Omega_{R} \rightarrow R$ given by
induced by projection to the zeroth term of the de Rham complex $\Omega^{\ast}_{R}$. Our goal in this section is to show that $\theta$ is the identity map: that is,
the natural transformation $u$ of Proposition \ref{proposition:make-comparison-map} is automatically compatible with $u_0$ (and is therefore an isomorphism).
This is a consequence of the following more general assertion:

\begin{proposition}\label{proposition:check-comparison-iso}
Let $\theta \colon \Omega_{(-)} \rightarrow \Omega_{(-)}$ be a morphism of commutative algebra objects of the derived category
$D( \Fun( \CAlg_{\F_p}^{\mathrm{reg}}, \Vect_{\F_p} ) )$ for which the diagram
$$ \xymatrix{ \Omega_{(-)} \ar[rr]^{\theta} \ar[dr]_-{ \epsilon^{\dR} } & & \Omega_{(-)} \ar[dl]^-{ \epsilon^{\dR} } \\
& \id & }$$
is commutative. Then $\theta$ is the identity map.
\end{proposition}

\begin{warning}
The $\infty$-categories $\calD(\Fun(\CAlg_{\F_p}^{\mathrm{reg}}, \Vect_{\F_p}))$ and $\Fun(\CAlg_{\F_p}^{\mathrm{reg}}, \mathcal{D}(\mathbf{F}_p))$ are equivalent, and have $D( \Fun( \CAlg_{\F_p}^{\mathrm{reg}}, \Vect_{\F_p} ) )$ as their homotopy categories. Consequently, in the statement of Proposition \ref{proposition:check-comparison-iso}, we can regard $\theta$ as an endomorphism of
the construction $R \mapsto \Omega_{R}$, regarded as a functor from the ordinary category $\CAlg_{\F_p}^{\mathrm{reg}}$ to the
$\infty$-category $\calD( \F_p )$. To conclude that $\theta$ is the identity, it is essential to regard the target of $\theta$ as an $\infty$-category, i.e., the analog of Proposition~\ref{proposition:check-comparison-iso} fails for endomorphisms of the  functor $\CAlg_{\F_p}^{\mathrm{reg}} \rightarrow D( \F_p )$ given by $R \mapsto \Omega_R$. Indeed, as $\F_p$ is a field, the derived category $D(\F_p)$ is equivalent to the category of graded vector spaces over $\F_p$ (via the construction $M \mapsto \mathrm{H}^{\ast}(M)$).
Consequently, when regarded as a commutative algebra object of $D(\F_p)$, the datum of the de Rham complex $\Omega_{R}$ is equivalent
to the datum of the de Rham cohomology ring $\mathrm{H}^{\ast}_{\dR}( \Spec(R) )$. This cohomology ring admits an endomorphism
$$ \mathrm{H}^{\ast}_{\dR}( \Spec(R) ) \rightarrow \mathrm{H}^{\ast}_{\dR}( \Spec(R) ) \quad \quad (x \in \mathrm{H}^{n}_{\dR}( \Spec(R) )
\mapsto \begin{cases} x & \text{ if $n =0$ } \\
0 & \text{ otherwise} \end{cases}$$
which depends functorially on $R$ and is the identity in degree zero. It follows from Proposition \ref{proposition:check-comparison-iso} that the resulting endomorphism $\overline{\theta}$ of the functor $\CAlg_{\F_p}^{\mathrm{reg}} \to D(\F_p)$ between ordinary categories given by $R \mapsto \Omega_R$, which satisfies the commutativity constraint formulated in Proposition~\ref{proposition:check-comparison-iso}, cannot lift to an endomorphism of the functor $\CAlg_{\F_p}^{\mathrm{reg}} \to \calD(\F_p)$ between $\infty$-categories determined by the same construction.
%
%
\end{warning}

We first prove a weak version of Proposition \ref{proposition:check-comparison-iso}.

\begin{lemma}\label{lemma:weak-iso-check}
Let $\theta \colon \Omega_{(-)} \rightarrow \Omega_{(-)}$ be as in the statement of Proposition \ref{proposition:check-comparison-iso}. Then,
for every regular $\F_p$-algebra $R$, the induced map $\mathrm{H}^{\ast}( \Omega_{R} ) \rightarrow \mathrm{H}^{\ast}( \Omega_{R} )$ is the identity map.
\end{lemma}

\begin{proof}
Without loss of generality, we may assume that $\Spec(R)$ is connected. To avoid
confusion, let us write $R^{(1)}$ to denote the image of the Frobenius map
$\varphi_{R} \colon R \rightarrow R$,
so that we have a Cartier isomorphism $\Cart \colon \Omega^{\ast}_{R^{(1)}}
\xrightarrow{\sim} \mathrm{H}^{\ast}_{\dR}( \Spec(R) )$ (Theorem \ref{avoidpop}). Note that the composite map
$$ R^{(1)} = \Omega^{0}_{ R^{(1)} } \xrightarrow[\sim]{ \Cart} \mathrm{H}^{0}_{\dR}( \Spec(R) ) \xrightarrow{ \epsilon^{\dR} } R$$
is the inclusion map. Since $\theta$ is assumed to be compatible with the map $\epsilon^{\dR}$, it follows that $\theta$ must induce the identity endomorphism of
$\mathrm{H}^{0}( \Omega_{R} )$. Note that the cohomology ring $\mathrm{H}^{\ast}(  \Omega_{R} )$ is generated, as an algebra over
$\mathrm{H}^{0}( \Omega_{R} )$, by elements of the form $\Cart( dx )$ for $x \in R$. It will therefore suffice to show that every such element is fixed by $\theta$.
To prove this, we may assume without loss of generality that $R = k[x]$ for some
perfect field $k$ of characteristic $p$. Enlarging $k$ if necessary, we may
further assume that it contains an element $t$ which is not algebraic over $\F_p$.

Since the ring $R = k[x]$ has Krull dimension $1$, the de Rham complex
$\Omega^{\ast}_{R}$ is concentrated in degrees $\leq 1$. Consequently, the cone
of the map $\epsilon_{R}^{\dR} \colon \Omega_{R} \rightarrow R$
can be identified with the module of K\"{a}hler differentials $\Omega^{1}_{R}$. Since $\theta$ is assumed to be compatible with $\epsilon^{\dR}$, it induces an endomorphism $\rho$ of $\Omega^{1}_{R}$
fitting into a commutative diagram
$$ \xymatrix{ \Omega_{R} \ar[r]^-{\epsilon^{\dR}_{R}} \ar[d]^{\theta} & R \ar[r]^-{d} \ar[d]^{\id} & \Omega^{1}_{R} \ar[d]^{\rho} \\
\Omega_{R} \ar[r]^{\epsilon^{\dR}_{R} } & R \ar[r]^{d} & \Omega^{1}_{R} }$$
in the $\infty$-category $\calD( \F_p )$. Note that the cohomology class $\Cart(dx)$ is represented by the $1$-form $\omega = x^{p-1} dx$. We will complete the proof by showing that
$\rho(\omega) = \omega$.

Note that every automorphism of the $k$-algebra $R = k[x]$ induces a $k$-linear automorphism of $\Omega^{1}_{R}$, which commutes with $\rho$ by naturality (here we use the fact that $\theta$ is a natural transformation
at the $\infty$-categorical level, where the formation of cones is functorial).
In particular, the map $x \mapsto tx$ induces an automorphism $\tau \colon \Omega^{1}_{R} \rightarrow \Omega^{1}_{R}$, and a simple
calculation shows that the eigenspace $(\Omega^{1}_{R})^{\tau = t^{p} }$ is the $1$-dimensional vector space $k\omega \subseteq \Omega^{1}_{R}$ spanned by the differential form $\omega$. Since $\rho$ commutes with $\tau$, it must carry the eigenspace $( \Omega^{1}_{R} )^{\tau = t^{p}}$ to itself, and therefore satisfies $\rho( \omega ) = c \omega$ for some element $c$ of the field $k$.

The construction $x \mapsto x + 1$ induces another $k$-linear automorphism
$\sigma \colon \Omega^{1}_{R} \rightarrow \Omega^{1}_{R}$, and an elementary calculation gives
$\sigma( \omega ) = \omega + df(x)$ where $f(x)$ denotes the polynomial $\sum_{i=1}^{p-1} \frac{ (p-1)!}{(p-i)! i!} x^{i}$.
Since $\sigma$ commutes with $\rho$, we have
\begin{align*}
df(x) + c \omega & =  df(x) + \rho(\omega)  =  \rho( df(x) + \omega) \\
& =  \rho( \sigma( \omega) )   =  \sigma( \rho(\omega) ) \\
& =  \sigma( c \omega )  =  cdf(x) + c \omega.
\end{align*}
Since $df(x)$ is a nonzero element of $\Omega^{1}_R{}$, it follows that $c=1$ so that $\rho(\omega) = \omega$, as desired.
\end{proof}

\begin{proof}[Proof of Proposition \ref{proposition:check-comparison-iso}]
Let us regard the construction of the derived de Rham complex $L \Omega_{(-)}$ as a commutative algebra object
in the homotopy category of functors from the ordinary $\CAlg_{\F_p}^{\mathrm{qs}}$ to the $\infty$-category $\calD( \F_p )$.
Since this functor is a left Kan extension of its restriction to
$\CAlg_{\F_p}^{\mathrm{reg}}$, the map $\theta$ admits an essentially unique extension to a natural transformation
$\overline{\theta} \colon L \Omega_{(-)} \rightarrow L \Omega_{(-)}$. 

Let us say that an $\F_p$-algebra $R$ is {\it standard} if it can be written as
a finite tensor product over $\mathbb{F}_p$ of algebras
of the form $\F_p[x]_{\perf}$ or $\F_p[x]_{\perf} / (x)$. If this condition is satisfied, then $R$ is quasiregular semiperfect, so the
derived de Rham complex $L \Omega_{R}$ is concentrated in cohomological degree zero and can therefore be identified with an ordinary commutative ring.
We will prove the following:
\begin{itemize}
\item[$(\ast)$] The natural transformation $\overline{\theta}$ is the identity when restricted to the category of standard $\F_p$-algebras.
\end{itemize}
Let us assume $(\ast)$ for the moment, and show that it leads to a proof of Proposition \ref{proposition:check-comparison-iso}.
Let $R$ be any finite type polynomial algebra over $\F_p$, let $R^{0} = R_{\perf}$ denote its perfection,
and let $R^{\bullet}$ denote the cosimplicial $R$-algebra given by the tensor powers of $R^{0}$ over $R$. As in the proof of Lemma \ref{lemma:right-kan-from-qrsp}, we note that the canonical map
$$ \Omega_{R} = L \Omega_{R} \rightarrow \Tot( L \Omega_{ R^{\bullet} } )$$
is an equivalence. Note that each term of the cosimplicial ring $R^{\bullet}$ is standard, and that $R^{0}$ is a weakly initial object in the category of standard $R$-algebras.
It follows that the restriction of the functor $L \Omega_{(-)}$ to the category of polynomial rings $\CAlg_{\F_p}^{\mathrm{poly}}$ is a right Kan extension of its restriction
to the category of standard $\F_p$-algebras. Invoking $(\ast)$ and the right
Kan extension property, we deduce that the natural transformation $\overline{\theta}$ is the identity when restricted to the category
$\CAlg_{\F_p}^{\mathrm{poly}}$. Since the functor $L \Omega_{(-)}$ is a left Kan extension of its restriction to $\CAlg_{\F_p}^{\mathrm{poly} }$, it follows that $\overline{\theta}$
is the identity and therefore $\theta$ is also the identity, as desired.

We now prove $(\ast)$. Let $R$ be a standard $\F_p$-algebra; we wish to show
that the map $\overline{\theta}_{R} \colon L \Omega_{R} \rightarrow L \Omega_{R}$ 
is equal to the identity. Factoring $R$ as a tensor product, we may assume that it is either a perfected polynomial ring $\F_p[x]_{\perf}$ or a quotient
$\F_p[x]_{\perf} / (x)$. In the first case, the projection map $\epsilon^{\dR}
\colon L \Omega_{R} \rightarrow R$ is an isomorphism, and the result is obvious. The second case is a special case of the following assertion:
\begin{itemize}
\item[$(\ast')$] Let $k$ be a perfect field of characteristic $p$ and let $R =
k[x]_{\perf} / (x)$. Then the homomorphism $\overline{\theta}_{R} \colon L \Omega_{R} \rightarrow L \Omega_{R}$ is the identity.
\end{itemize}
To prove $(\ast')$, we are free to enlarge $k$ to a larger perfect field: this follows by faithful flatness of extending scalars and the fact that $L\Omega_R \otimes_k K \simeq L\Omega_{R \otimes_k K}$ for any extension $K/k$ of perfect fields. We may therefore assume that $k$
is algebraically closed and contains an element $t$ which is transcendental over $\F_p$. Let us regard the commutative ring
$L \Omega_{R}$ as equipped with the conjugate filtration of Remark~\ref{ddr}. 
By Lemma~\ref{lemma:weak-iso-check} and left Kan extension from finite type
polynomial algebras,  $\overline{\theta_R}$ preserves the
conjugate filtration on $L \Omega_R$ (which is the left Kan extension of the
Postnikov filtration) and  induces the identity on associated
graded terms. 

Note that the construction the construction $x \mapsto tx$ extends to an automorphism of the ring $k[x]_{\perf}$ carrying the ideal $(x)$ to itself, and therefore induces an automorphism
$\tau$ of $L \Omega_{R}$. 
As in the proof of Lemma~\ref{lemma:compatibility-of-frobenius}, the action of $\tau$ on $L \Omega_R$ is semisimple on
associated graded terms $\mathrm{gr}^i L \Omega_R$, and the eigenvalues for
different $i$ are disjoint. 
It follows that $\tau$ induces a splitting $L \Omega_R \simeq \bigoplus_{i \geq
0} \mathrm{gr}^i L \Omega_R$, which is necessarily preserved by
$\overline{\theta_R}$ (which commutes with $\tau$). Since $\overline{\theta}_R$ is the
identity on associated graded terms, it follows that $\overline{\theta}_R$ is
the identity on $L \Omega_R$ as desired. 
\end{proof}

\subsection{Uniqueness of the Comparison Map}\label{subsection:unique-comparison}

Throughout this section, we fix a $p$-complete commutative algebra object $A(-)$
of $D( \Fun( \CAlg_{\F_p}^{\mathrm{reg}}, \Ab) )$, and an isomorphism $u_0 \colon \Omega_{(-)} \simeq \F_p \otimes^{L} A(-)$ of commutative algebras in $D( \Fun( \CAlg^{\mathrm{reg}}_{\F_p}, \Vect_{\F_p} ) )$. 
This equivalence induces a natural map $\epsilon \colon A(-) \to \mathrm{id}$. 
Our goal is to prove the uniqueness of the isomorphism $u$ appearing in the statement of Theorem \ref{theorem:uniqueness-of-crys}. This is an easy consequence of
the following rigidity result:

\begin{proposition}\label{proposition:rigid-A}
Let $v \colon A(-) \rightarrow A(-)$ be a morphism of commutative algebra objects of $D( \Fun( \CAlg_{\F_p}^{\mathrm{reg}}, \Ab) )$ for which the diagram
$$ \xymatrix{ A(-) \ar[rr]^{v} \ar[dr]_-{\epsilon} & & A(-) \ar[dl]^-{\epsilon} \\
& \id & }$$
is commutative. Then $v$ is the identity.
\end{proposition}

\begin{proof}[Proof of Theorem \ref{theorem:uniqueness-of-crys} from Proposition \ref{proposition:rigid-A}]
It follows from Proposition \ref{proposition:make-comparison-map} that there
exists a morphism $u \colon R \Gamma_{\crys}(-) \rightarrow A(-)$
satisfying $\epsilon \circ u = \epsilon^{\crys}$. Reducing modulo $p$, we obtain a comparison map
$$u'_0 \colon \Omega_{(-)} \simeq \F_p \otimes^{L} R \Gamma_{\crys}(-) \rightarrow \F_p \otimes^{L} A(-).$$
Since $u_0$ is an isomorphism, we can write $u'_0$ as a composition $u_0 \circ \theta$, where $\theta$ is an endomorphism of the de Rham functor
$\Omega_{(-)}$ which is compatible with $\epsilon^{\dR}$. It follows from Proposition \ref{proposition:check-comparison-iso} that $\theta$ is the identity map,
so that $u'_0 = u_0$. In particular, $u'_0$ is an isomorphism. Since $R \Gamma_{\crys}(-)$ and $A(-)$ are $p$-complete, it follows that $u$ is also an isomorphism. This proves the existence 
assertion of Theorem \ref{theorem:uniqueness-of-crys}.

To prove uniqueness, suppose that $u' \colon R \Gamma_{\crys}(-) \rightarrow
A(-)$ is some other isomorphism compatible with $u_0$. We can then write $u' = v
\circ u$, where $v \colon A(-) \simeq A(-)$
is an isomorphism which reduces modulo $p$ to the identity. In particular, we have $\epsilon \circ v = \epsilon$, so Proposition \ref{proposition:rigid-A} guarantees that $v$ is the identity
and $u = u'$, as desired.
\end{proof}

We now turn to the proof of Proposition \ref{proposition:rigid-A}. Let
$\overline{A}(-)$ denote the left Kan extension of $A(-)$ to the category of quasisyntomic $\F_p$-algebras
(Lemma \ref{lemma:make-lke}), and let $\overline{u} \colon R \Gamma_{\crys}(-) \rightarrow \overline{A}(-)$ be the natural transformation constructed in the proof of Proposition \ref{proposition:make-comparison-map2}.

\begin{lemma}\label{lemma:surjective-on-cohomology}
Let $R$ be a quasisyntomic $\F_p$-algebra. Then the morphism $\overline{u}_{R}
\colon R \Gamma_{\crys}( \Spec(R) ) \rightarrow \overline{A}(R)$ induces a surjection of cohomology rings
$$\mathrm{H}^{\ast}_{\crys}( \Spec(R) ) \twoheadrightarrow \mathrm{H}^{\ast}(\overline{A}(R)).$$
\end{lemma}

\begin{proof}
It follows from Proposition \ref{proposition:check-comparison-iso} that the map
$\overline{u}_{S} \colon R \Gamma_{\crys}( \Spec(S) ) \rightarrow \overline{A}(S)$ is an isomorphism
when $S$ is a regular Noetherian $\F_p$-algebra. Since the functor $\overline{A}(-)$ is a left Kan extension of its restriction to $\CAlg_{\F_p}^{\mathrm{reg}}$, the construction
$S \mapsto \overline{u}_{S}^{-1}$ extends to a natural transformation $w \colon \overline{A}(-) \rightarrow R \Gamma_{\crys}( - )$ which is a section of $\overline{u}$. It follows that
for every quasisyntomic $\F_p$-algebra $R$, the map of cohomology rings $\mathrm{H}^{\ast}_{\crys}( \Spec(R) ) \twoheadrightarrow \mathrm{H}^{\ast}(\overline{A}(R))$ also admits a section.
\end{proof}

\begin{proof}[Proof of Proposition \ref{proposition:rigid-A}]
Let $v \colon A(-) \rightarrow A(-)$ be a morphism of commutative algebra
objects of $D( \Fun( \CAlg^{\mathrm{reg}}_{\F_p}, \Ab ) )$, which we regard as an endomorphism of the functor
$A \colon \CAlg_{\F_p}^{\mathrm{reg}} \rightarrow \calD(\Z)$. Then $v$ admits an
essentially unique extension to a natural transformation $\overline{v} \colon \overline{A}(-) \rightarrow \overline{A}(-)$.
Let $R$ be a quasiregular semiperfect $\F_p$-algebra, so that $\overline{A}(R)$ can be identified with a commutative ring which is $p$-adically complete and $p$-torsion-free (Remark \ref{rasp}).
Let $I \subseteq \overline{A}(R)$ be the kernel of the projection map $\epsilon
\colon\overline{A}(R) \rightarrow R$, which we regard as a divided power ideal in $\overline{A}(R)$ (Lemma \ref{lemma:make-divided-powers}).
Since $\overline{A}(R)$ is $p$-torsion-free, the ring homomorphism
$\overline{v}_{R} \colon \overline{A}(R) \rightarrow \overline{A}(R)$ can be regarded as a morphism of divided power rings $( \overline{A}(R), I) \rightarrow (\overline{A}(R), I)$.
From the construction of the natural transformation $\overline{u}$ carried
out in the proof of Proposition~\ref{proposition:make-comparison-map2}, we see that the diagram
$$ \xymatrix{ & R \Gamma_{\crys}( \Spec(R) ) \ar[dl]_-{ \overline{u}_{R} } \ar[dr]^-{ \overline{u}_{R} } & \\
\overline{A}(R) \ar[rr]^{ \overline{v}_{R} } & & \overline{A}(R) }$$
is commutative. Since the vertical maps are surjective on cohomology (Lemma \ref{lemma:surjective-on-cohomology}), it follows that $\overline{v}_{R}$ is the identity map.
In other words, the natural transformation $\overline{v}$ is the identity when restricted to the category $\CAlg_{\F_p}^{\mathrm{qrsp}} \subseteq \CAlg_{\F_p}^{\mathrm{qs}}$ of quasiregular semiperfect $\F_p$-algebras.
Since the functor $\overline{A}(-)$ is a right Kan extension of its restriction to $\CAlg_{\F_p}^{\mathrm{qrsp}}$ (Lemma \ref{lemma:right-kan-from-qrsp}), it follows that $\overline{v}$ is the identity.
Restricting to regular Noetherian $\F_p$-algebras, we conclude that $v \colon A(-) \rightarrow A(-)$ is the identity.
\end{proof}

\newpage

\section{The Crystalline Comparison for $A\Omega$}
\label{cryscompsec}
\label{sec:crystallinecomp}

The main goal of this section is to provide a  simpler account of the
crystalline  comparison isomorphism for the $A\Omega$-complexes from \cite{BMS}
using the theory developed in this paper. In \cite{BMS}, one can find two proofs
of this result: one involving a comparison over Fontaine's $A_{\crys}$, and the
other involving relative de~Rham--Witt complexes for mixed characteristic rings.
By contrast, the proof presented here uses only the Hodge--Tate comparison and de~Rham--Witt complexes for smooth algebras over perfect fields of characteristic $p$.

\subsection{Review of the Construction of $A\Omega$}
\label{AOmegaReview}

We begin by recalling some pertinent details of the construction of \cite{BMS}. First, we introduce the basic rings involved.

\begin{notation}\label{AOmeganot}
Let $C/\mathbf{Q}_p$ denote a complete nonarchimedean extension, i.e., $C$ is the fraction field of $p$-adically complete rank $1$ valuation ring $\mathcal{O}_C$ that is $p$-torsion-free. Let $\mathfrak{m} \subseteq \mathcal{O}_C$ be the maximal ideal of $\mathcal{O}_{C}$,  $k = \mathcal{O}_C/\mathfrak{m}$ the residue field of $\mathcal{O}_{C}$, and $W = W(k)$ the ring of Witt vectors of $k$. We will assume that the field $C$ is {\it perfectoid}: that is, the valuation on $C$ is nondiscrete
and the Frobenius map $\varphi\colon  \mathcal{O}_{C} / p \mathcal{O}_{C} \rightarrow \mathcal{O}_{C} / p \mathcal{O}_{C}$ is surjective. Let $\mathcal{O}_{C}^{\flat}$ denote the tilt of $\mathcal{O}_{C}$, given by the inverse limit of the system
$$ \cdots \xrightarrow{\varphi} \mathcal{O}_{C} / p \mathcal{O}_{C} \xrightarrow{\varphi} \mathcal{O}_{C} / p \mathcal{O}_{C} \xrightarrow{\varphi} \mathcal{O}_{C} / p \mathcal{O}_{C}.$$
Let $A_{\inf}$ denote the ring of Witt vectors $W( \mathcal{O}_{C}^{\flat})$, and let $\varphi\colon  A_{\inf} \rightarrow A_{\inf}$ denote the automorphism of $A_{\inf}$ induced
by the Frobenius automorphism of $\mathcal{O}_{C}^{\flat}$. We let $\theta\colon  A_{\inf} \rightarrow \mathcal{O}_{C}$ denote the unique ring homomorphism for which the diagram
$$ \xymatrix{ A_{\inf} \ar[r]^{\theta} \ar[d] & \mathcal{O}_{C} \ar[d] \\
\mathcal{O}_{C}^{\flat} \ar[r] & \mathcal{O}_{C} / p \mathcal{O}_{C} }$$ commutes, and define $\tilde{\theta} := \theta \circ \varphi^{-1}$. 
Note that the map $A_{\inf} \xrightarrow{\tilde{\theta}} \mathcal{O}_C \to k$ lifts uniquely to a ring homomorphism $A_{\inf} \to W$, which we will sometimes refer to as the {\it (crystalline) specialization map}.  Unless otherwise specified, any map $A_{\inf} \to W$ is assumed to be the crystalline specialization map, and any map $A_{\inf} \to k$ is assumed to be its mod $p$ reduction, i.e., the map coming from $\tilde{\theta}$; see Remark~\ref{FrobeniusdRWComp} for a consequence of this choice.
\end{notation}

Next, we name certain elements that play an important role (see \cite[Example 3.16 and Proposition 3.17]{BMS}). 

\begin{notation}
In what follows, we will assume that the field $C$ contains a system of primitive $p^{n}$th roots of unity $\epsilon_{p^{n}} \in \mu_{p^{n}}(C)$ satisfying
$\epsilon_{p^n}^p = \epsilon_{p^{n-1}}$.  The system $\{ \epsilon_{p^{n}} \}_{n \geq 0}$ then determines an element $\underline{\epsilon} \in \mathcal{O}_{C}^{\flat}$; we normalize our choices to ensure that $\theta([\underline{\epsilon}]) = 1$ and $\tilde{\theta}([\underline{\epsilon}]) = \epsilon_p$. Let $[ \underline{\epsilon} ] \in A_{\inf}$ denote the Teichm\"{u}ller representative of $\underline{\epsilon}$ and set $\mu = [ \underline{\epsilon} ] - 1$, so that
$\varphi^{-1}(\mu)$ divides $\mu$. We will use the normalization of \cite{BMS}, so that 
$$ \xi =  \frac{\mu}{\varphi^{-1}(\mu)} = \frac{[\underline{\epsilon}]-1}{[\underline{\epsilon}^{1/p}]-1}$$
is a generator of the ideal $\ker(\theta) \subseteq A_{\inf}$, and
$$ \tilde{\xi} = \varphi(\xi) = \frac{ \varphi( \mu ) }{ \mu } = \frac{[\underline{\epsilon}^{p}]-1}{[\underline{\epsilon}]-1}$$
is a generator of the ideal $\ker( \tilde{\theta} ) \subseteq A_{\inf}$.
\end{notation}

Finally, we fix the main geometric object of interest.

\begin{notation}
Let $\mathfrak{X}$ be a smooth formal $\mathcal{O}_C$-scheme. We follow the same conventions for formal schemes as in \cite{BMS}: a smooth formal $\mathcal{O}_C$-scheme is one that locally has the form $\mathrm{Spf}(R)$ where $R$ is a $p$-adically complete and $p$-torsion-free $\mathcal{O}_C$-algebra equipped with the $p$-adic topology and for which $R/pR$ is smooth over $\mathcal{O}_C/p \mathcal{O}_C$. 
For such $R$, in this section we will write $\Omega^i_{R/\mathcal{O}_C}$ for the $p$-adically
completed module of $i$-forms (i.e., we suppress the  extra $p$-adic
completion).

For every commutative ring $\Lambda$, we let $\calD( \mathfrak{X}, \Lambda)$ denote the derived $\infty$-category of sheaves of $\Lambda$-modules on (the underlying topological space of) $\mathfrak{X}$,
and we let $D( \mathfrak{X}, \Lambda)$ denote its homotopy category (that is, the usual derived category of sheaves of $\Lambda$-modules on $\mathfrak{X}$).
Let $\mathfrak{X}_k$ denote the special fibre of $\mathfrak{X}$. As the topological spaces underlying $\mathfrak{X}$ and $\mathfrak{X}_k$ are identical, we can also view $D(\mathfrak{X},\Lambda)$ as the derived category $D(\mathfrak{X}_k,\Lambda)$ of sheaves of $\Lambda$-modules on $\mathfrak{X}_k$.
\end{notation}

The following summarizes the main construction from \cite{BMS}. 

\begin{construction}[The $A\Omega$-complex {\cite[\S 9]{BMS}}]\label{olosecon}
We let $A\Omega_{\mathfrak{X}}$
denote the object of the derived category $D(\mathfrak{X}, A_{\inf})$ given by the formula
\[ A\Omega_{\mathfrak{X}} := L\eta_\mu \left(R\nu_* A_{\inf,X}\right)\]
where $\nu\colon X_{proet} \to \mathfrak{X}$ is the nearby cycles map from the generic fibre $X$ of $\mathfrak{X}$. This is a commutative algebra object of $D(\mathfrak{X}, A_{\inf})$ which is equipped with a multiplicative isomorphism
\begin{equation}
\label{DividedFrobAOmega}
 \widetilde{\phi}_{\mathfrak{X}}\colon A\Omega_{\mathfrak{X}} \stackrel{\simeq}{\to} \varphi_* L\eta_{\tilde{\xi}} A\Omega_{\mathfrak{X}} 
 \end{equation}
that factors the commutative $A_{\inf}$-algebra map 
\[\phi_{\mathfrak{X}}\colon  A\Omega_{\mathfrak{X}} \to \varphi_* A\Omega_{\mathfrak{X}}\]
induced by $L\eta_\mu$ functoriality from the Frobenius automorphism $\phi_X$ of the complex $A_{\inf,X}$ on $X_{proet}$. We refer to $\widetilde{\varphi}_{\mathfrak{X}}$ informally as the {\it divided Frobenius} on $A\Omega_{\mathfrak{X}}$.
\end{construction}

Finally, let us recall the Hodge-Tate comparison theorem, which provides us with our main tool for controlling the local structure of $A\Omega$.

\begin{theorem}[Hodge-Tate comparison theorem]
\label{HodgeTateComp}
Let $\mathrm{Spf}(R) \subseteq \mathfrak{X}$ be an affine open subset. Write 
\[ A\Omega_R := R\Gamma(\mathrm{Spf}(R), A\Omega_{\mathfrak{X}}) \in \mathcal{D}(A_{\inf}) \quad \text{and} \quad \widetilde{\Omega}_R :=A\Omega_R \otimes_{A_{\inf},\ \widetilde{\theta}}^L \mathcal{O}_C \in \mathcal{D}(\mathcal{O}_C)\] 
for the displayed commutative algebra objects. Then there is a natural map of
commutative algebras $R \to \widetilde{\Omega}_R$ in $\mathcal{D}(\mathcal{O}_C)$, and the cohomology ring $H^*(\widetilde{\Omega}_R)$ can be identified with the exterior algebra $\bigwedge^*_R \Omega^1_{R/\mathcal{O}_C}$ as a graded $R$-algebra.
\end{theorem}

\begin{remark} 
We are implicitly suppressing all Breuil--Kisin twists above. 
\end{remark} 

\begin{proof}
This follows from \cite[Theorem 9.2 (i)]{BMS} (see \cite[Proposition 7.5 and Remark 7.11]{BhattSpecializing} for a simpler argument). More precisely, these references give an identification $A\Omega_{\mathfrak{X}} \otimes_{A_{\inf},\widetilde{\theta}}^L \mathcal{O}_C \simeq \widetilde{\Omega}_{\mathfrak{X}}$ in the derived $\infty$-category of sheaves, where the target is defined as $L\eta_{\epsilon_p-1} R\nu_* \mathcal{O}_X^+$. The claim above now follows from the calculation in \cite[Theorem 8.3]{BMS} by applying $R\Gamma(\mathrm{Spf}(R),-)$, and using the vanishing of coherent sheaf cohomology on formal affines. 

Alternately, we remark that in the sequel, we shall only use this result for small affine opens (see Definition~\ref{defsmall}), and here one can simply invoke \cite[Theorem 9.4 (i) and (iii)]{BMS} with $r=1$ and \cite[Theorem 8.7]{BMS}.
\end{proof}

\subsection{The First Formulation of the Main Comparison Theorem}

We follow the notation of \S \ref{AOmegaReview}. In particular, $\mathfrak{X}$ denotes a smooth formal $\mathcal{O}_C$-scheme. Our goal is to extract the de~Rham--Witt complex of $\mathfrak{X}_k$ (viewed as a sheaf of strict Dieudonn\'e algebras on $\mathfrak{X}_k$) from the pair $(A\Omega_{\mathfrak{X}}, \widetilde{\varphi}_{\mathfrak{X}})$ coming from Construction~\ref{olosecon}. A slightly imprecise version of this comparison can be formulated as follows, and gives a new proof of \cite[Theorem 1.10 (i)]{BMS}.

\begin{theorem}
\label{DRWAOmegaCoarse}
There is a natural identification $\varphi^*\left(A\Omega_{\mathfrak{X}} \widehat{\otimes}_{A_{\inf}} W\right) \simeq W\Omega^*_{\mathfrak{X}_k}$ of commutative algebras in $D(\mathfrak{X},W)$ that carries $\varphi_{\mathfrak{X},W}$ to $\varphi_{\mathfrak{X}_k}$. Moreover, this identification can be promoted to an isomorphism of presheaves of strict Dieudonn\'e $W$-algebras (see Theorem~\ref{dRWAOmegaMainComp} for a precise formulation).
\end{theorem}

There are roughly two steps in the (formulation and) proof of this result. First, in \S \ref{AOmegatoSDA}, we explain how to represent $A\Omega_{\mathfrak{X}} \widehat{\otimes}_{A_{\inf}}^L W$ by a presheaf of strict Dieudonn\'e $W$-algebras on a basis of the topology of $\mathfrak{X}$ (given by the small affine opens of Definition~\ref{defsmall}). Second, in \S \ref{AOmegaDRWComp}, we compare the aforementioned presheaf with the one given by the de~Rham--Witt complex using the universal property of the latter (Definition~\ref{def80}), and deduce Theorem~\ref{DRWAOmegaCoarse} using the Hodge-Tate comparison theorem.

\begin{remark}
\label{FrobeniusdRWComp}
Theorem~\ref{DRWAOmegaCoarse} has an extra Frobenius twist compared to \cite[Theorem 1.10 (i)]{BMS}. The reason for this discrepancy is that we have base changed along the map $A_{\inf} \to W$ coming from $\tilde{\theta}$ in formulating Theorem~\ref{DRWAOmegaCoarse} above, while \cite[Theorem 1.10 (i)]{BMS} uses the map $A_{\inf} \to W$ coming from $\theta$. Our choice of convention is dictated by the prismatic theory \cite{BhattScholzePrisms}, where it is more natural to have a Frobenius pullback present while relating to crystalline cohomology.
\end{remark}

\subsection{Extracting a Presheaf of Strict Dieudonn\'e Algebras from $A\Omega_{\mathfrak{X}}$}
\label{AOmegatoSDA}

We follow the notation of \S \ref{AOmegaReview}. In particular, $\mathfrak{X}$ denotes a smooth formal $\mathcal{O}_C$-scheme. In this section, we explain how to lift the object $A\Omega_{\mathfrak{X}} \widehat{\otimes}_{A_{\inf}} W \in D(\mathfrak{X},W)$ naturally to a presheaf of strict Dieudonn\'e $W$-algebras (Proposition~\ref{AOmegaSDAFinal}). The strategy roughly is to construct a divided Frobenius isomorphism for $A\Omega_{\mathfrak{X}} \widehat{\otimes}_{A_{\inf}} W$ from the one for $A\Omega_{\mathfrak{X}}$ (Proposition~\ref{GetSDC}), and then produce the desired rigidification by appealing to Corollary~\ref{1categorycor}. As in the local study of \cite{BMS}, we restrict attention to the following collection of affine open subsets of $\mathfrak{X}$.

\begin{definition}[Small affine opens]
\label{defsmall}
We will say that an open subset $U \subseteq \mathfrak{X}$ is {\it small} if it
is affine and admits an \'etale map to a (split) torus, i.e., we can write $U$ as a formal spectrum $\Spf(R)$, where $R$ is a $p$-adically complete and $p$-torsion free $\mathcal{O}_{C}$-algebra for which
there exists an \'etale $\mathcal{O}_C/p$-algebra map $(\mathcal{O}_C/p)[t_1^{\pm 1},...,t_d^{\pm 1}] \to R/pR$. In this case,
we will simply write $A\Omega_R$ for the commutative algebra object
\[ \RGamma( U, A \Omega_{\mathfrak{X}}|_{U} ) = \RGamma( \Spf(R), A \Omega_{\Spf(R)} )\] 
of the derived category $D( A_{\inf} )$.
\end{definition}

Let us introduce the relevant category of presheaves.

\begin{construction}[Presheaves on small affine opens]
\label{PresheafSmall}
Write $\mathcal{U}(\mathfrak{X})_{\sm}$ denote the poset of all small affine opens $U \subseteq \mathfrak{X}$, viewed as a category. For a commutative ring $\Lambda$, let $\calD(\Lambda)$ denote
the derived $\infty$-category of $\Lambda$, and let $\mathrm{Fun}(\mathcal{U}(\mathfrak{X})_{\sm}^{\op}, \calD(\Lambda))$ be the $\infty$-category of $\calD(\Lambda)$-valued presheaves on $\mathcal{U}(\mathfrak{X})_{\sm}$. 

Note that each $\sheafF \in \calD(\mathfrak{X}, \Lambda)$ defines an object of 
$i_{\Lambda}( \sheafF) \in \mathrm{Fun}(\mathcal{U}(\mathfrak{X})_{\sm}^{\op}, \calD(\Lambda))$ by the formula 
\[ i_{\Lambda}(\sheafF)(U) = \RGamma(U, \sheafF).\] 
This construction determines a fully faithful embedding from the derived category $D( \mathfrak{X}, \Lambda)$ to the homotopy category
of the $\infty$-category $\mathrm{Fun}(\mathcal{U}(\mathfrak{X})_{\sm}^{\op}, \calD(\Lambda))$ (see Remark~\ref{Kedlaya}).

The symmetric monoidal structure on the $\infty$-category $\calD( \Lambda )$ induces a symmetric monoidal structure on the functor $\infty$-category 
$\mathrm{Fun}(\mathcal{U}(\mathfrak{X})_{\sm}^{\op}, \calD(\Lambda))$, given by pointwise (derived) tensor product over $\Lambda$. In particular, we have a natural identification
\[ \mathrm{CAlg}(\mathrm{Fun}(\mathcal{U}(\mathfrak{X})_{\sm}^{\op}, \calD(\Lambda))) \simeq \mathrm{Fun}(\mathcal{U}(\mathfrak{X})_{\sm}^{\op}, \mathrm{CAlg}(\calD(\Lambda))) \]
of $\infty$-categories of commutative algebra objects. Similar remarks apply to variants involving presheaves of more structured objects.

We say that an object $F \in \mathrm{Fun}(\mathcal{U}(\mathfrak{X})_{\sm}^{\op}, \calD(\Lambda))$ is {\it $p$-complete} if $F(U)$ is $p$-complete, for each $U \in \mathcal{U}( \mathfrak{X})_{\sm}$.
\end{construction}

\begin{remark}
The $\infty$-category $\mathrm{Fun}(\mathcal{U}(\mathfrak{X})_{\sm}^{\op}, \calD(\Lambda))$ can be identified with the derived $\infty$-category of the ordinary abelian category of
presheaves of $\Lambda$-modules on $\mathcal{U}(\mathfrak{X})_{\sm}$. In particular, we can identify the homotopy category of $\mathrm{Fun}(\mathcal{U}(\mathfrak{X})_{\sm}^{\op}, \calD(\Lambda))$
with the derived category of $D(\mathcal{U}(\mathfrak{X})_{\sm}, \Lambda)$ where we endow the poset $\mathcal{U}(\mathfrak{X})_{\sm}$ with the indiscrete Grothendieck topology (so that all presheaves are sheaves).
Since $\mathcal{U}(\mathfrak{X})_{\sm}$ possesses a more natural Grothendieck topology coming from the Zariski topology of $\mathfrak{X}$, we prefer to not use this notation.
\end{remark}

\begin{remark}
\label{smallMatch}
In analogy with Definition~\ref{defsmall}, let us say that an open subset $V \subseteq \mathfrak{X}_k$ is {\em small} if it is affine and admits an \'etale map to a torus. Using the fact that \'etale morphisms are finitely presented, it is easy to see that the identification of $\mathfrak{X}_k \simeq \mathfrak{X}$ of topological spaces preserves the notion of small affine opens, i.e., an open $U \subseteq \mathfrak{X}$ is small if and only if $U_k \subseteq \mathfrak{X}_k$ is so. Consequently, we can identify the poset $\mathcal{U}(\mathfrak{X})_{\sm}$ of small affine opens in $\mathfrak{X}$ with the corresponding poset $\mathcal{U}(\mathfrak{X}_k)_{\sm}$ for the special fibre $\mathfrak{X}_k$. 
\end{remark}

\begin{remark}[Sheaves and presheaves]
\label{Kedlaya}
The collection of small affine opens forms a basis for the topology of $\mathfrak{X}$. In particular, if one equips the category $\mathcal{U}(\mathfrak{X})$ with the standard Grothendieck topology (i.e., the one inherited from the Zariski topology on $\mathfrak{X}$), then the resulting category of sheaves is identified with the category of sheaves on $\mathfrak{X}$. In particular, the construction
\[ i_{\Lambda}\colon  \calD(\mathfrak{X}, \Lambda) \to \mathrm{Fun}(\mathcal{U}(\mathfrak{X})_{\sm}^{\op}, \calD(\Lambda))\]
of Construction~\ref{PresheafSmall} determines a fully faithful embedding from the derived $\infty$-category of sheaves on $\mathfrak{X}$ to the derived $\infty$-category of presheaves on $\mathcal{U}(\mathfrak{X})_{\sm}$. This has a left adjoint 
\[ G_{\Lambda}\colon \mathrm{Fun}(\mathcal{U}(\mathfrak{X})_{\sm}^{\op}, \calD(\Lambda)) \to D(\mathfrak{X}, \Lambda) \]
given by sheafification for the Zariski topology. 

For future reference, we also record a variant of this observation for $p$-complete objects. Let $\widehat{D}(\mathfrak{X},W) \subseteq \calD(\mathfrak{X},W)$ denote the full subcategory spanned by $p$-complete objects, and $\mathrm{Fun}(\mathcal{U}(\mathfrak{X})_{\sm}^{\op}, \widehat{\calD(W)}) \subseteq \mathrm{Fun}(\mathcal{U}(\mathfrak{X})_{\sm}^{\op}, \calD(W))$ denote the full subcategory spanned by presheaves valued in the full subcategory $\widehat{\mathcal{D}}(W) \subseteq \calD(W)$ of $p$-complete objects in $\calD(W)$. Then the functor $i_W$ discussed above restricts to a fully faithful functor
\[ \widehat{i}_W\colon \widehat{\mathcal{D}}(\mathfrak{X},W) \to \mathrm{Fun}(\mathcal{U}(\mathfrak{X})_{\sm}^{\op}, \widehat{\mathcal{D}}(W))\]
which admits a left adjoint 
\[ \widehat{G}_W\colon \mathrm{Fun}(\mathcal{U}(\mathfrak{X})_{\sm}^{\op}, \widehat{\mathcal{D}}(W)) \to \widehat{\mathcal{D}}(\mathfrak{X},W)\]
given on objects by composing the sheafification functor $G_W$ mentioned above with the derived $p$-completion functor $\calD(\mathfrak{X},W) \to \widehat{\mathcal{D}}(\mathfrak{X},W)$. Unwinding definitions, we can write $\widehat{G}_W(K) = \varprojlim G_W(K \otimes_W^L W/p^n)$.
\end{remark}

Let us write $A\Omega_{\mathfrak{X}}^{\sm} \in \mathrm{Fun}(\mathcal{U}(\mathfrak{X})_{\sm}^{\op}, \calD(A_{\inf}))$ for the complex $A\Omega_{\mathfrak{X}}$ of Construction \ref{olosecon}
regarded as a presheaf: that is, the image of $A\Omega_{\mathfrak{X}} \in \calD(\mathfrak{X},A_{\inf})$ under the functor $i_{A_{\inf}}$ of Construction~\ref{PresheafSmall}. Beware
that composition with the functor $L \eta_{\tilde{\xi} }$ does not preserve the essential image of the functor $i_{ A_{\inf} }$: that is, it generally does not carry sheaves to sheaves.
Nevertheless, we have the following:
 
 \begin{proposition}
 \label{DividedFrobSmall}
 The Frobenius map $\widetilde{\varphi}_{\mathfrak{X}}$ on $A\Omega_{\mathfrak{X}}$ lifts to an identification
\[ \widetilde{\varphi}_{\mathfrak{X}}^{\sm}\colon A\Omega_{\mathfrak{X}}^{\sm} \simeq \varphi_* L\eta_{\tilde{\xi}} A\Omega_{\mathfrak{X}}^{\sm} \]
of commutative algebras in $\mathrm{Fun}(\mathcal{U}(\mathfrak{X})_{\sm}^{\op}, D(A_{\inf}))$. 
 \end{proposition}
 
 The proof below crucially uses the smallness restriction in the definition of $\mathcal{U}(\mathfrak{X})_{\sm}$. 
 
\begin{proof} 
By definition, for each small affine open $U \subseteq \mathfrak{X}$, we have $A\Omega_{\mathfrak{X}}^{\sm}(U) := \RGamma(U, A\Omega_{\mathfrak{X}})$. By \cite[Theorem 9.4]{BMS}, we have
 \[ A\Omega_{\mathfrak{X}}^{\sm}(U) \simeq L\eta_\mu \RGamma(U, A_{\inf,U}).\]
In particular, the Frobenius map $\widetilde{\varphi}_{\mathfrak{X}}$ refines to give an identification
\[  A\Omega_{\mathfrak{X}}^{\sm}(U) \simeq \varphi_* L\eta_{\tilde{\xi}} A\Omega_{\mathfrak{X}}^{\sm}(U)\]
for all small affine $U$, and hence we have an identification
\[ \varphi_{\mathfrak{X}}^{\sm}\colon A\Omega_{\mathfrak{X}}^{\sm} \simeq \varphi_* L\eta_{\tilde{\xi}} A\Omega_{\mathfrak{X}}^{\sm} \]
of commutative algebras in $\mathrm{Fun}(\mathcal{U}(\mathfrak{X})_{\sm}^{\op}, \calD(A_{\inf}))$, as desired.
\end{proof}

We can now introduce the crystalline specialization of $A\Omega$ with its divided Frobenius map.

\begin{construction}[The crystalline specialization]
\label{FrobBaseChange}
Let 
\[ A\Omega_{\mathfrak{X},W}^{\sm} := A\Omega_{\mathfrak{X}}^{\sm} \widehat{\otimes}^L_{A_{\inf}} W \in \mathrm{Fun}(\mathcal{U}(\mathfrak{X})_{\sm}^{\op}, \calD(W))\]
be the $\calD(W)$-valued presheaf on $\mathcal{U}(\mathfrak{X})_{\sm}$ defined by $A\Omega_{\mathfrak{X}}^{\sm}$ via $p$-adically completed base change along $A_{\inf} \to W$. Write $A\Omega_{R,W}$ for its value on a small affine open $\mathrm{Spf}(R) \subseteq \mathfrak{X}$, so $A\Omega_{R,W} := A\Omega_R \widehat{\otimes}_{A_{\inf}}^L W$. The map $\widetilde{\varphi}_{\mathfrak{X}}^{\sm}$ induces an isomorphism
\begin{equation}
\label{FBC1}
A\Omega_{\mathfrak{X},W}^{\sm}  \simeq \big(\varphi_* L\eta_{\tilde{\xi}} A\Omega_{\mathfrak{X}}^{\sm}\big) \widehat{\otimes}_{A_{\inf}}^L W.
\end{equation}
Since the crystalline specialization map $A_{\inf} \rightarrow W$ carries $\tilde{\xi}$ to $p$, the natural map 
$A\Omega_{\mathfrak{X}}^{\sm} \to A\Omega_{\mathfrak{X},W}^{\sm}$ induces by functoriality a commutative $A_{\inf}$-algebra map
\[ \varphi_* L\eta_{\tilde{\xi}} A\Omega_{\mathfrak{X}}^{\sm} \to \varphi_* L\eta_p A\Omega_{\mathfrak{X},W}^{\sm}. \]
As the codomain of the preceding map is $p$-complete, we obtain a map
\begin{equation}
\label{FBC2} 
\big(\varphi_* L\eta_{\tilde{\xi}} A\Omega_{\mathfrak{X}}^{\sm}\big) \widehat{\otimes}^L_{A_{\inf}} W \to \varphi_* L\eta_p A\Omega_{\mathfrak{X},W}^{\sm}
\end{equation}
Composing with the isomorphism in \eqref{FBC1}, we obtain a map
\[ \widetilde{\varphi}_{\mathfrak{X},W}^{\sm}\colon  A\Omega_{\mathfrak{X},W}^{\sm} \to \varphi_* L\eta_p A\Omega_{\mathfrak{X},W}^{\sm}\]
of commutative algebra objects of $\mathrm{Fun}(\mathcal{U}(\mathfrak{X})_{\sm}^{\op}, \widehat{\calD(W)} )$; we will refer to this map as
the {\it divided Frobenius} on $A\Omega_{\mathfrak{X},W}^{\sm}$. 
\end{construction}

To justify the ``divided Frobenius'' terminology, we prove the following:

\begin{lemma}
\label{AOmegaCoconn}
The presheaf $A\Omega_{\mathfrak{X},W}^{\sm} \otimes^L_W W/p^n$ takes coconnective values for all $n \geq 1$, i.e., the values of these sheaves lie in $D^{\geq 0}$. Consequently, there is a natural map 
\[ \varphi_* L\eta_p A\Omega_{\mathfrak{X},W}^{\sm} \to \varphi_* A\Omega_{\mathfrak{X},W}^{\sm}\]
of commutative algebras in $\mathrm{Fun}(\mathcal{U}(\mathfrak{X})_{\sm}^{\op}, \widehat{\mathcal{D}}(W))$.
\end{lemma}
\begin{proof}
Proceeding by induction on $n$, we can reduce to the case $n=1$. We must show that for every small affine open $\mathrm{Spf}(R) \subseteq \mathfrak{X}$, the object 
$$(A\Omega_R \otimes_{A_{\inf}}^L W) \otimes_W^L W/p \simeq A\Omega_R \otimes_{A_{\inf}}^L k$$
is coconnective. By the Hodge-Tate comparison (Theorem~\ref{HodgeTateComp}), we have
\[ A\Omega_R \otimes^L_{A_{\inf},\widetilde{\theta}} \mathcal{O}_C \simeq \widetilde{\Omega}_R,\]
where $\widetilde{\Omega}_R$ is a coconnective $R$-complex with cohomology groups given by $\HH^i(\widetilde{\Omega}_R) \simeq \Omega^i_{R/\mathcal{O}_C}$. The claim now follows by extending scalars along the map $\mathcal{O}_C \to k$ (and observing that each $\Omega^i_{R/\mathcal{O}_C}$ is a locally free $R$-module, hence flat over $\mathcal{O}_C$). 

The second assertion is immediate from the first; here we implicitly use that there is a symmetric monoidal natural transformation $L\eta_p(K) \to K$ defined on the full subcategory of $\calD(W)$ spanned by those objects $K$ for which $k \otimes_{W}^{L} K \in \calD(k)$ is coconnective, see \cite[Lemma 6.10]{BMS}.
\end{proof}

Composing the divided Frobenius $\widetilde{\varphi}_{\mathfrak{X},W}^{\sm}$ from Construction~\ref{FrobBaseChange} with the map from Lemma~\ref{AOmegaCoconn} gives a map 
\[ \varphi_{\mathfrak{X},W}^{\sm}\colon A\Omega_{\mathfrak{X},W}^{\sm} \to \varphi_* A\Omega_{\mathfrak{X},W}^{\sm}\]
of commutative algebras; we refer to this map as the {\it Frobenius} on $A\Omega_{\mathfrak{X},W}^{\sm}$.

\begin{remark}
\label{AOmegaSmallSheaf}
The object $A\Omega_{\mathfrak{X},W}^{\sm}$ from Construction~\ref{FrobBaseChange} belongs to the full subcategory $\mathrm{Fun}(\mathcal{U}(\mathfrak{X})_{\sm}^{\op}, \widehat{\mathcal{D}}(W)) \subseteq \mathrm{Fun}(\mathcal{U}(\mathfrak{X})_{\sm}^{\op}, \calD(W))$ of $p$-complete objects introduced in Remark~\ref{Kedlaya}. The completed sheafification functor $\widehat{G}_W$ of Remark~\ref{Kedlaya} carries $A\Omega_{\mathfrak{X},W}^{\sm}$ to the object $A\Omega_{\mathfrak{X},W} := A\Omega_{\mathfrak{X}} \widehat{\otimes}_{A_{\inf}}^L W \in \widehat{\mathcal{D}}(\mathfrak{X},W) \subseteq \calD(\mathfrak{X},W)$: this follows easily from the formula for $\widehat{G}_W$ given at the end of Remark~\ref{Kedlaya}.
\end{remark}

Let us now show that the divided Frobenius $\widetilde{\varphi}_{\mathfrak{X},W}^{\sm}$ constructed above is an isomorphism; this relies on the Hodge-Tate comparison theorem.

\begin{proposition}
\label{GetSDC}
The divided Frobenius map $\widetilde{\varphi}_{\mathfrak{X},W}^{\sm}$ of Construction~\ref{FrobBaseChange} is an isomorphism of commutative algebra objects of
$\mathrm{Fun}(\mathcal{U}(\mathfrak{X})_{\sm}^{\op}, \widehat{\calD(W)} )$.
\end{proposition}

\begin{proof}
It is enough to show that the natural map
\[ \big(\varphi_* L\eta_{\tilde{\xi}} A\Omega_{\mathfrak{X}}^{\sm}\big) \widehat{\otimes}^L_{A_{\inf}} W \to \varphi_* L\eta_p A\Omega_{\mathfrak{X},W}^{\sm} \]
appearing in Construction~\ref{FrobBaseChange} as \eqref{FBC2} is a quasi-isomorphism. Unwinding definitions (and neglecting Frobenius twists),  we must show that the natural map
\[ \big(L\eta_{\tilde{\xi}} A\Omega_{\mathfrak{X}}^{\sm}\big) \widehat{\otimes}^L_{A_{\inf}} W \to  L\eta_{\tilde{\xi}} (A\Omega_{\mathfrak{X}}^{\sm} \widehat{\otimes}^L_{A_{\inf}} W)\]
is a quasi-isomorphism, i.e., that $L\eta_{\tilde{\xi}}$ commutes with $p$-completed tensor product $W$; here we replace $p$ with $\tilde{\xi}$ since they have the same image in $W$. Fix a small affine open $\mathrm{Spf}(R) \subseteq \mathfrak{X}$; we wish to show that the map
\[ (L\eta_{\tilde{\xi}} A\Omega_R) \widehat{\otimes}_{A_{\inf}}^L W \to L\eta_{\tilde{\xi}} \big(A\Omega_R \widehat{\otimes}_{A_{\inf}}^L W)\]
is a quasi-isomorphism. 

Observe that we can write $W$ as $p$-adic completion of $\varinjlim_r A_{\inf}/\phi^{-r}(\mu)$. Moreover, since the equation $\tilde{\xi} = p$ holds in each quotient ring $A_{\inf}/\phi^{-r}(\mu)$, we can also realize $W$ as the $\tilde{\xi}$-adic completion of $\varinjlim_r A_{\inf}/\phi^{-r}(\mu)$. Let us first show that for $r \geq 0$, the natural map $A_{\inf} \to A_{\inf}/\phi^{-r}(\mu)$ induces a quasi-isomorphism
\[ (L\eta_{\tilde{\xi}} A\Omega_{R}) \otimes^L_{A_{\inf}} A_{\inf}/\phi^{-r}(\mu) \simeq L\eta_{\tilde{\xi}}\big(A\Omega_{R} \otimes^L_{A_{\inf}} A_{\inf}/\phi^{-r}(\mu)\big).\]
Using the criterion in \cite[Lemma 5.16]{BhattSpecializing}, it is enough to
show that the $\mathcal{O}_C$-modules $\mathrm{H}^i(A\Omega_{R}
\otimes^L_{A_{\inf},\tilde{\theta}} \mathcal{O}_C)$ have no
$\phi^{-r}(\mu)$-torsion for all $i$. Since the ring homomorphism
$\tilde{\theta}$ carries $\phi^{-r}(\mu)$ to the non-zero-divisor
$\epsilon_{p^{r+1}}-1 \in \mathcal{O}_C$, it enough to show that these cohomology
groups are flat over $\mathcal{O}_C$. This is immediate from the Hodge-Tate
comparison (Theorem~\ref{HodgeTateComp}), since each
$\Omega^i_{R/\mathcal{O}_C}$ is a projective $R$-module of finite rank (and therefore
$\mathcal{O}_C$-flat).

Since the functors $L\eta_{\tilde{\xi}}$ and $\otimes^L$ both preserve filtered colimits, we get a natural identification
\[ (L\eta_{\tilde{\xi}} A\Omega_{R}) \otimes^L_{A_{\inf}} \varinjlim_r A_{\inf}/\phi^{-r}(\mu) \simeq L\eta_{\tilde{\xi}}\big(A\Omega_{R} \otimes^L_{A_{\inf}} \varinjlim_r A_{\inf}/\phi^{-r}(\mu)\big)\]
in the $\infty$-category $\calD( A_{\inf} )$. Proposition \ref{GetSDC} now follows by applying derived $\tilde{\xi}$-adic completions to both sides, since $W$ is the derived $\tilde{\xi}$-adic completion of $\varinjlim_r A_{\inf}/\phi^{-r}(\mu)$ and the functor $L\eta_{\tilde{\xi}}(-)$ commutes with derived $\tilde{\xi}$-adic completions (see \cite[Lemma 6.20]{BMS}).
\end{proof}

We now use the divided Frobenius map $\widetilde{\varphi}_{\mathfrak{X},W}^{\sm}$ and the equivalence of Corollary~\ref{1categorycor} to promote $A\Omega_{\mathfrak{X},W}^{\sm}$ to a presheaf taking values in the ordinary category of strict Dieudonn\'e algebras. 

\begin{construction}
\label{AOmegaSDC}
By Proposition~\ref{GetSDC}, the pair $(A\Omega_{\mathfrak{X},W}^{\sm}, \widetilde{\varphi}_{\mathfrak{X},W}^{\sm})$ provides a lift of $A\Omega_{\mathfrak{X},W}^{\sm}$ along the forgetful functor
\[ \mathrm{Fun}(\mathcal{U}(\mathfrak{X})_{\sm}^{\op}, \widehat{\calD(W)}^{\varphi_* L\eta_p}) \simeq \mathrm{Fun}(\mathcal{U}(\mathfrak{X})_{\sm}^{\op}, \widehat{\calD(W)})^{\varphi_*  L\eta_p} \to\mathrm{Fun}(\mathcal{U}(\mathfrak{X})_{\sm}^{\op}, \calD(W)),\]
where we use the notation from Example~\ref{toko} (applied to $R = W$ and $\sigma = \varphi$). Applying Example~\ref{toko}, we can identify this lift with a presheaf
\[ A\Omega_{\mathfrak{X},W}^{\sm,\ast} \in \mathrm{Fun}(\mathcal{U}(\mathfrak{X})_{\sm}^{\op}, \mathrm{Mod}_W(\FrobCompComplete))\] 
taking values in the ordinary category of $W$-module objects in the category $\FrobCompComplete$ of strict Dieudonn\'e complexes. As all our constructions are compatible with the commutative algebra structure on $A\Omega_{\mathfrak{X}}$, it follows that $A\Omega_{\mathfrak{X},W}^{\sm,\ast}$ is a commutative algebra object in 
$\mathrm{Fun}(\mathcal{U}(\mathfrak{X})_{\sm}^{\op}, \mathrm{Mod}_W(\FrobCompComplete))$; see Example~\ref{torch} for four equivalent descriptions of this category.
\end{construction}

\begin{definition}
Let us regard the ring $W = W(k)$ of Witt vectors as a strict Dieudonn\'{e}
algebra (Example \ref{ex24}). A {\it strict Dieudonn\'{e} $W$-algebra} is a strict Dieudonn\'{e} algebra $A^{\ast}$
equipped with a morphism of Dieudonn\'{e} algebras $W \rightarrow A^{\ast}$. We let $(\FrobAlgComplete)_{W/}$ denote the category of strict Dieudonn\'{e} $W$-algebras, which
we will view as a full subcategory of the category $\CAlg( \Mod_{W}( \FrobComp) )$ of commutative $W$-algebras in $\FrobComp$.
\end{definition}

\begin{remark}\label{sulfi}
By virtue of Remark \ref{algdesc} and Proposition \ref{prop35}, a commutative algebra object $A^{\ast}$ of $\Mod_{W}( \FrobComp )$ belongs to $(\FrobAlgComplete)_{W/}$
if and only if the underlying Dieudonn\'{e} complex is strict, the groups $A^{n}$ vanish for $n < 0$, and the Frobenius map
satisfies the congruence $Fx \equiv x^{p} \pmod{ VA^{0} }$ for each $x \in A^{0}$.
\end{remark}

\begin{proposition}\label{AOmegaSDAFinal}
The presheaf $A\Omega_{\mathfrak{X},W}^{\sm,\ast}$ of Construction~\ref{AOmegaSDC} takes values in the
full subcategory $(\FrobAlgComplete)_{W/} \subseteq \CAlg( \Mod_{W}(\FrobComp) )$ of strict Dieudonn\'{e} $W$-algebras.
\end{proposition}

To prove Proposition \ref{AOmegaSDAFinal}, we must show that for every small open $U \subseteq \mathfrak{X}$, the
complex $A\Omega_{\mathfrak{X},W}^{\sm, \ast}(U)$ is a strict Dieudonn\'{e} $W$-algebra: that is, it satisfies
the requirements of Remark \ref{sulfi}. Note that $A\Omega_{\mathfrak{X},W}^{\sm, \ast}(U)$ is strict by construction,
and the groups $A\Omega_{\mathfrak{X},W}^{\sm, n}(U)$ vanish for $n < 0$ by virtue of Lemma~\ref{AOmegaCoconn} (together with
Proposition \ref{prop11}). Using Proposition~\ref{prop11} again, we are then reduced to proving the following:

\begin{proposition}
\label{AOmegaSDA}
Let $U \subseteq \mathfrak{X}$ be a small open. Then the Frobenius endomorphism of the $\mathbf{F}_p$-algebra $\HH^0(A\Omega_{\mathfrak{X},W}^{\sm,\ast}(U)/p)$ coincides with the map induced by the structure map $F\colon  A\Omega_{\mathfrak{X},W}^{\sm,\ast}(U) \to A\Omega_{\mathfrak{X},W}^{\sm,\ast}(U)$. 
\end{proposition}

\begin{proof}
Write $U = \mathrm{Spf}(R)$. Unwinding the definitions, we must check that the Frobenius endomorphism of the $\mathbf{F}_p$-algebra $\HH^0(A\Omega_R \otimes_{A_{\inf}}^L k)$ coincides with the one induced by the tensor product of the map $\varphi_R\colon  A\Omega_R \to A\Omega_R$ and the Frobenius map $\varphi_k\colon  k \to k$.
Using the smallness of $R$, we can choose a map $R \to R_\infty$ as in \cite[Definitions 8.5 and 8.6]{BMS}. In particular, $R_\infty$ is perfectoid, the map $R \to R_\infty$ is faithfully flat modulo $p$ and a pro-\'etale $\Gamma$-torsor after inverting $p$ (where $\Gamma = \mathbf{Z}_p(1)^{\oplus d}$), and $A\Omega_R \simeq L\eta_\mu R\Gamma(\Gamma, A_{\inf}(R_\infty))$ (by  \cite[Theorem 9.4 (iii)]{BMS}). We therefore obtain a $\phi$-equivariant homomorphism
\[ \eta\colon A\Omega_R \to A_{\inf}(R_\infty). \]
Recall that if $S$ is a perfectoid $\mathcal{O}_C$-algebra, then there is a Frobenius-equivariant isomorphism $A_{\inf}(S) \widehat{\otimes}^L_{A_{\inf}} W \simeq W(S_k)$ (see \cite[Lemma 3.13]{BMS}). Applying this to $S = R_\infty$, we obtain a $\phi$-equivariant map
\[ \mathrm{H}^0(\eta \otimes_{A_{\inf}}^L k)\colon  \mathrm{H}^0(A\Omega_R \otimes_{A_{\inf}}^L k) \to R_{\infty,k}.\]
As $\phi$ acts as the Frobenius on $R_{\infty,k}$, it also acts as the Frobenius on any $\phi$-stable subring. Consequently, to prove Proposition \ref{AOmegaSDA}, it will
suffice to show that the map $\mathrm{H}^0(\eta \otimes_{A_{\inf}}^L k)$ is injective. For this, we use the Hodge-Tate comparison isomorphism (Theorem~\ref{HodgeTateComp}), which
carries $\mathrm{H}^0(\eta \otimes_{A_{\inf}}^L k)$ to the map obtained by applying
the functor $\mathrm{H}^{0}$ to the map of chain complexes
\[ \widetilde{\Omega}_R \otimes_{\mathcal{O}_C}^L k \to R_{\infty,k}.\]
We now observe that on $0$th cohomology groups, this can be identified with the map $R_k \to R_{\infty,k}$ (using \cite[Theorem 8.7]{BMS} to identify the left hand side, and the construction of $\eta$ to identify the map). Since $R \to R_\infty$ is faithfully flat modulo $p$, it is also faithfully flat (hence injective) after extending scalars to $k$.
\end{proof}

\subsection{Comparison with the de~Rham--Witt Complex}
\label{AOmegaDRWComp}

We follow the notation of \S \ref{AOmegaReview}. Let $\mathfrak{X}$ be a smooth formal $\mathcal{O}_C$-scheme, and let $A\Omega_{\mathfrak{X},W}^{\sm,\ast}$ be the presheaf
of strict Dieudonn\'{e} $W$-algebras given by Proposition~\ref{AOmegaSDAFinal}. Our goal is to prove Theorem~\ref{DRWAOmegaCoarse} by identifying
$A\Omega_{\mathfrak{X},W}^{\sm, \ast}$ with the de~Rham--Witt complex of the special fiber $\mathfrak{X}_{k}$. 

\begin{notation}\label{dRWonX}
Let $\mathcal{O}_{ \mathfrak{X}_k}$ denote the structure sheaf of the $k$-scheme $\mathfrak{X}_{k}$ and let $\mathcal{O}_{\mathfrak{X}_k}^{\sm}$ denote its
restriction to the category $\mathcal{U}(\mathfrak{X})_{\sm} \simeq \mathcal{U}(\mathfrak{X}_k)_{\sm}$ of small open subsets of $\mathfrak{X}_k$. We let $W\Omega^{\sm,*}_{\mathfrak{X}_k}$ denote the presheaf of strict Dieudonn\'e $W$-algebras on $\mathcal{U}(\mathfrak{X})_{\sm}$ given by the formula
\[W\Omega^{\sm,*}_{\mathfrak{X}_k}(U) =  W\Omega^*_{\mathcal{O}(U_k)}.\] 
For each integer $i$, let $\Omega^{\sm,i}_{\mathfrak{X}_k}$ be the presheaf on $\mathcal{U}(\mathfrak{X})_{\sm}$ given by $U \mapsto \Omega^i_{\mathcal{O}(U_k)}$. Write \[ \eta_{W\Omega}\colon \mathcal{O}_{\mathfrak{X}_k}^{\sm} \to \varphi_* H^0(W\Omega^{\sm,*}_{\mathfrak{X}_k}/p)\]
for the natural $k$-algebra map induced by the degree $0$ case of the Cartier isomorphism in Proposition~\ref{prop11} (applicable since $\mathfrak{X}_k$ is smooth over $k$).
\end{notation}

\begin{remark}
\label{dRWSheafify}
Note that passage to the derived $\infty$-category (that is, composing with the forgetful functor $(\FrobAlgComplete)_{W/} \to \widehat{\calD(W)}$), carries the presheaf $W\Omega^{\sm,\ast}_{\mathfrak{X}_k}$ of Notation~\ref{dRWonX} to a commutative algebra object $W\Omega_{\mathfrak{X}_k}^{\sm,*} \in \mathrm{Fun}(\mathcal{U}(\mathfrak{X})_{\sm}^{\op},\widehat{\calD(W)})$ equipped with a Frobenius endomorphism. Applying the completed sheafification functor $\widehat{G}_W$ of Remark~\ref{Kedlaya} to this object yields the classical de~Rham--Witt complex $W\Omega_{\mathfrak{X}_k}^* \in \widehat{\calD}(\mathfrak{X}_k,W)$. 
\end{remark}

Our proof of Theorem~\ref{DRWAOmegaCoarse} will use the following recognition criterion for the de~Rham--Witt complex, which may be of independent interest:

\begin{proposition}
\label{dRWCrit}
Let $A^*$ be a presheaf of strict Dieudonn\'e $W$-algebras on $\mathcal{U}(\mathfrak{X}_k)_{\sm}$ equipped with a map 
\[ \eta_A\colon \mathcal{O}_{\mathfrak{X}_k}^{\sm} \to \varphi_* \HH^0(A/p A).\] 
of commutative $k$-algebras. Then $\HH^*(A/pA)$ has the structure of a presheaf of commutative differential graded algebras over $k$, with differential given by the Bockstein map
associated to $p$, and $\eta_A$ extends naturally to a map
\[ \widetilde{\eta}_A\colon  \Omega^{\sm,*}_{\mathfrak{X}_k} \to \varphi_* \HH^*(A/pA)\]
of presheaves of commutative differential graded algebras. If $\widetilde{\eta}_A$ is an isomorphism, then there is a unique isomorphism $W\Omega^{\sm,*}_{\mathfrak{X}_k} \simeq A^*$ that intertwines $\eta_{W\Omega}$ with $\eta_A$. 
\end{proposition}

\begin{proof}
Write $\beta_p$ for the Bockstein differential on $\HH^*(A/pA)$. Recall from Proposition~\ref{prop62} that we can naturally identify $(\HH^*(A/pA),\beta_p)$ with $(\eta_p A^*)/p$. Since this isomorphism is multiplicative, we learn that $(\HH^*(A/pA),\beta_p)$ is a presheaf of strictly commutative differential graded algebras with vanishing terms in degrees $< 0$. The universal property of the de Rham complex then yields the map $\widetilde{\eta}_A$ extending the map $\eta_A$. 

On the other hand, by the universal property of the de~Rham--Witt complex (Definition~\ref{def80} and Theorem~\ref{maintheoC}) and the Cartier isomorphism (Proposition~\ref{prop11}) $\mathcal{W}_1(A^*) \to \varphi_* \HH^*(A/pA)$ for $A^*$, the map $\eta_A$ lifts to uniquely to a map 
\[ \Psi\colon W\Omega^{\sm,*}_{\mathfrak{X}_k} \to A^*\] 
which intertwines $\eta_A$ with $\eta_{W\Omega}$. 

Now recall that the classical Cartier isomorphism gives an identification 
\[ C:\Omega^*_{\mathfrak{X}_k} \xrightarrow{\simeq} \varphi_* \HH^*(W\Omega^{\sm}_{\mathfrak{X}_k}/p)\] 
that intertwines the de~Rham differential with the Bockstein differential (e.g., by Proposition~\ref{prop11} for $W\Omega_{\mathfrak{X}_k}^{\sm,\ast}$). The map $\HH^*(\Psi/p)$ intertwines the Bockstein differentials, and coincides with the map $\widetilde{\eta}_A$ constructed above in degree $0$ by construction under the Cartier isomorphism. By the universal property of the de~Rham complex, we obtain an equality 
\[ \varphi_* \HH^*(\Psi/p) \circ C = \varphi_* \HH^*(\widetilde{\eta})\]
of maps $\Omega^*_{\mathfrak{X}_k} \to \varphi_* \HH^*(A/pA)$. Now $C$ is an isomorphism by construction while $\varphi_* \HH^*(\widetilde{\eta})$ is an isomorphism by assumption. It follows from the formula above that $\HH^*(\Psi/p)$ is an isomorphism, i.e., $\Psi/p$ is a quasi-isomorphism. 
Since the domain and codomain of $\Psi$ are $p$-torsion-free and $p$-adically complete, it follows that $\Psi$ is a quasi-isomorphism. Applying Corollary~\ref{cor15}, we deduce that $\Psi$ is an isomorphism.
The uniqueness of $\Psi$ is clear from the universal property of the saturated
de Rham--Witt complex. 
\end{proof}

\begin{theorem}
\label{dRWAOmegaMainComp}
There is a natural isomorphism $\varphi^* A\Omega_{\mathfrak{X},W}^{\sm,\ast} \simeq W\Omega_{\mathfrak{X}_k}^{\sm,\ast}$ of presheaves of strict Dieudonn\'e $W$-algebras on $\mathcal{U}(\mathfrak{X})_{\sm}$. 
\end{theorem}
\begin{proof}
We will show that the criterion of Proposition~\ref{dRWCrit} applies to the presheaf $A^* := \varphi^* A\Omega_{\mathfrak{X},W}^{\sm,\ast}$.
Recall that we have $A\Omega^{\sm}_{\mathfrak{X},W} := A\Omega_{\mathfrak{X}}^{\sm} \widehat{\otimes}_{A_{\inf}}^L W$ by definition (Construction~\ref{FrobBaseChange}). Moreover, $A\Omega^{\sm,\ast}_{\mathfrak{X},W}$ is a presheaf of commutative differential graded algebras over $W$ which represents $A\Omega^{\sm}_{\mathfrak{X},W}$ as a commutative algebra object of the the derived category (Construction~\ref{AOmegaSDC}). In particular, $\HH^0(A\Omega^{\sm,\ast}_{\mathfrak{X},W}/p)$ can be identified with $\HH^0(A\Omega_{\mathfrak{X}}^{\sm} \otimes_{A_{\inf}}^L k)$. By transitivity of tensor products, we obtain an isomorphism
\[ A\Omega_{\mathfrak{X}}^{\sm} \otimes^L_{A_{\inf}} k \simeq (A\Omega_{\mathfrak{X}}^{\sm} \otimes^L_{A_{\inf},\widetilde{\theta}} \mathcal{O}_C) \otimes^L_{\mathcal{O}_C} k. \]
We also have the Hodge-Tate comparison isomorphism (Theorem~\ref{HodgeTateComp})
 \[ \widetilde{\Omega}_{\mathfrak{X}}^{\sm} \simeq A\Omega_{\mathfrak{X}}^{\sm} \otimes^L_{A_{\inf},\widetilde{\theta}} \mathcal{O}_C\]
of commutative algebras in $\mathrm{Fun}(\mathcal{U}(\mathfrak{X})_{\sm}^{\op}, \calD(\mathcal{O}_C))$, where the left side is the presheaf which carries a 
small affine open $\mathrm{Spf}(R) \subseteq \mathfrak{X}$ to $\widetilde{\Omega}_R$. For such $R$, the $R$-module $\HH^i(\widetilde{\Omega}_R) \simeq \Omega^i_{R/\mathcal{O}_C}$ is locally free and hence flat over $\mathcal{O}_C$. It immediately follows that $\HH^0(A\Omega_{\mathfrak{X},W}^{\sm,\ast}/p)$ is identified with $\mathcal{O}_{\mathfrak{X}_k}^{\sm}$ as a commutative $k$-algebra, which supplies the map $\eta_A$ required in Proposition~\ref{dRWCrit}; note that the Frobenius twist appearing in Proposition~\ref{dRWCrit} is nullified by the Frobenius pullback present in the definition of $A^*$, i.e., we have
\[\varphi_* \HH^*(A^*/p) = \varphi_* \HH^*(\varphi^* A\Omega_{\mathfrak{X},W}^{\sm,*}/p) = \HH^*(A\Omega_{\mathfrak{X},W}^{\sm,*}/p).\]
 Moreover, this analysis also shows that the cohomology ring $\HH^*(A\Omega_{\mathfrak{X},W}^{\sm,\ast}/p)$ is an exterior algebra on 
\[ \HH^1(A\Omega_{\mathfrak{X},W}^{\sm,\ast}/p) \simeq \HH^1(\widetilde{\Omega}^{\sm}_{\mathfrak{X}}) \otimes_{\mathcal{O}_C} k \simeq \Omega^{\sm,1}_{\mathfrak{X}_k},\]
where the target is the presheaf determined by sending a small affine open $\mathrm{Spf}(R)$ to $\Omega^1_{R_k}$. To apply Proposition~\ref{dRWCrit}, it remains to check that the Bockstein map
\[ \HH^0(A\Omega_{\mathfrak{X},W}^{\sm,\ast}/p) \xrightarrow{\beta_p} \HH^1(A\Omega_{\mathfrak{X},W}^{\sm,\ast}/p) \]
corresponds to the to the de~Rham differential
\[ \mathcal{O}_{\mathfrak{X}_k}^{\sm} \to \Omega^{\sm,1}_{\mathfrak{X}_k} \]
under the preceding identifications. In fact, it is enough to check this compatibility after taking sections over a small affine open $\mathrm{Spf}(R)$. Over such an affine, we will deduce the desired compatibility from the de~Rham comparison theorem for $A\Omega_R$. Because the crystalline specialization map $A_{\inf} \rightarrow W$ carries $\tilde{\xi}$ to $p$, 
we obtain a commutative diagram in the derived category $D(A_{\inf})$ comparing the Bockstein constructions for $\tilde{\xi}$ and $p$:
\[ \xymatrix{ A\Omega_R/\tilde{\xi} \ar[r]^-{\tilde{\xi}} \simeq A\Omega_R \otimes_{A_{\inf},\widetilde{\theta}}^L \mathcal{O}_C \ar[d] & A\Omega/\tilde{\xi}^2 \ar[r] \ar[d] & A\Omega_R/\tilde{\xi} \simeq A\Omega_R \otimes_{A_{\inf},\widetilde{\theta}}^L \mathcal{O}_C\ar[d] \\
		  A\Omega_R/\tilde{\xi} \otimes_{A_{\inf}}^L W \ar[r]^-{\tilde{\xi}} \ar[d]^-{\simeq} & A\Omega/\tilde{\xi}^2 \otimes_{A_{\inf}}^L W \ar[r] \ar[d]^{\simeq} & A\Omega_R/\tilde{\xi} \otimes_{A_{\inf}}^L W \ar[d]^{\simeq} \\
		  (A\Omega_R \widehat{\otimes}^L_{A_{\inf}} W) \otimes_{\Z_p}^L \F_p \ar[r]^-{p} &  (A\Omega_R \widehat{\otimes}^L_{A_{\inf}} W) \otimes_{\Z_p}^L \Z/p^2 \ar[r] &  (A\Omega_R \widehat{\otimes}^L_{A_{\inf}} W) \otimes_{\Z_p}^L \F_p. }\]
Each row of this diagram is an exact triangle, and the second row is obtained from the first by extending scalars along the crystalline specialization $A_{\inf} \to W$.		  
Comparing boundary maps on cohomology for the first and last row above gives a commutative diagram
\[ \xymatrix{ \mathrm{H}^i(A\Omega_R \otimes_{A_{\inf},\widetilde{\theta}}^L \mathcal{O}_C) \ar[r]^-{\beta_{\tilde{\xi}}} \ar[d] & \mathrm{H}^{i+1}(A\Omega_R \otimes_{A_{\inf},\widetilde{\theta}}^L \mathcal{O}_C) \ar[d] \\
		   \mathrm{H}^i( (A\Omega_R \widehat{\otimes}^L_{A_{\inf}} W) \otimes_{\Z_p}^L \F_p) \ar[r]^-{\beta_p} &  \mathrm{H}^{i+1}( (A\Omega_R \widehat{\otimes}^L_{A_{\inf}} W) \otimes_{\Z_p}^L \F_p). }\]
Via the Hodge-Tate comparison (and ignoring Breuil--Kisin twists), this square is identified with
\[ \xymatrix{ \Omega^i_{R/\mathcal{O}_C} \ar[r]^-{\beta_{\tilde{\xi}}} \ar[d] & \Omega^{i+1}_{R/\mathcal{O}_C} \ar[d] \\
		   \Omega^i_{R_k} \simeq \Omega^i_{R_k/k} \ar[r]^-{\beta_p} & \Omega^{i+1}_{R_k} \simeq \Omega^{i+1}_{R_k/k}.}\]
where the vertical maps are the natural ones. It will therefore suffice to show that the top horizontal map is given by the de~Rham differential on $\Omega^{\ast}_{R/ \mathcal{O}_C}$,
which follows from the proof of the de~Rham comparison theorem for $A \Omega$ (see \cite[Proposition 7.9 and Remark 7.11]{BhattSpecializing}).
\end{proof}

\begin{proof}[Proof of Theorem~\ref{DRWAOmegaCoarse}]
Theorem~\ref{dRWAOmegaMainComp} gives an identification 
\[ \Psi\colon W\Omega^{\sm,*}_{\mathfrak{X}_k} \simeq \varphi^* A\Omega^{\sm,*}_{\mathfrak{X},W}\]
 of presheaves of strict Dieudonn\'e $W$-algebras on $\mathcal{U}(\mathfrak{X})_{\sm}$. This already proves the second statement of Theorem~\ref{DRWAOmegaCoarse}. 

Let us now deduce the first statement. Passing to the derived category (i.e., composing with the forgetful functor $(\FrobAlgComplete)_{/W} \to \widehat{D}(W)$), the isomorphism $\Psi$ induces an isomorphism 
\[ \Theta\colon W\Omega^{\sm,*}_{\mathfrak{X}_k}  \simeq \varphi^* A\Omega_{\mathfrak{X},W}^{\sm}\]
 of commutative algebra objects in $\mathrm{Fun}(\mathcal{U}(\mathfrak{X})_{\sm}^{\op}, \widehat{D}(W))$ that is compatible with the Frobenius maps. Applying the completed sheafification functor $\widehat{G}_W$ of Remark~\ref{Kedlaya} to $A\Omega_{\mathfrak{X},W}^{\sm,\ast}$ yields the complex $A\Omega_{\mathfrak{X},W} \in \widehat{D}(\mathfrak{X},W)$ (Remark~\ref{AOmegaSmallSheaf}). Similarly, by Remark~\ref{dRWSheafify}, applying $\widehat{G}_W$ to $W\Omega_{\mathfrak{X}_k}^{\sm,*}$ gives the de~Rham--Witt complex $W\Omega_{\mathfrak{X}_k}^* \in \widehat{D}(\mathfrak{X},W)$. Combining these observations with the Frobenius equivariant isomorphism $\Theta$, we obtain the first statement of Theorem~\ref{DRWAOmegaCoarse}.
\end{proof}

\newpage

\bibliographystyle{amsalpha}
\bibliography{Crystalline}

\end{document}